\let\chooseClass0   

\newcounter{whichClass}
\ifx\chooseClass0
\fi
\ifx\chooseClass1
\fi
\ifx\chooseClass2
\fi
\ifx\chooseClass3
\fi

\ifcase\value{whichClass}
\documentclass[a4paper,11pt]{article}
\or
\documentclass{tac}
\or
\documentclass[a4paper,onecolumn,superscriptaddress,10pt,shorttitle=Categories enriched over multicategories]{compositionalityarticle}
\pdfoutput=1
\or
\documentclass[12pt]{amsart}
\providecommand\texorpdfstring[2]{#1}
\fi

\usepackage[utf8]{inputenc}
\usepackage{amsmath}
\usepackage{amsfonts}
\usepackage{amssymb}
\usepackage{fancybox}
\usepackage{graphicx}

\ifcase\value{whichClass}
\usepackage{hyperref}
\or
\usepackage[colorlinks=true]{hyperref}
\hypersetup{allcolors=[rgb]{0.1,0.1,0.4}}
\or
\usepackage{hyperref}
\or
\usepackage{hyperref}
\fi

\usepackage[scr=true]{rsfso}
\usepackage{stmaryrd}
\usepackage{t-angles}
\usepackage{varioref}
\usepackage{verbatim}

\ifcase\value{whichClass}
\usepackage{amsthm}
\newtheoremstyle{boldhead}
{\topsep}
{\topsep}
{\slshape}
{}
{\bfseries}
{.}
{ }
{\thmname{#1}\thmnumber{ #2}\thmnote{ (#3)}}

\newtheoremstyle{nonumber}
{\topsep}
{\topsep}
{\slshape}
{}
{\bfseries}
{.}
{ }
{\thmname{#1}\thmnote{ #3}}

\newtheoremstyle{boldremark}
{\topsep}
{\topsep}
{\upshape}
{}
{\bfseries}
{.}
{ }
{\thmname{#1}\thmnumber{ #2}\thmnote{ (#3)}}

\theoremstyle{nonumber}
\newtheorem{theorem*}[]{Theorem}

\swapnumbers

\theoremstyle{boldhead}
\newtheorem{theorem}[subsubsection]{Theorem}
\newtheorem{claim}[subsubsection]{Claim}
\newtheorem{corollary}[subsubsection]{Corollary}
\newtheorem{lemma}[subsubsection]{Lemma}
\newtheorem{proposition}[subsubsection]{Proposition}

\theoremstyle{boldremark}
\newtheorem*{acknowledgement}{Acknowledgement}
\newtheorem*{assumption}{Assumption}
\newtheorem{conjecture}[subsubsection]{Conjecture}
\newtheorem{definition}[subsubsection]{Definition}
\newtheorem{describe}[subsubsection]{Description}
\newtheorem{example}[subsubsection]{Example}
\newtheorem{examples}[subsubsection]{Examples}
\newtheorem{exercise}[subsubsection]{Exercise}
\newtheorem{question}[subsubsection]{Question}
\newtheorem{remark}[subsubsection]{Remark}
\or
\newtheoremrm{exercise}{Exercise}
\newtheoremrm{question}{Question}
\or
\usepackage{amsthm}
\newtheoremstyle{boldhead}
{\topsep}
{\topsep}
{\slshape}
{}
{\bfseries}
{.}
{ }
{\thmname{#1}\thmnumber{ #2}\thmnote{ (#3)}}

\newtheoremstyle{nonumber}
{\topsep}
{\topsep}
{\slshape}
{}
{\bfseries}
{.}
{ }
{\thmname{#1}\thmnote{ #3}}

\newtheoremstyle{boldremark}
{\topsep}
{\topsep}
{\upshape}
{}
{\bfseries}
{.}
{ }
{\thmname{#1}\thmnumber{ #2}\thmnote{ (#3)}}

\theoremstyle{boldhead}
\newtheorem{theorem}[subsubsection]{Theorem}
\newtheorem{corollary}[subsubsection]{Corollary}
\newtheorem{lemma}[subsubsection]{Lemma}
\newtheorem{proposition}[subsubsection]{Proposition}

\theoremstyle{boldremark}
\newtheorem{definition}[subsubsection]{Definition}
\newtheorem{example}[subsubsection]{Example}
\newtheorem{exercise}[subsubsection]{Exercise}
\newtheorem{question}[subsubsection]{Question}
\newtheorem{remark}[subsubsection]{Remark}
\or
\usepackage{amsthm}
\newtheoremstyle{boldhead}
{\topsep}
{\topsep}
{\slshape}
{}
{\bfseries}
{.}
{ }
{\thmname{#1}\thmnumber{ #2}\thmnote{ (#3)}}

\newtheoremstyle{nonumber}
{\topsep}
{\topsep}
{\slshape}
{}
{\bfseries}
{.}
{ }
{\thmname{#1}\thmnote{ #3}}

\newtheoremstyle{boldremark}
{\topsep}
{\topsep}
{\upshape}
{}
{\bfseries}
{.}
{ }
{\thmname{#1}\thmnumber{ #2}\thmnote{ (#3)}}

\theoremstyle{nonumber}
\newtheorem{theorem*}[]{Theorem}

\swapnumbers

\theoremstyle{boldhead}
\newtheorem{theorem}[subsubsection]{Theorem}

\newtheorem{corollary}[subsubsection]{Corollary}
\newtheorem{lemma}[subsubsection]{Lemma}
\newtheorem{proposition}[subsubsection]{Proposition}

\theoremstyle{boldremark}
\newtheorem*{acknowledgement}{Acknowledgement}

\newtheorem{definition}[subsubsection]{Definition}

\newtheorem{example}[subsubsection]{Example}

\newtheorem{exercise}[subsubsection]{Exercise}

\newtheorem{remark}[subsubsection]{Remark}
\fi

\usepackage[sans]{dsfont}
\newcommand{\1}{{\mathds 1}}

\usepackage{diagrams}
\diagramstyle[height=2em,balance,righteqno,PostScript=dvips,nohug]

\def\rhaha{\raise.20ex\hbox{$\rightharpoonup$}\kern-1em\lower.20ex\hbox{$\rightharpoondown$}}
\def\lhaha{\raise.20ex\hbox{$\leftharpoonup$}\kern-1em\lower.20ex\hbox{$\leftharpoondown$}}
\def\dhaha{\kern.01em\downharpoonleft\kern-.26em\downharpoonright\kern.01em}
\def\uhaha{\kern.01em\upharpoonleft\kern-.26em\upharpoonright\kern.01em}
\newarrowhead{twoharpoons}\rhaha\lhaha\dhaha\uhaha

\newarrow{BTo}----{boldlittlevee}
\newarrow{DashTo}{}{dash}{}{dash}{->}
\newarrow{Epi}----{triangle}
\newarrow{Id}===={twoharpoons}
\newarrow{MapsTo}{mapsto}---{->}
\newarrow{Mono}{boldhook}---{->}
\newarrow{TTo}----{->}
\newarrow{Twoar}===={=>}

\newlength{\mylabelwidths}
\newenvironment{myitemize}{\begin{list}{}{%
			\setlength{\labelwidth}{\mylabelwidths}%
			\setlength{\leftmargin}{\mylabelwidths}\addtolength{\leftmargin}{0.5em}%
			\setlength{\itemsep}{-0.0\baselineskip}}}
	{\end{list}}

\newlength{\texthigh}
\newlength{\textwids}
\numberwithin{equation}{subsection}

\newcommand{\CC}{{\mathbb C}}
\newcommand{\FF}{{\mathbb F}}
\newcommand{\KK}{{\mathbb K}}
\newcommand{\LL}{{\mathbb L}}
\newcommand{\NN}{{\mathbb N}}
\newcommand{\RR}{{\mathbb R}}
\newcommand{\UU}{{\mathbb U}}
\newcommand{\ZZ}{{\mathbb Z}}

\newcommand{\ainf}[1]{$A_\infty$\nobreakdash-\hspace{0pt}}

\newcommand{\CAT}{\text-\mathcal{C}at}
\newcommand{\Cat}{\mathcal{C}at}
\newcommand{\ca}{{\mathcal A}}
\newcommand{\cb}{{\mathcal B}}
\newcommand{\cc}{{\mathcal C}}
\newcommand{\cD}{{\mathcal D}}
\newcommand{\cE}{{\mathcal E}}
\newcommand{\cF}{{\mathcal F}}

\newcommand{\cK}{{\mathcal K}}

\newcommand{\cO}{{\mathcal O}}
\newcommand{\cp}{{\mathcal P}}

\newcommand{\cS}{{\mathcal S}}

\newcommand{\cW}{{\mathcal W}}
\newcommand{\cx}{{\mathcal X}}
\newcommand{\cY}{{\mathcal Y}}

\newcommand{\lsmCat}{lsm\mathcal{C}at}
\newcommand{\mcC}{{\mathsf C}}
\newcommand{\mcV}{{\mathsf V}}

\newcommand{\n}[1]{\nobreakdash-\hspace{0pt}}
\newcommand{\op}{{\operatorname{op}}}
\newcommand{\pr}{\textup{pr}}
\newcommand{\Qu}{{\text-\mathcal{Q}u}}
\newcommand{\mcSet}{\mathsf{Set}}
\newcommand{\Set}{\mathcal{S}et}

\newcommand{\sk}{{\mathsf{sk}}}
\newcommand{\sMCat}{s\mathcal{MC}at}
\newcommand{\sS}[2]{\vphantom{#2}#1 #2}
\newcommand{\tdt}{\otimes\dots\otimes}
\newcommand{\VCat}{{\mcV\text-\mathcal{C}at}}

\newcommand{\Vect}{\textup{-Vect}}
\newcommand{\VQu}{{\mcV\text-\mathcal{Q}u}}
\newcommand{\XX}{{\mathsf{X\mkern-8mu X}}}

\let\boxt\boxtimes
\let\cA\ca
\let\con\triangledown
\let\cs\cS
\let\cv\cV
\let\ge\geqslant

\let\le\leqslant
\let\eps\varepsilon
\let\lto\xleftarrow
\let\mb\mathbf
\let\mcv\mcV
\let\ovl\overline
\let\rto\xrightarrow
\let\tens\otimes
\let\sss\scriptstyle

\let\und\underline
\let\wh\widehat
\let\:\colon

\DeclareMathOperator{\Col}{Col}

\DeclareMathOperator{\ev}{ev}

\DeclareMathOperator{\id}{id}
\DeclareMathOperator{\Id}{Id}
\DeclareMathOperator{\im}{Im}
\DeclareMathOperator{\inj}{in}
\DeclareMathOperator{\Ker}{Ker}
\DeclareMathOperator{\mcShort}{\mathsf{Short}}
\DeclareMathOperator{\mcsnS}{\mathsf{snS}}
\DeclareMathOperator{\Mor}{Mor}
\DeclareMathOperator{\ML}{ML}
\DeclareMathOperator{\Ob}{Ob}

\DeclareMathOperator{\cProp}{cProp}

\DeclareMathOperator{\sh}{sh}
\DeclareMathOperator{\Short}{\mathbf{Short}}
\DeclareMathOperator{\snS}{\mathbf{snS}}
\DeclareMathOperator{\src}{src}
\DeclareMathOperator{\tgt}{tgt}

\newcommand{\corolref}[1]{Corollary~\ref{#1}}
\newcommand{\defref}[1]{Definition~\ref{#1}}
\newcommand{\exaref}[1]{Example~\ref{#1}}

\newcommand{\figref}[1]{Figure~\ref{#1}}
\newcommand{\lemref}[1]{Lemma~\ref{#1}}
\newcommand{\propref}[1]{Proposition~\ref{#1}}
\newcommand{\remref}[1]{Remark~\ref{#1}}
\newcommand{\secref}[1]{Section~\ref{#1}}
\newcommand{\thmref}[1]{Theorem~\ref{#1}}

\ifcase\value{whichClass}
\author{Volodymyr Lyubashenko%
	\footnote{\textit{Email address:} lub@imath.kiev.ua}}
\or
\author{Volodymyr Lyubashenko}
\address{Institute of Mathematics, NAS Ukraine\\
	3 Tereshchenkivska st., Kyiv, 01024, Ukraine\\
\or
	\author{Volodymyr Lyubashenko}
	\affiliation{Institute of Mathematics, NAS Ukraine, 3 Tereshchenkivska st., Kyiv, 01024, Ukraine}
\or
\author{Volodymyr Lyubashenko}
\address{Institute of Mathematics, NAS Ukraine\\
	3 Tereshchenkivska st., Kyiv, 01024, Ukraine\\
\fi
	
\title{Categories enriched over symmetric closed multicategories}

\begin{document}

\ifcase\value{whichClass}
\maketitle
\begin{abstract}
	We construct a machine which takes as input a locally small symmetric closed complete multicategory $\mathsf V$.
	And its output is again a locally small symmetric closed complete multicategory $\mathsf V\text-\mathcal{C}at$, the multicategory of small $\mathsf V$-categories and multi-entry $\mathsf V$-functors.
	An example of such $\mathsf V$ is provided by short spaces (vector spaces with a system of seminorms) and short maps.
	When the ground multicategory $\mathsf V$ is $\mathsf{Set}$ we obtain strict 2-categories and their surroundings by iterating twice the construction of categories.
\end{abstract}
\or
\copyrightyear{2023}
\eaddress{lub@imath.kiev.ua}
\maketitle
\begin{abstract}
	We construct a machine which takes as input a locally small symmetric closed complete multicategory $\mathsf V$.
	And its output is again a locally small symmetric closed complete multicategory $\mathsf V\text-\mathcal{C}at$, the multicategory of small $\mathsf V$-categories and multi-entry $\mathsf V$-functors.
	An example of such $\mathsf V$ is provided by short spaces (vector spaces with a system of seminorms) and short maps.
	When the ground multicategory $\mathsf V$ is $\mathsf{Set}$ we obtain strict 2-categories and their surroundings by iterating twice the construction of categories.
\end{abstract}
\or
\email{lub@imath.kiev.ua}
\maketitle
\begin{abstract}
	We construct a machine which takes as input a locally small symmetric closed complete multicategory $\mathsf V$.
	And its output is again a locally small symmetric closed complete multicategory $\mathsf V\text-\mathcal{C}at$, the multicategory of small $\mathsf V$-categories and multi-entry $\mathsf V$-functors.
	An example of such $\mathsf V$ is provided by short spaces (vector spaces with a system of seminorms) and short maps.
	When the ground multicategory $\mathsf V$ is $\mathsf{Set}$ we obtain strict 2-categories and their surroundings by iterating twice the construction of categories.
\end{abstract}
\or
\email{lub@imath.kiev.ua}
\maketitle
\begin{abstract}
	We construct a machine which takes as input a locally small symmetric closed complete multicategory $\mathsf V$.
	And its output is again a locally small symmetric closed complete multicategory $\mathsf V\text-\mathcal{C}at$, the multicategory of small $\mathsf V$-categories and multi-entry $\mathsf V$-functors.
	An example of such $\mathsf V$ is provided by short spaces (vector spaces with a system of seminorms) and short maps.
	When the ground multicategory $\mathsf V$ is $\mathsf{Set}$ we obtain strict 2-categories and their surroundings by iterating twice the construction of categories.
\end{abstract}
\fi

\allowdisplaybreaks[1]
\ifx\chooseClass1
\thirdleveltheorems
\fi

\section{Introduction}
A complete multicategory $\mcv$ is a multicategory (=colored operad) which has all small products and all equalizers.
Warning: to say that the underlying category $\mcv_1$ has all small products and all equalizers is not enough.
One has to take into account the multicategory structure (Definitions \ref{def-multicategory-small-products} and \ref{def-equalizers}).
In fact, we view multicategories as monoidal categories for which the monoidal product does not exist.
Instead of monoidal products finite sequences of objects are used as an input.
Hence, conditions for products and equalizers have to be written for a finite sequence of objects, not only for a single object.
This point of view is supported by an adjunction between symmetric multicategories and colored props, see \secref{sec-Adjunction-multicategories-props}.
We assume also that $\mcv$ is a closed multicategory (that with internal homs, see around \eqref{eq-phi(XYZ)-iso}).
This notion was defined by Lambek \cite[p.~106]{LambekJ:dedsc2} (see also \cite[Definition~4.7]{BesLyuMan-book} for enriched case).
Furthermore, we assume that $\mcv$ is a symmetric multicategory (see the beginning of \secref{sec-ground-category-features}).

We start with a symmetric closed complete multicategory $\mcv$.
There is a technical notion of a small $\mcv$\n-quiver, which is a small quiver where instead of set of arrows between two vertices an object of $\mcv$ is used (\defref{def-V-quivers}).
A multi-entry $\mcv$\n-quiver morphism has several $\mcv$\n-quivers as a source and one as target (\defref{def-multi-entry-V-quiver-morphism}).
Collection of such morphisms is a symmetric multicategory $\VQu$ (\propref{pro-symmetric-multicategory-multi-entry-quiver-morphisms}).

However, what we really need are small $\mcv$\n-categories -- $\mcv$\n-quivers equipped with composition and identity morphisms (\defref{def-V-categories}).
Using composition we construct the evaluation multi-entry $\mcv$\n-quiver morphism in \propref{pro-und-VQu} and \defref{def-evVQu}.
Previously mentioned features (completeness and closedness of $\mcv$ and composition in the target) are used to define internal hom -- certain end in $\mcv$, which replaces the set of natural transformations.
When dealing with $\mcv$\n-categories, we use multi-entry $\mcv$\n-functors instead of multi-entry $\mcv$\n-quiver morphisms (\defref{def-multi-entry-V-functor}).
They form a symmetric multicategory $\VCat$ (\propref{pro-locally-small-symmetric-multicategory}).
The multi-entry $\mcv$\n-functors are identified with \(\FF\mcv\)\n-functors \(\boxt^{i\in I}\ca_i\to\cb\) (\propref{pro-multi-entry-V-functors}), where \(\FF\mcv\) is the colored prop associated with the symmetric multicategory $\mcv$ (\propref{pro-adjunction-FU}).
We define also natural $\mcv$\n-transformations (\defref{def-Natural-V-transformation}) and show that their set can be recovered from the internal hom (\propref{pro-natural-V-transformations}).

In the case of $\mcv$\n-categories the evaluation morphism is a multi-entry $\mcv$\n-functor (\propref{con-evVCat-VCat}).
Furthermore, the symmetric multicategory $\VCat$ is closed (\propref{pro-VCat-closed}).

We prove that the multicategory $\VCat$ has small products (\propref{pro-VCat-has-products}).
It also has equalizers (\propref{pro-VCat-has-equalizers}), thus, it is complete.
All mentioned results are summarized in

\ifcase\value{whichClass}
\begin{theorem*}[\ref{thm-machine}]
	Let $\mcv$ be a locally small symmetric closed complete multicategory.
	Then so is $\VCat$, the multicategory of small $\mcv$\n-categories and multi-entry $\mcv$\n-functors.
\end{theorem*}
\or
\subsubsection*{\bf \thmref{thm-machine}.}
	Let $\mcv$ be a locally small symmetric closed complete multicategory.
	Then so is $\VCat$, the multicategory of small $\mcv$\n-categories and multi-entry $\mcv$\n-functors.
\subparagraph*{}
\or
\subsubsection*{\bf \thmref{thm-machine}.}
Let $\mcv$ be a locally small symmetric closed complete multicategory.
Then so is $\VCat$, the multicategory of small $\mcv$\n-categories and multi-entry $\mcv$\n-functors.
\subparagraph*{}
\or
\begin{theorem*}[\ref{thm-machine}]
	Let $\mcv$ be a locally small symmetric closed complete multicategory.
	Then so is $\VCat$, the multicategory of small $\mcv$\n-categories and multi-entry $\mcv$\n-functors.
\end{theorem*}
\fi

We deduce whiskerings from the closed multicategory structure of $\VCat$ in \secref{sec-Compositions-whiskerings}.
The example of representable multicategory $\mcv$ is discussed in \secref{sec-Representable-multicategories}.
The examples of categories and strict 2\n-categories are presented in \secref{sec-Strict-2-categories}.

An example of such multicategory $\mcv$ is provided by short spaces (vector spaces over $\RR$ or $\CC$ with a system of seminorms) and short maps.
Seminorms are indexed by an element of a commutative partially ordered monoid $\LL$.
Further conditions on $\LL$ are listed in \secref{sec-Short-spaces}.
There is symmetric multicategory $\mcShort_\LL$ with short spaces as objects.
Morphisms are short multilinear maps (see \defref{def-Multicategory-Short}).
This multicategory is closed (\propref{pro-Short-closed}).
The internal hom object is a vector space of multilinear maps.
The symmetric multicategory $\mcShort_\LL$ has products (\propref{pro-multicategory-Short-has-small-products}) and kernels (equalizers) (\propref{pro-kernel}).
Summing up, the multicategory \(\mcShort_\LL\) is complete (\corolref{cor-Short-complete}).

We do not include explicitly in the definition the action of symmetric groups on symmetric multicategories.
So we have to deduce it in \corolref{cor-Action-symmetric-groups-symmetric-multicategory}.
Further interplay between the action of symmetric groups and the compositions in a symmetric multicategory is described in \propref{pro-equivariance-property}.

\tableofcontents

\ifcase\value{whichClass}
\begin{acknowledgement}
\or
\subsubsection*{Acknowledgement}
\or
\subsubsection*{Acknowledgement}
\or
\subsubsection*{Acknowledgement}
\fi
I thank Paul Taylor for writing a package `diagrams.sty', which I use a lot.
I am cordially grateful to the staff of the University of Zurich and of the Swiss Federal institute of Technology in Zurich who created an excellent environment for work.
I am grateful for clarifying discussions to Prof. Dr. Anna Beliakova, to Prof. Dr. Alberto Cattaneo and to Prof. Dr. Giovanni Felder.
All this was due to the generosity of the National Centre of Competence in Research SwissMAP of the Swiss National Science Foundation (grant number 205607), to whom I express my deep gratitude.
The author is grateful for the financial support within the program of support for priority research and technical (experimental) developments of the Section of Mathematics of the NAS of Ukraine for 2022-2023 Project "Innovative methods in the theory of differential equations, computational mathematics and mathematical modeling", No. 7/1/241 (State Registration No. 0122U000670).
I am really grateful to the Armed Forces of Ukraine who gave me the possibility to work quietly on the subject.
\ifx\chooseClass0
\end{acknowledgement}
\fi

\subsection{Conventions}
We work with a locally small closed symmetric multicategory $\mcV$ in the sense of \cite[Definitions 3.7, 4.7]{BesLyuMan-book}.
Locally small means that \(\mcV\bigl((X_i)_{i\in I};Y\bigr)\) are small.

When we write \(\mcV((X_i)_{i\in I};Y)\), we mean that $I$ is an object of $\cO_\sk$, the skeletal category of finite totally ordered sets with objects \(\mb n=\{1<2<\dots<n\}\), $n\ge0$, whose morphisms are non-decreasing maps.
A subset $J\subset I$ means a monomorphism in $\cO_\sk$.
We freely use the notation style of \cite{BesLyuMan-book}.
We use also the skeletal category $\cS_\sk$ of finite totally ordered sets, \(\Ob\cS_\sk=\Ob\cO_\sk\cong\NN\), whose morphisms are \emph{all} maps \(\mb n\to\mb m\) (ignoring the ordering).
Let \(f\:I\to J\in\cs_\sk\).
An element \(j\in J\) is a monomorphism \(\dot j\:\mb1\to J\) (\(1\mapsto j\)).
Its preimage \(f^{-1}(j)\) is the monomorphism \(\iota\:\mb k\to I\in\cO_\sk\), \(k=|f^{-1}(j)|\), which is the pullback of $\dot j$ along $f$ in the category \(\cs_\sk\)
\begin{diagram}
f^{-1}(j) \SEpbk=\mb k &\rTTo^p &\mb1
\\
\dMono<\iota &&\dMono>{\dot j}
\\
I &\rTTo^f &J
\end{diagram}

\subsection{Lax symmetric monoidal categories and functors: recollection}
We reproduce definition of lax symmetric monoidal categories from \cite[Definition~2.5]{BesLyuMan-book} (see also \cite[Definition~1.2.14]{Lyu-sqHopf} for symmetric monoidal categories and \cite{DayStreet-subst}, \cite[Definition~3.1.1]{math.CT/0305049}) in a simplified form.
Namely, instead of considering all finite sets we contend ourselves with the category \(\cS_\sk\)of finite ordinals \(\mb n=\{1<\dots<n\}\) and arbitrary maps of those.
The adjective `lax' is chosen so that a multicategory \(\wh\cv\) could be produced from a lax symmetric monoidal category $\cv$.
This agrees with \cite{BesLyuMan-book}, however, some authors including Day and Street \cite{DayStreet-subst} and Leinster \cite[Definition~3.1.1]{math.CT/0305049} would use `oplax' instead.

\begin{definition}\label{def-Monoidal-cat}
A \emph{lax symmetric monoidal category}\index{lax symmetric monoidal category} $(\cv,\tens_\cv^I,\lambda_\cv^f)$ consists of the following data:
	\begin{enumerate}
		\item A category $\cv$.
		\item A functor $\tens^I=\tens_\cv^I\:\cv^I\to\cv$, for every set $I\in\Ob\cs_\sk$.
		In particular, a map \(\tens_\cv^I\:\prod_{i\in I}\cv(X_i,Y_i)\to\cv(\tens^{i\in I}X_i,\tens^{i\in I}Y_i)\) is given.
		It is required that $\tens^{\mb1}=\tens_\cv^{\mb1}\:\cv^{\mb1}\to\cv$ is the identification of \(\cv^{\mb1}\) and $\cv$.
		
		For a map $f\:I\to J$ in $\Mor\cS_\sk$ introduce a functor $\tens^f=\tens^f_\cv\:\cv^I\to\cv^J$ which to a function $X\:I\to\Ob\cv$, $i\mapsto X_i$ assigns the function \(J\to\Ob\cv\), \(j\mapsto\tens^{i\in f^{-1}(j)}X_i\). The linear order on $f^{-1}(j)$ is induced by the embedding \(f^{-1}(j)\hookrightarrow I\).
		The functor $\tens^f_\cv\:\cv^I\to\cv^J$ acts on morphisms via the map
		\[ \prod_{i\in I}\cv(X_i,Y_i) \rTTo^\sim \prod_{j\in J}\prod_{i\in f^{-1}j}\cv(X_i,Y_i) \rTTo^{\prod_{j\in J}\tens^{f^{-1}j}} \prod_{j\in J} \cv(\tens^{i\in f^{-1}j}X_i,\tens^{i\in f^{-1}j}Y_i).
		\]
		
		\item A morphism of functors
\[ \lambda^f\: \tens^I \to \tens^J\circ\tens^f\: \cv^I \to \cv, \qquad \lambda^f\:\tens^{i\in I}X_i\to\tens^{j\in J}\tens^{i\in f^{-1}j}X_i,
\]
		for every map $f\:I\to J$ in $\Mor\cS_\sk$.
	\end{enumerate}
	These data are subject to the following axioms:
	\begin{enumerate}
		\renewcommand{\labelenumi}{(\roman{enumi})}
		\item for all sets \(I\in\Ob\cS_\sk\) \(\lambda^{\id_I}=\id\) and \(\lambda^{I\to\mb1}=\id\);
		
		\item for any pair of composable maps $I \rTTo^f J \rTTo^g K$ from $\cS_\sk$ the following equation holds:
		\begin{diagram}[h=2.2em,LaTeXeqno]
			\tens^{i\in I}X_i & \rTTo^{\lambda^f}
			& \tens^{j\in J}\tens^{i\in f^{-1}j}X_i \\
			\dTTo<{\lambda^{f\centerdot g}} & = & \dTTo>{\lambda^g} \\  
			\tens^{k\in K}\tens^{i\in f^{-1}g^{-1}k}X_i
			& \rTTo^{\tens^{k\in K}\lambda^{f|\:f^{-1}g^{-1}k\to g^{-1}k}}
			& \tens^{k\in K}\tens^{j\in g^{-1}k}\tens^{i\in f^{-1}j}X_i
			\label{eq-lambda-fg-non-enriched}
		\end{diagram}
	\end{enumerate}
	A symmetric monoidal category is a lax one for which all $\lambda^f$ are isomorphisms.
	A symmetric strict monoidal category \((\cv,\tens_\cv^I,\lambda_\cv^f)\) is lax symmetric monoidal one where \(\lambda_\cv^f\:\tens_\cv^I\to\tens_\cv^f\cdot\tens_\cv^J\) are identity morphisms for all isotonic maps \(f\:I\to J\).
A \emph{colax symmetric monoidal category}\index{colax symmetric monoidal category} (also called \emph{multitensor} in \cite{0803.3594,1209.2776}) is the above with the direction of arrows \(\lambda^f\) reversed, so they are morphisms of functors
\[ \lambda^f\: \tens^J\circ\tens^f \to\tens^I\: \cv^I \to\cv, \qquad \lambda^f\: \tens^{j\in J}\tens^{i\in f^{-1}j}X_i \to\tens^{i\in I}X_i.
\]
\end{definition}

\begin{definition}[cf. Definition~2.6 of \cite{BesLyuMan-book}]\label{def-Monoidal-fun}
A \emph{lax symmetric monoidal functor}\index{lax symmetric monoidal functor} between lax symmetric monoidal categories
	\[ (F,\phi^I)\: (\cc,\tens_\cc^I,\lambda_\cc^f) \to (\cD,\tens_\cD^I,\lambda_\cD^f)
	\]
	consists of
	\begin{enumerate}
		\renewcommand{\labelenumi}{\roman{enumi})}
		\item a functor $F\:\cc\to\cD$,
		\item a functorial morphism for each set $I\in\Ob\cs$
		\[ \phi^I\: \tens_\cD^I\circ F^I \to F\circ\tens_\cc^I\: \cc^I \to \cD, \qquad \phi^I\: \tens_\cD^{i\in I}FX_i \to F\tens_\cc^{i\in I}X_i,
		\]
	\end{enumerate}
	such that \(\phi^{\mb1}=\bigl(\tens^{\mb1}FX=FX=F\tens^{\mb1}X\bigr)=\id\), and for every map $f\:I\to J$ of $\cs_\sk$ and all families \((X_i)_{i\in I}\) of objects of $\cc$ the following equation holds:
	\begin{diagram}[h=2.2em]
		\tens_\cD^{i\in I}FX_i && \rTTo^{\phi^I} && F\tens_\cc^{i\in I}X_i \\
		\dTTo<{\lambda_\cD^f} && = && \dTTo<{F\lambda_\cc^f} \\
		\tens_\cD^{j\in J}\tens_\cD^{i\in f^{-1}j}FX_i
		& \rTTo^{\tens_\cD^{j\in J}\phi^{f^{-1}j}}
		& \tens_\cD^{j\in J}F\tens_\cc^{i\in f^{-1}j}X_i & \rTTo^{\phi^J}
		& F\tens_\cc^{j\in J}\tens_\cc^{i\in f^{-1}j}X_i
	\end{diagram}
	A lax symmetric monoidal functor \((F,\phi^I)\) is \emph{strict} if all \(\phi^I=\id\).
\end{definition}

The category of lax symmetric monoidal categories with lax symmetric monoidal functors as morphisms is denoted \(\lsmCat\).

There is also an appropriate definition of a morphism of lax symmetric monoidal functors \cite[Definition~2.7]{BesLyuMan-book}.
It is proven in \cite[Proposition~1.2.15]{Lyu-sqHopf} that the 2\n-categories of symmetric strict monoidal categories in the above sense and of symmetric strict monoidal categories in conventional sense (aka permutative categories \cite[Definition~3.1]{zbMATH05054351}, topological version is in \cite[Definition~1]{MAY1978225}) are isomorphic when we consider strict symmetric monoidal functors.
In particular, there is a correspondence assigning to each permutative category \(P=(P,\tens,\1,c)\) a symmetric strict monoidal category \(P^\clubsuit=(P,\tens^I,\lambda^f)\) with \(\tens^\varnothing=\dot\1\), \(\tens^I\) = iterated $\tens$, $\lambda^f=\id$ if the map $f\:I\to J\in\cS_\sk$ is order preserving.
If $f\:I\to I\in\cS_\sk$ is a bijection, then \(\lambda^f\:\tens^{i\in I}X_i\to\tens^{i\in I}X_{f^{-1}i}\) is an element of the symmetric group generated by \(1^{\tens a}\tens c\tens1^{\tens b}\).
The general map \(I\to K\in\cs_\sk\) can be presented as $fg$ where $f\:I\to I$ is a bijection and $g\:I\to K$ is order preserving.
Then \(\lambda^{fg}\) can be found from \eqref{eq-lambda-fg-non-enriched} as the composition
\[ \tens^{i\in I}X_i \rTTo^{\lambda^f} \tens^{i\in I}X_{f^{-1}i} =\tens^{k\in K}\tens^{i\in g^{-1}k}X_{f^{-1}i} \rTTo^{\tens^{k\in K}(\lambda^{f|\:f^{-1}g^{-1}k\to g^{-1}k})^{-1}} \tens^{k\in K}\tens^{i\in f^{-1}g^{-1}k}X_i.
\]
Being an isomorphism of 2\n-categories $\text-^\clubsuit$ is also isomorphism of categories.

\subsection{Multicategories: recollection}\label{sec-ground-category-features}
By \cite[Definition~3.7]{BesLyuMan-book} the structure maps of symmetric multicategory $\mcv$ are the following.
This is an intermediate notion between the ordinary definition of symmetric multicategory\index{symmetric multicategory} and Leinster's notion of fat symmetric multicategories \cite[Definition~A.2.1]{math.CT/0305049}.
Of course, it is equivalent to both, being a skeletal version of Leinster's notion.
\begin{myitemize}
\item[---] for each map \(\phi\:I\to J\) from \(\cS_\sk\) and objects \(X_i,Y_j,Z\in\Ob\mcV\), \(i\in I\), \(j\in J\), the composition map 
\[ \mu_\phi\: \bigl[ \prod_{j\in J}\mcV\bigl((X_i)_{i\in\phi^{-1}(j)};Y_j\bigr) \bigr] \times\mcv\bigl((Y_j)_{j\in J};Z\bigr) \to \mcv\bigl((X_i)_{i\in I};Z\bigr);
\]
\item[---] for each object \(X\in\Ob\mcv\) the identity -- an element\index{$1_X$ -- identity element of a multicategory} \(1_X\in\mcv(X;X)\).
\end{myitemize}


The above data have to satisfy the associativity equation and two unitality equations, see \cite[Definition~3.7]{BesLyuMan-book}.

\begin{myitemize}
\item[---] (Associativity) For each pair of composable maps \(I\rto\phi J\rto\psi K\) from $\cs_\sk$, the diagram shown 
\vpageref{dia-assoc-mu-multi} commutes.
	\begin{figure}
		\begin{center}
			\boldmath
			\resizebox{!}{.93\texthigh}{\rotatebox{90}{%
					\begin{diagram}[nobalance,inline,height=2.6em,thick]
&& \hspace*{-18em} \bigl[ \prod_{j\in J} \mcv\bigl((X_i)_{i\in\phi^{-1}j};Y_j\bigr)\bigr] \times \bigl[ \prod_{k\in K} 	\mcv\bigl((Y_j)_{j\in\psi^{-1}k};Z_k\bigr) \bigr] \times \mcv\bigl((Z_k)_{k\in K};W\bigr)  
\\
& \ldBTo_{\cong} & 
\\
\bigl[ \prod_{k\in K} \bigl( \bigl[ \prod_{j\in\psi^{-1}k} \mcv\bigl((X_i)_{i\in\phi_k^{-1}j};Y_j\bigr) \bigr] \times \mcv\bigl((Y_j)_{j\in\psi^{-1}k};Z_k\bigr) \bigr) \bigr] \times \mcv\bigl((Z_k)_{k\in K};W\bigr)
&& \dBTo>{1\times\mu_\psi}
\\
						\\
\dBTo<{(\prod_{k\in K}\mu_{\phi_k})\times1}
&& \hspace*{-5em} \bigl[ \prod_{j\in J} \mcv\bigl((X_i)_{i\in\phi^{-1}j};Y_j\bigr) \bigr] \times \mcv\bigl((Y_j)_{j\in J};W\bigr)
\\
						\\
\bigl[ \prod_{k\in K} \mcv\bigl((X_i)_{i\in(\phi\psi)^{-1}(k)};Z_k\bigr) \bigr] \times \mcv\bigl((Z_k)_{k\in K};W\bigr) 
&& \dBTo>{\mu_\phi} 
\\
& \rdBTo^{\mu_{\phi\psi}} & 
						\\
&& \mcv\bigl((X_i)_{i\in I};W\bigr)
					\end{diagram}
\hphantom{A}}}
		\end{center}
		\caption{Associativity in multicategories
			\label{dia-assoc-mu-multi}}
	\end{figure}
Here \(\phi_k=\phi|_{(\phi\psi)^{-1}(k)}\:(\phi\psi)^{-1}(k)\to\psi^{-1}(k)\), \(k\in K\), and \(\psi^{-1}(k)\) is understood as the pullback of the diagram \(\mb1=\{k\}\hookrightarrow K\lto\psi J\).
We define an operation \(\sqcup\:\cs_\sk\times\cs_\sk\to\cs_\sk\), \((\mb m,\mb n\mapsto\mb{m+n})\) (addition of finite ordinals) in an obvious way on morphisms.
Thus, the set \(I\sqcup J\) is a disjoint union of sets $I$ and $J$.
For all \(i\in I\) and \(j\in J\) we have $i<j$ in \(I\sqcup J\), and the embeddings \(I\hookrightarrow I\sqcup J\hookleftarrow J\) are increasing.
We often replace the commutative diagram \vpageref{dia-assoc-mu-multi} with its scheme:
\begin{gather*}
\prod_{j\in J} \mcv\bigl((X_i)_{i\in\phi^{-1}j};Y_j\bigr) \qquad \prod_{k\in K} \mcv\bigl((Y_j)_{j\in\psi^{-1}k};Z_k\bigr) \qquad \mcv\bigl((Z_k)_{k\in K};W\bigr)
\\
\begin{tanglec}
\hh \hstep \id \Step \id \Step \id 
	\\
\ffbox4{\prod_{k\in K}\mu_{\phi_k}} \step \id \hstep 
	\\
\hh \step[1.5] \id \step[3] \id 
	\\
\hh \step[1.5] \ffbox4{\mu_{\phi\psi}}
	\\
\hh \step[1.5] \id
\end{tanglec}
\;=\;
\begin{tanglec}
\hh	\id \step[1.5] \id \step \id 
	\\
\hh \hstep \id \step \ffbox2{\mu_\psi}
	\\
\hh	\id \Step \id \hstep
	\\
\hh	\ffbox3{\mu_\phi} \hstep
	\\
\hh	\id \hstep
\end{tanglec}
\\
\mcv\bigl((X_i)_{i\in I};W\bigr)
\end{gather*}
This scheme has to be read downwards.
Strings refer to (Cartesian products of) sets of morphisms, parallel strings refer to Cartesian products of sets, rectangles (coupons) denote (composition) maps.
This style of notations was used by Bartlett \cite{1409.2148}, who applied it to symmetric monoidal bicategories.

\item[---] (Identity) For \(\phi=\triangledown\:I\to\mb1\) the equation
\begin{equation}
\bigl[ \mcv((X_i)_{i\in I};Z) \rTTo^{1\times\dot1_Z} \mcv((X_i)_{i\in I};Z)\times\mcv(Z;Z) \rTTo^{\mu_{\triangledown\:I\to \mb1}} \mcv((X_i)_{i\in I};Z) \bigr] = \id
		\label{eq-axiom-unit-multi1etaZ}
\end{equation}
holds true.
	If \(\phi=\id\:I\to I\), then the equation
\begin{equation}
\bigl[ \mcv((X_i)_{i\in I};Z) \rTTo^{(\prod_{i\in I}\dot1_{X_i})\times1} \bigl(\prod_{i\in I}\mcv(X_i;X_i)\bigr)\times\mcv((X_i)_{i\in I};Z) \rTTo^{\mu_{\id_I}} \mcv((X_i)_{i\in I};Z) \bigr] = \id
		\label{eq-axiom-unit-multi2}
\end{equation}
	holds true.
\end{myitemize}

Here \(\dot1_Z\:\mb1\to\mcv(Z;Z)\), \(1\mapsto1_Z\), distinguishes the element $1_Z$.
In the following we omit the isomorphism at diagram \vpageref{dia-assoc-mu-multi}.
Thus, we do not distinguish sets $A\times B$ and $B\times A$.
This is done for the sake of economy of space.
Naturally, one can insert the flip symmetry \(A\times B\to B\times A\) wherever appropriate.

Recall \cite[p.~106]{LambekJ:dedsc2} (see also \cite[Definition~4.7]{BesLyuMan-book} for $\cv$\n-multicategories) that a plain multicategory $\mcv$ is \emph{closed}\index{closed multicategory} if for any collection \(((X_i)_{i\in I},Z)\), \(I\in\Ob\cs_\sk\), of objects of $\mcv$ there is an object \(\und\mcv((X_i)_{i\in I};Z)\) of $\mcv$ and an evaluation element
\[ \ev_{(X_i)_{i\in I};Z} \in \mcv\bigl((X_i)_{i\in I},\und\mcv((X_i)_{i\in I};Z);Z\bigr),
\]
such that the composition
\begin{multline}
\varphi_{(X_i)_{i\in I};(Y_j)_{j\in J};Z} =\bigl\{
\mcV\bigl((Y_j)_{j\in J};\und\mcV((X_i)_{i\in I};Z)\bigr) \rTTo^{\dot1_{X_1}\times\dots\times\dot1_{X_I}\times\id\times\dot\ev_{(X_i)_{i\in I};Z}}
\\
\bigl[\prod_{i\in I}\mcV(X_i;X_i)\bigr]\times\mcV\bigl((Y_j)_{j\in J};\und\mcV\bigl((X_i)_{i\in I};Z\bigr)\bigr)\times\mcV\bigl((X_i)_{i\in I},\und\mcV\bigl((X_i)_{i\in I};Z\bigr);Z\bigr)
\\
\rTTo^{\mu_{\id\sqcup\con\:I\sqcup J\to I\sqcup\mb1}} \mcV\bigl((X_i)_{i\in I},(Y_j)_{j\in J};Z\bigr) \bigr\}
\label{eq-phi(XYZ)-iso}
\end{multline}
is bijective for an arbitrary sequence \((Y_j)_{j\in J}\), \(J\in\Ob\cs_\sk\), of objects of $\mcV$.

Let $g\:(X_i)_{i\in I}\to Z$ be a morphism in a closed symmetric multicategory $\mcv$.
Generalizing the previous notation denote by \(\dot{g}\:()\to\und\mcv((X_i)_{i\in I};Z)\) the morphism \(\varphi_{(X_i)_{i\in I};();Z}^{-1}(g)\in\mcv(;\und\mcv((X_i)_{i\in I};Z))\).
Equation~\eqref{eq-phi(XYZ)-iso} for \(J=\varnothing\) implies that
\begin{multline}
\bigl[\prod_{i\in I}\mcV(X_i;X_i)\bigr] \times\mcV\bigl(;\und\mcV\bigl((X_i)_{i\in I};Z\bigr)\bigr) \times\mcV\bigl((X_i)_{i\in I},\und\mcV\bigl((X_i)_{i\in I};Z\bigr);Z\bigr)
\rTTo^{\mu_{\inj_1\:I\hookrightarrow I\sqcup\mb1}} \mcV\bigl((X_i)_{i\in I};Z\bigr),
\\
\bigl((1_{X_i})_{i\in I},\dot g,\ev_{(X_i)_{i\in I};Z}\bigr) \mapsto g.
\label{eq-mu-in1I-I1}
\end{multline}

\begin{definition}\label{def-multicategory-small-products}
A multicategory $\mcv$ has small products if the underlying category $\mcv_1$ has small products \(\pr_j\:\prod_{k\in J}M_k\to M_j\in\mcv\), \(j\in J\in\Set\), and for each family of morphisms \(\bigl(f_j\:(X_i)_{i\in I}\to M_j\in\mcv\bigr)_{j\in J}\) there is a unique morphism \(f\:(X_i)_{i\in I}\to\prod_{j\in J}M_j\in\mcv\) such that for all \(j\in J\)
\[ f_j =\bigl[ (X_i)_{i\in I} \rto f \prod_{j\in J}M_j \rTTo^{\pr_j} M_j \bigr].
\]
\end{definition}

For $I=\mb1$ this property is equivalent to \(\prod_{j\in J}M_j\) being a product in ordinary category $\mcv_1$.
In the following we \textbf{assume} that the multicategory $\mcv$ has small products.


\begin{definition}\label{def-equalizers}
A multicategory $\mcv$ has equalizers (of pairs of parallel morphisms) if for all pairs \(A\pile{\rTTo^f\\ \rTTo_g}B\in\mcv\) there is an object $K$ and a morphism \(e\:K\to A\) which is an equalizer of \((f,g)\) in ordinary category \(\mcv_1\) and, moreover, for each morphism \(h\:(X_i)_{i\in I}\to A\in\mcv\) such that \(h\centerdot f=h\centerdot g\) there exists a unique \(q\:(X_i)_{i\in I}\to K\) such that \(h=q\centerdot e\):
\begin{diagram}[h=1.8em]
&&K
\\
&\ruTTo^q &\dTTo>e
\\
(X_i)_{i\in I} &\rTTo_h &A &\pile{\rTTo^f\\ \rTTo_g} &B
\end{diagram}
\end{definition}

The equalizer for ordinary category \(\mcv_1\) is a particular case for $I=\mb1$.
In the following we \textbf{assume} that the multicategory $\mcv$ has equalizers.

\begin{corollary}
Let multicategory $\mcv$ have products and equalizers.
For any diagram \(J\to\mcv_1\), \(j\mapsto M_j\) (\(\mcv_1\) is an ordinary category $\mcv$), the limit \(\lim(J\to\mcv_1)\in\Ob\mcv\) satisfies also: for any morphism \(h=(h_j)\:(X_i)_{i\in I}\to\prod_{j\in J}M_j\) such that for all \(j\to k\in J\) the equation holds
\[ h_k =\bigl[ (X_i)_{i\in I} \rTTo^{h_j} M_j \to M_k \bigr]
\]
there exists a unique morphism \(g\:(X_i)_{i\in I}\to\lim(J\to\mcv_1)\) such that
\[ h =\bigl[ (X_i)_{i\in I} \rto g \lim(J\to\mcv_1) \to \prod_{j\in\Ob J}M_j \bigr].
\]
\end{corollary}

When the above holds, we say that multicategory $\mcv$ is complete\index{complete multicategory} and \textbf{assume} this from now on.

For morphisms of symmetric multicategories one may take symmetric multifunctors, see \cite[Definition~3.14]{BesLyuMan-book}.
More generally, one may take $k$\n-linear symmetric multifunctors of Elmendorf and Mandell \cite[Definition~2.6]{zbMATH05655952}, which are multifunctors with $k$ inputs and one output.
These were generalised to morphisms and bilinear maps of props by Hackney and Robertson \cite[Proposition~34]{zbMATH06469262}.

\section{About \texorpdfstring{$\mcV$}V-categories}
\subsection{Adjunction between symmetric multicategories and colored props}\label{sec-Adjunction-multicategories-props}
\begin{proposition}[{\cite[Theorem~4.2]{zbMATH05655952}, \cite[Proposition~11]{zbMATH06469262}, see also \cite[Theorem~2.3.2]{0809.2161}, \cite[Proposition~9.2]{0906.0015}}]\label{pro-adjunction-FU}
There is an adjunction between symmetric multicategories \(\sMCat\) and colored props \(\cProp\)
	\[ \FF\: \sMCat \leftrightarrows \cProp\: \UU.
	\]
\end{proposition}

It seems that in all cited sources the definition of symmetric multicategories uses explicit action of symmetric groups.
We use a different definition and give a different proof.

\begin{proof}
As any prop, the constructed \(\FF\mcv\) has the monoid of objects \((\Ob\FF\mcv,\tens)=(\Ob\mcv)^*\), the monoid (with the operation $\tens$) freely generated by \(\Ob\mcv\).
	Objects of \(\FF\mcv\) are denoted \(\tens^{i\in I}X_i=(X_i)_{i\in I}\), \(I\in\cS_\sk\).
	
	The morphism sets are
	\[ \FF\mcv\bigl((X_i)_{i\in I},(Y_j)_{j\in J}\bigr) =\coprod_{\phi\:I\to J\in\cS_\sk} \prod_{j\in J} \mcv\bigl((X_i)_{i\in\phi^{-1}j};Y_j\bigr).
	\]
	The composition is
	\begin{multline*}
	\FF\mcv\bigl((X_i)_{i\in I},(Y_j)_{j\in J}\bigr) \times \FF\mcv\bigl((Y_j)_{j\in J},(Z_k)_{k\in K}\bigr)
	\\
	\cong \coprod_{I\rto\phi J\rto\psi K\in\cS_\sk} \prod_{k\in K} \bigl\{\bigl[\prod_{j\in\psi^{-1}k}\mcv\bigl((X_i)_{i\in\phi^{-1}j};Y_j\bigr)\bigr]\times\mcv\bigl((Y_j)_{j\in\psi^{-1}k};Z_k\bigr)\bigr\}
	\\
	\rTTo^{\coprod_{(\phi,\psi)\mapsto\phi\centerdot\psi}\prod_{k\in K}\mu_{\phi|\:\phi^{-1}\psi^{-1}k\to\psi^{-1}k}} \hspace*{-1.5em} \coprod_{\xi\:I\to K\in\cS_\sk}\prod_{k\in K}\mcv\bigl((X_i)_{i\in\xi^{-1}k};Z_k\bigr) =\FF\mcv\bigl((X_i)_{i\in I},(Z_k)_{k\in K}\bigr).
	\end{multline*}
	Its associativity on summand indexed by \(I\rto\phi J\rto\psi K\rto\xi L\) follows from equation at \figref{dia-assoc-mu-multi} written for maps \(\phi^{-1}\psi^{-1}\xi^{-1}l\rto{\phi|}\psi^{-1}\xi^{-1}l\rto{\psi|}\xi^{-1}l\), \(l\in L\).
	
	The identity morphism $1$ in \(\FF\mcv\bigl((X_i)_{i\in I},(X_i)_{i\in I}\bigr)\) is \((1_{X_i})_{i\in I}\in\prod_{i\in I}\mcv(X_i;X_i)\) indexed by the identity map $\id_I$.
	The right unit property of $1$ on the summand indexed by \(\phi\:I\to J\) follows from equation~\eqref{eq-axiom-unit-multi1etaZ} for \(\triangledown\:\phi^{-1}j\to\mb1\), \(j\in J\).
	The left unit property of $1$ on the summand indexed by \(\phi\:I\to J\) follows from equation~\eqref{eq-axiom-unit-multi2} for \(\id\:\phi^{-1}j\to\phi^{-1}j\), \(j\in J\).
	
	The tensor multiplication on objects is the concatenation.
	On morphisms the tensor multiplication $\tens^K$ is the map (determined by maps \(I\rto fK\lto gJ\in\cO_\sk\))
	\begin{multline*}
	\tens^K\: \prod_{k\in K}\FF\mcv\bigl((X_i)_{i\in f^{-1}k},(Y_j)_{j\in g^{-1}k}\bigr) 
	\cong \coprod_{(\phi_k\:f^{-1}k\to g^{-1}k)_{k\in K}} \prod_{k\in K} \prod_{j\in g^{-1}k} \mcv\bigl((X_i)_{i\in\phi_k^{-1}j};Y_j\bigr)
	\\
	\rMono^{\;\;\coprod_{(\phi_k)\mapsto\phi}1}
	\coprod_{\xi\:I\to J\in\cS_\sk} \prod_{j\in J} \mcv\bigl((X_i)_{i\in\xi^{-1}j};Y_j\bigr) =\FF\mcv\bigl((X_i)_{i\in I},(Y_j)_{j\in J}\bigr),
	\end{multline*}
	where \(\phi\:I\to J\) is the only map, which satisfies the condition \(\phi|_{f^{-1}k}=\phi_k\).
	All such maps $\phi$ are characterized by the condition \((I\rto\phi J\rto gK)=f\).
	We shall see that the tensor multiplication is strictly associative.
	
	The unit object $\1$ (the image of \(\tens^{\mb0}\)) is the empty sequence \(()=()_\varnothing\).
	The left and the right unitors for this unit object are identity maps.
	We are going to prove that \((\FF\mcv,\tens,\1)\) is a strict monoidal category.
	
	Let \(h\:K\to J\in\cs_\sk\).
	The set \(\coprod_{j\in J}h^{-1}j=\{(j,k)\in J\times K\mid h(k)=j\}\) has a lexicographic ordering (for all \(k,k'\in K\) inequality $hk<hk'$ implies \((hk,k)<(hk',k')\), and if $hk=hk'$, then $k<k'$ implies \((hk,k)<(hk',k')\)).
	It follows that the map
	\[ t(h) =\bigl( \coprod_{j\in J}h^{-1}j =\{(j,k)\mid h(k)=j\} \subset J\times K \rTTo^{\pr_1} J \bigr)
	\]
	preserves the ordering.
	On the other hand, the map
	\[ \bigl( \coprod_{j\in J}h^{-1}j =\{(j,k)\mid h(k)=j\} \subset J\times K \rTTo^{\pr_2} K \bigr)
	\]
	is a bijection.
	Inverse to it bijection is denoted \(\sigma(h)\:K\to\coprod_{j\in J}h^{-1}j\).
	We adopt the point of view on this bijection as a permutation of elements of \(\{1<2<\dots<n\}=K\), sending $k\in K$ to $k\in K$, but the second $K$ has a different total ordering.
	Or we could view \(\sigma(h)\) as a self-bijection \(K\to K\), \(k\mapsto\sum_{j<h(k)}|h^{-1}j|+|\{k'\le k\mid h(k')=h(k)\}|\), but we shall not do it.
	Clearly,
	\begin{equation}
	\bigl( K \rTTo^{\sigma(h)} \coprod_{j\in J}h^{-1}j \rTTo^{t(h)} J \bigr) =h.
	\label{eq-sigma-t}
	\end{equation}
	For any colored prop $P$ the identity~\eqref{eq-lambda-fg-non-enriched} can be applied to the pair \((\sigma(h),t(h))\) from \eqref{eq-sigma-t}.
	Since \(\sigma(h)|\:h^{-1}j=\sigma(h)^{-1}t(h)^{-1}j\to t(h)^{-1}j=h^{-1}j\) is an order-preserving bijection, it is the identity map.
	Hence, equation~\eqref{eq-lambda-fg-non-enriched} can be written as \(\lambda_P^h\centerdot\tens^J1=\lambda_P^{\sigma(h)}\centerdot1\).
	We conclude that \(\lambda_P^h=\lambda_P^{\sigma(h)}\).
	
	In order to make \(\FF\mcv\) a lax symmetric monoidal category in the sense of \defref{def-Monoidal-cat} we assume given maps \(K\rto gI\rto fJ\), where \(g\in\cO_\sk\) and \(f\in\cs_\sk\).
	And we exhibit a natural transformation \(\lambda^f\:(X_k)_{k\in K}=\tens^{i\in I}(X_k)_{k\in g^{-1}i}\to\tens^{j\in J}\tens^{i\in f^{-1}j}(X_k)_{k\in g^{-1}i}=\bigl((X_k)_{k\in g^{-1}f^{-1}j}\bigr)_{j\in J}\).
	This is a morphism in \(\FF\mcv\) indexed by bijection $\sigma(g\centerdot f)\:K\to\coprod_{j\in J}g^{-1}f^{-1}j$.
	The element \(\lambda^f\in\prod_{j\in J}\prod_{k\in g^{-1}f^{-1}j}\mcv(X_k;X_k)\) is \(\lambda^f=\bigl((1_{X_k})_{k\in g^{-1}f^{-1}j}\bigr)_{j\in J}\).
	
	Naturality of \(\lambda^f\), \(f\in\cs_\sk\), amounts to commutative square
	\begin{diagram}[LaTeXeqno]
		(X_k)_{k\in K} &\rTTo^{\lambda^f} &\bigl((X_k)_{k\in g^{-1}f^{-1}j}\bigr)_{j\in J}
		\\
		\dTTo<{\tens^{i\in I}u_i} &&\dTTo>{\tens^{j\in J}\tens^{i\in f^{-1}j}u_i}
		\\
		(Y_l)_{l\in L} &\rTTo^{\lambda^f} &\bigl((Y_l)_{l\in h^{-1}f^{-1}j}\bigr)_{j\in J}
		\label{dia-XXYY-lambda-f}
	\end{diagram}
	for each pair of maps \(g,h\in\cO_\sk\) from
	\begin{diagram}[h=0.9em]
		K
		\\
		&\rdTTo^g
		\\
		\dTTo<\phi &&I &\rTTo^f &J
		\\
		&\ruTTo_h 
		\\
		L 
	\end{diagram}
	and all collections of morphisms \(u_i\:(X_k)_{k\in g^{-1}i}\to(Y_l)_{l\in h^{-1}i}\).
	Assume that $u_i$ is indexed by $\phi_i\:g^{-1}i\to h^{-1}i$.
	There is a unique map \(\phi\:K\to L\) such that \(\phi|_{g^{-1}i}=\phi_i\).
	Necessarily \(\phi\centerdot h=g\).
	Hence, \(u_i=(v_l)_{l\in h^{-1}i}\in\prod_{l\in h^{-1}i}\mcv\bigl((X_k)_{k\in\phi^{-1}l};Y_l\bigr)\).
	The diagram, formed by indexing maps for diagram~\eqref{dia-XXYY-lambda-f}
	\begin{diagram}[h=2.3em]
		K &\rTTo^{\sigma(g\centerdot f)} &\coprod_{j\in J}g^{-1}f^{-1}j
		\\
		\dTTo<\phi &&\dTTo>{\coprod_{j\in J}\phi|_{g^{-1}f^{-1}j}}
		\\
		L &\rTTo^{\sigma(h\centerdot f)} &\coprod_{j\in J}h^{-1}f^{-1}j
	\end{diagram}
	commutes, since both compositions map $k\in K$ to the same \(f(gk)=f(h\phi k)\).
	This is the only diagonal map of this square, independently of the ordering of source and target.
	One can verify that the diagonal map in \eqref{dia-XXYY-lambda-f}, represented by the family \(\bigl((v_l)_{l\in h^{-1}f^{-1}j}\bigr)_{j\in J}=\bigl((u_i)_{i\in f^{-1}j}\bigr)_{j\in J}\), equals the composition in the left--bottom path due to unitality \eqref{eq-axiom-unit-multi1etaZ} of multicategory $\mcv$, and equals the composition in the top--right path due to unitality property~\eqref{eq-axiom-unit-multi2}.
	Therefore, \eqref{dia-XXYY-lambda-f} commutes and \(\lambda^f\) is natural.
	
	Assume given maps \(L\rto hI\rto fJ\rto gK\), \(h\in\cO_\sk\), \(f,g\in\cs_\sk\).
	All vertices of the diagram
	\begin{diagram}[h=2.2em,LaTeXeqno]
		L &\rTTo^{\sigma(h\centerdot f)} &\coprod_{j\in J}h^{-1}f^{-1}j
		\\
		\dTTo<{\sigma(hfg)} &&\dTTo>{\coprod_{\sigma(g)}1}
		\\
		\coprod_{k\in K}h^{-1}f^{-1}g^{-1}k &\rTTo^{\coprod_{k\in K}\sigma(h\centerdot f|\:h^{-1}f^{-1}g^{-1}k\to g^{-1}k)} &\coprod_{k\in K}\coprod_{j\in g^{-1}k}h^{-1}f^{-1}j
		\label{dia-IIII}
	\end{diagram}
	are $L$ with various total orderings.
	All arrows map $i$ to $i$.
	Therefore diagram~\eqref{dia-IIII} commutes.
	Also diagram~\eqref{eq-lambda-fg-non-enriched} commutes, since \(1\centerdot 1=1\).
	
	In particular, \(\lambda^{\id_I}\:(X_i)_{i\in I}\to(X_i)_{i\in I}\), \(\lambda^{\id_I}=(1_{X_i})_{i\in I}\in\prod_{i\in I}\mcv(X_i;X_i)\), that is, \(\lambda^{\id_I}\) is the identity morphism of \((X_i)_{i\in I}\).
	Similarly, \(\lambda^{\triangledown\:I\to\mb1}\:(X_i)_{i\in I}\to(X_i)_{i\in I}\), is indexed by \(\sigma(\triangledown)=\id_I\) and \(\lambda^{\triangledown\:I\to\mb1}\:(X_i)_{i\in I}\to(X_i)_{i\in I}\), hence, \(\lambda^{\triangledown}=(1_{X_i})_{i\in I}\in\prod_{i\in I}\mcv(X_i;X_i)\) is the identity map.
	Summing up, \((\FF\mcv,\tens^I,\lambda^f)\) is a lax symmetric monoidal category.
	
	Furthermore, if \(f\in\cO_\sk\), then \(\lambda^f\:(X_k)_{k\in K}\to(X_k)_{k\in K}\), determined by \(K\rto gI\rto fJ\in\cO_\sk\) is indexed by $\id_K$ and equals \((1_{X_k})_{k\in K}\in\prod_{k\in K}\mcv(X_k;X_k)\).
	Therefore, \(\lambda^f=\id\) if $f$ preserves ordering.
	Thus \(\FF\mcv\) is a colored prop.
	
	In particular, it is symmetric with the symmetry \(c\:(X_i)_{i\in I}\sqcup(Y_j)_{j\in J}\to(Y_j)_{j\in J}\sqcup(X_i)_{i\in I}\) lying in the summand indexed by the block-wise permutation \(\sigma\:I\sqcup J\to J\sqcup I\).
	For \(1\le k\le|I|+|J|\)
	\[ \sigma(k)=
	\begin{cases}
	|J|+k, &\text{ for } k\le |I|,
	\\
	k-|I|, &\text{ for } k> |I|.
	\end{cases}
	\]
	The symmetry is \(\bigl((1_{Y_j})_{j\in J},(1_{X_i})_{i\in I}\bigr)\in\bigl[\prod_{j\in J}\mcv(Y_j;Y_j)\bigr]\times\bigl[\prod_{i\in I}\mcv(X_i;X_i)\bigr]\).
	
	The above construction being functorial, we get a functor \(\FF\:\sMCat\to\cProp\), where the latter category has symmetric strict monoidal functors \(F\:P\to Q\) as morphisms such that \(\Ob F\:\Ob P=(\Col P)^*\to(\Col Q)^*=\Ob Q\) is the morphism \((\Col F)^*\) of monoids induced by a map \(\Col F\:\Col P\to\Col Q\).
	
A functor \(\UU\:\cProp\to\sMCat\) is the composition which goes through lax symmetric monoidal categories \(\lsmCat\)
	\[ \cProp \rto{\text-^\clubsuit} \lsmCat \rto{\wh{\text-}} \sMCat,
	\]
	where the last functor is constructed in \cite[Proposition~3.22]{BesLyuMan-book}.
	On object (prop) $P$ the functor $\UU$ takes the value with \(\Ob\UU P=\Col P\), \(\UU P\bigl((X_i)_{i\in I};Y\bigr)=P\bigl((X_i)_{i\in I};Y\bigr)\), the units \(1_X\in P(X;X)\) and the composition
	\begin{multline*}
	\mu_f =\bigl\{ \bigl[ \prod_{j\in J} P\bigl((X_i)_{i\in f^{-1}j};Y_j\bigr) \bigr] \times P\bigl((Y_j)_{j\in J};Z\bigr) \rTTo^{\dot\lambda^f\times\tens^J\times1}
	\\
	P\bigl((X_i)_{i\in I};((X_i)_{i\in f^{-1}j})_{j\in J}\bigr) \times P\bigl(((X_i)_{i\in f^{-1}j})_{j\in J};(Y_j)_{j\in J}\bigr) \times P\bigl((Y_j)_{j\in J};Z\bigr)
	\\
	\rTTo^{\text{composition}} P\bigl((X_i)_{i\in I};Z\bigr) \bigr\}
	\end{multline*}
	for an arbitrary map \(f\:I\to J\in\cs_\sk\).
	Here \(\lambda^f\) is that of \(P^\clubsuit\).
	
	What is the natural bijection \(G\in\cProp(\FF\mcv,P)\cong\sMCat(\mcv,\UU P)\ni F\)?
	(Multi)\hspace{0pt}functors from the both sides have as the mapping on objects the same map \(\Ob F=\Ob G\:\Ob\mcv\to\Col P\), \(X\mapsto FX\) which we fix now.
	An element $F$ in the right hand side is the collection of mappings \(F_{(X_i)_{i\in I};Y}\:\mcv\bigl((X_i)_{i\in I};Y\bigr)\to P\bigl((FX_i)_{i\in I};FY\bigr)\) such that \((1^\mcv_X)F_{X;X}=1^P_{FX}\) and for any mapping \(f\:I\to J\)
	\begin{diagram}[h=2.6em,nobalance,bottom,LaTeXeqno]
		&&\bigl[ \prod_{j\in J} P\bigl((FX_i)_{i\in f^{-1}j};FY_j\bigr) \bigr] \times P\bigl((FY_j)_{j\in J};FZ\bigr)
		\\
		&\ruTTo[hug]_{[\prod_JF]\times F} &\dTTo>{\dot\lambda^f\times\tens^J\times1}
		\\
		\begin{array}{c}
			\bigl[ \prod_{j\in J} \mcv\bigl((X_i)_{i\in f^{-1}j};Y_j\bigr) \bigr]
			\\
			\times \mcv\bigl((Y_j)_{j\in J};Z\bigr)
		\end{array}
		&&\hspace*{-3.3em}
		\begin{array}{c}
			P\bigl((FX_i)_{i\in I};((FX_i)_{i\in f^{-1}j})_{j\in J}\bigr)
			\\[2pt]
			\times P\bigl(((FX_i)_{i\in f^{-1}j})_{j\in J};(FY_j)_{j\in J}\bigr) \times P\bigl((FY_j)_{j\in J};FZ\bigr)
		\end{array}
		\\
		\dTTo<{\mu^\mcv_f} &= &\dTTo>{\text{composition}}
		\\
		\mcv\bigl((X_i)_{i\in I};Z\bigr) &\rTTo^{F_{(X_i)_{i\in I};Z}} &P\bigl((FX_i)_{i\in I};FZ\bigr)
		\label{dia-VVV-PPPPPP}
	\end{diagram}
	An element $G$ in the left hand side is the collection of mappings
	\[ G^\phi\: \prod_{j\in J}\mcv\bigl((X_i)_{i\in\phi^{-1}j};Y_j\bigr) \to P\bigl((GX_i)_{i\in I};(GY_j)_{j\in J}\bigr),
	\]
	where mapping \(\phi\:I\to J\) runs over \(\cs_\sk\), such that $G$ is strictly compatible with the composition, the identities, the tensor products and \(\lambda^f\).
	
	Compatibility of $G$ with the tensor product \(\tens^J\), transformation \(\lambda^\phi\) and composition imply that the following diagram commutes:
	\begin{diagram}
		\prod_{j\in J} \FF\mcv\bigl((X_i)_{i\in\phi^{-1}j};Y_j\bigr) &\rTTo^{\prod_{j\in J}G^{\triangledown\:\phi^{-1}j\to\mb1}_{(X_i)_{i\in\phi^{-1}j};Y_j}} &\prod_{j\in J} P\bigl((GX_i)_{i\in\phi^{-1}j};GY_j\bigr)
		\\
		\dTTo<{\dot\lambda_{\FF\mcv}^\phi\times\tens_{\FF\mcv}^J}
		\\
		\FF\mcv\bigl((X_i)_{i\in I};((X_i)_{i\in\phi^{-1}j})_{j\in J}\bigr) \times \FF\mcv\bigl(((X_i)_{i\in\phi^{-1}j})_{j\in J};(Y_j)_{j\in J}\bigr) \hspace*{-13em} &&\dTTo>{\dot\lambda_P^\phi\times\tens_P^J}
		\\
		\dTTo<{\text{composition}_{\FF\mcv}} && \hspace*{-12.5em} P\bigl((GX_i)_{i\in I};((GX_i)_{i\in\phi^{-1}j})_{j\in J}\bigr) \times P\bigl(((GX_i)_{i\in\phi^{-1}j})_{j\in J};(GY_j)_{j\in J}\bigr)
		\\
		&&\dTTo>{\text{composition}_P}
		\\
		\FF\mcv\bigl((X_i)_{i\in I};(Y_j)_{j\in J}\bigr) &\rTTo^G &P\bigl((GX_i)_{i\in I};(GY_j)_{j\in J}\bigr)
	\end{diagram}
	Maps \(G^{\triangledown\:I\to\mb1}_{(X_i)_{i\in\phi^{-1}j};Y_j}\) are identified with \(F_{(X_i)_{i\in I};Y}\).
	As we are going to see this assignment determines all maps $G$ in a unique way.
	With this identification in mind we rewrite the above diagram as
	\begin{diagram}
		\prod_{j\in J} \mcv\bigl((X_i)_{i\in\phi^{-1}j};Y_j\bigr) &\rTTo^{\prod_{j\in J}F_{(X_i)_{i\in\phi^{-1}j};Y_j}} &\prod_{j\in J} P\bigl((GX_i)_{i\in\phi^{-1}j};GY_j\bigr)
		\\
		\dTTo<{\dot\lambda_{\FF\mcv}^\phi\times\tens_{\FF\mcv}^J}
		\\
		\bigl[ \prod_{j\in J}\prod_{i\in\phi^{-1}j}\mcv(X_i;X_i) \bigr]	\times \prod_{j\in J} \mcv\bigl((X_i)_{i\in\phi^{-1}j};Y_j\bigr) \hspace*{-5em} &&\dTTo>{\dot\lambda_P^\phi\times\tens_P^J}
		\\
		\dTTo<{\prod_{j\in J}\mu_{\id\:\phi^{-1}j\to\phi^{-1}j}} && \hspace*{-13em} P\bigl((GX_i)_{i\in I};((GX_i)_{i\in\phi^{-1}j})_{j\in J}\bigr) \times P\bigl(((GX_i)_{i\in\phi^{-1}j})_{j\in J};(GY_j)_{j\in J}\bigr)
		\\
		&&\dTTo>{\text{composition}_P}
		\\
		\prod_{j\in J} \mcv\bigl((X_i)_{i\in\phi^{-1}j};Y_j\bigr) &\rTTo^{G^\phi} &P\bigl((GX_i)_{i\in I};(GY_j)_{j\in J}\bigr)
	\end{diagram}
	Here the summand \(\prod_{j\in J}\prod_{i\in\phi^{-1}j}\mcv(X_i;X_i)\) is indexed by \(\sigma(\phi)\:I\to\coprod_{j\in J}\phi^{-1}j\).
	The summand \(\prod_{j\in J} \mcv\bigl((X_i)_{i\in\phi^{-1}j};Y_j\bigr)\) in the middle row is indexed by \(t(\phi)\:\coprod_{j\in J}\phi^{-1}j\to J\).
	Hence, the summand \(\prod_{j\in J} \mcv\bigl((X_i)_{i\in\phi^{-1}j};Y_j\bigr)\) in the bottom row is indexed by \((I\rTTo^{\sigma(\phi)} \coprod_{j\in J}\phi^{-1}j\rTTo^{t(\phi)} J)=\phi\) by \eqref{eq-sigma-t}.
	The left column composes to \(\id\) due to unitality~\eqref{eq-axiom-unit-multi2}.
	Therefore, for general \(\phi\:I\to J\in\cs_\sk\) we must have
	\begin{multline}
	G^\phi =\bigl[ \prod_{j\in J} \mcv\bigl((X_i)_{i\in\phi^{-1}j};Y_j\bigr) \rTTo^{\prod_{j\in J}F_{(X_i)_{i\in\phi^{-1}j};Y_j}} \prod_{j\in J} P\bigl((FX_i)_{i\in\phi^{-1}j};FY_j\bigr) \rTTo^{\dot\lambda_P^\phi\times\tens^J} 
	\\
	P\bigl((FX_i)_{i\in I};((FX_i)_{i\in\phi^{-1}j})_{j\in J}\bigr) \times P\bigl(((FX_i)_{i\in\phi^{-1}j})_{j\in J};(FY_j)_{j\in J}\bigr)
	\\
	\rTTo^{\text{composition}} P\bigl((FX_i)_{i\in I};(FY_j)_{j\in J}\bigr) \bigr].
	\label{eq-G-phi}
	\end{multline}
	
	Let us check that \eqref{dia-VVV-PPPPPP} and unitality are the only conditions imposed on $F$ by conditions on $G$.
	
	First of all we check that $G$ is compatible with tensor product $\tens^K$ due to ansatz~\eqref{eq-G-phi}, see diagram for \(f=(I\rto\phi J\rto gK)\), \(\phi\in\cs_\sk\), \(f,g\in\cO_\sk\)
	\[ \!
	\begin{diagram}[h=2.6em,inline]
	\prod_{k\in K} \prod_{j\in g^{-1}k} \mcv\bigl((X_i)_{i\in\phi^{-1}j};Y_j\bigr) &\rTTo^{\hspace*{-1.6em}\prod_{k\in K}\prod_{j\in g^{-1}k}F_{(X_i)_{i\in\phi^{-1}j};Y_j}\hspace*{-1em}} &\prod_{k\in K} \prod_{j\in g^{-1}k} P\bigl((FX_i)_{i\in\phi^{-1}j};FY_j\bigr)
	\\
	\dTTo<\cong &\rdTTo(2,4)[hug]_{\prod_{k\in K}G^{\phi|\:f^{-1}k\to g^{-1}k}} = \ldDashTo(2,4)~{\hspace*{5em}} &\dTTo~{\prod_{k\in K}\dot\lambda_P^{\phi|\:f^{-1}k\to g^{-1}k}\times\tens^{g^{-1}k}}
	\\
	\prod_{j\in J} \mcv\bigl((X_i)_{i\in\phi^{-1}j};Y_j\bigr) &&\hspace*{-4.5em}
	\begin{array}{r}
	\prod_{k\in K}\bigl[P\bigl((FX_i)_{i\in f^{-1}k};((FX_i)_{i\in\phi^{-1}j})_{j\in g^{-1}k}\bigr)
	\\
	\times P\bigl(((FX_i)_{i\in\phi^{-1}j})_{j\in g^{-1}k};(FY_j)_{j\in g^{-1}k}\bigr) \bigr]
	\end{array}
	\\
	\dTTo~{\prod_{j\in J}F_{(X_i)_{i\in\phi^{-1}j};Y_j}} &\rdTTo(2,4)[hug]_{G^{\phi\:I\to J}} \ldDashTo(2,4)[hug]^{\hspace*{4em}\tens^K\times\tens^K} &\dTTo>{\prod_K\text{composition}}
	\\
	\prod_{j\in J} P\bigl((FX_i)_{i\in\phi^{-1}j};FY_j\bigr)
	&&\hspace*{-1em} \prod_{k\in K} P\bigl((FX_i)_{i\in f^{-1}k};(FY_j)_{j\in g^{-1}k}\bigr)
	\\
	\dTTo<{\dot\lambda_P^\phi\times\tens^J} &= &\dTTo>{\tens^K}
	\\
	\begin{array}{c}
	P\bigl((FX_i)_{i\in I};((FX_i)_{i\in\phi^{-1}j})_{j\in J}\bigr)
	\\
	\times P\bigl(((FX_i)_{i\in\phi^{-1}j})_{j\in J};(FY_j)_{j\in J}\bigr)
	\end{array}
	&\rTTo^{\text{composition}} &P\bigl((FX_i)_{i\in I};(FY_j)_{j\in J}\bigr)
	\end{diagram}
	\]
	Upper right and lower left quadrilaterals commute due to equations~\eqref{eq-G-phi}.
	We prove that parallelogram in the middle commutes by considering the exterior and adding to the exterior two dashed arrows: the obvious bijection between upper right corner and the third set on the left and the dashed arrow marked by \(\tens^K\times\tens^K\).
	The obtained lower right quadrilateral commutes since \(\tens^K\) is a functor.
	The dashed parallelogram commutes since
	\begin{align*}
	\tens^K\: \prod_{k\in K} P\bigl((Z_i)_{i\in f^{-1}k};((Z_i)_{i\in\phi^{-1}j})_{j\in g^{-1}k}\bigr) &\to P\bigl((Z_i)_{i\in I};((Z_i)_{i\in\phi^{-1}j})_{j\in J}\bigr),
	\\
	(\lambda_P^{\phi|\:f^{-1}k\to g^{-1}k})_{k\in K} &\mapsto \lambda_P^\phi,
	\end{align*}
	due to equation~\eqref{eq-lambda-fg-non-enriched}, which takes the form \(\lambda_P^f\centerdot(\tens^{k\in K}\lambda_P^{\phi|\:f^{-1}k\to g^{-1}k})=\lambda_P^\phi\centerdot\lambda_P^g\).
	Notice that \(\lambda_P^f\) and \(\lambda_P^g\) are identity maps since $P$ is strictly monoidal.
	The obtained diagram commutes.
	
	Let \(K\rto gI\rto fJ\), where \(g\in\cO_\sk\) and \(f\in\cs_\sk\).
	We are going to prove that 
\[ G\:\FF\mcv\bigl((X_k)_{k\in K},\bigl((X_k)_{k\in g^{-1}f^{-1}j}\bigr)_{j\in J}\bigr)\to P\bigl((FX_k)_{k\in K},\bigl((FX_k)_{k\in g^{-1}f^{-1}j}\bigr)_{j\in J}\bigr)
\]
sends \(\lambda_{\FF\mcv}^f\) to \(\lambda_P^f\).
	Since \(\lambda_{\FF\mcv}^f\) is indexed by \(\sigma(g\centerdot f)\) we compute
	\begin{multline*}
	G^{\sigma(g\centerdot f)} =\bigl[ \prod_{j\in J}\prod_{k\in g^{-1}f^{-1}j} \mcv\bigl(X_k;X_k\bigr) \rTTo^{\prod_{j\in J}\prod_{k\in g^{-1}f^{-1}j} F_{X_k;X_k}} \prod_{j\in J}\prod_{k\in g^{-1}f^{-1}j} P\bigl(FX_k;FX_k\bigr)
	\\
	\rTTo^{\dot\lambda_P^{\sigma(g\centerdot f)}\times\tens^{j\in J}\tens^{k\in g^{-1}f^{-1}j}}
	\\
	P\bigl((FX_k)_{k\in K};((FX_k)_{k\in g^{-1}f^{-1}j})_{j\in J}\bigr) \times P\bigl(((FX_k)_{k\in g^{-1}f^{-1}j})_{j\in J};((FX_k)_{k\in g^{-1}f^{-1}j})_{j\in J}\bigr)
	\\
	\hfill \rTTo^{\text{composition}} P\bigl((FX_k)_{k\in K};((FX_k)_{k\in g^{-1}f^{-1}j})_{j\in J}\bigr) \bigr], \hskip\multlinegap
	\\
	\lambda_{\FF\mcv}^f =\bigl((1_{X_k})_{k\in g^{-1}f^{-1}j}\bigr)_{j\in J} \mapsto \bigl((1_{FX_k})_{k\in g^{-1}f^{-1}j}\bigr)_{j\in J} \mapsto (\lambda_P^{\sigma(g\centerdot f)},1_{((FX_k)_{k\in g^{-1}f^{-1}j})_{j\in J}})
	\\
	\mapsto \lambda_P^{\sigma(g\centerdot f)}.
	\end{multline*}
	Recall that \(\lambda_P^{\sigma(g\centerdot f)}=\lambda_P^{g\centerdot f}\) as noticed below \eqref{eq-sigma-t}.
	Since \(g|\:g^{-1}f^{-1}j\to f^{-1}j\) is order-preserving, equation~\eqref{eq-lambda-fg-non-enriched} for the pair $(g,f)$ gives \(\lambda_P^{g\centerdot f}\centerdot\tens^J1=1\centerdot\lambda_P^f\).
	Hence, for any \(K\rto gI\rto fJ\), where \(g\in\cO_\sk\) and \(f\in\cs_\sk\), and any family \((Z_k)_{k\in K}\) of objects of $P$ there is an equality \(\lambda_{(Z_k)_{k\in K}}^{\sigma(g\centerdot f)}=\lambda_{((Z_k)_{k\in g^{-1}i})_{i\in I}}^f\:(Z_k)_{k\in K}\to((Z_k)_{k\in g^{-1}f^{-1}j})_{j\in J}\).
	We conclude that $G$ sends \(\lambda_{\FF\mcv}^f\) to \(\lambda_P^f\).
	
	Compatibility of $G$ with the composition follows from commutativity of the diagram
	\begin{diagram}[nobalance,h=1.4em]
		&&\hspace*{-3em} \bigl[\prod_{j\in J}\mcv\bigl((X_i)_{i\in\phi^{-1}j};Y_j\bigr)\bigr] \times \prod_{k\in K}\mcv\bigl((Y_j)_{j\in\psi^{-1}k};Z_k\bigr)
		\\
		&\ldTTo_\cong &
		\\
		\prod_{k\in K} \bigl\{\bigl[\prod_{j\in\psi^{-1}k}\mcv\bigl((X_i)_{i\in\phi^{-1}j};Y_j\bigr)\bigr]\times\mcv\bigl((Y_j)_{j\in\psi^{-1}k};Z_k\bigr)\bigr\} \hspace*{-3em} &&\dTTo>{G^\phi\times G^\psi}
		\\
		\\
		\dTTo<{\prod_{k\in K}\mu_{\phi|\:\phi^{-1}\psi^{-1}k\to\psi^{-1}k}} &&\hspace*{-4.5em} P\bigl((GX_i)_{i\in I};(GY_j)_{j\in J}\bigr) \times P\bigl((GY_j)_{j\in J};(GZ_k)_{k\in K}\bigr)
		\\
		\\
		\prod_{k\in K} \mcv\bigl((X_i)_{i\in\phi^{-1}\psi^{-1}k};Z_k\bigr) &&\dTTo>{\text{composition}}
		\\
		&\rdTTo_{G^{\phi\centerdot\psi}} &
		\\
		&&P\bigl((GX_i)_{i\in I};(GZ_k)_{k\in K}\bigr)
	\end{diagram}
	for arbitrary maps \(I\rto\phi J\rto\psi K\in\cS_\sk\).
	Using ansatz~\eqref{eq-G-phi} we rewrite this diagram as
	\begin{diagram}[nobalance,h=1.4em]
		&&\hspace*{-2.5em} \bigl[\prod_{j\in J}\mcv\bigl((X_i)_{i\in\phi^{-1}j};Y_j\bigr)\bigr] \times \prod_{k\in K}\mcv\bigl((Y_j)_{j\in\psi^{-1}k};Z_k\bigr)
		\\
		&\ldTTo_\cong &
		\\
		\prod_{k\in K} \bigl\{\bigl[\prod_{j\in\psi^{-1}k}\mcv\bigl((X_i)_{i\in\phi^{-1}j};Y_j\bigr)\bigr]\times\mcv\bigl((Y_j)_{j\in\psi^{-1}k};Z_k\bigr)\bigr\} \hspace*{-2.5em} &&\dTTo~{\prod_{j\in J}F_{(X_i)_{i\in\phi^{-1}j};Y_j}\times\prod_{k\in K}F_{(Y_j)_{j\in\psi^{-1}k};Z_k}\hspace*{-2.35em}}
		\\
		\dTTo<{\prod_{k\in K}\mu_{\phi|\:\phi^{-1}\psi^{-1}k\to\psi^{-1}k}} &&\hspace*{-3em} \bigl[\prod_{j\in J}P\bigl((U_i)_{i\in\phi^{-1}j};V_j\bigr)\bigr] \times \prod_{k\in K}P\bigl((V_j)_{j\in\psi^{-1}k};W_k\bigr)
		\\
		\prod_{k\in K} \mcv\bigl((X_i)_{i\in\phi^{-1}\psi^{-1}k};Z_k\bigr) &&\dTTo>{\dot\lambda_P^\phi\times\tens^J\times\dot\lambda_P^\psi\times\tens^K}
		\\
		\dTTo<{\prod_{k\in K}F_{(X_i)_{i\in\phi^{-1}\psi^{-1}k};Z_k}} &&\hspace*{-8.6em}
		\begin{array}{c}
			P\bigl((U_i)_{i\in I};((U_i)_{i\in\phi^{-1}j})_{j\in J}\bigr) \times P\bigl(((U_i)_{i\in\phi^{-1}j})_{j\in J};(V_j)_{j\in J}\bigr) \times 
			\\[2pt]
			P\bigl((V_j)_{j\in J};((V_j)_{j\in\psi^{-1}k})_{k\in K}\bigr) \times P\bigl(((V_j)_{j\in\psi^{-1}k})_{k\in K};(W_k)_{k\in K}\bigr)
		\end{array}
		\\
		\prod_{k\in K} P\bigl((U_i)_{i\in\phi^{-1}\psi^{-1}k};W_k\bigr) &&\dTTo>{\text{composition}}
		\\
		\dTTo<{\dot\lambda_P^{\phi\centerdot\psi}\times\tens^K}	&&P\bigl((U_i)_{i\in I};(W_k)_{k\in K}\bigr)
		\\
		&\ruTTo^{\text{composition}}
		\\
		P\bigl((U_i)_{i\in I};((U_i)_{i\in\phi^{-1}\psi^{-1}k})_{k\in K}\bigr) \times P\bigl(((U_i)_{i\in\phi^{-1}\psi^{-1}k})_{k\in K};(W_k)_{k\in K}\bigr) \hspace*{-9em}
	\end{diagram}
	where we denote \(U_i=FX_i\), \(V_j=FY_j\), \(W_k=FZ_k\).
	The above diagram reduces to several equations~\eqref{dia-VVV-PPPPPP} (one for each map \(\phi|\:\phi^{-1}\psi^{-1}k\to\psi^{-1}k\), $k\in K$) and the following diagram which uses only structure maps of $P$:
	\begin{diagram}[nobalance,h=1.4em,bottom,LaTeXeqno]
		\prod_{k\in K} \bigl\{\bigl[\prod_{j\in\psi^{-1}k}P\bigl((U_i)_{i\in\phi^{-1}j};V_j\bigr)\bigr]\times P\bigl((V_j)_{j\in\psi^{-1}k};W_k\bigr)\bigr\} \hspace*{-2.5em} &&\hspace*{17em}
		\\
		&\rdTTo~\cong
		\\
		\dTTo~{\hspace*{-5em}\prod_{k\in K}\{\dot\lambda^{\phi|\:\phi^{-1}\psi^{-1}k\to\psi^{-1}k}\times\tens^{\psi^{-1}k}\times1\}} &&\hspace*{-4em} \bigl[\prod_{j\in J}P\bigl((U_i)_{i\in\phi^{-1}j};V_j\bigr)\bigr] \times \prod_{k\in K}P\bigl((V_j)_{j\in\psi^{-1}k};W_k\bigr)
		\\
		\begin{array}{c}
			\prod_{k\in K} \bigl\{ P\bigl((U_i)_{i\in\phi^{-1}\psi^{-1}k};((U_i)_{i\in\phi^{-1}j})_{j\in\psi^{-1}k}\bigr) \times
			\\[2pt]
			P\bigl(((U_i)_{i\in\phi^{-1}j})_{j\in\psi^{-1}k};(V_j)_{j\in\psi^{-1}k}\bigr) \times P\bigl((V_j)_{j\in\psi^{-1}k};W_k\bigr) \bigr\}
		\end{array}
		\hspace*{-4em}
		&&\dTTo>{\dot\lambda_P^\phi\times\tens^J\times\dot\lambda_P^\psi\times\tens^K}
		\\
		\dTTo<{\prod_K\text{composition}} &&\hspace*{-9em}
		\begin{array}{c}
			P\bigl((U_i)_{i\in I};((U_i)_{i\in\phi^{-1}j})_{j\in J}\bigr) \times P\bigl(((U_i)_{i\in\phi^{-1}j})_{j\in J};(V_j)_{j\in J}\bigr)\times 
			\\[2pt]
			P\bigl((V_j)_{j\in J};((V_j)_{j\in\psi^{-1}k})_{k\in K}\bigr) \times P\bigl(((V_j)_{j\in\psi^{-1}k})_{k\in K};(W_k)_{k\in K}\bigr)
		\end{array}
		\\
		\prod_{k\in K}P\bigl((U_i)_{i\in\phi^{-1}\psi^{-1}k};W_k\bigr) &&\dTTo>{\text{composition}}
		\\
		\dTTo<{\dot\lambda_P^{\phi\centerdot\psi}\times\tens^K}	&&P\bigl((U_i)_{i\in I};(W_k)_{k\in K}\bigr)
		\\
		&\ruTTo^{\text{composition}}
		\\
		P\bigl((U_i)_{i\in I};((U_i)_{i\in\phi^{-1}\psi^{-1}k})_{k\in K}\bigr) \times P\bigl(((U_i)_{i\in\phi^{-1}\psi^{-1}k})_{k\in K};(W_k)_{k\in K}\bigr) \hspace*{-8em}
		\label{dia-PPUUVVWW}
	\end{diagram}
	In order to prove its commutativity consider morphisms \(f_j\:(U_i)_{i\in\phi^{-1}j}\to V_j\), \(g_k\:(V_j)_{j\in\psi^{-1}k}\to W_k\) of $P$.
	Diagram~\eqref{dia-PPUUVVWW} is equivalent to commutativity of exterior of
	\begin{diagram}
		(U_i)_{i\in I} &\rTTo^{\lambda^\phi} &((U_i)_{i\in\phi^{-1}j})_{j\in J} &\rTTo^{\tens^{j\in J}f_j} &(V_j)_{j\in J}
		\\
		\dTTo<{\lambda^{\phi\centerdot\psi}} &= &&= &\dTTo>{\lambda^\psi}
		\\
		((U_i)_{i\in\phi^{-1}\psi^{-1}k})_{k\in K} &&\dTTo<{\lambda^\psi} &&((V_j)_{j\in\psi^{-1}k})_{k\in K}
		\\
		&\rdTTo_{\tens^{k\in K}\lambda^{\phi|\:\phi^{-1}\psi^{-1}k\to\psi^{-1}k}} &&\ruTTo^{\tens^{k\in K}\tens^{j\in\psi^{-1}k}f_j} &\dTTo>{\tens^{k\in K}g_k}
		\\
		&&\bigl(((U_i)_{i\in\phi^{-1}j})_{j\in J}\bigr)_{k\in K} &&(W_k)_{k\in K}
	\end{diagram}
	These equations hold due to equation~\eqref{eq-lambda-fg-non-enriched} and naturality of \(\lambda^\psi\).
	Thus a natural map 
\[ \theta\:\sMCat(\mcv,\UU P)\to\cProp(\FF\mcv,P), \quad F\mapsto G,
\]
is constructed.
	
	For \(\phi=\triangledown\:I\to\mb1\) we have \(G^\triangledown=F_{(X_i)_{i\in I};Y}\).
	Hence, the map $\theta$ is injective.
	It is also surjective, as obligatory formula~\eqref{eq-G-phi} shows.
	Therefore, $\theta$ is a natural bijection.
\end{proof}

As on any free monoid there is a length function \(l\:\Ob\FF\mcv=(\Ob\mcv)^*\to\NN\) on objects of \(\FF\mcv\).
Thus, \(\Ob\mcv=\{A\in\Ob\FF\mcv\mid l(X)=1\}\).

\subsection{Multicategory of \texorpdfstring{$\mcV$}V-quivers}
\begin{definition}\label{def-V-quivers}
Let $\mcv$ be a plain multicategory.
A small $\mcv$\n-quiver $\ca$ is
	\begin{myitemize}
		\item[---] a small set $\Ob\ca$ of objects;
\item[---] for each pair of objects \((X,Y)\) of $\ca$ an object \(\ca(X,Y)\) of $\mcv$, that is, an object \(\ca(X,Y)\in\Ob\FF\mcv\) such that \(l(\ca(X,Y))=1\).
	\end{myitemize}
\end{definition}

\begin{definition}\label{def-multi-entry-V-quiver-morphism}
Let $\mcV$ be a locally small multicategory.
Let $\cb$, $\ca_i$, $i\in I\in\cO_\sk$, be small $\mcv$\n-quivers.
A multi-entry $\mcv$\n-quiver morphism\index{morphism of quivers enriched in a multicategory} \(F\:(\ca_i)_{i\in I}\to\cb\) is
	\begin{myitemize}
		\item[---] a function \(F=\Ob F\:\Ob\ca_1\times\dots\times\Ob\ca_I\to\Ob\cb\);
\item[---] a collection of elements \(F=F_{(A_i),(D_i)}\in\mcv\bigl((\ca_i(A_i,D_i))_{i\in I};\cb((A_i)_{i\in I}F,(D_i)_{i\in I}F)\bigr)\).
	\end{myitemize}
\end{definition}

The small set of multi-entry $\mcv$\n-quiver morphisms \((\ca_i)_{i\in I}\to\cb\) is denoted
\[ \VQu\bigl((\ca_i)_{i\in I};\cb\bigr)= \hspace*{-2em} \bigsqcup_{F\:\prod_{i\in I}\Ob\ca_i\to\Ob\cb} \; \bigsqcap_{(A_i,D_i\in\Ob\ca_i)_{i\in I}} \hspace*{-2em} \mcv\bigl((\ca_i(A_i,D_i))_{i\in I};\cb((A_i)_{i\in I}F,(D_i)_{i\in I}F)\bigr).
\]

\begin{proposition}\label{pro-symmetric-multicategory-multi-entry-quiver-morphisms}
Let $\mcV$ be a locally small (symmetric) multicategory.
Small $\mcv$\n-quivers and multi-entry $\mcv$\n-quiver morphisms form a locally small (symmetric) multicategory $\VQu$.
\end{proposition}

\begin{proof}
Let \(\phi\:I\to J\in\cO_\sk\) (resp. \(\phi\:I\to J\in\cS_\sk\)).
Let \((\ca_i)_{i\in I}\), \((\cb_j)_{j\in J}\), $\cc$ be (families of) small $\mcv$\n-quivers.
Let \(F^j\:(\ca_i)_{i\in\phi^{-1}j}\to\cb_j\), $j\in J$, \(G\:(\cb_j)_{j\in J}\to\cc\) be multi-entry quiver morphisms.
We construct another multi-entry quiver morphism \(H\:(\ca_i)_{i\in I}\to\cc\) with
\begin{myitemize}
	\item[---] \(H=\Ob H\:(A_i)_{i\in I}\mapsto\bigl((A_i)_{i\in\phi^{-1}j}F^j\bigr)_{j\in J}G\).
	\item[---] \(H=H_{(A_i),(E_i)}\:(\cA_i(A_i,E_i))_{i\in I}\to\cc((A_i)_{i\in I}H,(E_i)_{i\in I}H)\) obtained from
\end{myitemize}
\begin{multline}
\mu^\mcv_\phi\: \prod_{j\in J} \mcv\bigl((\ca_i(A_i,E_i))_{i\in\phi^{-1}j};\cb_j((A_i)_{i\in\phi^{-1}j}F^j,(E_i)_{i\in\phi^{-1}j}F^j)\bigr) \times
	\\
\mcv\bigl((\cb_j((A_i)_{i\in\phi^{-1}j}F^j,(E_i)_{i\in\phi^{-1}j}F^j))_{j\in J};\cc(((A_i)_{i\in\phi^{-1}j}F^j)_{j\in J}G,((E_i)_{i\in\phi^{-1}j}F^j)_{j\in J}G)\bigr)
	\\
	\hfill \to \mcv\bigl((\ca_i(A_i,E_i))_{i\in I};\cc((A_i)_{i\in I}H,(E_i)_{i\in I}H)\bigr), \hskip\multlinegap
	\\
	\bigl((F^j_{(A_i)_{i\in\phi^{-1}j},(E_i)_{i\in\phi^{-1}j}})_{j\in J},G_{((A_i)_{i\in\phi^{-1}j}F^j)_{j\in J},((E_i)_{i\in\phi^{-1}j}F^j)_{j\in J}}\bigr) \mapsto H_{(A_i),(E_i)}.
\label{eq-FGH-quivers}
\end{multline}
This assignment is in fact a component of the map
\[ \mu_\phi^\VQu\: \bigl[ \prod_{j\in J}\VQu((\ca_i)_{i\in\phi^{-1}j};\cb_j) \bigr] \times\VQu((\cb_j)_{j\in J};\cc) \to \VQu((\ca_i)_{i\in I};\cc).
\]
Equivalently, it is a component of the map
\begin{multline*}
\bigsqcup_{(F^j\:\prod_{i\in\phi^{-1}j}\Ob\ca_i\to\Ob\cb_j)_{j\in J},G\:\prod_{j\in J}\Ob\cb_j\to\Ob\cc}
\\
\bigl(\prod_{j\in J} \; \prod_{(A_i,E_i\in\Ob\ca_i)_{i\in\phi^{-1}j}} \mcv\bigl((\ca_i(A_i,E_i))_{i\in\phi^{-1}j};\cb_j((A_i)_{i\in\phi^{-1}j}F^j,(E_i)_{i\in\phi^{-1}j}F^j)\bigr)\bigr)
\\
\times \bigl(\prod_{(B_j,D_j\in\Ob\cb_j)_{j\in J}} \mcv\bigl((\cb_j(B_j,D_j))_{j\in J};\cc((B_j)_{j\in J}G,(D_j)_{j\in J}G)\bigr)\bigr)
\\
\to \bigsqcup_{H\:\prod_{i\in I}\Ob\ca_i\to\Ob\cc} \; \prod_{(A_i,E_i\in\Ob\ca_i)_{i\in I}} \mcv\bigl((\ca_i(A_i,E_i))_{i\in I};\cb((A_i)_{i\in I}H,(E_i)_{i\in I}H)\bigr).
\end{multline*}

Let \((\ca_i)_{i\in I}\), \((\cb_j)_{j\in J}\), \((\cc_k)_{k\in K}\), $\cD$ be (families of) small $\mcv$\n-quivers, where \(I,J,K\in\Ob\cO_\sk\).
Let \(I\rto\phi J\rto\psi K\) be mappings in $\cO_\sk$ (in $\cs_\sk$).
Let \(F^j\:(\ca_i)_{i\in\phi^{-1}j}\to\cb_j\), $j\in J$, \(G^k\:(\cb_j)_{j\in\psi^{-1}k}\to\cc\), $k\in K$,  \(H\:(\cc_k)_{k\in K}\to\cD\) be multi-entry quiver morphisms.
Fix objects $A_i$, $E_i$ of $\ca_i$, $i\in I$.
Expanding entries of the associativity equation for $\VQu$ using \eqref{eq-FGH-quivers} we get diagram at \figref{dia-assoc-mu-multi} for \(X_i=\ca_i(A_i,E_i)\), \(Y_j=\cb_j\bigl((A_i)_{i\in\phi^{-1}j}F^j,(E_i)_{i\in\phi^{-1}j}F^j\bigr)\),
\begin{gather*}
Z_k=\cc_k\bigl(((A_i)_{i\in\phi^{-1}j}F^j)_{j\in\psi^{-1}k}G^k,((E_i)_{i\in\phi^{-1}j}F^j)_{j\in\psi^{-1}k}G^k\bigr),
\\
W=\cD\bigl((((A_i)_{i\in\phi^{-1}j}F^j)_{j\in\psi^{-1}k}G^k)_{k\in K}H,(((E_i)_{i\in\phi^{-1}j}F^j)_{j\in\psi^{-1}k}G^k)_{k\in K}H\bigr).
\end{gather*}
Therefore, for composition in $\VQu$ the associativity holds.

Define the identity $\mcv$\n-quiver morphism \(\Id\:\ca\to\ca\) with the identity map \(\id\:\Ob\ca\to\Ob\ca\) and \(1_{\ca(A,A)}\in\mcv\bigl(\ca(A,A);\ca(A,A)\bigr)\).
Clearly, both equations for identities are satisfied, hence, $\VQu$ is a (symmetric) multicategory.
\end{proof}

\subsection{\texorpdfstring{$\mcV$}V-categories}\label{sec-Bicategories-V-categories}
In mathematical literature there are at least two different notions called categories enriched in bicategories.
Let us consider categories\index{category enriched in a multicategory} enriched in multicategories.
This notion seems to appear for the first time in \cite[\S1, (MLC 4)]{Linton:multiYoneda}, translated to a modern language in \cite[\S2]{1802.07538}.
We use the definition of Leinster \cite[Example~2.2.1.iii]{Leinster:math.CT/9901139}, \cite[Example (2), page 399]{zbMATH01724909}.
First of all we show that categories enriched in multicategories are a particular case of $T$\n-algebras for some monad $T$.

Assume that $\mcv$ is a locally small symmetric multicategory.
According to \cite[Theorem~4.2]{zbMATH05655952}, \cite[Proposition~11]{zbMATH06469262} or \propref{pro-adjunction-FU} there is a colored prop $\cp=\FF\mcv$ associated with it.
Its free cocompletion \(\wh\cp=\und\Cat(\cp^\op,\Set)\) is a monoidally cocomplete category.
The monoidal product is the Day convolution \cite{Day-closed-categories-functors}
\begin{gather*}
	(P\:\cp^\op\to\Set,Q\:\cp^\op\to\Set) \mapsto P\tens Q\:\cp^\op\to\Set,
	\\
	(P\tens Q)(X) =\int^{Y,Z\in\cp} \cp(X,Y\tens Z) \times P(Y) \times Q(Z).
\end{gather*}
The Yoneda full embedding is a monoidal functor by the density formula which is called ninja Yoneda Lemma in \cite[Proposition~2.2.1]{1501.02503}.
Note that the Day convolution is compatible with the colimits in any argument.

The category \(\wh\cp\Qu\) has a full subcategory \(\wh\cp\Qu_S\) of \(\wh\cp\)\n-quivers with the set of objects $S$.
Such \(\wh\cp\)\n-quivers are functions \(S^2\to\Ob\wh\cp\).

The category \(\wh\cp\Qu_S\) has a monoidal structure.
The monoidal multiplication is
\begin{equation}
	\bigl(\bigotimes^{i\in I}\ca_i\bigr)(A,B) =\coprod_{(X_i\in S)_{i\in I-1}}^{X_0=A,X_I=B} \bigotimes^{i\in I} \bigl[\ca_i(X_{i-1},X_i)\bigr].
	\label{eq-AiAB}
\end{equation}
In particular, the unit object is
\[ \1^{\wh\cp\Qu_S}(A,B) =(\tens^\varnothing)(A,B) =
\begin{cases}
	\1^{\wh\cp}, &\text{ if }A=B,
	\\
	\emptyset, &\text{ if }A\ne B,
\end{cases}
\]
where \(\emptyset\) (\(\emptyset(P)=\varnothing\)) is the initial object of \(\wh\cp\).

Functor \eqref{eq-AiAB} preserves colimits in each argument.
This allows to write down the isomorphism \(\lambda_{\wh\cp\Qu_S}^\phi\:\tens^I\ca_i\to\tens^{j\in J}\tens^{i\in\phi^{-1}}\ca_i\) in \(\wh\cp\Qu_S\) for each order preserving map \(\phi\:I\to J\in\cO_\sk\).
It is
\begin{multline*}
	\coprod_{(X_i\in S)_{i\in I-1}}^{X_0=A,X_I=B} \bigotimes^{i\in I} \bigl[\ca_i(X_{i-1},X_i)\bigr] \rTTo^{\coprod\lambda_{\wh\cp}^\phi}_\cong \coprod_{(X_i\in S)_{i\in I-1}}^{X_0=A,X_I=B} \bigotimes^{j\in J} \bigotimes^{i\in\phi^{-1}j} \bigl[\ca_i(X_{i-1},X_i)\bigr]
	\\
	\rTTo^\cong \coprod_{(Y_j\in S)_{i\in J-1}}^{Y_0=A,Y_J=B} \bigotimes^{j\in J} \coprod_{(X_i\in S)_{i\in\phi^{-1}j}^{i<\max\phi^{-1}j}}^{X_{\min\phi^{-1}j}=Y_{j-1},X_{\max\phi^{-1}j}=Y_j} \bigotimes^{i\in\phi^{-1}j} \bigl[\ca_i(X_{i-1},X_i)\bigr].
\end{multline*}
Axioms i), ii) of \defref{def-Monoidal-cat} for \(\lambda_{\wh\cp\Qu_S}\) follow from those for \(\lambda_{\wh\cp}\).

Using this monoidal structure we define a functor
\[ T_S\:\wh\cp\Qu_S\to\wh\cp\Qu_S, \qquad Q\mapsto T_SQ =\coprod_{I\in\Ob\cO_\sk} Q^{\tens I} =\coprod_{n\ge0} Q^{\tens n}.
\]

\begin{proposition}
	The functor $T_S$ is a monad.
	$T_S$\n-algebras are precisely categories $\cc$ enriched in \(\wh\cp\) with \(\Ob\cc=S\).
\end{proposition}

\begin{proof}
	Since $\tens$ in \(\wh\cp\Qu_S\) commutes with coproducts, we have for \(A,B\in S\)
\begin{equation*}
	(T_S^2Q)(A,B) =\coprod_{J\in\cO_\sk} \bigotimes^{j\in J} \coprod_{I_j\in\cO_\sk} \bigotimes^{i\in I_j} Q \cong \coprod_{\phi\:I\to J\in\cO_\sk} \bigotimes^{i\in I} Q,
\end{equation*}
where \(I=\bigsqcup_{j\in J}I_j\) and \(\phi\) is the projection to the indexing set.
The multiplication \(m_S\:T_S^2\to T_S\) sends the summand indexed by \(\phi\:I\to J\in\cO_\sk\) to the summand indexed by $I$ identically.
The unit \(\eta_S\:\Id\to T_S\) sends \(Q\) to the summand indexed by $\mb1$ identically.
This shows that \((T_S,m_S,\eta_S)\) is a monad.

Clearly, algebras over this monad are identified with categories $\cc$ enriched in \(\wh\cp\) with the set of objects \(\Ob\cc=S\).
The composition \(\cc\tens\cc\to\cc\) is the restriction of the action \(T_S\cc\to\cc\) to \(\cc^{\tens2}\) and the unit \(\1\to\cc\) comes from the restriction of the action to \(\cc^{\tens0}\).
\end{proof}

The functors $T_S$ glue into a single functor \(T\:\wh\cp\Qu\to\wh\cp\Qu\) such that
\begin{diagram}[LaTeXeqno]
	\wh\cp\Qu_S &\rTTo^{T_S} &\wh\cp\Qu_S
	\\
	\dMono &= &\dMono
	\\
	\wh\cp\Qu &\rTTo^T &\wh\cp\Qu
	\label{dia-monad-T}
\end{diagram}
On morphisms \(F\:\ca\to\cb\) the functor $T$ is defined as \(TF=\coprod_{I\in\cO_\sk}F^{\tens I}\:\coprod_{I\in\cO_\sk}\ca^{\tens I}\to\coprod_{I\in\cO_\sk}\cb^{\tens I}\), where
\begin{multline*}
	F^{\tens I} =\biggl(\bigotimes^{i\in I}\bigl[\ca(X_{i-1},X_i)\bigr] \rTTo^{\tens^{i\in I}F_{X_{i-1},X_i}} \bigotimes^{i\in I}\bigl[\cb(FX_{i-1},FX_i)\bigr]
	\\
	\rTTo^{\iota_{Y_i=FX_i}} \coprod_{(Y_i\in\Ob\cb)_{i\in I-1}}^{Y_0=FA,Y_I=FB} \bigotimes^{i\in I} \bigl[\cb(Y_{i-1},Y_i)\bigr] \biggr)_{(X_i\in\Ob\ca)_{i\in I-1}}^{X_0=A,X_I=B}.
\end{multline*}
Moreover, the monad structures on $T_S$ glue together into a single monad structure on $T$.

$T$\n-algebras on the underlying $\cp$\n-quiver $\cc$ are defined as $T$\n-algebras such that the underlying \(\wh\cp\)\n-quiver factors through the Yoneda map \(\Ob\cY\:\Ob\cp\to\Ob\wh\cp\).
The object set \(\Ob\cp=(\Col\cp)^*=(\Ob\mcv)^*\) has a word length function \(l\:\Ob\cp\to\NN\).
\(\mcv\)\n-categories $\cc$ are distinguished among the latter kind of $T$\n-algebras by the requirement that the underlying quiver \((\Ob\cc)^2\to\Ob\cp\) takes values in the subset \(\{X\in\Ob\cp\mid l(X)=1\}\).
Equivalent and detailed definition of \(\mcv\)\n-categories is given below.

\begin{definition}\label{def-V-categories}
Let $\mcv$ be a plain multicategory.
A small $\mcv$\n-category $\cc$ is a small $\FF\mcv$\n-category $\cc$ with \(\cc(X,Y)\in\Ob\FF\mcv\) satisfying \(l(\cc(X,Y))=1\).
In detail, it is
\begin{myitemize}
\item[---] a small set $\Ob\cc$ of objects;
\item[---] for each pair of objects \((X,Y)\) of $\cc$ an object \(\cc(X,Y)\) of $\mcv$;
\item[---] for each triple of objects \((X,Y,Z)\) of $\cc$ a morphism \(\kappa_{X,Y,Z}\:\cc(X,Y),\cc(Y,Z)\to\cc(X,Z)\in\mcv\) -- the composition;
\item[---] for any object \(X\) of $\cc$ a morphism \(\id_X\:()\to\cc(X,X)\in\mcv\) -- the identity morphism\index{$\id_X$ -- identity element of an enriched category}
\end{myitemize}
such that
\begin{myitemize}
\item[---] for each quadruple of objects \((W,X,Y,Z)\) of $\cc$ the associativity holds:
\begin{diagram}[h=1.8em,LaTeXeqno]
\cc(W,X),\cc(X,Y),\cc(Y,Z) &\rTTo^{1,\kappa_{X,Y,Z}} &\cc(W,X),\cc(X,Z)
\\
\dTTo<{\kappa_{W,X,Y},1} &= &\dTTo>{\kappa_{W,X,Z}}
\\
\cc(W,Y),\cc(Y,Z) &\rTTo^{\kappa_{W,Y,Z}} &\cc(W,Z)
\label{dia-assoc-V-cat}
\end{diagram}
\item[---] for each pair of objects \((X,Y)\) of $\cc$
\begin{align}
\bigl[ \cc(X,Y) \rTTo^{\id_X,1} \cc(X,X),\cc(X,Y) \rTTo^{\kappa_{X,X,Y}} \cc(X,Y) \bigr] &=1,
\label{eq-idX1k1}
\\
\bigl[ \cc(X,Y) \rTTo^{1,\id_Y} \cc(X,Y),\cc(Y,Y) \rTTo^{\kappa_{X,Y,Y}} \cc(X,Y) \bigr] &=1.
\label{eq-1idYk1}
\end{align}
\end{myitemize}
\end{definition}

In detail \eqref{dia-assoc-V-cat} means equation $tr=lb(=\kappa_{W,X,Y,Z})$ where
\begin{multline}
\mcv\bigl(\cc(W,X);\cc(W,X)\bigr) \times \mcv\bigl(\cc(X,Y),\cc(Y,Z);\cc(X,Z)\bigr) \times \mcv\bigl(\cc(W,X),\cc(X,Z);\cc(W,Z)\bigr) 
\\
\hfill \rTTo^{\mu_{\mathsf{IV}\:\mb3\to\mb2}} \mcv\bigl(\cc(W,X),\cc(X,Y),\cc(Y,Z);\cc(W,Z)\bigr), \hskip\multlinegap
\\
(1_{\cc(W,X)},\kappa_{X,Y,Z},\kappa_{W,X,Z}) \mapsto tr,
\label{eq-V(CC)V(CCC)V(CCC)-V(CCCC)}
\end{multline}
\begin{multline}
\mcv\bigl(\cc(W,X),\cc(X,Y);\cc(W,Y)\bigr) \times \mcv\bigl(\cc(Y,Z);\cc(Y,Z)\bigr) \times \mcv\bigl(\cc(W,Y),\cc(Y,Z);\cc(W,Z)\bigr) 
	\\
\hfill \rTTo^{\mu_{\mathsf{VI}\:\mb3\to\mb2}} \mcv\bigl(\cc(W,X),\cc(X,Y),\cc(Y,Z);\cc(W,Z)\bigr), \hskip\multlinegap
	\\
(\kappa_{W,X,Y},1_{\cc(Y,Z)},\kappa_{W,Y,Z}) \mapsto lb.
\label{eq-V(CCC)V(CC)V(CCC)-V(CCCC)}
\end{multline}

%

\begin{proposition}\label{pro-und-VQu}
Let $\mcv$ be a locally small symmetric closed complete multicategory.
The symmetric multicategory $\VQu$ is equipped with the following.
Let \((\ca_i)_{i\in I}\), \(I\in\Ob\cs_\sk\), be a family of small $\mcv$\n-quivers.
Let $\cc$ be a small $\mcv$\n-category.
Then there is a small $\mcv$\n-category \(\und\VQu((\ca_i)_{i\in I};\cc)\) and a distinguished evaluation element
\[ \ev^\VQu_{(\ca_i)_{i\in I};\cc} \in \VQu\bigl((\ca_i)_{i\in I},\und\VQu((\ca_i)_{i\in I};\cc);\cc\bigr).
\]
\end{proposition}

\begin{proof}
Let \((\ca_i)_{i\in I}\), \(I\in\Ob\cs_\sk\), be a family of small $\mcv$\n-quivers.
Let $\cc$ be a small $\mcv$\n-category.
Define a small $\mcv$\n-quiver \(\und\VQu\bigl((\ca_i)_{i\in I};\cc\bigr)\) with
	\begin{myitemize}
		\item[---] \(\Ob\und\VQu\bigl((\ca_i)_{i\in I};\cc\bigr)=\VQu\bigl((\ca_i)_{i\in I};\cc\bigr)\);
\item[---] \(\und\VQu\bigl((\ca_i)_{i\in I};\cc\bigr)(F,G)=\) the object of $\mcv$\n-transformations \(F\to G\:(\ca_i)_{i\in I} \to\cc=\) the enriched end in $\mcv$
\[ \int_{(A_i\in\ca_i)_{i\in I}}\cc\bigl((A_i)_{i\in I}F,(A_i)_{i\in I}G\bigr)
\]
similar to \cite[\S~2.1]{KellyGM:bascec}, the equalizer in multicategory $\mcv$ of the pair of morphisms
	\end{myitemize}
	\begin{equation}
		\prod_{(A_i\in\ca_i)_{i\in I}} \hspace*{-1.1em} \cc\bigl((A_i)_{i\in I}F,(A_i)_{i\in I}G\bigr)
		\pile{\rTTo^{(\pr_{(D_i)}\centerdot\beta)} \\ \rTTo_{(\pr_{(A_i)}\centerdot\gamma)}} \hspace*{-0.3em}
		\prod_{(A_i,D_i\in\ca_i)_{i\in I}} \hspace*{-1.5em} \und\mcV\bigl((\ca_i(A_i,D_i))_{i\in I};\cc((A_i)_{i\in I}F,(D_i)_{i\in I}G)\bigr),
		\label{eq-int-c(FE-GE)-VQu}
	\end{equation}
	where \(\beta\:\cc\bigl((D_i)_{i\in I}F,(D_i)_{i\in I}G\bigr)\to\und{\mcV}\bigl((\ca_i(A_i,D_i))_{i\in I}; \cc((A_i)_{i\in I}F,(D_i)_{i\in I}G)\bigr)\) is adjunct to $\beta^\dagger$, obtained via
	\begin{multline}
		\mu_{\mathsf{\triangledown I}}\: \mcv\bigl((\ca_i(A_i,D_i))_{i\in I}; \cc((A_i)_{i\in I}F,(D_i)_{i\in I}F)\bigr)
		\\
		\times \mcv\bigl(\cc((D_i)_{i\in I}F,(D_i)_{i\in I}G);\cc((D_i)_{i\in I}F,(D_i)_{i\in I}G)\bigr)
		\\
\times \mcv\bigl(\cc((A_i)_{i\in I}F,(D_i)_{i\in I}F),\cc((D_i)_{i\in I}F,(D_i)_{i\in I}G); \cc((A_i)_{i\in I}F, (D_i)_{i\in I}G)\bigr) 
		\\
		\to \mcv\bigl((\ca_i(A_i,D_i))_{i\in I},\cc((D_i)_{i\in I}F,(D_i)_{i\in I}G);\cc((A_i)_{i\in I}F, (D_i)_{i\in I}G)\bigr),
		\\
		(F_{(A_i),(D_i)},1_{\cc((D_i)_{i\in I}F,(D_i)_{i\in I}G)},\centerdot) \mapsto \beta^\dagger,
\label{eq-beta-dagger}
	\end{multline}
	and \(\gamma\:\cc\bigl((A_i)_{i\in I}F,(A_i)_{i\in I}G\bigr)\to\und{\mcV}\bigl((\ca_i(A_i,D_i))_{i\in I}; \cc((A_i)_{i\in I}F,(D_i)_{i\in I}G)\bigr)\) is adjunct to \(\gamma^\dagger\), obtained via
	\begin{multline}
		\mu_{\mathsf{\triangledown I\centerdot X}}\: 
		\mcv\bigl(\cc((A_i)_{i\in I}F,(A_i)_{i\in I}G);\cc((A_i)_{i\in I}F,(A_i)_{i\in I}G)\bigr)
		\\
		\times\mcv\bigl((\ca_i(A_i,D_i))_{i\in I}; \cc((A_i)_{i\in I}G,(D_i)_{i\in I}G)\bigr)
		\\
		\times \mcv\bigl(\cc((A_i)_{i\in I}F,(A_i)_{i\in I}G),\cc((A_i)_{i\in I}G,(D_i)_{i\in I}G); \cc((A_i)_{i\in I}F, (D_i)_{i\in I}G)\bigr) 
		\\
		\to \mcv\bigl((\ca_i(A_i,D_i))_{i\in I},\cc((A_i)_{i\in I}F,(A_i)_{i\in I}G);\cc((A_i)_{i\in I}F, (D_i)_{i\in I}G)\bigr),
		\\
		(1_{\cc((A_i)_{i\in I}F,(A_i)_{i\in I}G)},G_{(A_i),(D_i)},\centerdot) \mapsto \gamma^\dagger.
\label{eq-gamma-dagger}
	\end{multline}
	Here \(\mathsf{\triangledown I}\:\mb{n+1}\to\mb2\), \(\mathsf{\triangledown I\centerdot X}=\bigl[\mb{n+1}\rto{\mathsf{\triangledown I}}\mb2\rTTo^{(12)} \mb2\bigr]\) (we read pictures from top to bottom).
	Notice that we may use Proposition from \ref{pro-mu-psi-r-sigma-mu-phi}.

\begin{definition}\label{def-evVQu}
	Define a multi-entry $\mcv$\n-quiver morphism
	\begin{align}
		\ev^\VQu\: (\ca_i)_{i\in I}, \und\VQu\bigl((\ca_i)_{i\in I};\cc\bigr) &\longrightarrow \cc \notag
		\\
		\bigl((A_i)_{i\in I};F\bigr) &\longmapsto (A_i)_{i\in I}F
\label{eq-ev-VQu-F(A)}
	\end{align}
	\begin{multline*}
		(\ca_i(A_i,D_i))_{i\in I}, \und\VQu\bigl((\ca_i)_{i\in I};\cc\bigr)(F,G) \rTTo^{(1)_I,\Xi}
		\\
\bigl(\ca_i(A_i,D_i)\bigr)_{i\in I},\und\mcV\bigl((\ca_i(A_i,D_i))_{i\in I};\cc((A_i)_{i\in I}F,(D_i)_{i\in I}G)\bigr) \rto{\ev^\mcv} \cc((A_i)_{i\in I}F,(D_i)_{i\in I}G),
	\end{multline*}
	where the morphism $\Xi$ is the diagonal in the \emph{commutative} square
\[ \hspace*{-3em}
\begin{diagram}[inline]
		\und\VQu\bigl((\ca_i)_{i\in I};\cc\bigr)(F,G) &\rTTo^{p_{(D_i)_{i\in I}}} &\cc((D_i)_{i\in I}F,(D_i)_{i\in I}G)
		\\
		\dTTo<{p_{(A_i)_{i\in I}}} &\rdTTo^\Xi &\dTTo>\beta
		\\
		\cc((A_i)_{i\in I}F,(A_i)_{i\in I}G) &\rTTo^\gamma &\und\mcV\bigl((\ca_i(A_i,D_i))_{i\in I};\cc((A_i)_{i\in I}F,(D_i)_{i\in I}G)\bigr)
\end{diagram}
\]
\end{definition}	

\begin{proof}[Detalisation]
Composing this diagram with \(\ev^\mcv\) as above, we obtain the morphism $\ev^\VQu$ as the diagonal in the commutative square
\begin{diagram}[LaTeXeqno,bottom]
	(\ca_i(A_i,D_i))_{i\in I},\und\VQu\bigl((\ca_i)_{i\in I};\cc\bigr)(F,G) &\rTTo^{(1)_I,p_{(D_i)_{i\in I}}} &(\ca_i(A_i,D_i))_{i\in I}, \cc((D_i)_{i\in I}F,(D_i)_{i\in I}G)
	\\
	\dTTo<{(1)_I,p_{(A_i)_{i\in I}}} &\rdTTo^{\ev^\VQu} &\dTTo>{\beta^\dagger}
	\\
	(\ca_i(A_i,D_i))_{i\in I},\cc((A_i)_{i\in I}F,(A_i)_{i\in I}G) &\rTTo^{\gamma^\dagger} &\cc((A_i)_{i\in I}F,(D_i)_{i\in I}G)
	\label{dia-Ai-VQu-B}
\end{diagram}
It is given below on the left.
Applying the associativity property at \figref{dia-assoc-mu-multi} for maps \(I\sqcup\mb1\rto\id I\sqcup\mb1\rto{\mathsf{\triangledown I}} \mb2\) we rewrite this expression getting the one on the right:
\begin{gather*}
\begin{array}{c}
	\prod_{i\in I}\mcv\bigl(\ca_i(A_i,D_i);\ca_i(A_i,D_i)\bigr) \times \mcv\bigl(\und\VQu\bigl((\ca_i)_{i\in I};\cc\bigr)(F,G);\cc((D_i)_{i\in I}F,(D_i)_{i\in I}G)\bigr)
	\\
	\times \mcv\bigl((\ca_i(A_i,D_i))_{i\in I}; \cc((A_i)_{i\in I}F,(D_i)_{i\in I}F)\bigr)
	\times \mcv\bigl(\cc((D_i)_{i\in I}F,(D_i)_{i\in I}G);\cc((D_i)_{i\in I}F,(D_i)_{i\in I}G)\bigr)
	\\
\times \mcv\bigl(\cc((A_i)_{i\in I}F,(D_i)_{i\in I}F),\cc((D_i)_{i\in I}F,(D_i)_{i\in I}G); \cc((A_i)_{i\in I}F, (D_i)_{i\in I}G)\bigr) 
\end{array}
\\
\vstretch 35
\begin{tanglec}
	\id \hstep \id \step[1.5] \id \step
	\\
	\id \hstep \id \hstep \ffbox2{\mu_{\mathsf{\triangledown I}}}
	\\
	\id \hstep \id \step[1.5] \id \step
	\\
	\ffbox3{\mu_{1_{I\sqcup\mb1}}} \step
	\\
	\id \step
\end{tanglec}
=\;
\begin{tanglec}
	\id \step[2.5] \id \step[1.5] \id 
	\\
	\ffbox2{\mu_{1_I}} \hstep \ffbox2{\mu_{1_{\mb1}}} \hstep \id \step 
	\\
	\id \step[2.5] \id \step[1.5] \id 
	\\
	\ffbox5{\mu_{\mathsf{\triangledown I}}}
	\\
	\id
\end{tanglec}
\\
\mcv\bigl((\ca_i(A_i,D_i))_{i\in I},\und\VQu\bigl((\ca_i)_{i\in I};\cc\bigr)(F,G);\cc((A_i)_{i\in I}F,(D_i)_{i\in I}G)\bigr)
\end{gather*}
On elements we have
\begin{diagram}[LaTeXeqno,h=1.4em]
&&\hspace*{-7em} \bigl((1_{\ca_i(A_i,D_i)})_{i\in I},p_{(D_i)_{i\in I}},F_{(A_i),(D_i)},1_{\cc((D_i)_{i\in I}F,(D_i)_{i\in I}G)},\centerdot\bigr) \hspace*{-7em} &&
\\
&\ldMapsTo &&\rdMapsTo &
\\
\bigl((1_{\ca_i(A_i,D_i)})_{i\in I},p_{(D_i)_{i\in I}},\beta^\dagger\bigr) &&&&\bigl(F_{(A_i),(D_i)},p_{(D_i)_{i\in I}},\centerdot\bigr)
\\
&\rdMapsTo &&\ldMapsTo &
\\
&&\ev^\VQu &&
\label{eq-ev-VQu-E}
\end{diagram}

Looking at another path of commutative diagram \eqref{dia-Ai-VQu-B} we get another presentation of $\ev^\VQu$ given below on the left.
Applying the associativity property at \figref{dia-assoc-mu-multi} for maps \(I\sqcup\mb1\rto\id I\sqcup\mb1\rto{\mathsf{\triangledown I\centerdot X}} \mb2\) we rewrite this expression on the right:
\begin{gather*}
\begin{array}{c}
	\prod_{i\in I}\mcv\bigl(\ca_i(A_i,D_i);\ca_i(A_i,D_i)\bigr) \times \mcv\bigl(\und\VQu\bigl((\ca_i)_{i\in I};\cc\bigr)(F,G);\cc((A_i)_{i\in I}F,(A_i)_{i\in I}G)\bigr)
	\\
	\times \mcv\bigl(\cc((A_i)_{i\in I}F,(A_i)_{i\in I}G);\cc((A_i)_{i\in I}F,(A_i)_{i\in I}G)\bigr)
	\times \mcv\bigl((\ca_i(A_i,D_i))_{i\in I}; \cc((A_i)_{i\in I}G,(D_i)_{i\in I}G)\bigr)
	\\
	\times \mcv\bigl(\cc((A_i)_{i\in I}F,(A_i)_{i\in I}G),\cc((A_i)_{i\in I}G,(D_i)_{i\in I}G);\cc((A_i)_{i\in I}F, (D_i)_{i\in I}G)\bigr)
\end{array}
\\
\vstretch 35
\begin{tanglec}
	\id \hstep \id \step[1.5] \id \step
	\\
	\id \hstep \id \hstep \ffbox2{\mu_{\mathsf{\triangledown I\centerdot X}}}
	\\
	\id \hstep \id \step[1.5] \id \step
	\\
	\ffbox3{\mu_{1_{I\sqcup\mb1}}} \step
	\\
	\id \step
\end{tanglec}
=\;
\begin{tanglec}
	\id \step[2.5] \id \step[1.5] \id 
	\\
	\ffbox2{\mu_{1_{\mb1}}} \hstep \ffbox2{\mu_{1_I}} \hstep \id \step 
	\\
	\id \step[2.5] \id \step[1.5] \id 
	\\
	\ffbox5{\mu_{\mathsf{\triangledown I\centerdot X}}}
	\\
	\id
\end{tanglec}
\\
\mcv\bigl((\ca_i(A_i,D_i))_{i\in I},\und\VQu\bigl((\ca_i)_{i\in I};\cc\bigr)(F,G);\cc((A_i)_{i\in I}F,(D_i)_{i\in I}G)\bigr)
\end{gather*}
On elements we have
\begin{diagram}[LaTeXeqno,h=1.4em]
&&\hspace*{-7em} \bigl((1_{\ca_i(A_i,D_i)})_{i\in I},p_{(A_i)_{i\in I}},1_{\cc((A_i)_{i\in I}F,(A_i)_{i\in I}G)},G_{(A_i),(D_i)},\centerdot\bigr) \hspace*{-7em} &&
	\\
	&\ldMapsTo &&\rdMapsTo &
	\\
\bigl((1_{\ca_i(A_i,D_i)})_{i\in I},p_{(A_i)_{i\in I}},\gamma^\dagger\bigr) &&&&\bigl(p_{(A_i)_{i\in I}},G_{(A_i),(D_i)},\centerdot\bigr)
	\\
	&\rdMapsTo &&\ldMapsTo &
	\\
	&&\ev^\VQu &&
\label{eq-ev-VQu-A}
\end{diagram}
Thus, \eqref{eq-ev-VQu-E} and \eqref{eq-ev-VQu-A} are giving the same element \(\ev^\VQu\).
\end{proof}

There is a composite map
\begin{multline}
	\VQu\bigl((\cb_j)_{j\in J};\und\VQu((\ca_i)_{i\in I};\cc)\bigr) \rTTo^{[\prod_{i\in I}\dot1_{\ca_i}]\times1\times\dot\ev^\VQu_{(\ca_i)_{i\in I};\cc}}
	\\
	\bigl[\prod_{i\in I}\VQu(\ca_i;\ca_i)\bigr]\times\VQu\bigl((\cb_j)_{j\in J};\und\VQu((\ca_i)_{i\in I};\cc)\bigr)
	\times\VQu\bigl((\ca_i)_{i\in I},\und\VQu((\ca_i)_{i\in I};\cc);\cc\bigr)
	\\
	\rTTo^{\mu^\VQu_{\id\sqcup\triangledown\:I\sqcup J\to I\sqcup\mb1}} \VQu\bigl((\ca_i)_{i\in I},(\cb_j)_{j\in J};\cc\bigr)
	\label{eq-phi(XYZ)-iso-VQu}
\end{multline}
for an arbitrary sequence \((\cb_j)_{j\in J}\), \(J\in\Ob\cs_\sk\), of $\mcv$\n-quivers.
	
Consider an element \(f\:(\cb_j)_{j\in J}\to\und\VQu\bigl((\ca_i)_{i\in I};\cc\bigr)\in\VQu\):
\begin{diagram}[w=1em,nobalance,LaTeXeqno]
	f\: &(\cb_j)_{j\in J} &\rTTo &\und\VQu\bigl((\ca_i)_{i\in I};\cc\bigr) \hspace*{-3em}
	\\
	&(B_j)_{j\in J} &\rMapsTo &(B_j)_{j\in J}f\: &(\ca_i)_{i\in I} &\rTTo &\cc
	\\
	&&&&(A_i)_{i\in I} &\rMapsTo &(A_i)_{i\in I}(B_j)_{j\in J}f
	\\
	&&&&\hspace*{-10.6em} (B_j)_{j\in J}f_{(A_i),(D_i)}\: (\ca_i(A_i,D_i))_{i\in I} &\rTTo &\cc((A_i)_{i\in I}(B_j)_{j\in J}f,(D_i)_{i\in I}(B_j)_{j\in J}f)
\label{dia-f-(B)-VQu((A)C)}
\end{diagram} 
\begin{multline}
	(\cb_j(B_j,E_j))_{j\in J} \to \und\VQu\bigl((\ca_i)_{i\in I};\cc\bigr) \bigl((B_j)_{j\in J}f,(E_j)_{j\in J}f\bigr)
	\\
	=\int_{(A_i\in\ca_i)_{i\in I}}\cc((A_i)_{i\in I}(B_j)_{j\in J}f,(A_i)_{i\in I}(E_j)_{j\in J}f).
\label{eq-(B)-VQu((A)C)((B)f(E)f)}
\end{multline}

\begin{lemma}
Map \eqref{eq-(B)-VQu((A)C)((B)f(E)f)} admits two presentations described below as \eqref{eq-(B)f-V(p(D))} and \eqref{eq-V(p(A))-(E)f}.
\end{lemma}

\begin{proof}
Apply composition \eqref{eq-phi(XYZ)-iso-VQu} to \eqref{eq-(B)-VQu((A)C)((B)f(E)f)} for this $f$.
We get
\begin{multline*}
\mcV\bigl((\cb_j(B_j,E_j))_{j\in J};\und\VQu((\ca_i)_{i\in I};\cc)((B_j)_{j\in J}f,(E_j)_{j\in J}f)\bigr) \rTTo^{[\prod_{i\in I}\dot1_{\ca_i(A_i,D_i)}]\times1\times\dot\ev^\VQu_{(\ca_i)_{i\in I};\cc}}
	\\
\bigl[\prod_{i\in I}\mcV(\ca_i(A_i,D_i);\ca_i(A_i,D_i))\bigr]
\times \mcV\bigl((\cb_j(B_j,E_j))_{j\in J};\und\VQu((\ca_i)_{i\in I};\cc)((B_j)_{j\in J}f,(E_j)_{j\in J}f)\bigr) \times
\\
\mcV\bigl((\ca_i(A_i,D_i))_{i\in I},\und\VQu((\ca_i)_{i\in I};\cc)((B_j)_{j\in J}f,(E_j)_{j\in J}f);
\cc((A_i)_{i\in I}(B_j)_{j\in J}f,(D_i)_{i\in I}(E_j)_{j\in J}f)\bigr)
	\\
\rTTo^{\mu^\mcV_{\id\sqcup\triangledown\:I\sqcup J\to I\sqcup\mb1}}
\mcV\bigl((\ca_i(A_i,D_i))_{i\in I},(\cb_j(B_j,E_j))_{j\in J};\cc((A_i)_{i\in I}(B_j)_{j\in J}f,(D_i)_{i\in I}(E_j)_{j\in J}f)\bigr).
\end{multline*}
Notice that \([(A_i)_{i\in I},(B_j)_{j\in J}f]\ev=(A_i)_{i\in I}(B_j)_{j\in J}f\).
Substituting \eqref{eq-ev-VQu-E} for \(\ev^\VQu\) we get the first map in the following diagram.
Apply associativity condition at \figref{dia-assoc-mu-multi} to maps \(I\sqcup J \rTTo^{1\sqcup\triangledown} I\sqcup\mb1\rto{\mathsf{\triangledown I}} \mb2\).
One obtains the second map.
Using the unitality of $\mcv$ we reduce map \eqref{eq-phi(XYZ)-iso-VQu} applied to \eqref{eq-(B)-VQu((A)C)((B)f(E)f)} for $f$ to the third map in
\begin{gather}
\mcV\bigl((\cb_j(B_j,E_j))_{j\in J};\und\VQu((\ca_i)_{i\in I};\cc)((B_j)_{j\in J}f,(E_j)_{j\in J}f)\bigr)
\label{eq-(B)f-V(p(D))}
\\
\begin{tanglec}
	\hh \id \step[6]
	\\
\ffbox6{\prod_{i\in I}\!\dot1_{\ca_i(A_i,D_i)}} \hstep \id \hstep \ffbox7{\dot{(B_j)_{j\in J}f}_{(A_i),(D_i)}} \hstep \ffbox3{\dot p_{(D_i)_{i\in I}}\!} \hstep \ffbox1{\dot\centerdot}
	\\
	\hh \step[2.5] \id \step[3.5] \id \step[4] \id \step[5.5] \id \step[2.5] \id
	\\
	\hh \step[3] \id \step[3.5] \id \step[4] \id \step[4.5] \ffbox4{\mu_{\mathsf{\triangledown I}}}
	\\
	\hh \step \id \step[3.5] \id \step[4] \id \step[6.5] \id
	\\
	\hh \step \ffbox{16}{\mu_{1\sqcup\triangledown\:I\sqcup J\to I\sqcup\mb1}}
	\\
	\hh \step \id
\end{tanglec}
\;=\;
\begin{tanglec}
	\hh \id \step[6]
	\\
\ffbox6{\prod_{i\in I}\!\dot1_{\ca_i(A_i,D_i)}} \hstep \id \hstep \ffbox7{\dot{(B_j)_{j\in J}f}_{(A_i),(D_i)}} \hstep \ffbox3{\dot p_{(D_i)_{i\in I}}\!} \hstep \ffbox1{\dot\centerdot}
	\\
	\step[2.5] \id \step[3.5] \XX \step[7.5] \id \step[2.5] \id
	\\
	\hh \step[1.5] \ffbox5{\mu_{1_I}} \step \ffbox9{\mu_{\triangledown\:J\to\mb1}} \step[1.5] \id
	\\
	\hh \step[4] \id \step[8] \id \step[6] \id
	\\
	\hh \step[4] \ffbox{15}{\mu_{\triangledown\triangledown\:I\sqcup J\to\mb2}}
	\\
	\hh \step[4] \id
\end{tanglec}
\notag
\\
=\;
\begin{tanglec}
	\hh \step[6] \id
	\\
	\ffbox7{\dot{(B_j)_{j\in J}f}_{(A_i),(D_i)}} \hstep \ffbox7{\mcv((1)_J;\dot p_{(D_i)_{i\in I}})} \hstep \ffbox1{\dot\centerdot}
	\\
	\hh \step[3] \id \step[7.5] \id \step[4.5] \id
	\\
	\hh \step[3] \ffbox{13}{\mu_{\triangledown\triangledown\:I\sqcup J\to\mb2}}
	\\
	\hh \step[3] \id
\end{tanglec}
\notag
\\
\mcV\bigl((\ca_i(A_i,D_i))_{i\in I},(\cb_j(B_j,E_j))_{j\in J};\cc((A_i)_{i\in I}(B_j)_{j\in J}f,(D_i)_{i\in I}(E_j)_{j\in J}f)\bigr) \notag
\end{gather}
\\

On the other hand, substituting \eqref{eq-ev-VQu-A} for \(\ev^\VQu\) we get the first map in the following diagram.
Using \figref{dia-assoc-mu-multi} for maps \(I\sqcup J\rTTo^{1\sqcup\triangledown} I\sqcup\mb 1\rTTo^{\triangledown\mathsf{I\centerdot X}} \mb2\) we rewrite this composition as the second map.
Using the unitality of $\mcv$ we reduce map \eqref{eq-phi(XYZ)-iso-VQu} applied to \eqref{eq-(B)-VQu((A)C)((B)f(E)f)} for $f$ to the third map in
\begin{gather}
\mcV\bigl((\cb_j(B_j,E_j))_{j\in J};\und\VQu((\ca_i)_{i\in I};\cc)((B_j)_{j\in J}f,(E_j)_{j\in J}f)\bigr)
\label{eq-V(p(A))-(E)f}
\\
\begin{tanglec}
	\hh \id \step[6]
	\\
\ffbox6{\prod_{i\in I}\!\dot1_{\ca_i(A_i,D_i)}} \hstep \id \hstep \ffbox3{\dot p_{(D_i)_{i\in I}}\!} \hstep \ffbox7{\dot{(E_j)_{j\in J}f}_{(A_i),(D_i)}} \hstep \ffbox1{\dot\centerdot}
	\\
	\step[2.5] \id \step[3.5] \id \step[3] \XX \step[7] \id
	\\
	\hh \step[3] \id \step[3.5] \id \step[2.5] \ffbox{10}{\mu_{\mathsf{\triangledown I}}}
	\\
	\hh \id \step[3.5] \id \step[7.5] \id \Step
	\\
	\hh \ffbox{12}{\mu_{1\sqcup\triangledown\:I\sqcup J\to I\sqcup\mb1}} \Step
	\\
	\hh \id \Step
\end{tanglec}
\;=\;
\begin{tanglec}
	\hh \id \step[6]
	\\
\ffbox6{\prod_{i\in I}\!\dot1_{\ca_i(A_i,D_i)}} \hstep \id \hstep \ffbox3{\dot p_{(D_i)_{i\in I}}\!} \hstep \ffbox7{\dot{(E_j)_{j\in J}f}_{(A_i),(D_i)}} \hstep \ffbox1{\dot\centerdot}
	\\
	\step[4] \XX \Step \id \step[4.5] \ne2 \step[4.5] \id
	\\
	\step[3] \ne2 \Step \XX \step[2.5] \ne2 \step[6.5] \id
	\\
	\hh \step[1.5] \ffbox5{\mu_{\triangledown\:J\to\mb1}} \step \ffbox3{\mu_{1_I}} \step[7.5] \id
	\\
	\hh \step[4] \id \step[5] \id \step[9] \id
	\\
	\hh \step[4] \ffbox{15}{\mu_{(\triangledown\sqcup\triangledown)\centerdot\mathsf{X}\:I\sqcup J\to\mb2}}
	\\
	\hh \step[4] \id
\end{tanglec}
\notag
\\
=\;
\begin{tanglec}
	\hh \step[6] \id
	\\
	\ffbox7{\mcv((1)_J;\dot p_{(D_i)_{i\in I}})} \hstep \ffbox7{\dot{(E_j)_{j\in J}f}_{(A_i),(D_i)}} \hstep \ffbox1{\dot\centerdot}
	\\
	\hh \step[3] \id \step[7.5] \id \step[4.5] \id
	\\
	\hh \step[3] \ffbox{13}{\mu_{(\triangledown\sqcup\triangledown)\centerdot\mathsf{X}\:I\sqcup J\to\mb2}}
	\\
	\hh \step[3] \id
\end{tanglec}
\notag
\\
\mcV\bigl((\ca_i(A_i,D_i))_{i\in I},(\cb_j(B_j,E_j))_{j\in J};\cc((A_i)_{i\in I}(B_j)_{j\in J}f,(D_i)_{i\in I}(E_j)_{j\in J}f)\bigr) \notag
\end{gather}
	\\
This element equals element \eqref{eq-(B)f-V(p(D))}.
\end{proof}

\begin{lemma}
Let $\mcv$ be a locally small symmetric closed complete multicategory.
The $\mcv$\n-subquiver \(\und\VQu((\ca_i)_{i\in I};\cc)\) embedded via 
\[ \iota\:\und\VQu((\ca_i)_{i\in I};\cc)(F,G)\hookrightarrow\prod_{(A_i\in\ca_i)}\cc((A_i)_{i\in I}F,(A_i)_{i\in I}G)
\]
is a $\mcv$\n-subcategory.
\end{lemma}

\begin{proof}
The vertical composition of objects of $\mcV$\n-transformations \(\und\VQu((\ca_i)_{i\in I},\cc)(F,G)\) comes from the composition in $\cc$:
\begin{equation}
\begin{diagram}[h=2.4em,inline]
\und\VQu((\ca_i)_{i\in I};\cc)(F,G),\und\VQu((\ca_i)_{i\in I};\cc)(G,H) &\rTTo^{\exists?\centerdot} &\und\VQu((\ca_i)_{i\in I},\cc)(F,H)
\\
\dTTo<{\iota,\iota} &= &\dTTo>\iota
\\
\prod_{(A_i\in\ca_i)_{i\in I}}\hspace*{-1em}\cc((A_i)_{i\in I}F,(A_i)_{i\in I}G), \hspace*{-0.8em}\prod_{(A_i\in\ca_i)_{i\in I}}\hspace*{-1em}\cc((A_i)_{i\in I}G,(A_i)_{i\in I}H) &\rTTo^{\exists!m} &\hspace*{-0.8em}\prod_{(A_i\in\ca_i)_{i\in I}}\hspace*{-1em}\cc((A_i)_{i\in I}F,(A_i)_{i\in I}H)
\\
\dTTo<{\pr_{(A_i)},\pr_{(A_i)}} &= &\dTTo>{\pr_{(A_i)}}
\\
\cc((A_i)_{i\in I}F,(A_i)_{i\in I}G),\cc((A_i)_{i\in I}G,(A_i)_{i\in I}H) &\rTTo^\centerdot &\cc((A_i)_{i\in I}F, (A_i)_{i\in I}H)
\end{diagram}
\label{dia-composition-undVCat}
\end{equation}

The multiplication $m$ exists due to the existence of products in multicategory $V$.
We have to prove the existence of the top arrow.
We use the abbreviation similar to that from Kelly's book \cite[\S~2.2]{KellyGM:bascec} \([(\ca_i)_{i\in I};\cc]=\und\VQu((\ca_i)_{i\in I};\cc)\).
First of all the exterior of the following diagram commutes
\[
\begin{diagram}[inline,h=2.2em]
&&\hspace*{-2em}[(\ca_i)_{i\in I};\cc](F,G),[(\ca_i)_{i\in I};\cc](G,H)\hspace*{-2em}
\\
&&\dTTo<{\iota,\iota}
\\
&&\hspace*{-8em}\prod_{(A_i\in\ca_i)_{i\in I}}\cc((A_i)_{i\in I}F,(A_i)_{i\in I}G),
\prod_{(A_i\in\ca_i)_{i\in I}}\cc((A_i)_{i\in I}G,(A_i)_{i\in I}H)\hspace*{-8em}
\\
&\ldTTo_{\pr_{(A_i)},\pr_{(A_i)}} &&\rdTTo_{\pr_{(D_i)},\pr_{(D_i)}} &
\\
\cc((A_i)_{i\in I}F,(A_i)_{i\in I}G),\cc((A_i)_{i\in I}G,(A_i)_{i\in I}H)\hspace*{-8em} &&&&\hspace*{-8em}\cc((D_i)_{i\in I}F,(D_i)_{i\in I}G),\cc((D_i)_{i\in I}G,(D_i)_{i\in I}H)
\\
\dTTo<\centerdot &= &\dTTo<m &= &\dTTo>\centerdot 
\\
\cc((A_i)_{i\in I}F,(A_i)_{i\in I}H) &\lTTo^{\pr_{(A_i)}} &\hspace*{-1em}\prod_{(A_i\in\ca_i)_{i\in I}}\hspace*{-1em}\cc((A_i)_{i\in I}F,(A_i)_{i\in I}H) &\rTTo^{\pr_{(D_i)}} &\cc((D_i)_{i\in I}F,(D_i)_{i\in I}H)
\\
&\rdTTo_\gamma &&\ldTTo_\beta &
\\
&&\hspace*{-3.5em}\und\mcV\bigl(((\ca_i(A_i,D_i)_{i\in I};\cc((A_i)_{i\in I}F,(D_i)_{i\in I}H)\bigr)\hspace*{-3.5em}
\end{diagram}
\]
In fact, it is adjoint to the equation $a=c$, where elements $a,b,c$ are introduced below.
\begin{multline*}
\mu_{\mathsf{III}}\: \mcV\bigl(((\ca_i(A_i,D_i)_{i\in I};\cc((A_i)_{i\in I}F,(D_i)_{i\in I}F)\bigr)
\times \mcv\bigl([(\ca_i)_{i\in I};\cc](F,G);\cc((D_i)_{i\in I}F,(D_i)_{i\in I}G)\bigr)
\\
\times \mcv\bigl([(\ca_i)_{i\in I};\cc](G,H);\cc((D_i)_{i\in I}G,(D_i)_{i\in I}H)\bigr) \times
\\
\mcv\bigl(\cc((A_i)_{i\in I}F,(D_i)_{i\in I}F),\cc((D_i)_{i\in I}F,(D_i)_{i\in I}G),\cc((D_i)_{i\in I}G,(D_i)_{i\in I}H);
\cc((A_i)_{i\in I}F,(D_i)_{i\in I}H)\bigr)
\\
\to \mcv\bigl(((\ca_i(A_i,D_i)_{i\in I},[(\ca_i)_{i\in I};\cc](F,G),[(\ca_i)_{i\in I};\cc](G,H);\cc((A_i)_{i\in I}F,(D_i)_{i\in I}H)\bigr),
\\
\hfill (F_{(A_i),(D_i)},p_{(D_i)},p_{(D_i)},\kappa_{(A_i)_{i\in I}F,(D_i)_{i\in I}F,(D_i)_{i\in I}G,(D_i)_{i\in I}H}) \mapsto a, \hskip\multlinegap
\\
\hskip\multlinegap \mu_{\mathsf{XI}}\: \mcv\bigl([(\ca_i)_{i\in I};\cc](F,G);\cc((A_i)_{i\in I}F,(A_i)_{i\in I}G)\bigr) 
\times \mcv\bigl(((\ca_i(A_i,D_i)_{i\in I};\cc((A_i)_{i\in I}G,(D_i)_{i\in I}G)\bigr) \hfill
\\
\times \mcv\bigl([(\ca_i)_{i\in I};\cc](G,H);\cc((D_i)_{i\in I}G,(D_i)_{i\in I}H)\bigr) \times
\\
\mcv\bigl(\cc((A_i)_{i\in I}F,(A_i)_{i\in I}G),\cc((A_i)_{i\in I}G,(D_i)_{i\in I}G),\cc((D_i)_{i\in I}G,(D_i)_{i\in I}H);
\cc((A_i)_{i\in I}F,(D_i)_{i\in I}H)\bigr)
\\
\to \mcv\bigl(((\ca_i(A_i,D_i)_{i\in I},[(\ca_i)_{i\in I};\cc](F,G),[(\ca_i)_{i\in I};\cc](G,H);\cc((A_i)_{i\in I}F,(D_i)_{i\in I}H)\bigr),
\\
\hfill (p_{(A_i)},G_{(A_i),(D_i)},p_{(D_i)},\kappa_{(A_i)_{i\in I}F,(A_i)_{i\in I}G,(D_i)_{i\in I}G,(D_i)_{i\in I}H}) \mapsto b, \hskip\multlinegap
\\
\mu_{(321)}\: \mcv\bigl([(\ca_i)_{i\in I};\cc](F,G);\cc((A_i)_{i\in I}F,(A_i)_{i\in I}G)\bigr) \times \mcv\bigl([(\ca_i)_{i\in I};\cc](G,H);\cc((A_i)_{i\in I}G,(A_i)_{i\in I}H)\bigr)
\\
\times \mcv\bigl(((\ca_i(A_i,D_i)_{i\in I};\cc((A_i)_{i\in I}H,(D_i)_{i\in I}H)\bigr) \times
\\
\mcv\bigl(\cc((A_i)_{i\in I}F,(A_i)_{i\in I}G),\cc((A_i)_{i\in I}G,(A_i)_{i\in I}H),\cc((A_i)_{i\in I}H,(D_i)_{i\in I}H);
\cc((A_i)_{i\in I}F,(D_i)_{i\in I}H)\bigr)
\\
\to \mcv\bigl(((\ca_i(A_i,D_i)_{i\in I},[(\ca_i)_{i\in I};\cc](F,G),[(\ca_i)_{i\in I};\cc](G,H);\cc((A_i)_{i\in I}F,(D_i)_{i\in I}H)\bigr),
\\
(p_{(A_i)},p_{(A_i)},H_{(A_i),(D_i)},\kappa_{(A_i)_{i\in I}F,(A_i)_{i\in I}G,(A_i)_{i\in I}H,(D_i)_{i\in I}H}) \mapsto c.
\end{multline*}
The elements \(\kappa_{(A_i)_{i\in I}F,(A_i)_{i\in I}G,(A_i)_{i\in I}H,(D_i)_{i\in I}H}\) refer to iterated composition in $\cc$.
Notice that actually $a=b=c$.
Equality between elements $a$, $b$, $c$ follows from the properties of \([(\ca_i)_{i\in I};\cc]=\und\VQu((\ca_i)_{i\in I};\cc)\).

The two (top) commutative squares imply that there is a unique arrow
\[ \centerdot\in\mcv\bigl([(\ca_i)_{i\in I};\cc](F,G),[(\ca_i)_{i\in I};\cc](G,H);[(\ca_i)_{i\in I};\cc](F,H)\bigr),
\]
denoted \(\exists?\) in diagram~\eqref{dia-composition-undVCat} in $\mcv$ which makes the diagram commutative.

Associativity of composition in $\cc$ implies associativity of composition $m$ in diagram~\eqref{dia-composition-undVCat}.
Hence the upper multiplication $\centerdot$ is associative as well.

The identity transformation \(\id_F\:()\to\und\VCat((\ca_i)_{i\in I};\cc)(F,F)\) is 
\[ \id_F=\bigl(\id_{(A_i)_{i\in I}F}\:()\to\cc((A_i)_{i\in I}F,(A_i)_{i\in I}F)\bigr)_{(A_i\in\ca_i)_{i\in I}}.
\]
It is a natural $\mcV$\n-transformation in the sense of \defref{def-Natural-V-transformation}, since the square
\begin{gather*}
(\ca_i(A_i,D_i))_{i\in I} \rto{F_{(A_i),(D_i)},\id_{(D_i)F}} \cc((A_i)_{i\in I}F,(D_i)_{i\in I}F), \cc((D_i)_{i\in I}F,(D_i)_{i\in I}F)
\\
\begin{diagram}[h=1.2em]
&\hspace*{18em} &
	\\
\dTTo<{\id_{(A_i)F},F_{(A_i),(D_i)}} &\rdTTo^{F_{(A_i),(D_i)}} &\dTTo>{\kappa_{(A_i)F,(D_i)F,(D_i)F}}
	\\
&\hspace*{18em} &
\end{diagram}
\\[-1em]
\cc((A_i)_{i\in I}F,(A_i)_{i\in I}F),\cc((A_i)_{i\in I}F,(D_i)_{i\in I}F) \rto{\kappa_{(A_i)F,(A_i)F,(D_i)F}} \cc((A_i)_{i\in I}F,(D_i)_{i\in I}F)
\end{gather*}
commutes.
The both triangles commute in $\mcv$ due to $\id$ being units of $\cc$.
\end{proof}
This proves \propref{pro-und-VQu}.
\end{proof}

\begin{example}
	Assume that $\cv$ is a complete closed symmetric monoidal category.
	For \(\mcv=\wh\cv\) (see \cite[Proposition~3.22]{BesLyuMan-book}) we get
	\begin{multline*}
		\beta^\dagger =\bigl[ \tens^{I\sqcup\mb1}[(\ca_i(A_i,D_i))_{i\in I},\cc((D_i)_{i\in I}F,(D_i)_{i\in I}G)]
		\\
		\rTTo^{F_{(A_i),(D_i)}\tens1} \cc((A_i)_{i\in I}F,(D_i)_{i\in I}F)\tens\cc((D_i)_{i\in I}F,(D_i)_{i\in I}G) \rto\centerdot \cc((A_i)_{i\in I}F,(D_i)_{i\in I}G) \bigr],
	\end{multline*}
	\begin{multline*}
		\gamma^\dagger =\bigl[ \tens^{I\sqcup\mb1}[(\ca_i(A_i,D_i))_{i\in I},\cc((A_i)_{i\in I}F,(A_i)_{i\in I}G)] \rTTo^{G_{(A_i),(D_i)}\tens1}
		\\
		\cc((A_i)_{i\in I}G,(D_i)_{i\in I}G)\tens\cc((A_i)_{i\in I}F,(A_i)_{i\in I}G)
		\\
\rto c \cc((A_i)_{i\in I}F,(A_i)_{i\in I}G)\tens\cc((A_i)_{i\in I}G,(D_i)_{i\in I}G) \rto\centerdot \cc((A_i)_{i\in I}F,(D_i)_{i\in I}G) \bigr].
	\end{multline*}
\end{example}

\begin{example}
\(\mcv=\wh\Set\), \(\VCat=\Cat\).
The quiver \(\und\VQu\bigl((\ca_i)_{i\in I};\cc\bigr)\) has
\begin{myitemize}
\item[---] \(\Ob\und\VQu\bigl((\ca_i)_{i\in I};\cc\bigr)=\VQu\bigl((\ca_i)_{i\in I};\cc\bigr)\);
\item[---] \(\und\VQu\bigl((\ca_i)_{i\in I};\cc\bigr)(F,G)=\int_{(A_i\in\ca_i)_{i\in I}}\cc((A_i)_{i\in I}F,(A_i)_{i\in I}G)\).
\end{myitemize}

\(g\in\VQu\bigl((\ca_i)_{i\in I},(\cb_j)_{j\in J};\cc\bigr)\) consists of
\begin{myitemize}
\item[---] a function \(g=\Ob g\:(\prod_{i\in I}\Ob\ca_i)\times(\prod_{j\in J}\Ob\cb_j)\to\Ob\cc\);
\item[---] elements \(g=g_{(A_i),(B_j),(D_i),(E_j)}\in\)
\[ \mcv\bigl((\ca_i(A_i,D_i))_{i\in I},(\cb_j(B_j,E_j))_{j\in J};\cc(((A_i)_{i\in I},(B_j)_{j\in J})g, ((D_i)_{i\in I},(E_j)_{j\in J})g)\bigr).
\]
\end{myitemize}

Consider an element \(f\:(\cb_j)_{j\in J}\to\und\VQu\bigl((\ca_i)_{i\in I};\cc\bigr)\in\VQu\) given by \eqref{dia-f-(B)-VQu((A)C)} and \eqref{eq-(B)-VQu((A)C)((B)f(E)f)}.
Map \eqref{eq-(B)-VQu((A)C)((B)f(E)f)} induces a map
\[ h_{(A_i)}\:(\cb_j(B_j,E_j))_{j\in J}\to\cc((A_i)_{i\in I}(B_j)_{j\in J}f,(A_i)_{i\in I}(E_j)_{j\in J}f).
\]
Let \(\alpha_i\in\ca_i(A_i,D_i)\), \(i\in I\), \(\beta_j\in\cb_j(B_j,E_j)\), \(j\in J\).
From the equality of compositions \eqref{eq-(B)f-V(p(D))} and \eqref{eq-V(p(A))-(E)f} we deduce that the square
\begin{diagram}
(A_i)_{i\in I}(B_j)_{j\in J}f &\rTTo^{(\alpha_i)(B_j)_{j\in J}f_{(A_i),(D_i)}} &(D_i)_{i\in I}(B_j)_{j\in J}f
\\
\dTTo<{(\beta_j)h_{(A_i)}} &&\dTTo>{(\beta_j)h_{(D_i)}}
\\
(A_i)_{i\in I}(E_j)_{j\in J}f &\rTTo^{(\alpha_i)(E_j)_{j\in J}f_{(A_i),(D_i)}} &(D_i)_{i\in I}(E_j)_{j\in J}f
\end{diagram}
commutes in $\cc$.
\end{example}

\subsection{Multicategory of \texorpdfstring{$\mcV$}V-categories}
$\mcv$\n-functors were defined in \cite[\S1, (MLC 4)]{Linton:multiYoneda}, translated to a modern language in \cite[\S2]{1802.07538}, and by Leinster \cite[Example~2.4.1.iii]{Leinster:math.CT/9901139}.
They can be recognised as $T$\n-algebra morphisms where $T$ comes from \eqref{dia-monad-T} (compare with \exaref{exa-I1}).
In fact, a morphism of $T$\n-algebras \(F\:(\cA,\xi)\to(\cb,\psi)\), that is a morphism of $\mcv$\n-quivers \(F\:\ca\to\cb\) which satisfies the equation
\begin{diagram}
	\ca T &\rTTo^\xi &\ca
	\\
	\dTTo<{FT} &= &\dTTo>F
	\\
	\cb T &\rTTo^\psi &\cb
\end{diagram}
can be described as a morphism of $\mcv$\n-quivers which satisfies \eqref{dia-AAABBB} and \eqref{eq-id-AAA-BAFAF}.

We shall use a version with several inputs, based on the exterior monoidal structure of \(\FF\mcv\Qu\) due to monoidal structure of the prop \(\FF\mcv\).

\begin{definition}\label{def-multi-entry-V-functor}
Let $\mcV$ be a locally small symmetric multicategory.
Let $\cb$, $\ca_i$, $i\in I$, be small $\mcv$\n-categories.
A multi-entry $\mcv$\n-functor\index{multi-entry functor} \(F\:(\ca_i)_{i\in I}\to\cb\) is an $\FF\mcv$\n-functor \(F\:\boxt^{i\in I}\ca_i\to\cb\).
\end{definition}

\begin{proposition}\label{pro-multi-entry-V-functors}
A multi-entry $\mcv$\n-functor \(F\:(\ca_i)_{i\in I}\to\cb\) is identified with the following data:
\begin{myitemize}
\item[---] a function \(F=\Ob F\:\Ob\ca_1\times\dots\times\Ob\ca_I\to\Ob\cb\);
\item[---] a collection of elements \(F=F_{(A_i),(E_i)}\in\mcv\bigl((\ca_i(A_i,E_i))_{i\in I};\cb((A_i)_{i\in I}F,(E_i)_{i\in I}F)\bigr)\);
\end{myitemize}
such that \(lb=tr\) where these elements come from
\begin{multline*}
\mu_{\triangledown\triangledown}\: \mcv\bigl((\ca_i(A_i,D_i))_{i\in I};\cb((A_i)_{i\in I}F,(D_i)_{i\in I}F)\bigr)
\times\mcv\bigl((\ca_i(D_i,E_i))_{i\in I};\cb((D_i)_{i\in I}F,(E_i)_{i\in I}F)\bigr)
\\
\times\mcv\bigl(\cb((A_i)_{i\in I}F,(D_i)_{i\in I}F),\cb((D_i)_{i\in I}F,(E_i)_{i\in I}F); \cb((A_i)_{i\in I}F, (E_i)_{i\in I}F)\bigr)
\\
\to \mcv\bigl((\ca_i(A_i,D_i))_{i\in I},(\ca_i(D_i,E_i))_{i\in I};\cb((A_i)_{i\in I}F, (E_i)_{i\in I}F)\bigr),
\\
\hfill (F_{(A_i),(D_i)},F_{(D_i),(E_i)},\centerdot) \mapsto lb, \hskip\multlinegap
\\
\hskip\multlinegap \mu_\chi\: \prod_{i\in I}\mcv\bigl(\ca_i(A_i,D_i),\ca_i(D_i,E_i);\ca_i(A_i,E_i)\bigr) \times\mcv\bigl((\ca_i(A_i,E_i))_{i\in I};\cb((A_i)_{i\in I}F, (E_i)_{i\in I}F)\bigr) \hfill
\\
\to \mcv\bigl((\ca_i(A_i,D_i))_{i\in I},(\ca_i(D_i,E_i))_{i\in I};\cb((A_i)_{i\in I}F, (E_i)_{i\in I}F)\bigr),
\\
((\kappa_{A_i,D_i,E_i})_{i\in I},F_{(A_i),(E_i)}) \mapsto tr.
\end{multline*}
Here
\begin{align}
\triangledown\triangledown &=
\begin{pmatrix}
1 &\dots &n &n+1 &\dots &2n
\\
1 &\dots &1 &2 &\dots &2
\end{pmatrix}
\:\mb{2n} \to \mb2,
\label{eq-triangledowntriangledown}
\\
\chi &=
\begin{pmatrix}
1 &2 &\dots &n &n+1 &n+2 &\dots &2n
	\\
1 &2 &\dots &n &1 &2 &\dots &n
\end{pmatrix}
\:\mb{2n} \to \mb n.
\label{eq-chi}
\end{align}
Another requirement is coherence with the units
\begin{equation}
\bigl[ () \rTTo^{(\id_{A_i})_{i\in I}} (\ca_i(A_i,A_i))_{i\in I} \rTTo^{F_{(A_i),(A_i)}} \cb((A_i)_{i\in I}F,(A_i)_{i\in I}F) \bigr] =\id_{(A_i)_{i\in I}F}.
\label{eq-coherence-with-units}
\end{equation}
\end{proposition}

\begin{proof}
An $\FF\mcv$\n-functor \(F\:\boxt^{i\in I}\ca_i\to\cb\) consists of a map \(F=\Ob F\:\prod_{i\in I}\Ob\ca_i\to\Ob\cb\) and a collection of elements 
\[ F=F_{(A_i),(E_i)}\in\mcv\bigl((\ca_i(A_i,E_i))_{i\in I};\cb((A_i)_{i\in I}F,(E_i)_{i\in I}F)\bigr).
\]
The \(\FF\mcv\)\n-functor has to satisfy the equation
	\begin{diagram}[nobalance,LaTeXeqno]
		(\ca_i(A_i,D_i))_{i\in I},(\ca_i(D_i,E_i))_{i\in I} \rto{\lambda^{\sh}} (\ca_i(A_i,D_i),\ca_i(D_i,E_i))_{i\in I} \rto{(\kappa)_I} (\ca_i(A_i,E_i))_{i\in I}
		\\
		\dTTo>{F_{(A_i),(D_i)},F_{(D_i),(E_i)}} \hspace*{8em} = \hspace*{8em} \dTTo<{F_{(A_i),(E_i)}}
		\\
		\cb((A_i)_{i\in I}F,(D_i)_{i\in I}F),\cb((D_i)_{i\in I}F,(E_i)_{i\in I}F) \rto{\kappa} \cb((A_i)_{i\in I}F,(E_i)_{i\in I}F)
		\label{dia-ADEsh-kappa}
	\end{diagram}
	where the shuffle \(\sh\:I\sqcup I\to I\sqcup I\) is given for \(I=\mb n\) by
	\[ \sh=
	\begin{pmatrix}
		1 &2 &\dots &n    &n+1 &n+2 &\dots &2n
		\\
		1 &3 &\dots &2n-1 &2   &4   &\dots &2n
	\end{pmatrix}
	\big\downarrow.
	\]
	The element \(\lambda^{\sh}=(1_{\ca_i(A_i,D_i)},1_{\ca_i(D_i,E_i)})_{i\in I}\) is indexed by $\sh$.
	The left--bottom path of diagram~\eqref{dia-ADEsh-kappa} gives
	\begin{multline*}
		\mu_{\triangledown\triangledown}\: \mcv\bigl((\ca_i(A_i,D_i))_{i\in I};\cb((A_i)_{i\in I}F,(D_i)_{i\in I}F)\bigr)
		\times\mcv\bigl((\ca_i(D_i,E_i))_{i\in I};\cb((D_i)_{i\in I}F,(E_i)_{i\in I}F)\bigr)
		\\
		\times\mcv\bigl(\cb((A_i)_{i\in I}F,(D_i)_{i\in I}F),\cb((D_i)_{i\in I}F,(E_i)_{i\in I}F); \cb((A_i)_{i\in I}F, (E_i)_{i\in I}F)\bigr)
		\\
		\to \mcv\bigl((\ca_i(A_i,D_i))_{i\in I},(\ca_i(D_i,E_i))_{i\in I};\cb((A_i)_{i\in I}F, (E_i)_{i\in I}F)\bigr),
		\\
		(F_{(A_i),(D_i)},F_{(D_i),(E_i)},\centerdot) \mapsto lb.
	\end{multline*}
The top--right path of diagram~\eqref{dia-ADEsh-kappa} gives the left path in the following diagram.
Apply the associativity equation at \figref{dia-assoc-mu-multi} for maps \(I\sqcup I \rto\sh I\sqcup I \rTTo^{\triangledown\triangledown\dots\triangledown} I\), whose composition is denoted $\chi$.
We get the right path in
\begin{equation*}
\hspace*{-1.2em}
\begin{diagram}[h=3.4em,inline,nobalance]
\begin{array}{c}
\prod_{i\in I}\bigl[\mcv\bigl(\ca_i(A_i,D_i);\ca_i(A_i,D_i)\bigr)
\\
\times\mcv\bigl(\ca_i(D_i,E_i);\ca_i(D_i,E_i)\bigr)\bigr] \times
\\
\bigl[\prod_{i\in I}\mcv\bigl(\ca_i(A_i,D_i),\ca_i(D_i,E_i);\ca_i(A_i,E_i)\bigr)\bigr]
\\
\times\mcv\bigl((\ca_i(A_i,E_i))_{i\in I};\cb((A_i)_{i\in I}F,(E_i)_{i\in I}F)\bigr)
\end{array}
&&\hphantom{\prod_{i\in I}\bigl[\mcv\bigl(\ca_i(A_i,D_i),\ca_i(D_i,E_i);\ca_i(A_i,E_i)\bigr)\bigr]}
\\
\dTTo<{1\times\mu_{\triangledown\triangledown\dots\triangledown\:2I\to I}} &\rdTTo^{\prod_I\mu_{\id_{\mb2}}\times1}
\\
\begin{array}{r}
\prod_{i\in I}\bigl[\mcv\bigl(\ca_i(A_i,D_i);\ca_i(A_i,D_i)\bigr) \hfill
\\
\times\mcv\bigl(\ca_i(D_i,E_i);\ca_i(D_i,E_i)\bigr)\bigr] \times
\\
\mcv\bigl((\ca_i(A_i,D_i),\ca_i(D_i,E_i))_{i\in I};\cb((A_i)_{i\in I}F,(E_i)_{i\in I}F)\bigr)
\end{array}
\hspace*{-0.7em} &&
\begin{array}{l}
\prod_{i\in I}\bigl[\mcv\bigl(\ca_i(A_i,D_i),\ca_i(D_i,E_i);\ca_i(A_i,E_i)\bigr)\bigr]
\\
\times\mcv\bigl((\ca_i(A_i,E_i))_{i\in I};\cb((A_i)_{i\in I}F,(E_i)_{i\in I}F)\bigr)
\end{array}
\\
&\rdTTo<{\mu_{\sh\:I\sqcup I\to I\sqcup I}} &\dTTo>{\mu_\chi}
\\
&&\hspace*{-7em} \mcv\bigl((\ca_i(A_i,D_i))_{i\in I},(\ca_i(D_i,E_i))_{i\in I};\cb((A_i)_{i\in I}F, (E_i)_{i\in I}F)\bigr)
\end{diagram}
\end{equation*}
We may consider only the right vertical arrow in
\begin{diagram}[nobalance,h=1.9em]
\bigl((1_{\ca_i(A_i,D_i)},1_{\ca_i(D_i,E_i)})_{i\in I},(\kappa_{A_i,D_i,E_i})_{i\in I},F_{(A_i),(E_i)}\bigr) &\rMapsTo &\bigl((\kappa_{A_i,D_i,E_i})_{i\in I},F_{(A_i),(E_i)}\bigr)
\\
\dMapsTo &&\dMapsTo
\\
\bigl((1_{\ca_i(A_i,D_i)},1_{\ca_i(D_i,E_i)})_{i\in I},?\bigr) &\rMapsTo &tr
\end{diagram}
Thus, equation~\eqref{dia-ADEsh-kappa} is the same as the equation $lb=tr$ discussed in the statement.
	
	Unitality condition for the \(\FF\mcv\)\n-functor \([\1\rto{\id}\boxt^{i\in I}\ca_i\rto F\cb]=\id\) in explicit form
	\begin{multline*}
\hspace*{-1em} \prod_{i\in I}\mcv\bigl(;\ca_i(A_i,A_i)\bigr)\! \times\! \mcv\bigl((\ca_i(A_i,A_i))_{i\in I};\cb((A_i)_{i\in I}F,(A_i)_{i\in I}F)\bigr)
\rTTo^{\mu_{\varnothing\to I}} \mcv\bigl(;\cb((A_i)_{i\in I}F,(A_i)_{i\in I}F)\bigr),
\\
\bigl((\id_{A_i})_{i\in I},F_{(A_i),(A_i)}\bigr) \mapsto \id_{(A_i)_{i\in I}F}
	\end{multline*}
coincides with unitality condition~\eqref{eq-coherence-with-units}.
\end{proof}

The small set of multi-entry $\mcv$\n-functors \((\ca_i)_{i\in I}\to\cb\) is denoted
\[ \VCat\bigl((\ca_i)_{i\in I};\cb\bigr)\subset\VQu\bigl((\ca_i)_{i\in I};\cb\bigr).
\]


\begin{example}\label{exa-varnothing}
Consider the particular case $I=\varnothing$.
What is a multi-entry $\mcv$\n-functor \(e\:()\to\cb\)?
By definition it consists of an object \(B\in\Ob\cb\), an element \(e\in\mcv\bigl(;\cb(B,B)\bigr)\) such that \(lb=tr\) where
\begin{align*}
\mu_{\varnothing\to\mb2}\: \mcv\bigl(;\cb(B,B)\bigr) \times \mcv\bigl(;\cb(B,B)\bigr) \times \mcv\bigl(\cb(B,B),\cb(B,B);\cb(B,B)\bigr) &\to \mcv\bigl(;\cb(B,B)\bigr),
\\
(e,e,\centerdot) &\mapsto lb,
\\
\mu_{\varnothing\to\varnothing}=\id\: \mcv\bigl(;\cb(B,B)\bigr) \to \mcv\bigl(;\cb(B,B)\bigr), e &\mapsto tr=e,
\end{align*}
(see \eqref{eq-axiom-unit-multi2} for $I=\varnothing$) (that is, $e$ is an idempotent) and \eqref{eq-coherence-with-units} holds.
The latter condition, \(e\mu_{\varnothing\to\varnothing}=\id_B\), fixes the value of $e$ as \(e=\id_B\).
Thus, \(\VCat(;\cb)\cong\Ob\cb\).
The multi-entry $\mcv$\n-functor corresponding to an object $B$ is denoted \(\ddot B\:()\to\cb\).
\end{example}

\begin{example}\label{exa-I1}
	Consider the particular case $I=\mb1$.
A $\mcv$\n-functor \(F\:\ca\to\cb\) is a multi-entry $\mcv$\n-functor with the set of entries indexed by $I=\mb1$.
Thus, it is
\begin{myitemize}
	\item[---] a function \(F=\Ob F\:\Ob\ca\to\Ob\cb\);
	\item[---] a collection of elements \(F=F_{A,E}\in\mcv\bigl(\ca(A,E);\cb(AF,EF)\bigr)\);
\end{myitemize}
such that \(lb=tr\) where these elements come from
\begin{multline*}
\mu_{\mathsf{II}}\: \mcv\bigl((\ca(A,D);\cb(AF,DF)\bigr) \times\mcv\bigl(\ca(D,E);\cb(DF,EF)\bigr)
\\
\times\mcv\bigl(\cb(AF,DF),\cb(DF,EF);\cb(AF,EF)\bigr)
\to \mcv\bigl(\ca(A,D),\ca(D,E);\cb(AF,EF)\bigr),
\\
\hfill (F_{A,D},F_{D,E},\kappa_{AF,DF,EF}) \mapsto lb, \hskip\multlinegap
\\
\hskip\multlinegap \mu_{\mathsf{V}}\: \mcv\bigl(\ca(A,D),\ca(D,E);\ca(A,E)\bigr)
\times\mcv\bigl(\ca(A,E);\cb(AF,EF)\bigr) \hfill
\\
\to \mcv\bigl(\ca(A,D),\ca(D,E);\cb(AF,EF)\bigr), \qquad
(\kappa_{A,D,E},F_{A,E}) \mapsto tr.
\end{multline*}
The equation $lb=tr$ is a commutative square in $\mcv$:
\begin{diagram}[LaTeXeqno]
\ca(A,D),\ca(D,E) &\rTTo^{\kappa_{A,D,E}} &\ca(A,E)
\\
\dTTo<{F_{A,D},F_{D,E}} &= &\dTTo>{F_{A,E}}
\\
\cb(AF,DF),\cb(DF,EF) &\rTTo^{\kappa_{AF,DF,EF}} &\cb(AF,EF)
\label{dia-AAABBB}
\end{diagram}
And, furthermore, coherence with units is required:
\begin{equation}
\bigl[ () \rTTo^{\id_A} \ca(A,A) \rTTo^{F_{A,A}} \cb(AF,AF) \bigr] =\id_{AF}.
\label{eq-id-AAA-BAFAF}
\end{equation}
\end{example}

\begin{proposition}\label{pro-locally-small-symmetric-multicategory}
Let $\mcV$ be a locally small symmetric multicategory.
Small $\mcv$\n-categories and multi-entry $\mcv$\n-functors form a locally small symmetric multicategory $\VCat$.
\end{proposition}

\begin{proof}
Let \(\phi\:I\to J\in\cS_\sk\).
Let \((\ca_i)_{i\in I}\), \((\cb_j)_{j\in J}\), $\cc$ be (families of) small $\mcv$\n-categories.
Let \(F^j\:(\ca_i)_{i\in\phi^{-1}j}\to\cb_j\), $j\in J$, \(G\:(\cb_j)_{j\in J}\to\cc\) be multi-entry functors.
Similarly to the $\mcv$\n-quiver case considered in \propref{pro-symmetric-multicategory-multi-entry-quiver-morphisms} we construct another multi-entry functor \(H\:(\ca_i)_{i\in I}\to\cc\) with
\begin{myitemize}
\item[---] \(H=\Ob H\:(A_i)_{i\in I}\mapsto\bigl((A_i)_{i\in\phi^{-1}j}F^j\bigr)_{j\in J}G\).
\item[---] \(H=H_{(A_i),(E_i)}\:(\cA_i(A_i,E_i))_{i\in I}\to\cc((A_i)_{i\in I}H,(E_i)_{i\in I}H)\) obtained from
\end{myitemize}
\begin{multline}
\mu^\mcv_\phi\: \prod_{j\in J} \mcv\bigl((\ca_i(A_i,E_i))_{i\in\phi^{-1}j};\cb_j((A_i)_{i\in\phi^{-1}j}F^j,(E_i)_{i\in\phi^{-1}j}F^j)\bigr) \times
\\
\mcv\bigl((\cb_j((A_i)_{i\in\phi^{-1}j}F^j,(E_i)_{i\in\phi^{-1}j}F^j))_{j\in J};\cc(((A_i)_{i\in\phi^{-1}j}F^j)_{j\in J}G,((E_i)_{i\in\phi^{-1}j}F^j)_{j\in J}G)\bigr)
\\
\hfill \to \mcv\bigl((\ca_i(A_i,E_i))_{i\in I};\cc((A_i)_{i\in I}H,(E_i)_{i\in I}H)\bigr), \hskip\multlinegap
\\
\bigl((F^j_{(A_i)_{i\in\phi^{-1}j},(E_i)_{i\in\phi^{-1}j}})_{j\in J},G_{((A_i)_{i\in\phi^{-1}j}F^j)_{j\in J},((E_i)_{i\in\phi^{-1}j}F^j)_{j\in J}}\bigr) \mapsto H_{(A_i),(E_i)}.
\label{eq-FGH}
\end{multline}
Let us check that this assignment is in fact a map 
\[ \mu_\phi^\VCat\: \bigl[ \prod_{j\in J}\VCat((\ca_i)_{i\in\phi^{-1}(j)};\cb_j) \bigr] \times\VCat((\cb_j)_{j\in J};\cc) \to \VCat((\ca_i)_{i\in I};\cc).
\]
The equality \(lb=tr\) for $H$ comes from
\begin{multline*}
\mu_{\triangledown\triangledown}\:
\mcv\bigl((\ca_i(A_i,D_i))_{i\in I};\cc(((A_i)_{i\in\phi^{-1}j}F^j)_{j\in J}G,((D_i)_{i\in\phi^{-1}j}F^j)_{j\in J}G)\bigr)\times
	\\
\mcv\bigl((\ca_i(D_i,E_i))_{i\in I};\cc(((D_i)_{i\in\phi^{-1}j}F^j)_{j\in J}G,((E_i)_{i\in\phi^{-1}j}F^j)_{j\in J}G)\bigr)\times
	\\
\mcv\bigl(\cc(((A_i)_{i\in\phi^{-1}j}F^j)_{j\in J}G,((D_i)_{i\in\phi^{-1}j}F^j)_{j\in J}G),\cc(((D_i)_{i\in\phi^{-1}j}F^j)_{j\in J}G,((E_i)_{i\in\phi^{-1}j}F^j)_{j\in J}G);
\\
\hfill \cc(((A_i)_{i\in\phi^{-1}j}F^j)_{j\in J}G,((E_i)_{i\in\phi^{-1}j}F^j)_{j\in J}G)\bigr) \hskip\multlinegap
	\\
\to \mcv\bigl((\ca_i(A_i,D_i))_{i\in I},(\ca_i(D_i,E_i))_{i\in I};\cc(((A_i)_{i\in\phi^{-1}j}F^j)_{j\in J}G,((E_i)_{i\in\phi^{-1}j}F^j)_{j\in J}G)\bigr)\bigr),
	\\
(H_{(A_i),(D_i)},H_{(D_i),(E_i)},\centerdot) \mapsto lb,
\end{multline*}
\begin{multline*}
\mu_\chi\: \prod_{i\in I}\mcv\bigl(\ca_i(A_i,D_i),\ca_i(D_i,E_i);\ca_i(A_i,E_i)\bigr)
	\\
\times\mcv\bigl((\ca_i(A_i,E_i))_{i\in I};\cc(((A_i)_{i\in\phi^{-1}j}F^j)_{j\in J}G,((E_i)_{i\in\phi^{-1}j}F^j)_{j\in J}G)\bigr)
	\\
\to \mcv\bigl((\ca_i(A_i,D_i))_{i\in I},(\ca_i(D_i,E_i))_{i\in I};\cc(((A_i)_{i\in\phi^{-1}j}F^j)_{j\in J}G,((E_i)_{i\in\phi^{-1}j}F^j)_{j\in J}G)\bigr),
	\\
((\kappa_{A_i,D_i,E_i})_{i\in I},H_{(A_i),(E_i)}) \mapsto tr.
\end{multline*}
The second expression in details comes from
\begin{multline}
\prod_{i\in I}\mcv\bigl(\ca_i(A_i,D_i),\ca_i(D_i,E_i);\ca_i(A_i,E_i)\bigr)
	\\
	\times\prod_{j\in J} \mcv\bigl((\ca_i(A_i,E_i))_{i\in\phi^{-1}j};\cb_j((A_i)_{i\in\phi^{-1}j}F^j,(E_i)_{i\in\phi^{-1}j}F^j)\bigr)
	\\
	\times \mcv\bigl((\cb_j((A_i)_{i\in\phi^{-1}j}F^j,(E_i)_{i\in\phi^{-1}j}F^j))_{j\in J};\cc(((A_i)_{i\in\phi^{-1}j}F^j)_{j\in J}G,((E_i)_{i\in\phi^{-1}j}F^j)_{j\in J}G)\bigr)
	\\
	\hskip\multlinegap \rTTo^{1\times\mu_\phi} \prod_{i\in I}\mcv\bigl(\ca_i(A_i,D_i),\ca_i(D_i,E_i);\ca_i(A_i,E_i))\bigr) \hfill
	\\
	\times\mcv\bigl((\ca_i(A_i,E_i))_{i\in I};\cc(((A_i)_{i\in\phi^{-1}j}F^j)_{j\in J}G,((E_i)_{i\in\phi^{-1}j}F^j)_{j\in J}G)\bigr)
	\\
	\hfill \rTTo^{\mu_\chi} \mcv\bigl((\ca_i(A_i,D_i))_{i\in I},(\ca_i(D_i,E_i))_{i\in I};\cc(((A_i)_{i\in\phi^{-1}j}F^j)_{j\in J}G,((E_i)_{i\in\phi^{-1}j}F^j)_{j\in J}G)\bigr), \hskip\multlinegap
	\\
	\hskip\multlinegap ((\kappa_{A_i,D_i,E_i})_{i\in I},(F^j_{(A_i)_{i\in\phi^{-1}j},(E_i)_{i\in\phi^{-1}j}})_{j\in J},G_{((A_i)_{i\in\phi^{-1}j}F^j)_{j\in J},((E_i)_{i\in\phi^{-1}j}F^j)_{j\in J}}) \hfill
	\\
	\mapsto ((\kappa_{A_i,D_i,E_i})_{i\in I},H_{(A_i),(E_i)}) \mapsto tr.
	\label{eq-equality-lb=tr-for-H}
\end{multline}
The first expression is presented by the left map in the following diagram.
Using \figref{dia-assoc-mu-multi} for maps \(I\sqcup I\rTTo^{\phi\sqcup\phi} J\sqcup J\rTTo^{\triangledown\triangledown} \mb2\) we rewrite this as the right map in
\begin{gather*}
\begin{array}{c}
	\prod_{j\in J} \mcv\bigl((\ca_i(A_i,D_i))_{i\in\phi^{-1}j};\cb_j((A_i)_{i\in\phi^{-1}j}F^j,(D_i)_{i\in\phi^{-1}j}F^j)\bigr)
	\\
\times \mcv\bigl((\cb_j((A_i)_{i\in\phi^{-1}j}F^j,(D_i)_{i\in\phi^{-1}j}F^j))_{j\in J}; 
\cc(((A_i)_{i\in\phi^{-1}j}F^j)_{j\in J}G,((D_i)_{i\in\phi^{-1}j}F^j)_{j\in J}G)\bigr) 
	\\
\times \prod_{j\in J}\mcv\bigl((\ca_i(D_i,E_i))_{i\in\phi^{-1}j};\cb_j((D_i)_{i\in\phi^{-1}j}F^j,(E_i)_{i\in\phi^{-1}j}F^j)\bigr)
	\\
\times \mcv\bigl((\cb_j((D_i)_{i\in\phi^{-1}j}F^j,(E_i)_{i\in\phi^{-1}j}F^j))_{j\in J}; 
\cc(((D_i)_{i\in\phi^{-1}j}F^j)_{j\in J}G,((E_i)_{i\in\phi^{-1}j}F^j)_{j\in J}G)\bigr) 
	\\
\times \mcv\bigl(\cc(((A_i)_{i\in\phi^{-1}j}F^j)_{j\in J}G,((D_i)_{i\in\phi^{-1}j}F^j)_{j\in J}G),
\cc(((D_i)_{i\in\phi^{-1}j}F^j)_{j\in J}G,((E_i)_{i\in\phi^{-1}j}F^j)_{j\in J}G);
	\\
	\hfill \cc(((A_i)_{i\in\phi^{-1}j}F^j)_{j\in J}G,((E_i)_{i\in\phi^{-1}j}F^j)_{j\in J}G)\bigr) \hfill
\end{array}
\\
\begin{tanglec}
\hh \id \step[1.5] \id \step[1.5] \id
\\
\hh \ffbox1{\mu_\phi} \hstep \ffbox1{\mu_\phi} \step \id \hstep
\\
\hh \id \step[1.5] \id \step[1.5] \id
\\
\hh \ffbox5{\mu_{\triangledown\triangledown\:I\sqcup I\to\mb2}}
\\
\hh \id
\end{tanglec}
=
\begin{tanglec}
\id \step \XX \step \id \step \id \step
	\\
\hh \hstep \id \step \id \hstep \ffbox5{\mu_{\triangledown\triangledown\:J\sqcup J\to\mb2}}
	\\
\hh \id \step \id \step[3] \id \Step
	\\
\hh \ffbox6{\mu_{\phi\sqcup\phi\:I\sqcup I\to J\sqcup J}} \Step
	\\
\hh \id \Step
\end{tanglec}
\\
\mcv\bigl((\ca_i(A_i,D_i))_{i\in I},(\ca_i(D_i,E_i))_{i\in I};\cc(((A_i)_{i\in\phi^{-1}j}F^j)_{j\in J}G,((E_i)_{i\in\phi^{-1}j}F^j)_{j\in J}G)\bigr)
\end{gather*}
On elements we have
\begin{diagram}[h=3em]
\begin{array}{r}
\bigl((F^j_{(A_i)_{i\in\phi^{-1}j},(D_i)_{i\in\phi^{-1}j}})_{j\in J}, \hfill
\\
G_{((A_i)_{i\in\phi^{-1}j}F^j)_{j\in J},((D_i)_{i\in\phi^{-1}j}F^j)_{j\in J}},
\\
(F^j_{(D_i)_{i\in\phi^{-1}j},(E_i)_{i\in\phi^{-1}j}})_{j\in J},
\\
G_{((D_i)_{i\in\phi^{-1}j}F^j)_{j\in J},((E_i)_{i\in\phi^{-1}j}F^j)_{j\in J}},\centerdot\bigr)
\end{array}
&\rMapsTo &
\begin{array}{l}
\bigl((F^j_{(A_i)_{i\in\phi^{-1}j},(D_i)_{i\in\phi^{-1}j}})_{j\in J},
\\
(F^j_{(D_i)_{i\in\phi^{-1}j},(E_i)_{i\in\phi^{-1}j}})_{j\in J}, \hfill
\\
\mu_{\triangledown\triangledown\:J\sqcup J\to\mb2}(G_{((A_i)_{i\in\phi^{-1}j}F^j)_{j\in J},((D_i)_{i\in\phi^{-1}j}F^j)_{j\in J}},
\\
\hfill G_{((D_i)_{i\in\phi^{-1}j}F^j)_{j\in J},((E_i)_{i\in\phi^{-1}j}F^j)_{j\in J}},\centerdot))
\end{array}
\\
\dMapsTo &&\dMapsTo
\\
(H_{(A_i),(D_i)},H_{(D_i),(E_i)},\centerdot) &\rMapsTo &lb
\end{diagram}

Using the condition on $G$ as a multi-entry functor we rewrite the above as the left map in the following diagram.
Using \figref{dia-assoc-mu-multi} for maps \(I\sqcup I\rTTo^{\phi\sqcup\phi} J\sqcup J\rTTo^\chi J\) we rewrite this as the right map in
\begin{gather*}
\begin{array}{c}
\prod_{j\in J} \bigl[\mcv\bigl((\ca_i(A_i,D_i))_{i\in\phi^{-1}j};\cb_j((A_i)_{i\in\phi^{-1}j}F^j,(D_i)_{i\in\phi^{-1}j}F^j)\bigr)
	\\
\times \mcv\bigl((\ca_i(D_i,E_i))_{i\in\phi^{-1}j};\cb_j((D_i)_{i\in\phi^{-1}j}F^j,(E_i)_{i\in\phi^{-1}j}F^j)\bigr) 
	\\
\times \mcv\bigl(\cb_j((A_i)_{i\in\phi^{-1}j}F^j,(D_i)_{i\in\phi^{-1}j}F^j),\cb_j((D_i)_{i\in\phi^{-1}j}F^j,(E_i)_{i\in\phi^{-1}j}F^j); \hfill
\\
\hfill c\cb_j((A_i)_{i\in\phi^{-1}j}F^j,(E_i)_{i\in\phi^{-1}j}F^j)\bigr) \bigr]
	\\
\times \mcv\bigl((\cb_j((A_i)_{i\in\phi^{-1}j}F^j,(E_i)_{i\in\phi^{-1}j}F^j))_{j\in J};
\cc(((A_i)_{i\in\phi^{-1}j}F^j)_{j\in J}G,((E_i)_{i\in\phi^{-1}j}F^j)_{j\in J}G)\bigr)
\end{array}
\\
\begin{tanglec}
\hh \id \hstep \id \step[1.5] \id \Step \id \step
\\
\hh \id \hstep \id \hstep \ffbox4{\mu_{\chi\:\!\!J\sqcup J\to J}}
\\
\hh \id \hstep \id \step[2.5] \id \Step
\\
\hh \ffbox4{\mu_{\phi\sqcup\phi}} \Step
\\
\hh \id \Step
\end{tanglec}
=\quad
\begin{tanglec}
\hh \hstep \id \step \id \step[7] \id \step[1.5] \id
	\\
\ffbox{10}{\prod_{j\in J}\mu_{\triangledown\triangledown\:\phi^{-1}j\sqcup\phi^{-1}j\to\mb2}} \hstep \id \hstep
	\\
\hh \step[4.5] \id \step[5.5] \id
	\\
\hh \step[4] \ffbox7{\mu_{(\phi\sqcup\phi)\centerdot\chi}}
	\\
\hh \step[4] \id
\end{tanglec}
\\
\mcv\bigl((\ca_i(A_i,D_i))_{i\in I},(\ca_i(D_i,E_i))_{i\in I};\cc(((A_i)_{i\in\phi^{-1}j}F^j)_{j\in J}G,((E_i)_{i\in\phi^{-1}j}F^j)_{j\in J}G)\bigr)
\end{gather*}
On elements we have
\begin{diagram}[h=2.8em,nobalance]
	\begin{array}{r}
\bigl((F^j_{(A_i)_{i\in\phi^{-1}j},(D_i)_{i\in\phi^{-1}j}}, \hfill
\\
F^j_{(D_i)_{i\in\phi^{-1}j},(E_i)_{i\in\phi^{-1}j}},\centerdot)_{j\in J},
\\
G_{((A_i)_{i\in\phi^{-1}j}F^j)_{j\in J}, ((E_i)_{i\in\phi^{-1}j}F^j)_{j\in J}}\bigr)
	\end{array}
	&\rMapsTo &\hspace*{-2em}
	\begin{array}{r}
\bigl(((\mu_{\triangledown\triangledown\:\phi^{-1}j\sqcup\phi^{-1}j\to\mb2})(F^j_{(A_i)_{i\in\phi^{-1}j},(D_i)_{i\in\phi^{-1}j}},
\\
F^j_{(D_i)_{i\in\phi^{-1}j},(E_i)_{i\in\phi^{-1}j}},\centerdot))_{j\in J},
\\
G_{((A_i)_{i\in\phi^{-1}j}F^j)_{j\in J}, ((E_i)_{i\in\phi^{-1}j}F^j)_{j\in J}}\bigr) 
	\end{array}
	\\
	\dMapsTo &&\dMapsTo
	\\
	\begin{array}{r}
\bigl((F^j_{(A_i)_{i\in\phi^{-1}j},(D_i)_{i\in\phi^{-1}j}})_{j\in J},(F^j_{(D_i)_{i\in\phi^{-1}j},(E_i)_{i\in\phi^{-1}j}})_{j\in J}, \hfill
\\
\mu_{\chi\:J\sqcup J\to J}((\centerdot)_{j\in J},G_{((A_i)_{i\in\phi^{-1}j}F^j)_{j\in J}, ((E_i)_{i\in\phi^{-1}j}F^j)_{j\in J}})\bigr)	\end{array}
 &\rMapsTo &lb
\end{diagram}

Using the conditions on $F^j$ being multi-entry functors we replace the element
\[ (\mu_{\triangledown\triangledown\:\phi^{-1}j\sqcup\phi^{-1}j\to\mb2}) (F^j_{(A_i)_{i\in\phi^{-1}j},(D_i)_{i\in\phi^{-1}j}},F^j_{(D_i)_{i\in\phi^{-1}j},(E_i)_{i\in\phi^{-1}j}},\centerdot)
\]
with \(\mu_\chi((\centerdot)_{i\in\phi^{-1}j},F^j_{(D_i)_{i\in\phi^{-1}j},(E_i)_{i\in\phi^{-1}j}})\):
\begin{multline*}
\prod_{j\in J} \bigl[ \prod_{i\in\phi^{-1}j} \mcv\bigl(\ca_i(A_i,D_i),\ca_i(D_i,E_i);\ca_i(A_i,E_i)\bigr)
\\
\times \mcv\bigl((\ca_i(A_i,E_i))_{i\in\phi^{-1}j};\cb_j((A_i)_{i\in\phi^{-1}j}F^j,(E_i)_{i\in\phi^{-1}j}F^j)\bigr) \bigr]
	\\
\times \mcv\bigl((\cb_j((A_i)_{i\in\phi^{-1}j}F^j,(E_i)_{i\in\phi^{-1}j}F^j))_{j\in J};\cc(((A_i)_{i\in\phi^{-1}j}F^j)_{j\in J}G,((E_i)_{i\in\phi^{-1}j}F^j)_{j\in J}G)\bigr)
	\\
\rTTo^{(\prod_{j\in J}\mu_{\chi\:\phi^{-1}j\sqcup\phi^{-1}j\to\phi^{-1}j})\times1}
\\
\prod_{j\in J} \mcv\bigl((\ca_i(A_i,D_i))_{i\in\phi^{-1}j},(\ca_i(D_i,E_i))_{i\in\phi^{-1}j}; \cb_j((A_i)_{i\in\phi^{-1}j}F^j,(E_i)_{i\in\phi^{-1}j}F^j)\bigr)
	\\
	\times \mcv\bigl((\cb_j((A_i)_{i\in\phi^{-1}j}F^j,(E_i)_{i\in\phi^{-1}j}F^j))_{j\in J}; \cc(((A_i)_{i\in\phi^{-1}j}F^j)_{j\in J}G,((E_i)_{i\in\phi^{-1}j}F^j)_{j\in J}G)\bigr)
	\\
	\hfill \rTTo^{\mu_{(\phi\sqcup\phi)\centerdot\chi}} \mcv\bigl((\ca_i(A_i,D_i))_{i\in I},(\ca_i(D_i,E_i))_{i\in I};\cc(((A_i)_{i\in\phi^{-1}j}F^j)_{j\in J}G,((E_i)_{i\in\phi^{-1}j}F^j)_{j\in J}G)\bigr)\bigr), \hskip\multlinegap
	\\
\bigl(((\centerdot)_{i\in\phi^{-1}j},F^j_{(A_i)_{i\in\phi^{-1}j},(E_i)_{i\in\phi^{-1}j}})_{j\in J},
G_{((A_i)_{i\in\phi^{-1}j}F^j)_{j\in J}, ((E_i)_{i\in\phi^{-1}j}F^j)_{j\in J}}\bigr) \mapsto
\\
\bigl((\mu_{\chi\:\phi^{-1}j\sqcup\phi^{-1}j\to\phi^{-1}j}((\centerdot)_{i\in\phi^{-1}j},F^j_{(A_i)_{i\in\phi^{-1}j},(E_i)_{i\in\phi^{-1}j}}))_{j\in J},
G_{((A_i)_{i\in\phi^{-1}j}F^j)_{j\in J}, ((E_i)_{i\in\phi^{-1}j}F^j)_{j\in J}}\bigr)
\mapsto lb.
\end{multline*}
We may apply \figref{dia-assoc-mu-multi} for maps \(I\sqcup I\rto\chi I\rto\phi J\).
Indeed, the compositions are equal
\[ (I\sqcup I\rto\chi I\rto\phi J) =(I\sqcup I\rTTo^{\phi\sqcup\phi} J\sqcup J\rto\chi J),
\]
and
\begin{diagram}[h=1.7em]
\phi^{-1}j\sqcup\phi^{-1}j &\rTTo^\chi &\phi^{-1}j
\\
\dMono &= &\dMono
\\
I\sqcup I &\rTTo^\chi &I
\end{diagram}
We get the composition of \eqref{eq-equality-lb=tr-for-H} with $lb$ as the final element.
We conclude that \(tr=lb\) for $H$.

Now let us check unitality condition \eqref{eq-coherence-with-units} for $H$.
The condition to verify is
\begin{multline*}
\bigl[ \prod_{i\in I} \mcv\bigl(;\ca_i(A_i,A_i)\bigr) \bigr] \times \mcv\bigl((\ca_i(A_i,A_i))_{i\in I};\cc((A_i)_{i\in I}H,(A_i)_{i\in I}H)\bigr)
\\
\rTTo^{\mu_{\varnothing\to I}} \mcv\bigl(;\cc((A_i)_{i\in I}H,(A_i)_{i\in I}H)\bigr), \quad
\bigl((\id_{A_i})_{i\in I},H_{(A_i)_{i\in I},(A_i)_{i\in I}}\bigr) \mapsto \id_{(A_i)_{i\in I}H}.
\end{multline*}
Providing details we get the left path in the following diagram.
This composition can be rewritten using \figref{dia-assoc-mu-multi} for maps \(\varnothing\to I\rto\phi J\).
We get the right path in
\[ \hspace*{-1em}
\begin{diagram}[h=3.2em,inline]
	\begin{array}{r}
\bigl[ \prod_{i\in I} \mcv\bigl(;\ca_i(A_i,A_i)\bigr) \bigr] \times \bigl[ \prod_{j\in J} \mcv\bigl((\ca_i(A_i,A_i))_{i\in\phi^{-1}j};\cb_j((A_i)_{i\in\phi^{-1}j}F^j,(A_i)_{i\in\phi^{-1}j}F^j)\bigr) \bigr] 
\\
\times \mcv\bigl((\cb_j((A_i)_{i\in\phi^{-1}j}F^j,(A_i)_{i\in\phi^{-1}j}F^j))_{j\in J};
\cc(((A_i)_{i\in\phi^{-1}j}F^j)_{j\in J}G,((A_i)_{i\in\phi^{-1}j}F^j)_{j\in J}G)\bigr)
	\end{array}
\hspace*{-18.6em}
	\\
\dTTo<{1\times\mu_\phi} &\rdTTo^{(\prod_{j\in J}\mu_{\varnothing\to\phi^{-1}j})\times1} &
	\\
	\begin{array}{l}
\bigl[ \prod_{i\in I} \mcv\bigl(;\ca_i(A_i,A_i)\bigr) \bigr] \times
\\
\mcv\bigl((\ca_i(A_i,A_i))_{i\in I};
\\
\cc(((A_i)_{i\in\phi^{-1}j}F^j)_{j\in J}G,((A_i)_{i\in\phi^{-1}j}F^j)_{j\in J}G)\bigr)
	\end{array}
	\hspace*{-0em} &&\hspace*{-0em}
	\begin{array}{l}
\bigl[ \prod_{j\in J} \mcv\bigl(;\cb_j((A_i)_{i\in\phi^{-1}j}F^j,(A_i)_{i\in\phi^{-1}j}F^j)\bigr) \bigr]
\\
\times \mcv\bigl((\cb_j((A_i)_{i\in\phi^{-1}j}F^j,(A_i)_{i\in\phi^{-1}j}F^j))_{j\in J};
\\
\cc(((A_i)_{i\in\phi^{-1}j}F^j)_{j\in J}G,((A_i)_{i\in\phi^{-1}j}F^j)_{j\in J}G)\bigr)
	\end{array}
	\\
&\rdTTo<{\mu_{\varnothing\to I}} &\dTTo>{\mu_{\varnothing\to J}}
	\\
&&\hspace*{-2em} \mcv\bigl(;\cc(((A_i)_{i\in\phi^{-1}j}F^j)_{j\in J}G,((A_i)_{i\in\phi^{-1}j}F^j)_{j\in J}G)\bigr)
\end{diagram}
\]
On elements we have
\begin{diagram}[h=2.2em]
	\begin{array}{r}
\bigl((\id_{A_i})_{i\in I},(F^j_{(A_i)_{i\in\phi^{-1}j},(A_i)_{i\in\phi^{-1}j}})_{j\in J},
\\
G_{((A_i)_{i\in\phi^{-1}j}F^j)_{j\in J},((A_i)_{i\in\phi^{-1}j}F^j)_{j\in J}}\bigr)
	\end{array}
	&\rMapsTo &
\begin{array}{l}
\bigl((\id_{(A_i)_{i\in\phi^{-1}j}F^j})_{j\in J},
\\
G_{((A_i)_{i\in\phi^{-1}j}F^j)_{j\in J},((A_i)_{i\in\phi^{-1}j}F^j)_{j\in J}}\bigr)
\end{array}
	\\
\dMapsTo &&\dMapsTo>!
	\\
\bigl((\id_{A_i})_{i\in I},H_{(A_i)_{i\in I},(A_i)_{i\in I}}\bigr) &\rMapsTo^? &\id_{((A_i)_{i\in\phi^{-1}j}F^j)_{j\in J}G}
\end{diagram}
due to the unitality of $F^j$ and $G$.
Unitality of $H$ is proven.

Let \((\ca_i)_{i\in I}\), \((\cb_j)_{j\in J}\), \((\cc_k)_{k\in K}\), $\cD$ be (families of) $\mcv$\n-categories, where \(I,J,K\in\Ob\cs_\sk\).
Let \(I\rto\phi J\rto\psi K\) be mappings (in $\cs_\sk$).
Let \(F^j\:(\ca_i)_{i\in\phi^{-1}j}\to\cb_j\), $j\in J$, \(G^k\:(\cb_j)_{j\in\psi^{-1}k}\to\cc\), $k\in K$,  \(H\:(\cc_k)_{k\in K}\to\cD\) be multi-entry functors.
Fix objects $A_i$, $E_i$ of $\ca_i$, $i\in I$.
Expanding entries of associativity equation for $\VCat$ using \eqref{eq-FGH} we get diagram at \figref{dia-assoc-mu-multi} for \(X_i=\ca_i(A_i,E_i)\), \(Y_j=\cb_j((A_i)_{i\in\phi^{-1}j}F^j,(E_i)_{i\in\phi^{-1}j}F^j)\),  \(Z_k=\cc_k(((A_i)_{i\in\phi^{-1}j}F^j)_{j\in\psi^{-1}k}G^k,((E_i)_{i\in\phi^{-1}j}F^j)_{j\in\psi^{-1}k}G^k)\),
\[ W=\cD((((A_i)_{i\in\phi^{-1}j}F^j)_{j\in\psi^{-1}k}G^k)_{k\in K}H,(((E_i)_{i\in\phi^{-1}j}F^j)_{j\in\psi^{-1}k}G^k)_{k\in K}H).
\]
Therefore, for composition in $\VCat$ the associativity holds.

Define the identity $\mcv$\n-functor \(\Id\:\ca\to\ca\) with the identity map \(\id\:\Ob\ca\to\Ob\ca\) and \(1_{\ca(A,A)}\in\mcv\bigl(\ca(A,A);\ca(A,A)\bigr)\).
Clearly, both equations for identities are satisfied, hence, $\VCat$ is a symmetric multicategory.
\end{proof}

\subsection{Natural \texorpdfstring{$\mcV$}V-transformations}
\begin{definition}\label{def-Natural-V-transformation}
Natural $\mcv$\n-transformation\index{natural transformation between multi-entry functors} \(\lambda\:F\to G\:(\ca_i)_{i\in I}\to\cc\) is a family \((\lambda_{A_1,\dots,A_I})_{(A_i\in\ca_i)}\), \(\lambda_{A_1,\dots,A_I}\in\mcv\bigl(;\cc((A_i)_{i\in I}F,(A_i)_{i\in I}G)\bigr)\), such that for all objects $A_i$, $D_i$ of $\ca_i$, $i\in I$, the square
	\begin{diagram}[h=2.4em]
		(\ca_i(A_i,D_i))_{i\in I} \rto{F_{(A_i),(D_i)},\lambda_{(D_i)}} \cc((A_i)_{i\in I}F,(D_i)_{i\in I}F),\cc((D_i)_{i\in I}F,(D_i)_{i\in I}G)
		\\
		\dTTo<{\lambda_{(A_i)},G_{(A_i),(D_i)}} \hspace*{26em} \dTTo>{\kappa_{(A_i)F,(D_i)F,(D_i)G}}
		\\
		\cc((A_i)_{i\in I}F,(A_i)_{i\in I}G),\cc((A_i)_{i\in I}G,(D_i)_{i\in I}G) \rto{\kappa_{(A_i)F,(A_i)G,(D_i)G}} \cc((A_i)_{i\in I}F,(D_i)_{i\in I}G)
	\end{diagram}
	commutes in $\mcv$.
	In detail, elements $b'$ and $g'$ of \(\mcv\bigl((\ca_i(A_i,D_i)_{i\in I};\cc((A_i)_{i\in I}F,(D_i)_{i\in I}G)\bigr)\) are equal, where
	\begin{multline}
\mu_{\mathsf{\triangledown\centerdot}}\: \mcv\bigl((\ca_i(A_i,D_i)_{i\in I}; \cc((A_i)_{i\in I}F,(A_i)_{i\in I}F)\bigr)\times \mcv\bigl(;\cc((D_i)_{i\in I}F,(D_i)_{i\in I}G)\bigr)
\\
\times\mcv\bigl(\cc((A_i)_{i\in I}F,(D_i)_{i\in I}F), \cc((D_i)_{i\in I}F,(D_i)_{i\in I}G);\cc((A_i)_{i\in I}F,(D_i)_{i\in I}G)\bigr) \to
\\
\mcv\bigl((\ca_i(A_i,D_i)_{i\in I};\cc((A_i)_{i\in I}F,(D_i)_{i\in I}G)\bigr), \quad (F_{(A_i),(D_i)},\lambda_{(D_i)},\kappa_{(A_i)F,(D_i)F,(D_i)G}) \mapsto b',
\label{eq-Fb38}
\end{multline}
\begin{multline}
\mu_{\mathsf{\centerdot\triangledown}}\: \mcv\bigl(;\cc((A_i)_{i\in I}F,(A_i)_{i\in I}G)\bigr)\times \mcv\bigl((\ca_i(A_i,D_i)_{i\in I};\cc((A_i)_{i\in I}G,(D_i)_{i\in I}G)\bigr)
\\
\times \mcv\bigl(\cc((A_i)_{i\in I}F,(A_i)_{i\in I}G),\cc((A_i)_{i\in I}G,(D_i)_{i\in I}G);\cc((A_i)_{i\in I}F,(D_i)_{i\in I}G)\bigr) \to
\\
\mcv\bigl((\ca_i(A_i,D_i)_{i\in I};\cc((A_i)_{i\in I}F,(D_i)_{i\in I}G)\bigr), \quad (\lambda_{(A_i)},G_{(A_i),(D_i)},\kappa_{(A_i)F,(A_i)G,(D_i)G}) \mapsto g'.
\label{eq-Gg38}
	\end{multline}
	Here \(\mathsf{\triangledown\centerdot}=
	\begin{pmatrix}
		1 &\dots &I
		\\
		1 &\dots &1
	\end{pmatrix}
	\:I\to\mb2\) and \(\mathsf{\centerdot\triangledown}=
	\begin{pmatrix}
		1 &\dots &I
		\\
		2 &\dots &2
	\end{pmatrix}
	\:I\to\mb2\).
\end{definition}

\begin{proposition}\label{pro-natural-V-transformations}
	The set \(\VCat((\ca_i)_{i\in I},\cc)(F,G)\) of natural $\mcv$\n-transformations \(\lambda\:F\to G\:(\ca_i)_{i\in I}\to\cc\) is in bijection with the set \(\mcv\bigl(;\int_{(A_i\in\ca_i)}\cc((A_i)_{i\in I}F,(A_i)_{i\in I}G)\bigr)\).
\end{proposition}

\begin{proof}
	The latter set is
	\begin{multline*}
		\mcv\biggl(;\int_{(A_i\in\ca_i)}\cc((A_i)_{i\in I}F,(A_i)_{i\in I}G)\biggr)
		\\
		=\Bigl\{ \lambda=(\lambda_{(A_i)}) \in \prod_{(A_i\in\ca_i)} \mcv\bigl(;\cc((A_i)_{i\in I}F,(A_i)_{i\in I}G)\bigr) \mid 
		\\
		\begin{diagram}[inline,h=1.9em]
			() &\rTTo^{\lambda_{(D_i)}} &\cc((D_i)_{i\in I}F,(D_i)_{i\in I}G)
			\\
			\dTTo<{\lambda_{(A_i)}} &= &\dTTo>\beta
			\\
			\cc((A_i)_{i\in I}F,(A_i)_{i\in I}G) &\rTTo^\gamma &\und\mcv\bigl((\ca_i(A_i,D_i))_{i\in I};\cc((A_i)_{i\in I}F,(D_i)_{i\in I}G)\bigr)
		\end{diagram}
		\hspace*{4em} \Bigr\}
	\end{multline*}
	Equivalently the condition can be written as: for all families of objects \((A_j,E_j\in\cE_j)_{j\in J}\)
	\begin{diagram}[h=1.9em]
		(\ca_i(A_i,D_i))_{i\in I} &\rTTo^{(1)_I,\lambda_{(D_i)}} &(\ca_i(A_i,D_i))_{i\in I},\cc((D_i)_{i\in I}F,(D_i)_{i\in I}G)
		\\
		\dTTo<{(1)_I,\lambda_{(A_i)}} &= &\dTTo>{\beta^\dagger}
		\\
		(\ca_i(A_i,D_i))_{i\in I},\cc((A_i)_{i\in I}F,(A_i)_{i\in I}G) &\rTTo^{\gamma^\dagger} &\cc((A_i)_{i\in I}F,(D_i)_{i\in I}G)
	\end{diagram}
	Equivalently, $tr=lb$ where:
	\begin{multline}
		\prod_{i\in I}\mcv\bigl(\ca_i(A_i,D_i);\ca_i(A_i,D_i)\bigr)	\times \mcv\bigl(;\cc((D_i)_{i\in I}F,(D_i)_{i\in I}G)\bigr)
		\\
		\times \mcv\bigl((\ca_i(A_i,D_i))_{i\in I},\cc((D_i)_{i\in I}F,(D_i)_{i\in I}G);\cc((A_i)_{i\in I}F, (D_i)_{i\in I}G)\bigr) 
		\\
		\rTTo^{\mu_{\inj_1\:I\hookrightarrow I\sqcup\mb1}} \mcv\bigl((\ca_i(A_i,D_i))_{i\in I};\cc((A_i)_{i\in I}F, (D_i)_{i\in I}G)\bigr),
		\\
		\bigl((1_{\ca_i(A_i,D_i)})_{i\in I},\lambda_{(D_i)},\beta^\dagger\bigr) \mapsto tr,
\label{eq-1ADD-tr}
	\end{multline}
	\begin{multline}
		\prod_{i\in I}\mcv\bigl(\ca_i(A_i,D_i);\ca_i(A_i,D_i)\bigr)	\times \mcv\bigl(;\cc((A_i)_{i\in I}F,(A_i)_{i\in I}G)\bigr)
		\\
		\times \mcv\bigl((\ca_i(A_i,D_i))_{i\in I},\cc((A_i)_{i\in I}F,(A_i)_{i\in I}G);\cc((A_i)_{i\in I}F, (D_i)_{i\in I}G)\bigr) 
		\\
		\rTTo^{\mu_{\inj_1\:I\hookrightarrow I\sqcup\mb1}} \mcv\bigl((\ca_i(A_i,D_i))_{i\in I};\cc((A_i)_{i\in I}F, (D_i)_{i\in I}G)\bigr),
		\\
		\bigl((1_{\ca_i(A_i,D_i)})_{i\in I},\lambda_{(A_i)},\gamma^\dagger\bigr) \mapsto lb.
\label{eq-1ADA-lb}
	\end{multline}
In more detail from \eqref{eq-1ADD-tr} we obtain the left path of the following diagram.
Applying the associativity property from \figref{dia-assoc-mu-multi} for maps \(I \rto{\inj_1} I\sqcup\mb1 \rto{\mathsf{\triangledown I}} \mb2\) we get another expression for $tr$ via the right path of (see \eqref{eq-Fb38}):
\[ \hspace*{-1.6em}
\begin{diagram}[h=3.8em,inline]
	\begin{array}{l}
\prod_{i\in I}\mcv\bigl(\ca_i(A_i,D_i);\ca_i(A_i,D_i)\bigr)	\times \mcv\bigl(;\cc((D_i)_{i\in I}F,(D_i)_{i\in I}G)\bigr)
\\
\times \mcv\bigl((\ca_i(A_i,D_i))_{i\in I}; \cc((A_i)_{i\in I}F,(D_i)_{i\in I}F)\bigr)
\\
\times \mcv\bigl(\cc((D_i)_{i\in I}F,(D_i)_{i\in I}G);\cc((D_i)_{i\in I}F,(D_i)_{i\in I}G)\bigr)
\\
\times \mcv\bigl(\cc((A_i)_{i\in I}F,(D_i)_{i\in I}F),\cc((D_i)_{i\in I}F,(D_i)_{i\in I}G);
\\
\cc((A_i)_{i\in I}F, (D_i)_{i\in I}G)\bigr)
	\end{array}
	\hspace*{-9em}
	\\
\dTTo<{1\times1\times\mu_{\mathsf{\triangledown I}}} &\rdTTo^{\mu_{\id_I}\times\mu_{\varnothing\to\mb1}\times1} &
	\\
	\begin{array}{l}
\prod_{i\in I}\mcv\bigl(\ca_i(A_i,D_i);\ca_i(A_i,D_i)\bigr)
\\
\times \mcv\bigl(;\cc((D_i)_{i\in I}F,(D_i)_{i\in I}G)\bigr) \times
\\
\mcv\bigl((\ca_i(A_i,D_i))_{i\in I},\cc((D_i)_{i\in I}F,(D_i)_{i\in I}G);
\\
\hfill \cc((A_i)_{i\in I}F, (D_i)_{i\in I}G)\bigr) 
	\end{array}
	\hspace*{-0em} &&\hspace*{-0em}
	\begin{array}{l}
\mcv\bigl((\ca_i(A_i,D_i))_{i\in I}; \cc((A_i)_{i\in I}F,(D_i)_{i\in I}F)\bigr)
\\
\times \mcv\bigl(;\cc((D_i)_{i\in I}F,(D_i)_{i\in I}G)\bigr) \times
\\
\mcv\bigl(\cc((A_i)_{i\in I}F,(D_i)_{i\in I}F),\cc((D_i)_{i\in I}F,(D_i)_{i\in I}G);
\\
\hfill \cc((A_i)_{i\in I}F,(D_i)_{i\in I}G)\bigr) 
	\end{array}
	\\
&\rdTTo<{\mu_{\inj_1\:I\hookrightarrow I\sqcup\mb1}} &\dTTo>{\mu_{\triangledown\centerdot\:I\to\mb2}}
	\\
&&\hspace*{-0em} \mcv\bigl((\ca_i(A_i,D_i))_{i\in I};\cc((A_i)_{i\in I}F, (D_i)_{i\in I}G)\bigr)
\end{diagram}
\]
On elements:
\begin{diagram}[h=2em,nobalance,LaTeXeqno]
\bigl((1_{\ca_i(A_i,D_i)})_{i\in I},\lambda_{(D_i)},F_{(A_i),(D_i)},1_{\cc((D_i)_{i\in I}F,(D_i)_{i\in I}G)},\centerdot\bigr) &\rMapsTo &\bigl(F_{(A_i),(D_i)},\lambda_{(D_i)},\centerdot\bigr)
	\\
\dMapsTo &&\dMapsTo
	\\
\bigl((1_{\ca_i(A_i,D_i)})_{i\in I},\lambda_{(D_i)},\beta^\dagger\bigr) &\rMapsTo &tr=b'
\label{eq-tr-11mm}
\end{diagram}

Similarly from \eqref{eq-1ADA-lb} we obtain the left path of the following diagram.
Applying the associativity property from \figref{dia-assoc-mu-multi} for maps \(I \rto{\inj_1} I\sqcup\mb1 \rto{\mathsf{\triangledown I\centerdot X}} \mb2\) we get another expression for $lb$ via the right path of (see \eqref{eq-Gg38}):
\[ \hspace*{-1.6em}
\begin{diagram}[h=3.8em,inline]
	\begin{array}{l}
\prod_{i\in I}\mcv\bigl(\ca_i(A_i,D_i);\ca_i(A_i,D_i)\bigr) \times \mcv\bigl(;\cc((A_i)_{i\in I}F,(A_i)_{i\in I}G)\bigr)
\\
\times \mcv\bigl(\cc((A_i)_{i\in I}F,(A_i)_{i\in I}G);\cc((A_i)_{i\in I}F,(A_i)_{i\in I}G)\bigr)
\\
\times \mcv\bigl((\ca_i(A_i,D_i))_{i\in I}; \cc((A_i)_{i\in I}G,(D_i)_{i\in I}G)\bigr)
\\
\times \mcv\bigl(\cc((A_i)_{i\in I}F,(A_i)_{i\in I}G),\cc((A_i)_{i\in I}G,(D_i)_{i\in I}G);
\\
\cc((A_i)_{i\in I}F, (D_i)_{i\in I}G)\bigr) 
	\end{array}
	\hspace*{-9em}
	\\
\dTTo<{1\times1\times\mu_{\mathsf{\triangledown I\centerdot X}}} &\rdTTo^{\mu_{\varnothing\to\mb1}\times\mu_{\id_I}\times1} &
	\\
	\begin{array}{l}
\prod_{i\in I}\mcv\bigl(\ca_i(A_i,D_i);\ca_i(A_i,D_i)\bigr)
\\
\times \mcv\bigl(;\cc((A_i)_{i\in I}F,(A_i)_{i\in I}G)\bigr) \times
\\
\mcv\bigl((\ca_i(A_i,D_i))_{i\in I},\cc((A_i)_{i\in I}F,(A_i)_{i\in I}G);
\\
\hfill \cc((A_i)_{i\in I}F, (D_i)_{i\in I}G)\bigr) 
	\end{array}
	\hspace*{-0em} &&\hspace*{-0em}
	\begin{array}{l}
\mcv\bigl(;\cc((A_i)_{i\in I}F,(A_i)_{i\in I}G)\bigr)
\\
\times \mcv\bigl((\ca_i(A_i,D_i))_{i\in I}; \cc((A_i)_{i\in I}G,(D_i)_{i\in I}G)\bigr) \times
\\
\mcv\bigl(\cc((A_i)_{i\in I}F,(A_i)_{i\in I}G),\cc((A_i)_{i\in I}G,(D_i)_{i\in I}G);
\\
\hfill \cc((A_i)_{i\in I}F,(D_i)_{i\in I}G)\bigr) 
	\end{array}
	\\
&\rdTTo<{\mu_{\inj_1\:I\hookrightarrow I\sqcup\mb1}} &\dTTo>{\mu_{\centerdot\triangledown\:I\to\mb2}}
	\\
&&\hspace*{-0em} \mcv\bigl((\ca_i(A_i,D_i))_{i\in I};\cc((A_i)_{i\in I}F, (D_i)_{i\in I}G)\bigr)
\end{diagram}
\]
On elements:
\begin{diagram}[h=2em,nobalance,LaTeXeqno]
\bigl((1_{\ca_i(A_i,D_i)})_{i\in I},\lambda_{(A_i)},1_{\cc((A_i)_{i\in I}F,(A_i)_{i\in I}G)},G_{(A_i),(D_i)},\centerdot\bigr) &\rMapsTo &\bigl((\lambda_{(A_i)},G_{(A_i),(D_i)},\centerdot\bigr)
	\\
	\dMapsTo &&\dMapsTo
	\\
\bigl((1_{\ca_i(A_i,D_i)})_{i\in I},\lambda_{(A_i)},\gamma^\dagger\bigr) &\rMapsTo &lb=g'
	\label{eq-lb-11mm}
\end{diagram}
Thus, equations $tr=lb$ and $b'=g'$ from \eqref{eq-Fb38} and \eqref{eq-Gg38} coincide identically.
\end{proof}

\subsection{Closedness of the multicategory of \texorpdfstring{$\mcV$}V-categories}
\begin{proposition}\label{con-evVCat-VCat}
Let $\mcv$ be a locally small symmetric closed complete multicategory.
Let \((\ca_i)_{i\in I}\), \(I\in\Ob\cs_\sk\), $\cc$, be (a family of) small $\mcv$\n-categories.
Then
\[ \ev^\VCat=\ev^\VQu_{(\ca_i)_{i\in I};\cc}\vert\in\VCat\bigl((\ca_i)_{i\in I},\und\VCat((\ca_i)_{i\in I};\cc);\cc\bigr).
\]
\end{proposition}

\begin{proof}
We have to prove that \(lb=tr\), where
\begin{multline}
\mu_{\triangledown\triangledown}\: \mcv\bigl((\ca_i(A_i,D_i))_{i\in I},[(\ca_i)_{i\in I};\cc](F,G);\cc((A_i)_{i\in I}F,(D_i)_{i\in I}G)\bigr)
\\
\times \mcv\bigl((\ca_i(D_i,E_i))_{i\in I},[(\ca_i)_{i\in I};\cc](G,H);\cc((D_i)_{i\in I}G,(E_i)_{i\in I}H)\bigr)
\\
\times
\mcv\bigl(\cc((A_i)_{i\in I}F,(D_i)_{i\in I}G),\cc((D_i)_{i\in I}G,(E_i)_{i\in I}H);\cc((A_i)_{i\in I}F,(E_i)_{i\in I}H)\bigr) \to
\\
\mcv\bigl((\ca_i(A_i,D_i))_{i\in I},[(\ca_i)_{i\in I};\cc](F,G),(\ca_i(D_i,E_i))_{i\in I},[(\ca_i)_{i\in I};\cc](G,H); 
\cc((A_i)_{i\in I}F,(E_i)_{i\in I}H)\bigr),
\\
(\ev^\VQu_{(A_i),F,(D_i),G},\ev^\VQu_{(D_i),G,(E_i),H},\centerdot) \mapsto lb,
\label{eq-ev-ev-lb}
\end{multline}
\begin{multline}
\mu_\chi\: \prod_{i\in I} \mcv\bigl(\ca_i(A_i,D_i),\ca_i(D_i,E_i);\ca_i(A_i,E_i)\bigr)
\\
\times \mcv\bigl([(\ca_i)_{i\in I};\cc](F,G),[(\ca_i)_{i\in I};\cc](G,H);[(\ca_i)_{i\in I};\cc](F,H)\bigr)
	\\
\times \mcv\bigl((\ca_i(A_i,E_i))_{i\in I},[(\ca_i)_{i\in I};\cc](F,H);\cc((A_i)_{i\in I}F,(E_i)_{i\in I}H)\bigr) \to
	\\
\mcv\bigl((\ca_i(A_i,D_i))_{i\in I},[(\ca_i)_{i\in I};\cc](F,G),(\ca_i(D_i,E_i))_{i\in I},[(\ca_i)_{i\in I};\cc](G,H);
\cc((A_i)_{i\in I}F,(E_i)_{i\in I}H)\bigr),
	\\
\bigl((\kappa_{A_i,D_i,E_i})_{i\in I},\centerdot,\ev^\VQu_{(A_i),F,(E_i),H}\bigl) \mapsto tr.
\label{eq-kappa-ev-tr}
\end{multline}
In detail \eqref{eq-ev-ev-lb} can be described as the left map in the following diagram.
Using the associativity property from \figref{dia-assoc-mu-multi} for maps \((I\sqcup\mb1)\sqcup(I\sqcup\mb1)\rTTo^{(\triangledown\mathsf{I})\sqcup(\triangledown\mathsf{I})} \mb2\sqcup\mb2\rto{\triangledown\triangledown} \mb2\) we get the right map in:
\begin{gather*}
\begin{array}{c}
	\mcv\bigl((\ca_i(A_i,D_i))_{i\in I};\cc((A_i)_{i\in I}F,(D_i)_{i\in I}F)\bigr)
	\times \mcv\bigl([(\ca_i)_{i\in I};\cc](F,G);\cc((D_i)_{i\in I}F,(D_i)_{i\in I}G)\bigr)
	\\
\times \mcv\bigl(\cc((A_i)_{i\in I}F,(D_i)_{i\in I}F),\cc((D_i)_{i\in I}F,(D_i)_{i\in I}G);\cc((A_i)_{i\in I}F,(D_i)_{i\in I}G)\bigr)
	\\
\times \mcv\bigl((\ca_i(D_i,E_i))_{i\in I};\cc((D_i)_{i\in I}G,(E_i)_{i\in I}G)\bigr)
\times \mcv\bigl([(\ca_i)_{i\in I};\cc](G,H);\cc((E_i)_{i\in I}G,(E_i)_{i\in I}H)\bigr)
	\\
	\times \mcv\bigl(\cc((D_i)_{i\in I}G,(E_i)_{i\in I}G),\cc((E_i)_{i\in I}G,(E_i)_{i\in I}H);\cc((D_i)_{i\in I}G,(E_i)_{i\in I}H)\bigr)
	\\
	\times \mcv\bigl(\cc((A_i)_{i\in I}F,(D_i)_{i\in I}G),\cc((D_i)_{i\in I}G,(E_i)_{i\in I}H);\cc((A_i)_{i\in I}F,(E_i)_{i\in I}H)\bigr)
\end{array}
\\
\begin{tanglec}
\hh \step \id \step \id \step \id \step[2.5] \id \step \id \step \id \step[1.5] \id 
\\
\hh \ffbox4{\mu_{\triangledown\mathsf{I}\:\!\!I\sqcup\mb1\to\mb2}} \hstep \ffbox4{\mu_{\triangledown\mathsf{I}\:\!\!I\sqcup\mb1\to\mb2}} \hstep \id
\\
\hh \Step \id \step[4.5] \id \step[2.5] \id
\\
\hh \Step \ffbox8{\mu_{\triangledown\triangledown\:(I\sqcup\mb1)\sqcup(I\sqcup\mb1)\to\mb2}}
\\
\hh \Step \id 
\end{tanglec}
=
\begin{tanglec}
\id \step \id \step \XX \step[3] \ne2 \step \id \step \id
	\\
\id \step \id \step \id \Step \XX \step[3] \id \step \id
\\
\hh \hstep \id \step \id \step \id \Step \id \step[1.5] \ffbox5{\mu_{\triangledown\triangledown\:\mb2\sqcup\mb2\to\mb2}}
	\\
\hh \id \step \id \step \id \Step \id \step[4] \id \Step
	\\
\hh \ffbox{10}{\mu_{(\triangledown\mathsf{I})\sqcup(\triangledown\mathsf{I})\:\!\!(I\sqcup\mb1)\sqcup(I\sqcup\mb1)\to\mb2\sqcup\mb2}} \step
	\\
\hh \id \step
\end{tanglec}
\\
\mcv\bigl((\ca_i(A_i,D_i))_{i\in I},[(\ca_i)_{i\in I};\cc](F,G),(\ca_i(D_i,E_i))_{i\in I},[(\ca_i)_{i\in I};\cc](G,H);
\cc((A_i)_{i\in I}F,(E_i)_{i\in I}H)\bigr)
\end{gather*}
On elements
\begin{diagram}[h=2em,nobalance,LaTeXeqno]
(F_{(A_i),(D_i)},p_{(D_i)_{i\in I}},\centerdot,G_{(D_i),(E_i)},p_{(E_i)_{i\in I}},\centerdot,\centerdot) &\rMapsTo &(F_{(A_i),(D_i)},p_{(D_i)_{i\in I}},G_{(D_i),(E_i)},p_{(E_i)_{i\in I}},\centerdot^{(4)})
	\\
	\dMapsTo &&\dMapsTo
	\\
(\ev^\VQu_{(A_i),F,(D_i),G},\ev^\VQu_{(D_i),G,(E_i),H},\centerdot) &\rMapsTo &lb
\label{eq-FpGp-lb}
\end{diagram}

Expanding the expression for $tr$ from \eqref{eq-kappa-ev-tr} we get the left map in the following diagram.
Using \figref{dia-assoc-mu-multi} for maps \((I\sqcup\mb1)\sqcup(I\sqcup\mb1) \rTTo^{\chi_{I\sqcup\mb1}} I\sqcup\mb1 \rTTo^{\triangledown\mathsf{I}} \mb2\) we rewrite the composition as the second map.
Let us use the fact that $F$ is a multi-entry $\mcv$\n-functor (see \defref{def-multi-entry-V-functor}).
Also we use explicit form of composition in $\mcv$\n-category \([(\ca_i)_{i\in I};\cc]\) given by diagram~\eqref{dia-composition-undVCat}.
This gives the third map.
Using \figref{dia-assoc-mu-multi} for maps \(I\sqcup\mb1\sqcup I\sqcup\mb1 \rTTo^{\triangledown\mathsf{I}\triangledown\mathsf{I}\centerdot(23)} \mb2\sqcup\mb2 \rTTo^{\triangledown\triangledown} \mb2\) we rewrite the composition as the right map in
\begin{gather*}
\hspace*{-0.8em}
\begin{array}{r}
	\prod_{i\in I} \mcv\bigl(\ca_i(A_i,D_i),\ca_i(D_i,E_i);\ca_i(A_i,E_i)\bigr)
	\\
	\times \mcv\bigl([(\ca_i)_{i\in I};\cc](F,G),[(\ca_i)_{i\in I};\cc](G,H);
	\\
	{}[(\ca_i)_{i\in I};\cc](F,H)\bigr) 
	\\
	\times \mcv\bigl((\ca_i(A_i,E_i))_{i\in I}; \cc((A_i)_{i\in I}F,(E_i)_{i\in I}F)\bigr)
	\\
	\times \mcv\bigl([(\ca_i)_{i\in I};\cc](F,H);\cc((E_i)_{i\in I}F,(E_i)_{i\in I}H)\bigr) \times
	\\
	\mcv\bigl(\cc((A_i)_{i\in I}F,(E_i)_{i\in I}F),\cc((E_i)_{i\in I}F,(E_i)_{i\in I}H);
	\\
	\cc((A_i)_{i\in I}F,(E_i)_{i\in I}H)\bigr)
\end{array}
\hspace*{-0.5em}
\begin{array}{l}
	\mcv\bigl((\ca_i(A_i,D_i))_{i\in I};\cc((A_i)_{i\in I}F,(D_i)_{i\in I}F)\bigr)
	\\
	\times \mcv\bigl((\ca_i(D_i,E_i))_{i\in I};\cc((D_i)_{i\in I}F,(E_i)_{i\in I}F)\bigr)
	\\
	\times \mcv\bigl([(\ca_i)_{i\in I};\cc](F,G);\cc((E_i)_{i\in I}F,(E_i)_{i\in I}G)\bigr)
	\\
	\times \mcv\bigl([(\ca_i)_{i\in I};\cc](G,H);\cc((E_i)_{i\in I}G,(E_i)_{i\in I}H)\bigr)
	\\
	\times \mcv\bigl(\cc((A_i)_{i\in I}F,(D_i)_{i\in I}F),\cc((D_i)_{i\in I}F,(E_i)_{i\in I}F);
	\\
	\hfill \cc((A_i)_{i\in I}F,(E_i)_{i\in I}F)\bigr)
	\\
	\times \mcv\bigl(\cc((E_i)_{i\in I}F,(E_i)_{i\in I}G),\cc((E_i)_{i\in I}G,(E_i)_{i\in I}H);
	\\
	\hfill \cc((E_i)_{i\in I}F,(E_i)_{i\in I}H)\bigr)
	\\
	\times \mcv\bigl(\cc((A_i)_{i\in I}F,(E_i)_{i\in I}F),\cc((E_i)_{i\in I}F,(E_i)_{i\in I}H);
	\\
	\hfill \cc((A_i)_{i\in I}F,(E_i)_{i\in I}H)\bigr)
\end{array}
\\
\begin{tanglec}
\hh \id \hstep \id \step[1.5] \id \step \id \step \id \hstep
\\
\hh \hstep \id \hstep \id \hstep \ffbox4{\mu_{\mathsf{\triangledown I}\:\!\!I\sqcup\mb1\to\mb2}}
\\
\hh \id \hstep \id \step[2.5] \id \step[1.5]
\\
\hh \ffbox4{\mu_{\chi_{I\sqcup\mb1}}} \step[1.5]
\\
\hh \id \step[1.5]
\end{tanglec}
=
\begin{tanglec}
\id \step[3] \XX \step \id \step \id
\\
\hh \ffbox4{\mu_{\chi\:I\sqcup I\to I}} \step \ffbox2{\mu_{\mb2\to\!\mb1}} \hstep \id \hstep
\\
\hh \step[1.5] \id \step[4] \id \step[1.5] \id
\\
\hh \ffbox8{\mu_{\mathsf{\triangledown I\triangledown I}\centerdot\chi_{\mb2}\:I\sqcup\mb1\sqcup I\sqcup\mb1\to\mb2}}
\\
\hh \id
\end{tanglec}
\qquad
\begin{tanglec}
\hstep \id \step \id \step[2.5] \id \Step \XX \step \id \step \id
	\\
\hstep \id \step \id \step[2.5] \XX \Step \id \step \id \step \id
\\
\hh \ffbox5{\mu_{\triangledown\triangledown\:I\sqcup I\to\mb2}} \step \ffbox4{\mu_{1\:\mb2\to\mb2}} \hstep \id \hstep
	\\
\hh \Step \id \step[5.5] \id \step[2.5] \id
	\\
\hh \Step \ffbox9{\mu_{\mathsf{\triangledown I\triangledown I}\centerdot\chi_{\mb2}\:I\sqcup\mb1\sqcup I\sqcup\mb1\to\mb2}}
	\\
\hh \Step \id
\end{tanglec}
=
\begin{tanglec}
\hh \id \hstep \id \step \id \hstep \id \step[2.5] \id \step \id \Step \id \hstep
	\\
\hh \hstep \id \hstep \id \step \id \hstep \id \step[1.5] \ffbox5{\mu_{\triangledown\triangledown\:\mb2\sqcup\mb2\to\mb2}}
	\\
\hh \id \hstep \id \step \id \hstep \id \step[4] \id \Step
	\\
\hh \ffbox9{\mu_{\mathsf{\triangledown I\triangledown I}\centerdot(23)\:\!\!I\sqcup\mb1\sqcup I\sqcup\mb1\to\mb2\sqcup\mb2}}
	\\
\hh \id
\end{tanglec}
\\
\mcv\bigl((\ca_i(A_i,D_i))_{i\in I},[(\ca_i)_{i\in I};\cc](F,G),(\ca_i(D_i,E_i))_{i\in I},[(\ca_i)_{i\in I};\cc](G,H);
\cc((A_i)_{i\in I}F,(E_i)_{i\in I}H)\bigr)
\end{gather*}
On elements:
	\begin{diagram}[LaTeXeqno]
\bigl((\kappa_{A_i,D_i,E_i})_{i\in I},\centerdot,F_{(A_i),(E_i)},p_{(E_i)}\centerdot\bigl) &\rMapsTo &\bigl(b,\nu,\centerdot\bigl) &\lMapsTo &\bigl(F_{(A_i),(D_i)},F_{(D_i),(E_i)},p_{(E_i)},p_{(E_i)},\centerdot,\centerdot,\centerdot\bigl)
		\\
\dMapsTo &&\dMapsTo &&\dMapsTo
		\\
\bigl((\kappa_{A_i,D_i,E_i})_{i\in I},\centerdot,\ev^\VQu_{(A_i),F,(E_i),H}\bigl) &\rMapsTo &tr &\lMapsTo &\bigl(F_{(A_i),(D_i)},F_{(D_i),(E_i)},p_{(E_i)},p_{(E_i)},\centerdot^{(4)}\bigl)
\label{eq-FFpp-tr}
	\end{diagram}

Recall the morphism \(\ev^\VQu\) defined as the diagonal of the commutative square~\eqref{dia-Ai-VQu-B}.
Using it we define a morphism
\begin{multline*}
M =\bigl[ (\ca_i(A_i,D_i))_{i\in I},(\ca_i(D_i,E_i))_{i\in I},[(\ca_i)_{i\in I};\cc](F,G),[(\ca_i)_{i\in I};\cc](G,H)
\\
\rTTo^{F_{(A_i),(D_i)},\ev^\VQu,p_{(E_i)_{i\in I}}}
\cc((A_i)_{i\in I}F,(D_i)_{i\in I}F),\cc((D_i)_{i\in I}F,(E_i)_{i\in I}G),\cc((E_i)_{i\in I}G,(E_i)_{i\in I}H)
\\
\rto{\centerdot^{(3)}} \cc((A_i)_{i\in I}F,(E_i)_{i\in I}H) \bigr].
\end{multline*}
In detail:
\begin{multline}
\mcv\bigl((\ca_i(A_i,D_i))_{i\in I};\cc((A_i)_{i\in I}F,(D_i)_{i\in I}F)\bigr)
\\
\times \mcv\bigl((\ca_i(D_i,E_i))_{i\in I},[(\ca_i)_{i\in I};\cc](F,G);\cc((D_i)_{i\in I}F,(E_i)_{i\in I}G)\bigr)
\\
\times \mcv\bigl([(\ca_i)_{i\in I};\cc](G,H);\cc((E_i)_{i\in I}G,(E_i)_{i\in I}H)\bigr)
	\\
\hskip\multlinegap \times \mcv\bigl(\cc((A_i)_{i\in I}F,(D_i)_{i\in I}F),\cc((D_i)_{i\in I}F,(E_i)_{i\in I}G),\cc((E_i)_{i\in I}G,(E_i)_{i\in I}H); \hfill
\\
\hfill \cc((A_i)_{i\in I}F,(E_i)_{i\in I}H)\bigr) \rTTo^{\mu_{\mathsf{\triangledown\triangledown I}\:I\sqcup(I\sqcup\mb1)\sqcup\mb1\to\mb3}} \hskip\multlinegap
	\\
\mcv\bigl((\ca_i(A_i,D_i))_{i\in I},(\ca_i(D_i,E_i))_{i\in I},[(\ca_i)_{i\in I};\cc](F,G),[(\ca_i)_{i\in I};\cc](G,H);
\cc((A_i)_{i\in I}F,(E_i)_{i\in I}H)\bigr),
	\\
\bigl(F_{(A_i),(D_i)},\ev^\VQu,p_{(E_i)},\centerdot^{(3)}\bigl) \mapsto M.
\label{eq-F-evVQu-p}
\end{multline}
Recall that \(\ev^\VQu\) has two presentations: \eqref{eq-ev-VQu-E} and \eqref{eq-ev-VQu-A}.
The first one gives
\begin{multline*}
\mcv\bigl((\ca_i(A_i,D_i))_{i\in I};\cc((A_i)_{i\in I}F,(D_i)_{i\in I}F)\bigr)
\times\mcv\bigl((\ca_i(D_i,E_i))_{i\in I};\cc((D_i)_{i\in I}F,(E_i)_{i\in I}F)\bigr)
	\\
\times \mcv\bigl([(\ca_i)_{i\in I};\cc](F,G);\cc((E_i)_{i\in I}F,(E_i)_{i\in I}G)\bigr)
\times \mcv\bigl([(\ca_i)_{i\in I};\cc](G,H);\cc((E_i)_{i\in I}G,(E_i)_{i\in I}H)\bigr)
	\\
\hskip\multlinegap \times \mcv\bigl(\cc((A_i)_{i\in I}F,(D_i)_{i\in I}F),\cc((D_i)_{i\in I}F,(E_i)_{i\in I}F),\cc((E_i)_{i\in I}F,(E_i)_{i\in I}G), \hfill
	\\
\hfill \cc((E_i)_{i\in I}G,(E_i)_{i\in I}H);\cc((A_i)_{i\in I}F,(E_i)_{i\in I}H)\bigr)
\rTTo^{\mu_{\mathsf{\triangledown\triangledown II}\:I\sqcup I\sqcup\mb1\sqcup\mb1\to\mb4}} \hskip\multlinegap
	\\
\mcv\bigl((\ca_i(A_i,D_i))_{i\in I},(\ca_i(D_i,E_i))_{i\in I},[(\ca_i)_{i\in I};\cc](F,G),[(\ca_i)_{i\in I};\cc](G,H);
\cc((A_i)_{i\in I}F,(E_i)_{i\in I}H)\bigr),
	\\
\bigl(F_{(A_i),(D_i)},F_{(D_i),(E_i)},p_{(E_i)},p_{(E_i)},\centerdot^{(4)}\bigl) \mapsto M.
\end{multline*}

Apply \propref{pro-equivariance-property} to the square
\begin{diagram}[LaTeXeqno]
I\sqcup\mb1\sqcup I\sqcup\mb1 &\rTTo^{\triangledown\mathsf{I\triangledown I}} &\mb4
	\\
\dTTo<{1\sqcup\varpi\sqcup1}>{=\pi} &\rdTTo[hug]^{\triangledown\mathsf{I\triangledown I}\centerdot(23)} &\dTTo>{(23)}
	\\
I\sqcup I\sqcup\mb1\sqcup\mb1 &\rTTo^{\triangledown\triangledown\mathsf{II}} &\mb4
\label{dia-DI-X-pi-ID4}
\end{diagram}
where \(\varpi=(I+1I\dots21)\) and \(\varpi^{-1}=(12\dots II+1)\).
We have \(\pi_k=1\), \(k=1,2,3,4\).
Equation~\eqref{eq-mur-miao} gives, in particular,
\begin{multline*}
\mu_{\triangledown\mathsf{I\triangledown I}\centerdot(23)} =\bigl[ \mcv((X_i)_{i\in I};Y_1) \times \mcv((U_i)_{i\in I};Y_2) \times \mcv(Z_3;Y_3) \times \mcv(Z_4;Y_4) \times \mcv(Y_1,Y_2,Y_3,Y_4;W)
\\
\rTTo^{\mu_{\triangledown\triangledown\mathsf{II}}} \mcv((X_i)_{i\in I},(U_i)_{i\in I},Z_3,Z_4;W) \rto{r_\pi} \mcv((X_i)_{i\in I},Z_2,(U_i)_{i\in I},Z_4;W) \bigr].
\end{multline*}
This implies \(tr=(M)r_{1\sqcup\varpi\sqcup1}\).

The value of \((M)r_{1\sqcup\varpi\sqcup1}\) is found via presentation~\eqref{eq-ev-VQu-A} of \(\ev^\VQu\) together with \eqref{eq-F-evVQu-p}
\begin{multline}
\mcv\bigl((\ca_i(A_i,D_i))_{i\in I};\cc((A_i)_{i\in I}F,(D_i)_{i\in I}F)\bigr)
\times \mcv\bigl([(\ca_i)_{i\in I};\cc](F,G);\cc((D_i)_{i\in I}F,(D_i)_{i\in I}G)\bigr)
	\\
	\times \mcv\bigl((\ca_i(D_i,E_i))_{i\in I};\cc((D_i)_{i\in I}G,(E_i)_{i\in I}G)\bigr)
	\\
	\times \mcv\bigl(\cc((D_i)_{i\in I}F,(D_i)_{i\in I}G),\cc((D_i)_{i\in I}G,(E_i)_{i\in I}G);\cc((D_i)_{i\in I}F,(E_i)_{i\in I}G)\bigr)
	\\
\times \mcv\bigl([(\ca_i)_{i\in I};\cc](G,H);\cc((E_i)_{i\in I}G,(E_i)_{i\in I}H)\bigr) \times
	\\
\mcv\bigl(\cc((A_i)_{i\in I}F,(D_i)_{i\in I}F),\cc((D_i)_{i\in I}F,(E_i)_{i\in I}G),\cc((E_i)_{i\in I}G,(E_i)_{i\in I}H);
\cc((A_i)_{i\in I}F,(E_i)_{i\in I}H)\bigr)
	\\
	\rTTo^{1\times\mu_{\mathsf{\triangledown I\centerdot X}}\times1\times1}
	\mcv\bigl((\ca_i(A_i,D_i))_{i\in I};\cc((A_i)_{i\in I}F,(D_i)_{i\in I}F)\bigr)
	\\
	\times \mcv\bigl((\ca_i(D_i,E_i))_{i\in I},[(\ca_i)_{i\in I};\cc](F,G);\cc((D_i)_{i\in I}F,(E_i)_{i\in I}G)\bigr)
	\\
\times \mcv\bigl([(\ca_i)_{i\in I};\cc](G,H);\cc((E_i)_{i\in I}G,(E_i)_{i\in I}H)\bigr)
	\\
\hskip\multlinegap \times \mcv\bigl(\cc((A_i)_{i\in I}F,(D_i)_{i\in I}F),\cc((D_i)_{i\in I}F,(E_i)_{i\in I}G),\cc((E_i)_{i\in I}G,(E_i)_{i\in I}H); \hfill
\\
\hfill \cc((A_i)_{i\in I}F,(E_i)_{i\in I}H)\bigr) \rTTo^{\mu_{\mathsf{\triangledown\triangledown I}\:I\sqcup(I\sqcup\mb1)\sqcup\mb1\to\mb3}} \hskip\multlinegap
	\\
\hskip\multlinegap \mcv\bigl((\ca_i(A_i,D_i))_{i\in I},(\ca_i(D_i,E_i))_{i\in I},[(\ca_i)_{i\in I};\cc](F,G),[(\ca_i)_{i\in I};\cc](G,H); \hfill
\\
\hfill \cc((A_i)_{i\in I}F,(E_i)_{i\in I}H)\bigr) \rto{r_{1\sqcup\varpi\sqcup1}} \hskip\multlinegap
\\
\mcv\bigl((\ca_i(A_i,D_i))_{i\in I},[(\ca_i)_{i\in I};\cc](F,G),(\ca_i(D_i,E_i))_{i\in I},[(\ca_i)_{i\in I};\cc](G,H);
\cc((A_i)_{i\in I}F,(E_i)_{i\in I}H)\bigr),
	\\
\bigl(F_{(A_i),(D_i)},p_{(D_i)_{i\in I}},G_{(D_i),(E_i)},\centerdot,p_{(E_i)},\centerdot^{(3)}\bigl) \mapsto
\bigl(F_{(A_i),(D_i)},\ev^\VQu,p_{(E_i)},\centerdot^{(3)}\bigl)
\\
\mapsto M \mapsto (M)r_{1\sqcup\varpi\sqcup1}.
\label{eq-11V}
\end{multline}
Apply \propref{pro-equivariance-property} to the square
\begin{diagram}
I\sqcup(\mb1\sqcup I)\sqcup\mb1 &\rTTo^{\triangledown\triangledown\mathsf{I}} &\mb3
	\\
\dTTo<{1\sqcup\varpi\sqcup1}>{=\pi} &\rdTTo[hug]^{\triangledown\triangledown\mathsf{I}} &\dEq
	\\
I\sqcup(I\sqcup\mb1)\sqcup\mb1 &\rTTo^{\triangledown\triangledown\mathsf{I}} &\mb3
\end{diagram}
where \(\varpi=(I+1I\dots21)\) and \(\varpi^{-1}=(12\dots II+1)\).
We have \(\pi_1=1\), \(\pi_2=\varpi\:\mb1\sqcup I\to I\sqcup\mb1\), \(\pi_3=1\).
Equation~\eqref{eq-mur-miao} gives, in particular,
\begin{multline*}
\bigl[ \mcv((X_i)_{i\in I};Y_1) \times \mcv((U_i)_{i\in I},Z;Y_2) \times \mcv(Q;Y_3) \times \mcv(Y_1,Y_2,Y_3;W) \rTTo^{1\times r_{\pi_2}\times1\times1} 
\\
\mcv((X_i)_{i\in I};Y_1) \times \mcv(Z,(U_i)_{i\in I};Y_2) \times \mcv(Q;Y_3) \times \mcv(Y_1,Y_2,Y_3;W) \rTTo^{\mu_{\triangledown\triangledown\mathsf{I}\:I\sqcup(\mb1\sqcup I)\sqcup\mb1\to\mb3}}
\\
\hfill \mcv((X_i)_{i\in I},Z,(U_i)_{i\in I},Q;W) \bigr] \hskip\multlinegap
\\
\hskip\multlinegap =\bigl[ \mcv((X_i)_{i\in I};Y_1) \times \mcv((U_i)_{i\in I},Z;Y_2) \times \mcv(Q;Y_3) \times \mcv(Y_1,Y_2,Y_3;W) \rTTo^{\mu_{\triangledown\triangledown\mathsf{I}\:I\sqcup(\mb1\sqcup I)\sqcup\mb1\to\mb3}} \hfill
\\
\mcv((X_i)_{i\in I},(U_i)_{i\in I},Z,Q;W) \rTTo^{r_\pi} \mcv((X_i)_{i\in I},Z,(U_i)_{i\in I},Q;W) \bigr].
\end{multline*}
Using this we can transform expression~\eqref{eq-11V} to the left map of
\begin{gather}
\begin{array}{c}
	\mcv\bigl((\ca_i(A_i,D_i))_{i\in I};\cc((A_i)_{i\in I}F,(D_i)_{i\in I}F)\bigr)
	\\
	\times \mcv\bigl([(\ca_i)_{i\in I};\cc](F,G);\cc((D_i)_{i\in I}F,(D_i)_{i\in I}G)\bigr)
	\\
\times \mcv\bigl((\ca_i(D_i,E_i))_{i\in I};\cc((D_i)_{i\in I}G,(E_i)_{i\in I}G)\bigr)
	\\
\times \mcv\bigl(\cc((D_i)_{i\in I}F,(D_i)_{i\in I}G),\cc((D_i)_{i\in I}G,(E_i)_{i\in I}G);\cc((D_i)_{i\in I}F,(E_i)_{i\in I}G)\bigr)
	\\
\times \mcv\bigl([(\ca_i)_{i\in I};\cc](G,H);\cc((E_i)_{i\in I}G,(E_i)_{i\in I}H)\bigr) 
	\\
\hspace*{-0.4em} \times \mcv\bigl(\cc((A_i)_{i\in I}F,(D_i)_{i\in I}F),\cc((D_i)_{i\in I}F,(E_i)_{i\in I}G),\cc((E_i)_{i\in I}G,(E_i)_{i\in I}H);\cc((A_i)_{i\in I}F,(E_i)_{i\in I}H)\bigr)
\end{array}
\notag
\\
\begin{tanglec}
\hh \id \Step \id \step \id \step \id \step \id \step \id
\\
\hh \id \step[1.5] \ffbox3{\mu_{\mathsf{\triangledown I\centerdot X}}} \hstep \id \step \id
\\
\hh \id \step[3] \id \Step \id \step \id
\\
\hh \id \Step \ffbox2{r_{\pi_2}} \step \id \step \id
\\
\hh \id \step[3] \id \Step \id \step \id
\\
\hh \ffbox7{\mu_{\triangledown\triangledown\mathsf{I}\:\!\!I\sqcup(\mb1\sqcup I)\sqcup\mb1\to\mb3}}
\\
\hh \id
\end{tanglec}
\;=\;
\begin{tanglec}
	\hh \id \Step \id \step \id \step \id \step \id \step \id
	\\
\hh \id \step[1.5] \ffbox3{\mu_{\mathsf{I}\triangledown}} \hstep \id \step \id
	\\
	\hh \id \step[3] \id \Step \id \step \id
	\\
	\hh \ffbox7{\mu_{\triangledown\triangledown\mathsf{I}\:\!\!I\sqcup(\mb1\sqcup I)\sqcup\mb1\to\mb3}}
	\\
	\hh \id
\end{tanglec}
\label{eq-14V}
\\
\mcv\bigl((\ca_i(A_i,D_i))_{i\in I},[(\ca_i)_{i\in I};\cc](F,G),(\ca_i(D_i,E_i))_{i\in I},[(\ca_i)_{i\in I};\cc](G,H);\cc((A_i)_{i\in I}F,(E_i)_{i\in I}H)\bigr), \notag
\end{gather}

On elements we have
\begin{multline*}
\bigl(F_{(A_i),(D_i)},p_{(D_i)_{i\in I}},G_{(D_i),(E_i)},\centerdot,p_{(E_i)},\centerdot^{(3)}\bigl) \mapsto
\bigl(F_{(A_i),(D_i)},\ev^\VQu,p_{(E_i)},\centerdot^{(3)}\bigl)
\\
\mapsto \bigl(F_{(A_i),(D_i)},?,p_{(E_i)},\centerdot^{(3)}\bigl) \mapsto (M)r_{1\sqcup\varpi\sqcup1} =tr.
\end{multline*}

In order to justify passage from the left to the right map, let us analyse a piece of the left map
\begin{multline*}
\mcv\bigl([(\ca_i)_{i\in I};\cc](F,G);\cc((D_i)_{i\in I}F,(D_i)_{i\in I}G)\bigr)
\times \mcv\bigl((\ca_i(D_i,E_i))_{i\in I};\cc((D_i)_{i\in I}G,(E_i)_{i\in I}G)\bigr)
	\\
\times \mcv\bigl(\cc((D_i)_{i\in I}F,(D_i)_{i\in I}G),\cc((D_i)_{i\in I}G,(E_i)_{i\in I}G);\cc((D_i)_{i\in I}F,(E_i)_{i\in I}G)\bigr)
	\\
\rTTo^{\mu_{\mathsf{\triangledown I\centerdot X}}}
\mcv\bigl((\ca_i(D_i,E_i))_{i\in I},[(\ca_i)_{i\in I};\cc](F,G);\cc((D_i)_{i\in I}F,(E_i)_{i\in I}G)\bigr)
	\\
\rTTo^{r_{\varpi\:\mb1\sqcup I\to I\sqcup\mb1}} 
\mcv\bigl([(\ca_i)_{i\in I};\cc](F,G),(\ca_i(D_i,E_i))_{i\in I};\cc((D_i)_{i\in I}F,(E_i)_{i\in I}G)\bigr),
	\\
\bigl(p_{(D_i)_{i\in I}},G_{(D_i),(E_i)},\centerdot\bigl) \mapsto ? \mapsto ??.
\end{multline*}
Apply \propref{pro-equivariance-property} to the square
\begin{diagram}
\mb1\sqcup I &\rTTo^{\mathsf I\triangledown} &\mb2
	\\
\dTTo<{\varpi} &\rdTTo^{\mathsf I\triangledown} &\dEq
	\\
I\sqcup\mb1 &\rTTo^{\triangledown\mathsf{I\centerdot X}} &\mb2
\end{diagram}
We have \(\varpi_1=1_I\), \(\varpi_2=1_{\mb1}\).
Equation~\eqref{eq-mur-miao} gives, in particular, that the above composition equals
\begin{multline*}
\mcv\bigl([(\ca_i)_{i\in I};\cc](F,G);\cc((D_i)_{i\in I}F,(D_i)_{i\in I}G)\bigr)
\times \mcv\bigl((\ca_i(D_i,E_i))_{i\in I};\cc((D_i)_{i\in I}G,(E_i)_{i\in I}G)\bigr)
	\\
	\times \mcv\bigl(\cc((D_i)_{i\in I}F,(D_i)_{i\in I}G),\cc((D_i)_{i\in I}G,(E_i)_{i\in I}G);\cc((D_i)_{i\in I}F,(E_i)_{i\in I}G)\bigr)
	\\
\rTTo^{\mu_{\mathsf{I}\triangledown\:\mb1\sqcup I\to\mb2}}
\mcv\bigl([(\ca_i)_{i\in I};\cc](F,G),(\ca_i(D_i,E_i))_{i\in I};\cc((D_i)_{i\in I}F,(E_i)_{i\in I}G)\bigr),
	\\
\bigl(p_{(D_i)_{i\in I}},G_{(D_i),(E_i)},\centerdot\bigl) \mapsto ??.
\end{multline*}
Hence, \eqref{eq-14V} sends \(\bigl(F_{(A_i),(D_i)},p_{(D_i)_{i\in I}},G_{(D_i),(E_i)},\centerdot,p_{(E_i)},\centerdot^{(3)}\bigl)\) to \((M)r_{1\sqcup\varpi\sqcup1}=tr\).
Due to associativity of composition in $\cc$ this map coincides with the right vertical map in \eqref{eq-FpGp-lb}.
Therefore, $lb=tr$.

Let us prove coherence of \(\ev^\VQu\) with the units \eqref{eq-coherence-with-units}:
\begin{multline}
\bigl[ () \rTTo^{(\id_{A_i})_{i\in I},\id_F} (\ca_i(A_i,A_i))_{i\in I},[(\ca_i)_{i\in I};\cc](F,F) \rTTo^{\ev^\VQu} \cc((A_i)_{i\in I}F,(A_i)_{i\in I}F) \bigr]
\\
=\id_{(A_i)_{i\in I}F}.
\label{eq-coherence-ev-with-units}
\end{multline}
Recall that \(\id_F=\bigl(\id_{(A_i)_{i\in I}F}\:()\to\cc((A_i)_{i\in I}F,(A_i)_{i\in I}F)\bigr)_{(A_i\in\ca_i)_{i\in I}}\).
Using \eqref{eq-ev-VQu-A} we conclude that the left hand side ($lhs$) is obtained via the left map in the following diagram.
Using associativity equation at \figref{dia-assoc-mu-multi} for maps \(\varnothing\to I\sqcup\mb1 \rTTo^{\mathsf{\triangledown I\centerdot X}} \mb2\) we get the right map in
\begin{gather*}
\begin{array}{c}
	\prod_{i\in I} \mcv\bigl(;\ca_i(A_i,A_i)\bigr) \times \mcv\bigl(;[(\ca_i)_{i\in I};\cc](F,F)\bigr)
	\\
	\times \mcv\bigl([(\ca_i)_{i\in I};\cc](F,F);\cc((A_i)_{i\in I}F,(A_i)_{i\in I}F)\bigr)
	\\
	\times \mcv\bigl((\ca_i(A_i,A_i))_{i\in I}; \cc((A_i)_{i\in I}F,(A_i)_{i\in I}F)\bigr) 
	\\
\times \mcv\bigl(\cc((A_i)_{i\in I}F,(A_i)_{i\in I}F),\cc((A_i)_{i\in I}F,(A_i)_{i\in I}F);\cc((A_i)_{i\in I}F,(A_i)_{i\in I}F)\bigr)
\end{array}
\\
\begin{tanglec}
	\hh \id \step \id \step \id \step \id \step \id
	\\
\hh \hstep \id \step \id \hstep \ffbox3{\mu_{\mathsf{\triangledown I\centerdot X}}}
	\\
\hh \id \step \id \Step \id \step
	\\
\hh \ffbox4{\mu_{\varnothing\to I\sqcup\mb1}} \step
	\\
\hh \id \step
\end{tanglec}
\;=\;
\begin{tanglec}
\XX \Step \id \Step \id \step \id
	\\
\id \Step \XX \Step \id \step \id
	\\
\hh \ffbox3{\mu_{\varnothing\to\mb1}} \step \ffbox3{\mu_{\varnothing\to I}} \hstep \id \hstep
	\\
\hh \step \id \step[4] \id \Step \id
	\\
\hh \step \ffbox7{\mu_{\varnothing\to\mb2}}
	\\
\hh \step \id
\end{tanglec}
\\
\mcv\bigl(;\cc((A_i)_{i\in I}F,(A_i)_{i\in I}F)\bigr)
\end{gather*}
On elements we have
\begin{diagram}
\bigl((\id_{A_i})_{i\in I},\id_F,p_{(A_i)_{i\in I}},F_{(A_i),(A_i)},\centerdot\bigr) &\rMapsTo &\bigl(\id_{(A_i)_{i\in I}F},\id_{(A_i)_{i\in I}F},\centerdot\bigr)
	\\
	\dMapsTo &&\dMapsTo
	\\
\bigl((\id_{A_i})_{i\in I},\id_F,\ev^\VQu\bigr) &\rMapsTo &lhs =\id_{(A_i)_{i\in I}F}
\end{diagram}
This proves equation~\eqref{eq-coherence-ev-with-units}.
\end{proof}

\begin{example}
Assume that $\cv$ is a complete closed symmetric monoidal category.
For \(\mcv=\wh\cv\) (see \cite[Proposition~3.22]{BesLyuMan-book}) we get
\begin{multline*}
	\cv\bigl(\tens^{i\in I}(\ca_i(A_i,D_i)),\cc((A_i)_{i\in I}F,(D_i)_{i\in I}F)\bigr)
\times \cv\bigl(\tens^{i\in I}(\ca_i(D_i,E_i)),\cc((D_i)_{i\in I}F,(E_i)_{i\in I}F)\bigr) 
\\
\times \cv\bigl([(\ca_i)_{i\in I};\cc](F,G),\cc((E_i)_{i\in I}F,(E_i)_{i\in I}G)\bigr)
	\\
\times \cv\bigl(\cc((D_i)_{i\in I}F,(E_i)_{i\in I}F)\tens\cc((E_i)_{i\in I}F,(E_i)_{i\in I}G),\cc((D_i)_{i\in I}F,(E_i)_{i\in I}G)\bigr)
	\\
\times \cv\bigl([(\ca_i)_{i\in I};\cc](G,H),\cc((E_i)_{i\in I}G,(E_i)_{i\in I}H)\bigr) \times
	\\
\cv\bigl(\cc((A_i)_{i\in I}F,(D_i)_{i\in I}F)\tens\cc((D_i)_{i\in I}F,(E_i)_{i\in I}G)\tens\cc((E_i)_{i\in I}G,(E_i)_{i\in I}H),
\cc((A_i)_{i\in I}F,(E_i)_{i\in I}H)\bigr)
	\\
\rTTo^{1\times\mu_{\mathsf{I\triangledown\centerdot X}\:\mb1\sqcup I\to\mb2}\times1\times1}
\cv\bigl(\tens^{i\in I}(\ca_i(A_i,D_i)),\cc((A_i)_{i\in I}F,(D_i)_{i\in I}F)\bigr)
	\\
\times \cv\bigl([(\ca_i)_{i\in I};\cc](F,G)\tens\tens^{i\in I}(\ca_i(D_i,E_i)),\cc((D_i)_{i\in I}F,(E_i)_{i\in I}G)\bigr)
	\\
	\times \cv\bigl([(\ca_i)_{i\in I};\cc](G,H),\cc((E_i)_{i\in I}G,(E_i)_{i\in I}H)\bigr)
	\\
\hskip\multlinegap \times \cv\bigl(\cc((A_i)_{i\in I}F,(D_i)_{i\in I}F)\tens\cc((D_i)_{i\in I}F,(E_i)_{i\in I}G)\tens\cc((E_i)_{i\in I}G,(E_i)_{i\in I}H), \hfill
\\
\hfill \cc((A_i)_{i\in I}F,(E_i)_{i\in I}H)\bigr) \rTTo^{\mu_{\triangledown\triangledown\mathsf{I}\:I\sqcup(\mb1\sqcup I)\sqcup\mb1\to\mb3}} \hskip\multlinegap
	\\
\hskip\multlinegap \cv\bigl(\tens^{i\in I}(\ca_i(A_i,D_i))\tens[(\ca_i)_{i\in I};\cc](F,G)\tens\tens^{i\in I}(\ca_i(D_i,E_i))\tens[(\ca_i)_{i\in I};\cc](G,H), \hfill
\\
\hfill \cc((A_i)_{i\in I}F,(E_i)_{i\in I}H)\bigr), \hskip\multlinegap
	\\
	(F_{(A_i),(D_i)},F_{(D_i),(E_i)},p_{(E_i)_{i\in I}},\centerdot,p_{(E_i)_{i\in I}},\centerdot^{(3)}) \mapsto
	(F_{(A_i),(D_i)},a,p_{(E_i)_{i\in I}},\centerdot^{(3)}) \mapsto lb,
\end{multline*}
Here
\begin{multline*}
a =\bigl[ [(\ca_i)_{i\in I};\cc](F,G)\tens\tens^{i\in I}(\ca_i(D_i,E_i)) \rto c \tens^{i\in I}(\ca_i(D_i,E_i))\tens[(\ca_i)_{i\in I};\cc](F,G)
\\
\rTTo^{F_{(A_i),(D_i)}\tens p_{(E_i)_{i\in I}}} \cc((D_i)_{i\in I}F,(E_i)_{i\in I}F)\tens\cc((E_i)_{i\in I}F,(E_i)_{i\in I}G)
\rto\centerdot \cc((D_i)_{i\in I}F,(E_i)_{i\in I}G) \bigr].
\end{multline*}
Hence,
\begin{multline*}
lb =\bigl[ \tens^{i\in I}(\ca_i(A_i,D_i))\tens[(\ca_i)_{i\in I};\cc](F,G)\tens\tens^{i\in I}(\ca_i(D_i,E_i))\tens[(\ca_i)_{i\in I};\cc](G,H)
\\
\rTTo^{F_{(A_i),(D_i)}\tens a\tens p_{(E_i)_{i\in I}}}
\cc((A_i)_{i\in I}F,(D_i)_{i\in I}F)\tens\cc((D_i)_{i\in I}F,(E_i)_{i\in I}G)\tens\cc((E_i)_{i\in I}G,(E_i)_{i\in I}H)
\\
\hfill \rto{\centerdot^{(3)}} \cc((A_i)_{i\in I}F,(E_i)_{i\in I}H) \bigr] \hskip\multlinegap
\\
=\bigl[ \tens^{i\in I}(\ca_i(A_i,D_i))\tens[(\ca_i)_{i\in I};\cc](F,G)\tens\tens^{i\in I}(\ca_i(D_i,E_i))\tens[(\ca_i)_{i\in I};\cc](G,H)
\\
\rTTo^{1\tens c\tens1} \tens^{i\in I}(\ca_i(A_i,D_i))\tens\tens^{i\in I}(\ca_i(D_i,E_i))\tens[(\ca_i)_{i\in I};\cc](F,G)\tens[(\ca_i)_{i\in I};\cc](G,H)
\\
\rTTo^{F_{(A_i),(D_i)}\tens F_{(D_i),(E_i)}\tens p_{(E_i)_{i\in I}}\tens p_{(E_i)_{i\in I}}}
\\
\cc((A_i)_{i\in I}F,(D_i)_{i\in I}F)\tens\cc((D_i)_{i\in I}F,(E_i)_{i\in I}F)\tens\cc((E_i)_{i\in I}F,(E_i)_{i\in I}G)
\tens\cc((E_i)_{i\in I}G,(E_i)_{i\in I}H)
\\
\rto{\centerdot\tens\centerdot} \cc((A_i)_{i\in I}F,(E_i)_{i\in I}F)\tens\cc((E_i)_{i\in I}F,(E_i)_{i\in I}H) \rto\centerdot \cc((A_i)_{i\in I}F,(E_i)_{i\in I}H) \bigr]
\\
=\bigl[ \tens^{i\in I}(\ca_i(A_i,D_i))\tens[(\ca_i)_{i\in I};\cc](F,G)\tens\tens^{i\in I}(\ca_i(D_i,E_i))\tens[(\ca_i)_{i\in I};\cc](G,H) 
\\
\rTTo^{F_{(A_i),(D_i)}\tens p_{(E_i)_{i\in I}}\tens F_{(D_i),(E_i)}\tens p_{(E_i)_{i\in I}}}
\\
\cc((A_i)_{i\in I}F,(D_i)_{i\in I}F)\tens\cc((E_i)_{i\in I}F,(E_i)_{i\in I}G)\tens\cc((D_i)_{i\in I}F,(E_i)_{i\in I}F)
\tens\cc((E_i)_{i\in I}G,(E_i)_{i\in I}H)
\\
\hskip\multlinegap \rTTo^{1\tens c\tens1} \cc((A_i)_{i\in I}F,(D_i)_{i\in I}F)\tens\cc((D_i)_{i\in I}F,(E_i)_{i\in I}F)\tens\cc((E_i)_{i\in I}F,(E_i)_{i\in I}G) \hfill
\\
\hfill \tens\cc((E_i)_{i\in I}G,(E_i)_{i\in I}H) \hskip\multlinegap
\\
\rto{\centerdot\tens\centerdot} \cc((A_i)_{i\in I}F,(E_i)_{i\in I}F)\tens\cc((E_i)_{i\in I}F,(E_i)_{i\in I}H) \rto\centerdot \cc((A_i)_{i\in I}F,(E_i)_{i\in I}H) \bigr]
\\
=\bigl[ \tens^{i\in I}(\ca_i(A_i,D_i))\tens[(\ca_i)_{i\in I};\cc](F,G)\tens\tens^{i\in I}(\ca_i(D_i,E_i))\tens[(\ca_i)_{i\in I};\cc](G,H)
\\
\rTTo^{1\tens c\tens1} \tens^{i\in I}(\ca_i(A_i,D_i))\tens\tens^{i\in I}(\ca_i(D_i,E_i))\tens[(\ca_i)_{i\in I};\cc](F,G)\tens[(\ca_i)_{i\in I};\cc](G,H)
\\
\rTTo^{\lambda^{\chi\:\mb2I\to I}\tens\centerdot} \tens^{i\in I}(\ca_i(A_i,D_i)\tens\ca_i(D_i,E_i))\tens[(\ca_i)_{i\in I};\cc](F,H)
\\
\rTTo^{\tens^I(\centerdot)\tens p_{(E_i)_{i\in I}}}
\tens^{i\in I}(\ca_i(A_i,E_i))\tens\cc((E_i)_{i\in I}F,(E_i)_{i\in I}H) \rTTo^{F_{(A_i),(E_i)}\tens1} 
\\
\cc((A_i)_{i\in I}F,(E_i)_{i\in I}F)\tens\cc((E_i)_{i\in I}F,(E_i)_{i\in I}H) \rto\centerdot \cc((A_i)_{i\in I}F,(E_i)_{i\in I}H) \bigr] =tr.
\end{multline*}
\end{example}

\begin{proposition}\label{pro-VCat-closed}
Let $\mcv$ be a locally small symmetric closed complete multicategory.
The symmetric multicategory $\VCat$ is closed.
\end{proposition}

\begin{proof}
Let \((\ca_i)_{i\in I}\), $\cc$ be (families of) small $\mcv$\n-categories.
Define a full $\mcv$\n-subquiver \(\und\VCat\bigl((\ca_i)_{i\in I};\cc\bigr)\) of a small $\mcv$\n-quiver \(\und\VQu\bigl((\ca_i)_{i\in I};\cc\bigr)\) introduced in \propref{pro-und-VQu}:
\begin{myitemize}
\item[---] \(\Ob\und\VCat\bigl((\ca_i)_{i\in I};\cc\bigr)=\VCat\bigl((\ca_i)_{i\in I};\cc\bigr)\);
\item[---] \(\und\VCat\bigl((\ca_i)_{i\in I};\cc\bigr)(F,G)=\und\VQu\bigl((\ca_i)_{i\in I};\cc\bigr)(F,G)={\displaystyle\int_{(A_i\in\ca_i)_{i\in I}}}\hspace*{-1.4em}\cc((A_i)_{i\in I}F,(A_i)_{i\in I}G)\), the equalizer in multicategory $\mcv$ of pair of morphisms \eqref{eq-int-c(FE-GE)-VQu}.
\end{myitemize}

Use the multi-entry $\mcv$\n-functor (see \propref{con-evVCat-VCat})
\begin{align*}
\ev^\VCat =\ev^\VQu\big| \: (\ca_i)_{i\in I}, \und\VCat\bigl((\ca_i)_{i\in I};\cc\bigr) &\longrightarrow \cc
\\
\bigl((A_i)_{i\in I};F\bigr) &\longmapsto (A_i)_{i\in I}F.
\end{align*}

Restricting \eqref{eq-ev-VQu-E} to $\VCat$ we get that the evaluation element can be obtained via
\begin{multline}
\mcv\bigl((\ca_i(A_i,E_i))_{i\in I}; \cc((A_i)_{i\in I}F,(E_i)_{i\in I}F)\bigr)
\times \mcv\bigl(\und\VCat\bigl((\ca_i)_{i\in I};\cc\bigr)(F,G);\cc((E_i)_{i\in I}F,(E_i)_{i\in I}G)\bigr)
\\
\times \mcv\bigl(\cc((A_i)_{i\in I}F,(E_i)_{i\in I}F),\cc((E_i)_{i\in I}F,(E_i)_{i\in I}G);\cc((A_i)_{i\in I}F, (E_i)_{i\in I}G)\bigr)
\\
\rTTo^{\mu_{\mathsf{\triangledown I}}} \mcv\bigl((\ca_i(A_i,E_i))_{i\in I},\und\VCat\bigl((\ca_i)_{i\in I};\cc\bigr)(F,G);\cc((A_i)_{i\in I}F,(E_i)_{i\in I}G)\bigr),
\\
\bigl(F_{(A_i),(E_i)},p_{(E_i)_{i\in I}},\centerdot\bigr) \mapsto \ev^\VCat.
\label{eq-evVCat-E}
\end{multline}

Looking at another path of commutative diagram \eqref{dia-Ai-VQu-B} we get another presentation of $\ev^\VCat$.
Restricting \eqref{eq-ev-VQu-A} we conclude that the evaluation element can be obtained via
\begin{multline}
\mcv\bigl(\und\VCat\bigl((\ca_i)_{i\in I};\cc\bigr)(F,G);\cc((A_i)_{i\in I}F,(A_i)_{i\in I}G)\bigr)
\times \mcv\bigl((\ca_i(A_i,E_i))_{i\in I}; \cc((A_i)_{i\in I}G,(E_i)_{i\in I}G)\bigr)
\\
\times \mcv\bigl(\cc((A_i)_{i\in I}F,(A_i)_{i\in I}G),\cc((A_i)_{i\in I}G,(E_i)_{i\in I}G);\cc((A_i)_{i\in I}F, (E_i)_{i\in I}G)\bigr)
\\
\rTTo^{\mu_{\mathsf{\triangledown I\centerdot X}}} \mcv\bigl((\ca_i(A_i,E_i))_{i\in I},\und\VCat\bigl((\ca_i)_{i\in I}; \cc\bigr)(F,G);\cc((A_i)_{i\in I}F,(E_i)_{i\in I}G)\bigr),
\\
\bigl(p_{(A_i)_{i\in I}},G_{(A_i),(E_i)},\centerdot\bigr) \mapsto \ev^\VCat.
\label{eq-evVCat-A}
\end{multline}
Thus, \eqref{eq-evVCat-E} and \eqref{eq-evVCat-A} are giving the same element \(\ev^\VCat\).

Let \((\ca_i)_{i\in I}\), \((\cb_j)_{j\in J}\), $\cc$ be (families of) small $\mcv$\n-categories.
According to \propref{pro-und-VQu} there is a map
\[ \varPhi\:\VQu\bigl((\cb_j)_{j\in J};\und\VCat((\ca_i)_{i\in I};\cc)\bigr)\to\VQu\bigl((\ca_i)_{i\in I},(\cb_j)_{j\in J};\cc\bigr).
\]
Let us provide a map in the other direction
\[ \varPsi\:\VCat\bigl((\ca_i)_{i\in I},(\cb_j)_{j\in J};\cc\bigr)\to\VQu\bigl((\cb_j)_{j\in J};\und\VCat((\ca_i)_{i\in I};\cc)\bigr).
\]
Let \(g\:(\ca_i)_{i\in I},(\cb_j)_{j\in J}\to\cc\in\VCat\).
For any family of objects \(B_j\in\Ob\cb_j\), $j\in J$, define a multi-entry $\mcv$\n-functor
\begin{equation}
(B_j)_{j\in J}f =\bigl[ (\ca_i)_{i\in I} \rTTo^{(\Id)_I,(\ddot B_j)_{j\in J}} (\ca_i)_{i\in I},(\cb_j)_{j\in J} \rto g \cc \bigr] \in\VCat.
\label{eq-Bf-A-AB-C-VCat}
\end{equation}
In detail:
\begin{multline*}
\hspace*{-0.5em} \bigl[ \prod_{i\in I}\VCat(\ca_i;\ca_i) \bigr] \times  \bigl[ \prod_{j\in J}\VCat(;\cb_j) \bigr] \times \VCat\bigl((\ca_i)_{i\in I},(\cb_j)_{j\in J};\cc\bigr)
\rTTo^{\mu_{\inj_1\:I\hookrightarrow I\sqcup J}} \VCat\bigl((\ca_i)_{i\in I};\cc\bigr)
\\
\bigl((\Id_{\ca_i})_{i\in I},(\ddot B_j)_{j\in J},g\bigr) \mapsto (B_j)_{j\in J}f.
\end{multline*}
This defines a map \(\Ob f\:\prod_{j\in J}\Ob\cb_j\to\Ob\und\VCat\bigl((\ca_i)_{i\in I};\cc\bigr)=\VCat\bigl((\ca_i)_{i\in I};\cc\bigr)\).
On morphisms we have
\begin{multline*}
(B_j)_{j\in J}f_{(A_i),(D_i)}=\bigl[(\ca_i(A_i,D_i))_{i\in I} \rTTo^{(1)_I,(\id)_J} (\ca_i(A_i,D_i))_{i\in I},(\cb_j(B_j,B_j))_{j\in J}
\\
\rTTo^{g_{(A_i),(B_j),(D_i),(B_j)}} \cc(((A_i)_{i\in I},(B_j)_{j\in J})g,((D_i)_{i\in I},(B_j)_{j\in J})g)
\bigr].
\end{multline*}
In detail:
\begin{multline}
\mu_{\inj_1\:I\hookrightarrow I\sqcup J}\: \prod_{i\in I}\mcv\bigl(\ca_i(A_i,D_i);\ca_i(A_i,D_i)\bigr) \times \prod_{j\in J}\mcv(;\cb_j(B_j,B_j))
\\
\times \mcV\bigl((\ca_i(A_i,D_i))_{i\in I},(\cb_j(B_j,B_j))_{j\in J};\cc(((A_i)_{i\in I},(B_j)_{j\in J})g,((D_i)_{i\in I},(B_j)_{j\in J})g)\bigr)
\\
\hfill \to \mcV\bigl((\ca_i(A_i,D_i))_{i\in I};\cc(((A_i)_{i\in I},(B_j)_{j\in J})g,((D_i)_{i\in I},(B_j)_{j\in J})g)\bigr) \hskip\multlinegap
\\
\bigl((1_{\ca_i(A_i,D_i)})_{i\in I},(\id_{B_j})_{j\in J},g_{(A_i),(B_j),(D_i),(B_j)}\bigr) \mapsto 
(B_j)_{j\in J}f_{(A_i),(D_i)}.
\label{eq-(Bj)f}
\end{multline}

Let us introduce a $\mcv$\n-quiver \(\ovl\VCat\bigl((\ca_i)_{i\in I};\cc\bigr)\) with
\begin{myitemize}
	\item[---] \(\Ob\ovl\VCat\bigl((\ca_i)_{i\in I};\cc\bigr)=\VCat\bigl((\ca_i)_{i\in I};\cc\bigr)\),
\item[---] \(\ovl\VCat\bigl((\ca_i)_{i\in I};\cc\bigr)(F,G)=\prod_{(A_i\in\ca_i)_{i\in I}}\cc((A_i)_{i\in I}F,(A_i)_{i\in I}G)\).
\end{myitemize}

With $g$ we are given elements
\begin{multline*}
g_{(A_i),(B_j),(D_i),(E_j)}
\\
\in\mcv\bigl((\ca_i(A_i,D_i))_{i\in I},(\cb_j(B_j,E_j))_{j\in J};\cc(((A_i)_{i\in I},(B_j)_{j\in J})g,((D_i)_{i\in I},(E_j)_{j\in J})g)\bigr).
\end{multline*}
Using them we define elements
\begin{multline}
\mu_{\inj_2\:J\hookrightarrow I\sqcup J}\: \bigl[ \prod_{i\in I}\mcv\bigl(;\ca_i(A_i,A_i)\bigr) \bigr] \times \bigl[ \prod_{j\in J}\mcv(\cb_j(B_j,E_j);\cb_j(B_j,E_j)) \bigr]
	\\
\times \mcV\bigl((\ca_i(A_i,A_i))_{i\in I},(\cb_j(B_j,E_j))_{j\in J};\cc(((A_i)_{i\in I},(B_j)_{j\in J})g,((A_i)_{i\in I},(E_j)_{j\in J})g)\bigr)
\\
\hfill \to \mcV\bigl((\cb_j(B_j,E_j))_{j\in J};\cc(((A_i)_{i\in I},(B_j)_{j\in J})g,((A_i)_{i\in I},(E_j)_{j\in J})g)\bigr) \hskip\multlinegap
	\\
\bigl((\id_{A_i})_{i\in I},(1_{\cb_j(B_j,E_j)})_{j\in J},g_{(A_i),(B_j),(A_i),(E_j)}\bigr) \mapsto (A_i)_{i\in I}\bar f_{(B_j),(E_j)}.
\label{eq-(Ai)fbar}
\end{multline}
So we define \(\bar f\:(\cb_j)_{j\in J}\to\ovl\VCat\bigl((\ca_i)_{i\in I};\cc\bigr)\), \((B_j)_{j\in J}\mapsto(B_j)_{j\in J}f\) as
\begin{multline*}
\bar f_{(B_j),(E_j)} =((A_i)_{i\in I}\bar f_{(B_j),(E_j)})_{(A_i\in\ca_i)_{i\in I}}
\\
\in \prod_{(A_i\in\ca_i)_{i\in I}} \mcV\bigl((\cb_j(B_j,E_j))_{j\in J};\cc(((A_i)_{i\in I},(B_j)_{j\in J})g,((A_i)_{i\in I},(E_j)_{j\in J})g)\bigr)
\\
\cong \mcV\bigl((\cb_j(B_j,E_j))_{j\in J};\ovl\VCat((\ca_i)_{i\in I};\cc)((B_j)_{j\in J}f,(E_j)_{j\in J}f)\bigr).
\end{multline*}
Let us show that this element is sent by the following two maps to the same element:
\begin{multline*}
\mcv\bigl((\cb_j(B_j,E_j))_{j\in J};\hspace*{-0.5em} \prod_{(A_i\in\ca_i)_{i\in I}} \hspace*{-0.5em} \cc(((A_i)_{i\in I},(B_j)_{j\in J})g,((A_i)_{i\in I},(E_j)_{j\in J})g)\bigr)
\pile{\rTTo^{\mcv((1)_J;(\pr_{(D_i)}\centerdot\beta))} \\ \rTTo_{\mcv((1)_J;(\pr_{(A_i)}\centerdot\gamma))}}
\\
\mcv\bigl((\cb_j(B_j,E_j))_{j\in J};\hspace*{-1.2em} \prod_{(A_i,D_i\in\ca_i)_{i\in I}} \hspace*{-1.2em} \und\mcV\bigl((\ca_i(A_i,D_i))_{i\in I};\cc(((A_i)_{i\in I},(B_j)_{j\in J})g,((D_i)_{i\in I},(E_j)_{j\in J})g)\bigr),
\end{multline*}
Equivalently, for any $A_i,D_i\in\Ob\ca_i$, $i\in I$, \(B_j,E_j\in\Ob\cb_j\), $j\in J$, the following square is commutative:
\begin{diagram}[nobalance]
\bigl(\cb_j(B_j,E_j)\bigr)_{j\in J} &\rTTo^{\hspace*{-2.5em}(D_i)_{i\in I}\bar f_{(B_j),(E_j)}} &\cc(((D_i)_{i\in I},(B_j)_{j\in J})g,((D_i)_{i\in I},(E_j)_{j\in J})g)
\\
\dTTo<{(A_i)_{i\in I}\bar f_{(B_j),(E_j)}} &&\dTTo>\beta
\\
\cc(((A_i)_{i\in I},(B_j)_{j\in J})g,((A_i)_{i\in I},(E_j)_{j\in J})g)\hspace*{-2.6em} &&
\\
&\rdTTo(2,1)^\gamma
&\hspace*{-8.5em}\und\mcV\bigl((\ca_i(A_i,D_i))_{i\in I};\cc(((A_i)_{i\in I},(B_j)_{j\in J})g,((D_i)_{i\in I},(E_j)_{j\in J})g)\bigr)
\end{diagram}
By closedness of $\mcv$ this is equivalent to commutativity of
\begin{diagram}
\bigl(\ca_i(A_i,D_i)\bigr)_{i\in I},\bigl(\cb_j(B_j,E_j)\bigr)_{j\in J}
\\
\dTTo<{(1)_I,(A_i)_{i\in I}\bar f_{(B_j),(E_j)}} &\rdTTo(2,1)>{(1)_I,(D_i)_{i\in I}\bar f_{(B_j),(E_j)}}
&\hspace*{-7em} \bigl(\ca_i(A_i,D_i)\bigr)_{i\in I},\cc\bigl(((D_i)_{i\in I},(B_j)_{j\in J})g,((D_i)_{i\in I},(E_j)_{j\in J})g\bigr)
	\\
\bigl(\ca_i(A_i,D_i)\bigr)_{i\in I},\cc\bigl(((A_i)_{i\in I},(B_j)_{j\in J})g,((A_i)_{i\in I},(E_j)_{j\in J})g\bigr) \hspace*{-7em} &&\dTTo>{\beta^\dagger}
	\\
&\rdTTo(2,1)~{\gamma^\dagger} &\cc\bigl(((A_i)_{i\in I},(B_j)_{j\in J})g,((D_i)_{i\in I},(E_j)_{j\in J})g\bigr)
\end{diagram}
where \(\beta^\dagger\) is given by \eqref{eq-beta-dagger} and \(\gamma^\dagger\) is given by \eqref{eq-gamma-dagger}.
In detail, $tr=lb$ where these elements are obtained as follows.
$tr$ comes from
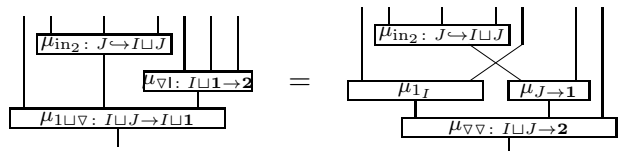
\begin{gather*}
\begin{array}{c}
\bigl[\prod_{i\in I}\mcv\bigl(\ca_i(A_i,D_i);\ca_i(A_i,D_i)\bigr)\bigr] \times \bigl[\prod_{i\in I}\mcv\bigl(;\ca_i(D_i,D_i)\bigr)\bigr] \times \bigl[ \prod_{j\in J}\mcv\bigl(\cb_j(B_j,E_j);\cb_j(B_j,E_j)\bigr) \bigr]
	\\
	\times \mcV\bigl((\ca_i(D_i,D_i))_{i\in I},(\cb_j(B_j,E_j))_{j\in J};\cc(((D_i)_{i\in I},(B_j)_{j\in J})g,((D_i)_{i\in I},(E_j)_{j\in J})g)\bigr)
	\\
\times \mcv\bigl((\ca_i(A_i,D_i))_{i\in I};\cc(((A_i)_{i\in I},(B_j)_{j\in J})g,((D_i)_{i\in I},(B_j)_{j\in J})g)\bigr) 
	\\
\times \mcv\bigl(\cc(((D_i)_{i\in I},(B_j)_{j\in J})g,((D_i)_{i\in I},(E_j)_{j\in J})g);\cc(((D_i)_{i\in I},(B_j)_{j\in J})g,((D_i)_{i\in I},(E_j)_{j\in J})g)\bigr)
	\\
	\times \mcv\bigl(\cc(((A_i)_{i\in I},(B_j)_{j\in J})g,((D_i)_{i\in I},(B_j)_{j\in J})g), \hfill
	\\
	\cc(((D_i)_{i\in I},(B_j)_{j\in J})g,((D_i)_{i\in I},(E_j)_{j\in J})g);\cc(((A_i)_{i\in I},(B_j)_{j\in J})g,((D_i)_{i\in I},(E_j)_{j\in J})g)\bigr)
\end{array}
\\
\begin{tanglec}
\hh \id \step \id \Step \id \Step \id \step \id \step \id \step \id
	\\
\hh \id \hstep \ffbox5{\mu_{\inj_2\:J\hookrightarrow I\sqcup J}} \hstep \id \step \id \step \id
\\
\hh \id \step[3] \id \step[3] \id \step \id \step \id
\\
\hh \hstep \id \step[3] \id \step[1.5] \ffbox4{\mu_{\mathsf{\triangledown I}\:\!\!I\sqcup\mb1\to\mb2}}
	\\
\hh \id \step[3] \id \step[3.5] \id \step[1.5]
	\\
\hh \ffbox8{\mu_{1\sqcup\triangledown\:I\sqcup J\to I\sqcup\mb1}} \step
	\\
\hh \id \step
\end{tanglec}
\quad =\quad
\begin{tanglec}
\hh \id \step \id \Step \id \Step \id \step \id \Step \id \step \id
	\\
\hh \id \hstep \ffbox5{\mu_{\inj_2\:J\hookrightarrow I\sqcup J}} \hstep \id \Step \id \step \id
	\\
\id \step[4] \XX \Step \id \step \id
	\\
\hh \ffbox5{\mu_{1_I}} \step \ffbox3{\mu_{J\to\mb1}} \hstep \id \hstep
	\\
\hh \Step \id \step[5] \id \Step \id
	\\
\hh \Step \ffbox8{\mu_{\triangledown\triangledown\:I\sqcup J\to\mb2}}
	\\
\hh \Step \id
\end{tanglec}
\\
\mcV\bigl((\ca_i(A_i,D_i))_{i\in I},(\cb_j(B_j,E_j))_{j\in J};\cc(((A_i)_{i\in I},(B_j)_{j\in J})g,((D_i)_{i\in I},(E_j)_{j\in J})g)\bigr)
\end{gather*}
Here we have used associativity equation at \figref{dia-assoc-mu-multi} for maps \(I\sqcup J \rTTo^{\id\sqcup\triangledown} I\sqcup\mb1 \rTTo^{\mathsf{\triangledown I}} \mb2\).
The above maps define $tr$ via
\begin{diagram}[h=2.6em,LaTeXeqno,top]
\begin{array}{r}
\bigl( (1_{\ca_i(A_i,D_i)})_{i\in I},(\id_{D_i})_{i\in I},(1_{\cb_j(B_j,E_j)})_{j\in J},g_{(D_i),(B_j),(D_i),(E_j)},
\\
(B_j)_{j\in J}f_{(A_i),(D_i)},1_{\cc(((D_i)_{i\in I},(B_j)_{j\in J})g,((D_i)_{i\in I},(E_j)_{j\in J})g)},\centerdot \bigr)
\end{array}
\hspace*{-4em}
\\
\dMapsTo &&
\\
\begin{array}{r}
\bigl( (1_{\ca_i(A_i,D_i)})_{i\in I},(D_i)_{i\in I}\bar f_{(B_j),(E_j)},(B_j)_{j\in J}f_{(A_i),(D_i)}, \hfill
\\
\hfill 1_{\cc(((D_i)_{i\in I},(B_j)_{j\in J})g,((D_i)_{i\in I},(E_j)_{j\in J})g)},\centerdot \bigr)
\end{array}
&\rMapsTo &
\bigl((B_j)_{j\in J}f_{(A_i),(D_i)},(D_i)_{i\in I}\bar f_{(B_j),(E_j)},\centerdot\bigr) \hspace*{-5em}
	\\
\dMapsTo &&\dMapsTo
	\\
\bigl((1_{\ca_i(A_i,D_i)})_{i\in I}, (D_i)_{i\in I}\bar f_{(B_j),(E_j)},\beta^\dagger \bigr) &\rMapsTo &tr
\label{dia-tr-41}
\end{diagram}

$lb$ comes from the following commutative diagram.
First via the left map, then using associativity equation at \figref{dia-assoc-mu-multi} for maps \(I\sqcup J \rTTo^{\id\sqcup\triangledown} I\sqcup\mb1 \rTTo^{\mathsf{\triangledown I\centerdot X}} \mb2\) we get $lb$ via the right map
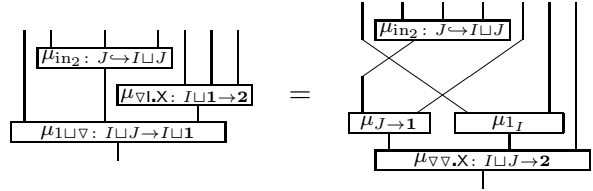
\begin{gather*}
	\begin{array}{c}
\bigl[\prod_{i\in I}\mcv\bigl(\ca_i(A_i,D_i);\ca_i(A_i,D_i)\bigr)\bigr] \times \bigl[\prod_{i\in I}\mcv\bigl(;\ca_i(A_i,A_i)\bigr)\bigr] \times \bigl[ \prod_{j\in J}\mcv\bigl(\cb_j(B_j,E_j);\cb_j(B_j,E_j)\bigr) \bigr]
	\\
	\times \mcV\bigl((\ca_i(A_i,A_i))_{i\in I},(\cb_j(B_j,E_j))_{j\in J};\cc(((A_i)_{i\in I},(B_j)_{j\in J})g,((A_i)_{i\in I},(E_j)_{j\in J})g)\bigr) 
	\\
\times \mcv\bigl(\cc(((A_i)_{i\in I},(B_j)_{j\in J})g,((A_i)_{i\in I},(E_j)_{j\in J})g);\cc(((A_i)_{i\in I},(B_j)_{j\in J})g,((A_i)_{i\in I},(E_j)_{j\in J})g)\bigr)
	\\
\times \mcv\bigl((\ca_i(A_i,D_i))_{i\in I};\cc(((A_i)_{i\in I},(E_j)_{j\in J})g,((D_i)_{i\in I},(E_j)_{j\in J})g)\bigr) 
	\\
\times \mcv\bigl(\cc(((A_i)_{i\in I},(B_j)_{j\in J})g,((A_i)_{i\in I},(E_j)_{j\in J})g),\cc(((A_i)_{i\in I},(E_j)_{j\in J})g,((D_i)_{i\in I},(E_j)_{j\in J})g);
	\\
\hfill \cc(((A_i)_{i\in I},(B_j)_{j\in J})g,((D_i)_{i\in I},(E_j)_{j\in J})g)\bigr)
\end{array}
\\
\begin{tanglec}
	\hh \id \step \id \Step \id \Step \id \step \id \step \id \step \id
	\\
	\hh \id \hstep \ffbox5{\mu_{\inj_2\:J\hookrightarrow I\sqcup J}} \hstep \id \step \id \step \id
	\\
	\hh \id \step[3] \id \step[3] \id \step \id \step \id
	\\
\hh \hstep \id \step[3] \id \hstep \ffbox5{\mu_{\mathsf{\triangledown I\centerdot X}\:I\sqcup\mb1\to\mb2}}
	\\
	\hh \id \step[3] \id \step[3.5] \id \step[1.5]
	\\
	\hh \ffbox8{\mu_{1\sqcup\triangledown\:I\sqcup J\to I\sqcup\mb1}} \step
	\\
	\hh \id \step
\end{tanglec}
\quad =\quad
\begin{tanglec}
\hh \id \step[1.5] \id \step[1.5] \id \step[1.5] \id \step[1.5] \id \step \id \step \id
	\\
\hh \id \hstep \ffbox5{\mu_{\inj_2\:J\hookrightarrow I\sqcup J}} \hstep \id \step \id \step \id
	\\
\XX \step[3] \ne2 \step \id \step \id
	\\
\id \Step \XX \step[3] \id \step \id
\\
\hh \ffbox3{\mu_{J\to\mb1}} \step \ffbox4{\mu_{1_I}} \hstep \id \hstep
	\\
\hh \step \id \step[4.5] \id \step[2.5] \id
	\\
\hh \step \ffbox8{\mu_{\triangledown\triangledown\centerdot\mathsf X\:I\sqcup J\to\mb2}}
	\\
\hh \step \id
\end{tanglec}
\\
\mcV\bigl((\ca_i(A_i,D_i))_{i\in I},(\cb_j(B_j,E_j))_{j\in J};\cc(((A_i)_{i\in I},(B_j)_{j\in J})g,((D_i)_{i\in I},(E_j)_{j\in J})g)\bigr)
\end{gather*}
The above maps define $lb$ via
\begin{diagram}[h=1.3em,LaTeXeqno,top,nobalance]
	\begin{array}{r}
\bigl( (1_{\ca_i(A_i,D_i)})_{i\in I},(\id_{A_i})_{i\in I},(1_{\cb_j(B_j,E_j)})_{j\in J},g_{(A_i),(B_j),(A_i),(E_j)},
\\
1_{\cc(((A_i)_{i\in I},(B_j)_{j\in J})g,((A_i)_{i\in I},(E_j)_{j\in J})g)},(E_j)_{j\in J}f_{(A_i),(D_i)},\centerdot \bigr)
 	\end{array}
	\hspace*{-4em}
&&\phantom{\bigl((A_i)_{i\in I}\bar f_{(B_j),(E_j)},(E_j)_{j\in J}f_{(A_i),(D_i)},\centerdot \bigr)AA}
\\
\\
\dMapsTo &&\hspace*{-6em} \bigl((A_i)_{i\in I}\bar f_{(B_j),(E_j)},(E_j)_{j\in J}f_{(A_i),(D_i)},\centerdot \bigr)
\\
&\ruMapsTo &
\\
	\begin{array}{l}
\bigl( (1_{\ca_i(A_i,D_i)})_{i\in I},(A_i)_{i\in I}\bar f_{(B_j),(E_j)},
\\
1_{\cc(((A_i)_{i\in I},(B_j)_{j\in J})g,((A_i)_{i\in I},(E_j)_{j\in J})g)},(E_j)_{j\in J}f_{(A_i),(D_i)},\centerdot \bigr) 
	\end{array}
\hspace*{-4em} &&\dMapsTo
\\
\\
\dMapsTo &&lb
\\
&\ruMapsTo &
\\
\bigl((1_{\ca_i(A_i,D_i)})_{i\in I}, (A_i)_{i\in I}\bar f_{(B_j),(E_j)},\gamma^\dagger \bigr) &&
\label{dia-lb-41}
\end{diagram}

Combining \eqref{dia-tr-41} with \eqref{eq-(Bj)f} and \eqref{eq-(Ai)fbar} we obtain the left map in the following diagram for $tr$.
Applying the associativity property from \figref{dia-assoc-mu-multi} for maps \(I\sqcup J \rTTo^{\inj_1\sqcup\inj_2} (I\sqcup J)\sqcup(I\sqcup J) \rto{\triangledown\triangledown} \mb2\) we get the second map.
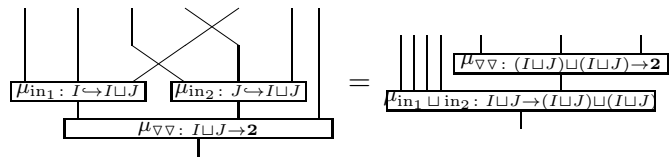
\begin{gather}
\begin{array}{c}
\prod_{i\in I}\mcv\bigl(\ca_i(A_i,D_i);\ca_i(A_i,D_i)\bigr) \times \prod_{j\in J}\mcv(;\cb_j(B_j,B_j))
	\\
\times \prod_{i\in I}\mcv\bigl(;\ca_i(D_i,D_i)\bigr) \times \prod_{j\in J}\mcv(\cb_j(B_j,E_j);\cb_j(B_j,E_j))
	\\
\times \mcV\bigl((\ca_i(A_i,D_i))_{i\in I},(\cb_j(B_j,B_j))_{j\in J};
\cc(((A_i)_{i\in I},(B_j)_{j\in J})g,((D_i)_{i\in I},(B_j)_{j\in J})g)\bigr)
	\\
\times \mcV\bigl((\ca_i(D_i,D_i))_{i\in I},(\cb_j(B_j,E_j))_{j\in J};
\cc(((D_i)_{i\in I},(B_j)_{j\in J})g,((D_i)_{i\in I},(E_j)_{j\in J})g)\bigr)
	\\
\times \mcv\bigl(\cc(((A_i)_{i\in I},(B_j)_{j\in J})g,((D_i)_{i\in I},(B_j)_{j\in J})g),
\cc(((D_i)_{i\in I},(B_j)_{j\in J})g,((D_i)_{i\in I},(E_j)_{j\in J})g);
	\\
\hfill \cc(((A_i)_{i\in I},(B_j)_{j\in J})g,((D_i)_{i\in I},(E_j)_{j\in J})g)\bigr)
\end{array}
\notag
\\
\begin{tanglec}
\id \Step \id \Step \id \Step \XX \Step \id \step \id
	\\
\id \Step \id \Step \XX \Step \id \Step \id \step \id
\\
\hh \ffbox5{\mu_{\inj_1\:I\hookrightarrow I\sqcup J}} \step \ffbox5{\mu_{\inj_2\:J\hookrightarrow I\sqcup J}} \hstep \id \hstep 
	\\
\hh \Step \id \step[6] \id \step[3] \id
	\\
\hh \Step \ffbox{10}{\mu_{\triangledown\triangledown\:I\sqcup J\to\mb2}}
	\\
\hh \Step \id
\end{tanglec}
\;=\;
\begin{tanglec}
\hh \id \hstep \id \hstep \id \hstep \id \step[1.5] \id \step[3] \id \step[3] \id \step
	\\
\hh \id \hstep \id \hstep \id \hstep \id \hstep \ffbox8{\mu_{\triangledown\triangledown\:(I\sqcup J)\sqcup(I\sqcup J)\to\mb2}}
	\\
\hh \id \hstep \id \hstep \id \hstep \id \step[4.5] \id \step[4]
\\
\hh \ffbox{10}{\mu_{\inj_1\sqcup\inj_2\:\!\!I\sqcup J\to(I\sqcup J)\sqcup(I\sqcup J)}} \step
	\\
\hh \id \step
\end{tanglec}
\label{eq-AAD-BBE-1}
\\
\mcV\bigl((\ca_i(A_i,D_i))_{i\in I},(\cb_j(B_j,E_j))_{j\in J};\cc(((A_i)_{i\in I},(B_j)_{j\in J})g,((D_i)_{i\in I},(E_j)_{j\in J})g)\bigr) \notag
\end{gather}
Due to $g$ being a multi-entry $\mcv$\n-functor (see \defref{def-multi-entry-V-functor}) we get the third map.
Applying the associativity property from \figref{dia-assoc-mu-multi} for maps \(I\sqcup J \rTTo^{\inj_1\sqcup\inj_2} (I\sqcup J)\sqcup(I\sqcup J) \rto\chi I\sqcup J\) we get the right map in
\begin{gather}
\begin{array}{c}
\prod_{i\in I}\mcv\bigl(\ca_i(A_i,D_i);\ca_i(A_i,D_i)\bigr) \times \prod_{j\in J}\mcv(;\cb_j(B_j,B_j))
	\\
\times \prod_{i\in I}\mcv\bigl(;\ca_i(D_i,D_i)\bigr) \times \prod_{j\in J}\mcv(\cb_j(B_j,E_j);\cb_j(B_j,E_j))
	\\
\times \prod_{i\in I}\mcv\bigl(\ca_i(A_i,D_i),\ca_i(D_i,D_i);\ca_i(A_i,D_i)\bigr)
\times \prod_{j\in J}\mcv(\cb_j(B_j,B_j),\cb_j(B_j,E_j);\cb_j(B_j,E_j))
	\\
\times \mcV\bigl((\ca_i(A_i,D_i))_{i\in I},(\cb_j(B_j,E_j))_{j\in J};
\cc(((A_i)_{i\in I},(B_j)_{j\in J})g,((D_i)_{i\in I},(E_j)_{j\in J})g)\bigr)
\end{array}
\notag
\\
\begin{tanglec}
	\hh \id \hstep \id \hstep \id \hstep \id \step[1.5] \id \step[3] \id \step[3] \id \step
	\\
	\hh \id \hstep \id \hstep \id \hstep \id \hstep \ffbox8{\mu_{\chi\:(I\sqcup J)\sqcup(I\sqcup J)\to I\sqcup J}}
	\\
	\hh \id \hstep \id \hstep \id \hstep \id \step[4.5] \id \step[4]
	\\
	\hh \ffbox{10}{\mu_{\inj_1\sqcup\inj_2\:\!\!I\sqcup J\to(I\sqcup J)\sqcup(I\sqcup J)}} \step
	\\
\hh \id \step
\end{tanglec}
=
\begin{tanglec}
	\step \id \step \XX \Step \XX \step \id \step \id
	\\
	\step \id \step \id \Step \XX \Step \id \step \id \step \id
	\\
	\hstep \ffbox4{\prod_{i\in I}\mu_{\mathsf{I}\centerdot}} \step \ffbox4{\prod_{j\in J}\mu_{\centerdot\mathsf{I}}} \hstep \id 
	\\
	\hh \step[2.5] \id \step[5] \id \step[2.5] \id
	\\
	\hh \Step \ffbox{9}{\mu_{1\:I\sqcup J\to I\sqcup J}}
	\\
	\hh \Step \id
\end{tanglec}
\label{eq-AAD-BBE-2}
\\
\mcV\bigl((\ca_i(A_i,D_i))_{i\in I},(\cb_j(B_j,E_j))_{j\in J};\cc(((A_i)_{i\in I},(B_j)_{j\in J})g,((D_i)_{i\in I},(E_j)_{j\in J})g)\bigr) \notag
\end{gather}
On elements \eqref{eq-AAD-BBE-1} and \eqref{eq-AAD-BBE-2} give:
\begin{diagram}[h=1.3em,nobalance]
&&\hspace*{-1em}
\begin{array}{c}
	\bigl((1_{\ca_i(A_i,D_i)})_{i\in I},(\id_{B_j})_{j\in J},
	\\
	(\id_{D_i})_{i\in I},(1_{\cb_j(B_j,E_j)})_{j\in J},
	\\
	(g_{(A_i),(B_j),(D_i),(B_j)}, \hfill
	\\
	g_{(D_i),(B_j),(D_i),(E_j)})\kappa\bigr)
\end{array}
\hspace*{-1em} &&\phantom{\bigl((1_{\ca_i(A_i,D_i)})_{i\in I},(\id_{B_j})_{j\in J},(\id_{D_i})_{i\in I},}
\\
&\ruMapsTo(2,3) &&\luMapsTo(2,3) &
\\
\\
\begin{array}{r}
\bigl((1_{\ca_i(A_i,D_i)})_{i\in I},(\id_{B_j})_{j\in J},
\\
(\id_{D_i})_{i\in I},(1_{\cb_j(B_j,E_j)})_{j\in J},
\\
g_{(A_i),(B_j),(D_i),(B_j)},
\\
g_{(D_i),(B_j),(D_i),(E_j)},\centerdot \bigr)
\end{array}
&&&&
\begin{array}{r}
\bigl((1_{\ca_i(A_i,D_i)})_{i\in I},(\id_{B_j})_{j\in J},(\id_{D_i})_{i\in I},
\\
(1_{\cb_j(B_j,E_j)})_{j\in J},(\kappa_{A_i,D_i,D_i})_{i\in I},
\\
(\kappa_{B_j,B_j,E_j})_{j\in J},g_{(A_i),(B_j),(D_i),(E_j)} \bigr)
\end{array}
\\
\\
\dMapsTo &&\dMapsTo &&\dMapsTo
\\
\\
\begin{array}{l}
\bigl((B_j)_{j\in J}f_{(A_i),(D_i)},
\\
(D_i)_{i\in I}\bar f_{(B_j),(E_j)},\centerdot\bigr)
\end{array}
&&&&\hspace*{-0.6em}
\begin{array}{r}
\bigl((1_{\ca_i(A_i,D_i)})_{i\in I},(1_{\cb_j(B_j,E_j)})_{j\in J},
\\
g_{(A_i),(B_j),(D_i),(E_j)}\bigr)
\end{array}
\\
&\rdMapsTo &&\ldMapsTo 
\\
&&tr =g_{(A_i),(B_j),(D_i),(E_j)} 
\end{diagram}
Thus, \(tr=g_{(A_i),(B_j),(D_i),(E_j)}\).

Combining \eqref{dia-lb-41} with \eqref{eq-(Bj)f} and \eqref{eq-(Ai)fbar} we obtain the left map in the following diagram for $lb$.
Applying the associativity property from \figref{dia-assoc-mu-multi} for maps \(I\sqcup J \rTTo^{\XX\centerdot(\inj_2\sqcup\inj_1)} (I\sqcup J)\sqcup(I\sqcup J) \rto{\triangledown\triangledown} \mb2\), where \(\XX\:I\sqcup J\to J\sqcup I\) denotes the block-order-changing bijection, we get the second map.
\begin{gather*}
\begin{array}{c}
	\prod_{i\in I}\mcv\bigl(;\ca_i(A_i,A_i)\bigr) \times \prod_{j\in J}\mcv(\cb_j(B_j,E_j);\cb_j(B_j,E_j))
	\\
	\times \prod_{i\in I}\mcv\bigl(\ca_i(A_i,D_i);\ca_i(A_i,D_i)\bigr) \times \prod_{j\in J}\mcv(;\cb_j(E_j,E_j))
	\\
\times \mcV\bigl((\ca_i(A_i,A_i))_{i\in I},(\cb_j(B_j,E_j))_{j\in J};
\cc(((A_i)_{i\in I},(B_j)_{j\in J})g,((A_i)_{i\in I},(E_j)_{j\in J})g)\bigr)
	\\
\times \mcV\bigl((\ca_i(A_i,D_i))_{i\in I},(\cb_j(E_j,E_j))_{j\in J};
\cc(((A_i)_{i\in I},(E_j)_{j\in J})g,((D_i)_{i\in I},(E_j)_{j\in J})g)\bigr)
	\\
\times \mcv\bigl(\cc(((A_i)_{i\in I},(B_j)_{j\in J})g,((A_i)_{i\in I},(E_j)_{j\in J})g),
\cc(((A_i)_{i\in I},(E_j)_{j\in J})g,((D_i)_{i\in I},(E_j)_{j\in J})g); 
	\\
\hfill \cc(((A_i)_{i\in I},(B_j)_{j\in J})g,((D_i)_{i\in I},(E_j)_{j\in J})g)\bigr)
\end{array}
\\
\begin{tanglec}
\id \Step \id \Step \id \Step \XX \Step \id \step \id
	\\
\id \Step \id \Step \XX \Step \id \Step \id \step \id
	\\
\hh \ffbox5{\mu_{\inj_2\:J\hookrightarrow I\sqcup J}} \step \ffbox5{\mu_{\inj_1\:I\hookrightarrow I\sqcup J}} \hstep \id \hstep
	\\
\hh \Step \id \step[6] \id \step[3] \id
	\\
\hh \Step \ffbox{10}{\mu_{\triangledown\triangledown\centerdot\mathsf X\:I\sqcup J\to\mb2}}
	\\
	\hh \Step \id
\end{tanglec}
\;=\;
\begin{tanglec}
	\hh \id \hstep \id \hstep \id \hstep \id \step[1.5] \id \step[3] \id \step[3] \id \step
	\\
\hh \id \hstep \id \hstep \id \hstep \id \hstep \ffbox8{\mu_{\triangledown\triangledown\:(I\sqcup J)\sqcup(I\sqcup J)\to\mb2}}
	\\
	\hh \id \hstep \id \hstep \id \hstep \id \step[4.5] \id \step[4]
	\\
\hh \ffbox{12}{\mu_{\mathsf{X\mkern-6mu X}\centerdot(\inj_2\sqcup\inj_1)\:I\sqcup J\to(I\sqcup J)\sqcup(I\sqcup J)}}
	\\
\hh \id
\end{tanglec}
\\
\mcV\bigl((\ca_i(A_i,D_i))_{i\in I},(\cb_j(B_j,E_j))_{j\in J};\cc(((A_i)_{i\in I},(B_j)_{j\in J})g,((D_i)_{i\in I},(E_j)_{j\in J})g)\bigr)
\end{gather*}
Due to $g$ being a multi-entry $\mcv$\n-functor (see \defref{def-multi-entry-V-functor}) we get the third map.
Applying the associativity property from \figref{dia-assoc-mu-multi} for maps \(I\sqcup J \rTTo^{\XX\centerdot(\inj_2\sqcup\inj_1)} (I\sqcup J)\sqcup(I\sqcup J) \rto\chi I\sqcup J\) we get the right map in
\begin{gather*}
\begin{array}{c}
\prod_{i\in I}\mcv\bigl(;\ca_i(A_i,A_i)\bigr) \times \prod_{j\in J}\mcv(\cb_j(B_j,E_j);\cb_j(B_j,E_j))
	\\
\times \prod_{i\in I}\mcv\bigl(\ca_i(A_i,D_i);\ca_i(A_i,D_i)\bigr) \times \prod_{j\in J}\mcv(;\cb_j(E_j,E_j))
	\\
\times \prod_{i\in I}\mcv\bigl(\ca_i(A_i,A_i),\ca_i(A_i,D_i);\ca_i(A_i,D_i)\bigr)
\times \prod_{j\in J}\mcv(\cb_j(B_j,E_j),\cb_j(E_j,E_j);\cb_j(B_j,E_j))
	\\
\times \mcV\bigl((\ca_i(A_i,D_i))_{i\in I},(\cb_j(B_j,E_j))_{j\in J};
\cc(((A_i)_{i\in I},(B_j)_{j\in J})g,((D_i)_{i\in I},(E_j)_{j\in J})g)\bigr)
\end{array}
\\
\begin{tanglec}
	\hh \id \hstep \id \hstep \id \hstep \id \step[1.5] \id \step[3] \id \step[3] \id \step
	\\
\hh \id \hstep \id \hstep \id \hstep \id \hstep \ffbox8{\mu_{\chi\:(I\sqcup J)\sqcup(I\sqcup J)\to I\sqcup J}}
	\\
	\hh \id \hstep \id \hstep \id \hstep \id \step[4.5] \id \step[4]
	\\
	\hh \ffbox{12}{\mu_{\mathsf{X\mkern-6mu X}\centerdot(\inj_2\sqcup\inj_1)\:I\sqcup J\to(I\sqcup J)\sqcup(I\sqcup J)}}
	\\
	\hh \id
\end{tanglec}
\;=
\begin{tanglec}
	\step \id \step \XX \Step \XX \step \id \step \id
	\\
	\step \id \step \id \Step \XX \Step \id \step \id \step \id
	\\
\hstep \ffbox4{\prod_{i\in I}\mu_{\centerdot\mathsf{I}}} \step \ffbox4{\prod_{j\in J}\mu_{\mathsf{I}\centerdot}} \hstep \id 
	\\
	\hh \step[2.5] \id \step[5] \id \step[2.5] \id
	\\
	\hh \Step \ffbox{9}{\mu_{1\:I\sqcup J\to I\sqcup J}}
	\\
	\hh \Step \id
\end{tanglec}
\\
\mcV\bigl((\ca_i(A_i,D_i))_{i\in I},(\cb_j(B_j,E_j))_{j\in J};\cc(((A_i)_{i\in I},(B_j)_{j\in J})g,((D_i)_{i\in I},(E_j)_{j\in J})g)\bigr)
\end{gather*}
On elements the above four maps give:
\begin{diagram}[h=1.3em,nobalance]
	&&\hspace*{-1em}
	\begin{array}{c}
	\bigl((\id_{A_i})_{i\in I},(1_{\cb_j(B_j,E_j)})_{j\in J},
	\\
	(1_{\ca_i(A_i,D_i)})_{i\in I},(\id_{E_j})_{j\in J},
	\\
	(g_{(A_i),(B_j),(A_i),(E_j)}, \hfill
	\\
	g_{(A_i),(E_j),(D_i),(E_j)})\kappa \bigr)
\end{array}
\hspace*{-1em} &&\phantom{(1_{\ca_i(A_i,D_i)})_{i\in I},(\id_{E_j})_{j\in J},\kappa_{A_i,A_i,D_i},}
\\
&\ruMapsTo(2,3) &&\luMapsTo(2,3) &
\\
\\
	\begin{array}{r}
\bigl((\id_{A_i})_{i\in I},(1_{\cb_j(B_j,E_j)})_{j\in J},
\\
(1_{\ca_i(A_i,D_i)})_{i\in I},(\id_{E_j})_{j\in J},
\\
g_{(A_i),(B_j),(A_i),(E_j)},
\\
g_{(A_i),(E_j),(D_i),(E_j)},\centerdot \bigr)
	\end{array}
&&&&
	\begin{array}{r}
\bigl((\id_{A_i})_{i\in I},(1_{\cb_j(B_j,E_j)})_{j\in J},
\\
(1_{\ca_i(A_i,D_i)})_{i\in I},(\id_{E_j})_{j\in J},\kappa_{A_i,A_i,D_i},
\\
\kappa_{B_j,E_j,E_j},g_{(A_i),(B_j),(D_i),(E_j)}\bigr)
	\end{array}
\\
\\
\dMapsTo &&\dMapsTo &&\dMapsTo
\\
\\
	\begin{array}{l}
\bigl((A_i)_{i\in I}F_{(B_j),(E_j)},
\\
(E_j)_{j\in J}f_{(A_i),(D_i)},\centerdot \bigr)
	\end{array}
&&&&\hspace*{-0.6em}
	\begin{array}{r}
\bigl((1_{\ca_i(A_i,D_i)})_{i\in I},(1_{\cb_j(B_j,E_j)})_{j\in J},
\\
g_{(A_i),(B_j),(D_i),(E_j)}\bigr)
	\end{array}
\\
&\rdMapsTo &&\ldMapsTo 
\\
&&lb =g_{(A_i),(B_j),(D_i),(E_j)}
\end{diagram}
Thus, \(lb=g_{(A_i),(B_j),(D_i),(E_j)}=tr\). 
Therefore, $\varPsi$ is a map
\[ \varPsi\:\VCat\bigl((\ca_i)_{i\in I},(\cb_j)_{j\in J};\cc\bigr)\to \VCat\bigl((\cb_j)_{j\in J};\und\VCat((\ca_i)_{i\in I};\cc)\bigr).
\]

\propref{con-evVCat-VCat} implies that composition~\eqref{eq-phi(XYZ)-iso} takes the form
\begin{multline}
\varPhi =\bigl[ \VCat\bigl((\cb_j)_{j\in J};\und\VCat((\ca_i)_{i\in I};\cc)\bigr) \rTTo^{\dot\Id_{\ca_1}\times\dots\times\dot\Id_{\ca_I}\times\id\times\dot\ev_{(\ca_i)_{i\in I};\cc}}
	\\
\hskip\multlinegap \bigl[\prod_{i\in I}\VCat(\ca_i;\ca_i)\bigr]\times\VCat\bigl((\cb_j)_{j\in J};\und\VCat((\ca_i)_{i\in I};\cc)\bigr) \hfill
\\
\hfill \times\VCat\bigl((\ca_i)_{i\in I},\und\VCat((\ca_i)_{i\in I};\cc);\cc\bigr)
\rTTo^{\mu^\VCat_{\id\sqcup\con\:I\sqcup J\to I\sqcup\mb1}} \VCat\bigl((\ca_i)_{i\in I},(\cb_j)_{j\in J};\cc\bigr) \bigr], \hskip\multlinegap
\\
\hskip\multlinegap f \mapsto \bigl((\Id_{\ca_i})_{i\in I},f,\ev^\VCat\bigr) \hfill
\\
\mapsto f\varPhi =\bigl[ (\ca_i)_{i\in I},(\cb_j)_{j\in J} \rTTo^{(\Id_{\ca_i})_{i\in I},f} (\ca_i)_{i\in I},\und\VCat((\ca_i)_{i\in I};\cc) \rTTo^{\ev^\VCat} \cc \bigr].
\label{eq-fvarPhi-VCat}
\end{multline}
On objects \(f\varPhi\:\bigl((A_i)_{i\in I},(B_j)_{j\in J}\bigr)\mapsto((A_i)_{i\in I},(B_j)_{j\in J}f)\mapsto(A_i)_{i\in I}(B_j)_{j\in J}f\).
On morphisms
\begin{multline}
\prod_{i\in I}\mcv\bigl(\ca_i(A_i,D_i);\ca_i(A_i,D_i)\bigr)
\times \mcv\bigl((\cb_j(B_j,E_j))_{j\in J};\und\VCat((\ca_i)_{i\in I};\cc)((B_j)_{j\in J}f,(E_j)_{j\in J}f)\bigr) \times
\\
\mcV\bigl((\ca_i(A_i,D_i))_{i\in I},\und\VCat((\ca_i)_{i\in I};\cc)((B_j)_{j\in J}f,(E_j)_{j\in J}f);
\cc((A_i)_{i\in I}(B_j)_{j\in J}f,(D_i)_{i\in I}(E_j)_{j\in J}f)\bigr)
\\
\rTTo^{\mu^\mcv_{\id\sqcup\con\:I\sqcup J\to I\sqcup\mb1}} 
\mcV\bigl((\ca_i(A_i,D_i))_{i\in I},(\cb_j(B_j,E_j))_{j\in J};\cc((A_i)_{i\in I}(B_j)_{j\in J}f,(D_i)_{i\in I}(E_j)_{j\in J}f)\bigr),
\\
\bigl((1_{\ca_i(A_i,D_i)})_{i\in I},f_{(B_j),(E_j)},\ev^\VCat\bigr) \mapsto (f\varPhi)_{(A_i),(B_j),(D_i),(E_j)}.
\label{eq-fF-AA-BACBfEf}
\end{multline}
In place of \(\ev^\VCat\) we can substitute formula~\eqref{eq-ev-VQu-E} or \eqref{eq-ev-VQu-A}.
First \eqref{eq-ev-VQu-E}, which gives the left map in the following diagram.
Applying the associativity property from \figref{dia-assoc-mu-multi} for maps \(I\sqcup J \rTTo^{\id\sqcup\con} I\sqcup\mb1 \rto{\mathsf{\triangledown I}} \mb2\) we get the right map in
\begin{gather*}
\begin{array}{l}
\prod_{i\in I}\mcv\bigl(\ca_i(A_i,D_i);\ca_i(A_i,D_i)\bigr)
\times \mcv\bigl((\ca_i(A_i,D_i))_{i\in I};\cc((A_i)_{i\in I}(B_j)_{j\in J}f,(D_i)_{i\in I}(B_j)_{j\in J}f)\bigr)
	\\
	\times \mcv\bigl((\cb_j(B_j,E_j))_{j\in J};\und\VCat((\ca_i)_{i\in I};\cc)((B_j)_{j\in J}f,(E_j)_{j\in J}f)\bigr)
	\\
\times \mcv\bigl(\und\VCat\bigl((\ca_i)_{i\in I};\cc\bigr)((B_j)_{j\in J}f,(E_j)_{j\in J}f);
\cc((D_i)_{i\in I}(B_j)_{j\in J}f,(D_i)_{i\in I}(E_j)_{j\in J}f)\bigr)
	\\
	\times \mcv\bigl(\cc((A_i)_{i\in I}(B_j)_{j\in J}f,(D_i)_{i\in I}(B_j)_{j\in J}f),
\cc((D_i)_{i\in I}(B_j)_{j\in J}f,(D_i)_{i\in I}(E_j)_{j\in J}f); \hfill
	\\
	\hfill \cc((A_i)_{i\in I}(B_j)_{j\in J}f, (D_i)_{i\in I}(E_j)_{j\in J}f)\bigr) 
\end{array}
\\
\begin{tanglec}
\id \step \XX \step \id \step \id
	\\
\hh \hstep \id \step \id \step[1.5] \ffbox3{\mu_{\mathsf{\triangledown I}}}
	\\
\hh \id \step \id \step[3] \id \step
	\\
\hh \ffbox6{\mu_{1\sqcup\con\:I\sqcup J\to I\sqcup\mb1}}
	\\
\hh \id
\end{tanglec}
\;=\;
\begin{tanglec}
\hh \id \Step \id \step[1.5] \id \Step \id \step \id
	\\
\hh \ffbox3{\mu_{1_I}} \hstep \ffbox3{\mu_{\triangledown\:\!\!J\to\mb1\!}} \hstep \id \hstep
	\\
\hh \step \id \step[3.5] \id \Step \id
	\\
\hh \hstep \ffbox7{\mu_{\con\sqcup\con\:I\sqcup J\to\mb2}}
	\\
\hh \hstep \id
\end{tanglec}
\\
\mcV\bigl((\ca_i(A_i,D_i))_{i\in I},(\cb_j(B_j,E_j))_{j\in J};\cc((A_i)_{i\in I}(B_j)_{j\in J}f,(D_i)_{i\in I}(E_j)_{j\in J}f)\bigr)
\end{gather*}

Due to \eqref{eq-axiom-unit-multi2} applied to unit elements it simplifies to the map
\begin{gather*}
\begin{array}{l}
	\mcv\bigl((\ca_i(A_i,D_i))_{i\in I};\cc((A_i)_{i\in I}(B_j)_{j\in J}f,(D_i)_{i\in I}(B_j)_{j\in J}f)\bigr)
	\\
\times \mcv\bigl((\cb_j(B_j,E_j))_{j\in J};
\und\VCat((\ca_i)_{i\in I};\cc)((B_j)_{j\in J}f,(E_j)_{j\in J}f)\bigr)
	\\
\times \mcv\bigl(\und\VCat\bigl((\ca_i)_{i\in I};\cc\bigr)((B_j)_{j\in J}f,(E_j)_{j\in J}f);
\cc((D_i)_{i\in I}(B_j)_{j\in J}f,(D_i)_{i\in I}(E_j)_{j\in J}f)\bigr)
	\\
\times \mcv\bigl(\cc((A_i)_{i\in I}(B_j)_{j\in J}f,(D_i)_{i\in I}(B_j)_{j\in J}f),
\cc((D_i)_{i\in I}(B_j)_{j\in J}f,(D_i)_{i\in I}(E_j)_{j\in J}f);
	\\
	\hfill \cc((A_i)_{i\in I}(B_j)_{j\in J}f, (D_i)_{i\in I}(E_j)_{j\in J}f)\bigr)
\end{array}
\\
\begin{tanglec}
\hh \id \step \id \Step \id \step \id
	\\
\hh \id \hstep \ffbox3{\mu_{\triangledown\:\!J\to\mb1\!}} \hstep \id
	\\
\hh \id \Step \id \Step \id
	\\
\hh \ffbox5{\mu_{\con\sqcup\con\:\!I\sqcup J\to\mb2}}
	\\
\hh \id
\end{tanglec}
	\\
\mcV\bigl((\ca_i(A_i,D_i))_{i\in I},(\cb_j(B_j,E_j))_{j\in J};\cc((A_i)_{i\in I}(B_j)_{j\in J}f,(D_i)_{i\in I}(E_j)_{j\in J}f)\bigr)
\end{gather*}
On elements we obtain:
\begin{diagram}[nobalance,LaTeXeqno,bottom]
\bigl((1_{\ca_i(A_i,D_i)})_{i\in I},[(B_j)_{j\in J}f]_{(A_i),(D_i)},f_{(B_j),(E_j)},p_{(D_i)_{i\in I}},\centerdot\bigr) &\  &\bigl([(B_j)_{j\in J}f]_{(A_i),(D_i)},f_{(B_j),(E_j)},p_{(D_i)_{i\in I}},\centerdot\bigr)
\\
\dMapsTo &\rdMapsTo &\dMapsTo
\\
\bigl((1_{\ca_i(A_i,D_i)})_{i\in I},f_{(B_j),(E_j)},\ev^\VCat\bigr) &&\bigl([(B_j)_{j\in J}f]_{(A_i),(D_i)},f_{(B_j),(E_j)}\centerdot p_{(D_i)_{i\in I}},\centerdot\bigr)
\\
&\rdMapsTo &\dMapsTo
\\
&&(f\varPhi)_{(A_i),(B_j),(D_i),(E_j)}
\label{eq-(fvarPhi)-V}
\end{diagram}

Now to get \(f\varPhi\) we use \eqref{eq-ev-VQu-A}, obtaining the left map of the following diagram.
Applying the associativity property from \figref{dia-assoc-mu-multi} for maps \(I\sqcup J \rTTo^{\id\sqcup\con} I\sqcup\mb1 \rTTo^{\mathsf{\triangledown I\centerdot X}} \mb2\) we get the right expression
\begin{gather*}
\begin{array}{l}
	\mcv\bigl((\cb_j(B_j,E_j))_{j\in J};\und\VCat((\ca_i)_{i\in I};\cc)((B_j)_{j\in J}f,(E_j)_{j\in J}f)\bigr)
	\\
\times \mcv\bigl(\und\VCat((\ca_i)_{i\in I};\cc)((B_j)_{j\in J}f,(E_j)_{j\in J}f);
\cc((A_i)_{i\in I}(B_j)_{j\in J}f,(A_i)_{i\in I}(E_j)_{j\in J}f)\bigr)
	\\
\times \prod_{i\in I}\mcv\bigl(\ca_i(A_i,D_i);\ca_i(A_i,D_i)\bigr) 
	\\
	\times \mcv\bigl((\ca_i(A_i,D_i))_{i\in I}; \cc((A_i)_{i\in I}(E_j)_{j\in J}f,(D_i)_{i\in I}(E_j)_{j\in J}f)\bigr)
	\\
\times \mcv\bigl(\cc((A_i)_{i\in I}(B_j)_{j\in J}f,(A_i)_{i\in I}(E_j)_{j\in J}f),
\cc((A_i)_{i\in I}(E_j)_{j\in J}f,(D_i)_{i\in I}(E_j)_{j\in J}f);
	\\
\hfill \cc((A_i)_{i\in I}(B_j)_{j\in J}f, (D_i)_{i\in I}(E_j)_{j\in J}f)\bigr)
\end{array}
\\
\begin{tanglec}
\id \Step \XX \step \id \step \id
	\\
\XX \Step \id \step \id \step \id
\\
\hh \hstep \id \Step \id \step[1.5] \ffbox3{\mu_{\mathsf{\triangledown I\centerdot X}}}
	\\
\hh \id \Step \id \step[3] \id \step
	\\
\hh \ffbox6{\mu_{1\sqcup\con\:I\sqcup J\to I\sqcup\mb1}} \step
	\\
\hh \id \step
\end{tanglec}
\;=\;
\begin{tanglec}
	\hh \id \Step \id \step[1.5] \id \Step \id \step \id
	\\
\hh \ffbox3{\mu_{\triangledown\:\!J\to\mb1\!}} \hstep \ffbox3{\mu_{1_I}} \hstep \id \hstep
	\\
	\hh \step \id \step[3.5] \id \Step \id
	\\
\hh \hstep \ffbox7{\mu_{\con\con\centerdot\mathsf X\:I\sqcup J\to\mb2}}
	\\
	\hh \hstep \id
\end{tanglec}
\\
\mcV\bigl((\ca_i(A_i,D_i))_{i\in I},(\cb_j(B_j,E_j))_{j\in J};\cc((A_i)_{i\in I}(B_j)_{j\in J}f,(D_i)_{i\in I}(E_j)_{j\in J}f)\bigr)
\end{gather*}
Due to \eqref{eq-axiom-unit-multi2} it simplifies to the map
\begin{gather*}
\begin{array}{l}
	\mcv\bigl((\cb_j(B_j,E_j))_{j\in J};\und\VCat((\ca_i)_{i\in I};\cc)((B_j)_{j\in J}f,(E_j)_{j\in J}f)\bigr)
	\\
\times \mcv\bigl(\und\VCat((\ca_i)_{i\in I};\cc)((B_j)_{j\in J}f,(E_j)_{j\in J}f);
\cc((A_i)_{i\in I}(B_j)_{j\in J}f,(A_i)_{i\in I}(E_j)_{j\in J}f)\bigr)
	\\
\times \mcv\bigl((\ca_i(A_i,D_i))_{i\in I};\cc((A_i)_{i\in I}(E_j)_{j\in J}f,(D_i)_{i\in I}(E_j)_{j\in J}f)\bigr)
	\\
\times \mcv\bigl(\cc((A_i)_{i\in I}(B_j)_{j\in J}f,(A_i)_{i\in I}(E_j)_{j\in J}f),
\cc((A_i)_{i\in I}(E_j)_{j\in J}f,(D_i)_{i\in I}(E_j)_{j\in J}f);
	\\
	\hfill \cc((A_i)_{i\in I}(B_j)_{j\in J}f, (D_i)_{i\in I}(E_j)_{j\in J}f)\bigr) 
\end{array}
\\
\begin{tanglec}
\hh \id \Step \id \step \id \step \id
	\\
\hh \ffbox3{\mu_{\triangledown\:\!J\to\mb1\!}} \hstep \id \step \id \hstep
	\\
\hh \step \id \Step \id \step \id
	\\
\hh \ffbox5{\mu_{\con\con\centerdot\mathsf X\:\!\!I\sqcup J\to\mb2}}
	\\
\hh \id
\end{tanglec}
\\
\mcV\bigl((\ca_i(A_i,D_i))_{i\in I},(\cb_j(B_j,E_j))_{j\in J};\cc((A_i)_{i\in I}(B_j)_{j\in J}f,(D_i)_{i\in I}(E_j)_{j\in J}f)\bigr)
\end{gather*}
On elements we obtain:
\begin{diagram}[nobalance,bottom]
\bigl(f_{(B_j),(E_j)},p_{(A_i)_{i\in I}},(1_{\ca_i(A_i,D_i)})_{i\in I},[(E_j)_{j\in J}f]_{(A_i),(D_i)},\centerdot\bigr) &\  &\bigl(f_{(B_j),(E_j)},p_{(A_i)_{i\in I}},[(E_j)_{j\in J}f]_{(A_i),(D_i)},\centerdot\bigr)
	\\
	\dMapsTo &\rdMapsTo &\dMapsTo
	\\
\bigl((1_{\ca_i(A_i,D_i)})_{i\in I},f_{(B_j),(E_j)},\ev^\VCat\bigr) &&\bigl(f_{(B_j),(E_j)}\centerdot p_{(A_i)_{i\in I}},[(E_j)_{j\in J}f]_{(A_i),(D_i)},\centerdot\bigr)
	\\
	&\rdMapsTo &\dMapsTo
	\\
	&&(f\varPhi)_{(A_i),(B_j),(D_i),(E_j)}
\end{diagram}

Start from \(f\:(\cb_j)_{j\in J}\to\und\VCat((\ca_i)_{i\in I};\cc)\).
Produce \(g=f\varPhi\) and \(f'=g\varPsi\).
Then\\
\(\Ob f'\:\prod_{j\in J}\Ob\cb_j\to\VCat((\ca_i)_{i\in I};\cc)\) is given by
\begin{multline*}
\prod_{i\in I}\VCat(\ca_i;\ca_i) \times \prod_{j\in J}\VCat(;\cb_j) \times \VCat((\ca_i)_{i\in I},(\cb_j)_{j\in J};\cc)
\rTTo^{\mu_{\inj_1\:I\hookrightarrow I\sqcup J}} \VCat((\ca_i)_{i\in I};\cc),
\\
(B_j)_{j\in J} \mapsto \bigl((\Id_{\ca_i})_{i\in I},(\ddot B_j)_{j\in J},f\varPhi\bigr) \mapsto (B_j)_{j\in J}f'.
\end{multline*}
\begin{align*}
(B_j)_{j\in J}f' &=\bigl[ (\ca_i)_{i\in I} \rTTo^{(\Id_{\ca_i})_{i\in I},(\ddot B_j)_{j\in J}} (\ca_i)_{i\in I},(\cb_j)_{j\in J}
\\
&\hspace*{6em} \rTTo^{(\Id_{\ca_i})_{i\in I},f} (\ca_i)_{i\in I},\und\VCat((\ca_i)_{i\in I};\cc) \rTTo^{\ev^\VCat} \cc \bigr]
\\
&=\bigl[ (\ca_i)_{i\in I} \rTTo^{(\Id_{\ca_i})_{i\in I},(\ddot B_j)_{j\in J}\centerdot f} (\ca_i)_{i\in I},\und\VCat((\ca_i)_{i\in I};\cc) \rTTo^{\ev^\VCat} \cc \bigr]
\\
&=\bigl[ (\ca_i)_{i\in I} \rTTo^{(\Id_{\ca_i})_{i\in I},\ddot h} (\ca_i)_{i\in I},\und\VCat((\ca_i)_{i\in I};\cc) \rTTo^{\ev^\VCat} \cc \bigr] =h,
\end{align*}
where \(h=(B_j)_{j\in J}f\in\VCat((\ca_i)_{i\in I};\cc)\).
Notice that \((\ddot B_j)_{j\in J}\centerdot f=\ddot h\) due to \exaref{exa-varnothing}.
The last equation follows from the

\begin{lemma}
For an arbitrary multi-entry $\mcv$\n-functor $h\:(\ca_i)_{i\in I}\to\cc$ \eqref{eq-mu-in1I-I1} holds for \(\mcC=\VCat\):
\[ \bigl[ (\ca_i)_{i\in I} \rTTo^{(\Id_{\ca_i})_{i\in I},\ddot h} (\ca_i)_{i\in I},\und\VCat((\ca_i)_{i\in I};\cc) \rTTo^{\ev^\VCat} \cc \bigr] =h.
\]
\end{lemma}

\begin{proof}
The left hand side sends a tuple of objects \((A_i)_{i\in I}\) to \(\bigl((A_i)_{i\in I},h\bigr)\mapsto(A_i)_{i\in I}h\), thus, it acts on objects like $\Ob h$.
On morphisms the left hand side is a particular case of map \(\varPhi\) for \(J=\varnothing\) (see \eqref{eq-fvarPhi-VCat}):
\begin{multline*}
\varPhi_0 =\bigl[ \VCat\bigl((\ca_i)_{i\in I};\cc)\bigr) \rTTo^{(\dot\Id_{\ca_i})_{i\in I}\times\id\times\dot\ev^\VCat_{(\ca_i)_{i\in I};\cc}}
\\
\prod_{i\in I}\VCat(\ca_i;\ca_i)\times\VCat\bigl((\ca_i)_{i\in I};\cc\bigr) \times\VCat\bigl((\ca_i)_{i\in I},\und\VCat((\ca_i)_{i\in I};\cc);\cc\bigr)
\\
\hfill \rTTo^{\mu^\VCat_{\inj_1\:I\hookrightarrow I\sqcup\mb1}} \VCat\bigl((\ca_i)_{i\in I};\cc\bigr) \bigr], \hskip\multlinegap
\\
h \mapsto \bigl((\Id_{\ca_i})_{i\in I},h,\ev^\VCat\bigr)
\mapsto h\varPhi_0 =\bigl[ (\ca_i)_{i\in I} \rTTo^{(\Id_{\ca_i})_{i\in I},\ddot h} (\ca_i)_{i\in I},\und\VCat((\ca_i)_{i\in I};\cc) \rTTo^{\ev^\VCat} \cc \bigr].
\end{multline*}
We have to prove that \(\varPhi_0=\id\).
Equation \eqref{eq-fvarPhi-VCat} on morphisms transforms to \eqref{eq-(fvarPhi)-V}.
As a consequence we can write \(h\varPhi_0\) on morphisms as
\begin{multline*}
\prod_{i\in I}\mcv\bigl(\ca_i(A_i,D_i);\ca_i(A_i,D_i)\bigr) \times \mcv\bigl(;\und\VCat((\ca_i)_{i\in I};\cc)(h,h)\bigr) \times
	\\
	\mcV\bigl((\ca_i(A_i,D_i))_{i\in I},\und\VCat((\ca_i)_{i\in I};\cc)(h,h);\cc((A_i)_{i\in I}h,(D_i)_{i\in I}h)\bigr)
\\
	\rTTo^{\mu^\mcv_{\inj_1\:I\hookrightarrow I\sqcup\mb1}} 
	\mcV\bigl((\ca_i(A_i,D_i))_{i\in I};\cc((A_i)_{i\in I}h,(D_i)_{i\in I}h)\bigr),
	\\
	\bigl((1_{\ca_i(A_i,D_i)})_{i\in I},\id_h,\ev^\VCat\bigr) \mapsto (h\varPhi_0)_{(A_i),(D_i)},
\end{multline*}
which transforms further to
\begin{multline*}
\mcv\bigl((\ca_i(A_i,D_i))_{i\in I};\cc((A_i)_{i\in I}h,(D_i)_{i\in I}h)\bigr) \times \mcv\bigl(;\und\VCat((\ca_i)_{i\in I};\cc)(h,h)\bigr) 
\\
\times \mcv\bigl(\und\VCat\bigl((\ca_i)_{i\in I};\cc\bigr)(h,h);\cc((D_i)_{i\in I}h,(D_i)_{i\in I}h)\bigr)
\\
\times \mcv\bigl(\cc((A_i)_{i\in I}h,(D_i)_{i\in I}h),\cc((D_i)_{i\in I}h,(D_i)_{i\in I}h);\cc((A_i)_{i\in I}h, (D_i)_{i\in I}h)\bigr) \rTTo^{1\times\mu^\mcv_{\varnothing\to\mb1}\times1}
\\
\mcv\bigl((\ca_i(A_i,D_i))_{i\in I};\cc((A_i)_{i\in I}h,(D_i)_{i\in I}h)\bigr) \times \mcv\bigl(;\cc((D_i)_{i\in I}h,(D_i)_{i\in I}h)\bigr)
\\
\times \mcv\bigl(\cc((A_i)_{i\in I}h,(D_i)_{i\in I}h),\cc((D_i)_{i\in I}h,(D_i)_{i\in I}h);\cc((A_i)_{i\in I}h,(D_i)_{i\in I}h)\bigr)
\\
\rTTo^{\mu^\mcv_{\con\centerdot\inj_1\:I\to\mb2}} 
\mcV\bigl((\ca_i(A_i,D_i))_{i\in I};\cc((A_i)_{i\in I}h,(D_i)_{i\in I}h)\bigr),
\\
\bigl(h_{(A_i),(D_i)},\id_h,p_{(D_i)_{i\in I}},\centerdot\bigr) \mapsto \bigl(h_{(A_i),(D_i)},\id_{(D_i)_{i\in I}h},\centerdot\bigr) 
\mapsto (h\varPhi_0)_{(A_i),(D_i)}.
\end{multline*}
Due to \eqref{eq-1idYk1} \((h\varPhi_0)_{(A_i),(D_i)}=h_{(A_i),(D_i)}\).
Thus, \(h\varPhi_0=h\).
\end{proof}

We conclude that \(\Ob f'=\Ob f\).

On morphisms \(f'_{(B_j),(E_j)}p_{(A_i)}\) is determined by \eqref{eq-(Ai)fbar}:
\begin{multline*}
\prod_{i\in I}\mcv\bigl(;\ca_i(A_i,A_i)\bigr) \times \prod_{j\in J}\mcv(\cb_j(B_j,E_j);\cb_j(B_j,E_j))
	\\
\times \mcV\bigl((\ca_i(A_i,A_i))_{i\in I},(\cb_j(B_j,E_j))_{j\in J};\cc(((A_i)_{i\in I},(B_j)_{j\in J})g,((A_i)_{i\in I},(E_j)_{j\in J})g)\bigr)
	\\
\hfill \rTTo^{\mu_{\inj_2\:J\hookrightarrow I\sqcup J}} \mcV\bigl((\cb_j(B_j,E_j))_{j\in J};\cc(((A_i)_{i\in I},(B_j)_{j\in J})g,((A_i)_{i\in I},(E_j)_{j\in J})g)\bigr) \hskip\multlinegap
	\\
\bigl((\id_{A_i})_{i\in I},(1_{\cb_j(B_j,E_j)})_{j\in J},g_{(A_i),(B_j),(A_i),(E_j)}\bigr) \mapsto f'_{(B_j),(E_j)}p_{(A_i)},
\end{multline*}
where \(g=f\varPhi\) and \(f'=g\varPsi\).
Due to \eqref{eq-fF-AA-BACBfEf} we can write \(f'_{(B_j),(E_j)}p_{(A_i)}\) as the left map in the following diagram.
Applying the associativity property from \figref{dia-assoc-mu-multi} for maps \(J \rTTo^{\inj_2} I\sqcup J \rTTo^{\id\sqcup\triangledown} I\sqcup\mb1\) we get the second map.
Expanding \(\ev^\VCat\) using \eqref{eq-ev-VQu-E} we get the third map.
Applying the associativity property from \figref{dia-assoc-mu-multi} for maps \(J \rTTo^{\triangledown\centerdot\inj_2} I\sqcup\mb1 \rTTo^{\triangledown\mathsf{I}} \mb2\) we obtain the right map in
\begin{gather*}
\hspace*{-0.4em}
\begin{array}{r}
	\prod_{i\in I}\bigl[\mcv\bigl(;\ca_i(A_i,A_i)\bigr) \times \mcv\bigl(\ca_i(A_i,A_i);\ca_i(A_i,A_i)\bigr)\bigr]
	\\
	\times \prod_{j\in J}\mcv(\cb_j(B_j,E_j);\cb_j(B_j,E_j)) \times
	\\
	\mcv\bigl((\cb_j(B_j,E_j))_{j\in J}; \hfill
	\\
	\und\VCat((\ca_i)_{i\in I};\cc)((B_j)_{j\in J}f,(E_j)_{j\in J}f)\bigr) 
	\\
	\times \mcV\bigl((\ca_i(A_i,A_i))_{i\in I}, \hfill
	\\
	\hfill \und\VCat((\ca_i)_{i\in I};\cc)((B_j)_{j\in J}f,(E_j)_{j\in J}f); \hfill
	\\
	\hfill \cc((A_i)_{i\in I}(B_j)_{j\in J}f,(A_i)_{i\in I}(E_j)_{j\in J}f)\bigr) 
\end{array}
\hspace*{-0.5em}
\begin{array}{l}
	\prod_{i\in I}\mcv\bigl(;\ca_i(A_i,A_i)\bigr) \times \mcv\bigl((\cb_j(B_j,E_j))_{j\in J};
	\\
	\hfill \und\VCat((\ca_i)_{i\in I};\cc)((B_j)_{j\in J}f,(E_j)_{j\in J}f)\bigr)
	\\
	\times \mcv\bigl((\ca_i(A_i,A_i))_{i\in I};
	\\
	\hfill \cc((A_i)_{i\in I}(B_j)_{j\in J}f,(A_i)_{i\in I}(B_j)_{j\in J}f)\bigr)
	\\
	\times \mcv\bigl(\und\VCat\bigl((\ca_i)_{i\in I};\cc\bigr)((B_j)_{j\in J}f,(E_j)_{j\in J}f);
	\\
	\hfill \cc((A_i)_{i\in I}(B_j)_{j\in J}f,(A_i)_{i\in I}(E_j)_{j\in J}f)\bigr)
	\\
	\times \mcv\bigl(\cc((A_i)_{i\in I}(B_j)_{j\in J}f,(A_i)_{i\in I}(B_j)_{j\in J}f),
	\\
	\hfill \cc((A_i)_{i\in I}(B_j)_{j\in J}f,(A_i)_{i\in I}(E_j)_{j\in J}f); \hfill
	\\
	\hfill \cc((A_i)_{i\in I}(B_j)_{j\in J}f, (A_i)_{i\in I}(E_j)_{j\in J}f)\bigr)
\end{array}
\\
\begin{tanglec}
\id \hstep \XX \Step \id \Step \id
	\\
\hh \hstep \id \hstep \id \hstep \ffbox6{\mu_{1\sqcup\con\:I\sqcup J\to I\sqcup\mb1}}
	\\
\hh \id \hstep \id \step[3.5] \id \step[2.5]
	\\
\hh \ffbox6{\mu_{\inj_2\:J\hookrightarrow I\sqcup J}} \step[1.5]
	\\
\hh \id \step[1.5]
\end{tanglec}
\;=
\begin{tanglec}
\hh \step \id \Step \id \Step \id \step \id \step \id
	\\
\ffbox4{\prod_I\mu_{\varnothing\to\mb1}} \hstep \ffbox2{\mu_{1_J}} \hstep \id
	\\
\hh \Step \id \step[3.5] \id \step[1.5] \id
	\\
\hh \Step \ffbox6{\mu_{\triangledown\centerdot\inj_2\:J\to I\sqcup\mb1}}
	\\
\hh \Step \id
\end{tanglec}
\qquad
\begin{tanglec}
	\hh \id \step \id \step \id \step \id \step \id
	\\
\hh \hstep \id \step \id \hstep \ffbox3{\mu_{\mathsf{\triangledown I}}}
	\\
	\hh \id \step \id \Step \id \step
	\\
\hh \ffbox6{\mu_{\triangledown\centerdot\inj_2\:J\to I\sqcup\mb1}}
	\\
\hh \id
\end{tanglec}
\;=\;
\begin{tanglec}
\id \Step \XX \Step \id \step \id
	\\
\hh \ffbox3{\mu_{\varnothing\to I}} \step \ffbox3{\mu_{\triangledown\:\!\!J\to\mb1\!}} \hstep \id \hstep
	\\
\hh \step \id \step[4] \id \Step \id
	\\
\hh \step \ffbox7{\mu_{\triangledown\centerdot\inj_2\:J\to\mb2}}
	\\
\hh \step \id
\end{tanglec}
\\
\mcV\bigl((\cb_j(B_j,E_j))_{j\in J};\cc((A_i)_{i\in I}(B_j)_{j\in J}f,(A_i)_{i\in I}(E_j)_{j\in J}f)\bigr)
\end{gather*}
On elements:
\begin{diagram}[h=2.5em]
	\begin{array}{l}
\bigl((\id_{A_i},1_{\ca_i(A_i,A_i)})_{i\in I},
\\
(1_{\cb_j(B_j,E_j)})_{j\in J},f_{(B_j),(E_j)},\ev^\VCat\bigr)
	\end{array}
	&\rMapsTo &
	\begin{array}{l}
\bigl((\id_{A_i})_{i\in I},
\\
f_{(B_j),(E_j)},\ev^\VCat\bigr)
	\end{array}
	&\lMapsTo &
	\begin{array}{l}
\bigl((\id_{A_i})_{i\in I},f_{(B_j),(E_j)},
\\
{}[(B_j)_{j\in J}f]_{(A_i),(A_i)},p_{(A_i)_{i\in I}},\centerdot\bigr)
	\end{array}
	\\
	\dMapsTo &&\dMapsTo &&\dMapsTo
	\\
	\begin{array}{r}
\bigl((\id_{A_i})_{i\in I},(1_{\cb_j(B_j,E_j)})_{j\in J},
\\
g_{(A_i),(B_j),(A_i),(E_j)}\bigr)
	\end{array}
&\rMapsTo &f'_{(B_j),(E_j)}p_{(A_i)} &\lMapsTo &\hspace*{-0.6em}
	\begin{array}{l}
\bigl(\id_{(A_i)_{i\in I}(B_j)_{j\in J}f},
\\
f_{(B_j),(E_j)},p_{(A_i)_{i\in I}},\centerdot\bigr)
	\end{array}
\end{diagram}
Therefore, \(f'=f\) and \(\varPhi\centerdot\varPsi=id\).

Start from \(g\:(\ca_i)_{i\in I},(\cb_j)_{j\in J}\to\cc\).
Produce \(f=g\varPsi\) and \(g''=f\varPhi=g\varPsi\varPhi\).
Then\\
\(\Ob g''\:\prod_{i\in I}\Ob\ca_i\times\prod_{j\in J}\Ob\cb_j\to\Ob\cc\) is given by
\begin{multline*}
g''=\bigl[ (\ca_i)_{i\in I},(\cb_j)_{j\in J} \rTTo^{(\Id_{\ca_i})_{i\in I},g\varPsi} (\ca_i)_{i\in I},\und\VCat((\ca_i)_{i\in I};\cc) \rTTo^{\ev^\VCat} \cc \bigr],
\\
\hskip\multlinegap \bigl((A_i)_{i\in I},(B_j)_{j\in J}\bigr) \mapsto \bigl((A_i)_{i\in I},\bigl[ (\ca_i)_{i\in I} \rTTo^{(\Id)_I,(\ddot B_j)_{j\in J}} (\ca_i)_{i\in I},(\cb_j)_{j\in J} \rto g \cc \bigr] \bigr) \hfill
\\
\mapsto \bigl((A_i)_{i\in I},(B_j)_{j\in J}\bigr)g.
\end{multline*}
Thus, \(\Ob g''=\Ob g\).

In order to describe $g''$ on morphisms let us rewrite \eqref{eq-fF-AA-BACBfEf} substituting \eqref{eq-Bf-A-AB-C-VCat} into it and using \eqref{eq-evVCat-E} for \(\ev^\VCat\).
We get the left map of the following diagram.
Applying the associativity property from \figref{dia-assoc-mu-multi} for maps \(I\sqcup J \rTTo^{\id\sqcup\triangledown} I\sqcup\mb1 \rTTo^{\triangledown\mathsf{I}} \mb2\) we get the right map in
\begin{gather*}
\begin{array}{r}
\prod_{i\in I}\mcv\bigl(\ca_i(A_i,D_i);\ca_i(A_i,D_i)\bigr) 
	\\
\times \mcv\bigl((\cb_j(B_j,E_j))_{j\in J};\und\VCat((\ca_i)_{i\in I};\cc)
\bigl(\bigl[ (\ca_i)_{i\in I} \rto{(\Id)_I,(\ddot B_j)_{j\in J}} (\ca_i)_{i\in I},(\cb_j)_{j\in J} \rto g \cc \bigr], \hfill
\\
\bigl[ (\ca_i)_{i\in I} \rto{(\Id)_I,(\ddot E_j)_{j\in J}} (\ca_i)_{i\in I},(\cb_j)_{j\in J} \rto g \cc \bigr])\bigr)
	\\
\times \mcv\bigl((\ca_i(A_i,D_i))_{i\in I}; \cc(((A_i)_{i\in I},(B_j)_{j\in J})g,((D_i)_{i\in I},(B_j)_{j\in J})g)\bigr) 
	\\
\times \mcv\bigl(\und\VCat\bigl((\ca_i)_{i\in I};\cc\bigr)(\bigl[ (\ca_i)_{i\in I} \rto{(\Id)_I,(\ddot B_j)_{j\in J}} (\ca_i)_{i\in I},(\cb_j)_{j\in J} \rto g \cc \bigr], \hfill
	\\
	\bigl[ (\ca_i)_{i\in I} \rto{(\Id)_I,(\ddot E_j)_{j\in J}} (\ca_i)_{i\in I},(\cb_j)_{j\in J} \rto g \cc \bigr]);\cc(((D_i)_{i\in I},(B_j)_{j\in J})g,((D_i)_{i\in I},(E_j)_{j\in J})g)\bigr)
	\\
	\times \mcv\bigl(\cc(((A_i)_{i\in I},(B_j)_{j\in J})g,((D_i)_{i\in I},(B_j)_{j\in J})g), \hfill
	\\
	\cc(((D_i)_{i\in I},(B_j)_{j\in J})g,((D_i)_{i\in I},(E_j)_{j\in J})g);\cc(((A_i)_{i\in I},(B_j)_{j\in J})g, ((D_i)_{i\in I},(E_j)_{j\in J})g)\bigr)
\end{array}
\\
\begin{tanglec}
\hh \id \hstep \id \step[1.5] \id \step \id \step \id \step
	\\
\hh \id \hstep \id \hstep \ffbox4{\mu_{\mathsf{\triangledown I}\:\!\!I\sqcup\mb1\to\mb2}}
	\\
\hh \id \hstep \id \step[2.5] \id \Step
	\\
\hh \ffbox6{\mu_{1\sqcup\con\:I\sqcup J\to I\sqcup\mb1}}
	\\
\hh \id
\end{tanglec}
\;=\;
\begin{tanglec}
	\id \Step \XX \Step \id \step \id
	\\
\hh \ffbox3{\mu_{1_I}} \step \ffbox3{\mu_{\triangledown\:\!J\to\mb1\!}} \hstep \id \hstep
	\\
	\hh \step \id \step[4] \id \Step \id
	\\
\hh \step \ffbox7{\mu_{\con\con\:I\sqcup J\to\mb2}}
	\\
	\hh \step \id
\end{tanglec}
\\
\mcV\bigl((\ca_i(A_i,D_i))_{i\in I},(\cb_j(B_j,E_j))_{j\in J};\cc(((A_i)_{i\in I},(B_j)_{j\in J})g,((D_i)_{i\in I},(E_j)_{j\in J})g)\bigr)
\end{gather*}
On elements
\begin{diagram}[h=2em,nobalance]
\bigl((1_{\ca_i(A_i,D_i)})_{i\in I},f_{(B_j),(E_j)},(B_j)_{j\in J}f_{(A_i),(D_i)},p_{(D_i)},\centerdot\bigr) &\rMapsTo &\bigl((B_j)_{j\in J}f_{(A_i),(D_i)},f_{(B_j),(E_j)}\centerdot p_{(D_i)},\centerdot\bigr)
	\\
	\dMapsTo &&\dMapsTo
	\\
\bigl((1_{\ca_i(A_i,D_i)})_{i\in I},f_{(B_j),(E_j)},\ev^\VCat\bigr) &\rMapsTo &g''_{(A_i),(B_j),(D_i),(E_j)}
\end{diagram}

Let us use the embedding
\begin{multline*}
\und\VCat\bigl((\ca_i)_{i\in I};\cc\bigr)
\\
(\bigl[ (\ca_i)_{i\in I} \rTTo^{(\Id)_I,(\ddot B_j)_{j\in J}} (\ca_i)_{i\in I},(\cb_j)_{j\in J} \rto g \cc \bigr], \bigl[ (\ca_i)_{i\in I} \rTTo^{(\Id)_I,(\ddot E_j)_{j\in J}} (\ca_i)_{i\in I},(\cb_j)_{j\in J} \rto g \cc \bigr])
\\
\subset \prod_{(X_i\in\cA_i)_{i\in I}}\cc\bigl(((X_i)_{i\in I},(B_j)_{j\in J})g,((X_i)_{i\in I},(E_j)_{j\in J})g\bigr).
\end{multline*}
Together with \eqref{eq-axiom-unit-multi2} it allows to rewrite the above replacing \((B_j)_{j\in J}f_{(A_i),(D_i)}\) with its definition \eqref{eq-(Bj)f} and \((D_i)_{i\in I}\bar f_{(B_j),(E_j)}\) with appropriately modified \eqref{eq-(Ai)fbar}.
We get the left map in equation~\eqref{eq-AAD-BBE-1}.
Applying the associativity property from \figref{dia-assoc-mu-multi} for maps \(I\sqcup J \rTTo^{\inj_1\sqcup\inj_2} (I\sqcup J)\sqcup(I\sqcup J) \rto{\triangledown\triangledown} \mb2\) we get the second map.
Using the fact that $g$ is a multi-entry $\mcv$\n-functor (see \defref{def-multi-entry-V-functor}), written for 3 tuples -- for \(((A_i)_{i\in I},(B_j)_{j\in J})\), for \(((D_i)_{i\in I},(B_j)_{j\in J})\) and for \(((D_i)_{i\in I},(E_j)_{j\in J})\) we get the left hand side of equation~\eqref{eq-AAD-BBE-2}.
Applying the associativity property from \figref{dia-assoc-mu-multi} for maps \(I\sqcup J \rTTo^{\inj_1\sqcup\inj_2} (I\sqcup J)\sqcup(I\sqcup J) \rto{\chi} I\sqcup J\) we get another expression for \(g''_{(A_i),(B_j),(D_i),(E_j)}\) from the right hand side of equation~\eqref{eq-AAD-BBE-2}.
On elements:
\begin{diagram}[h=1.3em,nobalance]
	&&\hspace*{-1em}
\begin{array}{l}
	\bigl((1_{\ca_i(A_i,D_i)})_{i\in I},(\id_{B_j})_{j\in J},
	\\
	(\id_{D_i})_{i\in I},(1_{\cb_j(B_j,E_j)})_{j\in J},?\bigr)
\end{array}
\hspace*{-1em} &&\phantom{(\centerdot)_I,(\centerdot)_J,g_{(A_i),(B_j),(D_i),(E_j)}\bigr)}
\\
&\ruMapsTo(2,3) &&\luMapsTo(2,3) &
\\
\\
	\begin{array}{l}
\bigl((1_{\ca_i(A_i,D_i)})_{i\in I},(\id_{B_j})_{j\in J},
\\
(\id_{D_i})_{i\in I},(1_{\cb_j(B_j,E_j)})_{j\in J},
\\
g_{(A_i),(B_j),(D_i),(B_j)},g_{(D_i),(B_j),(D_i),(E_j)},\centerdot\bigr)
	\end{array}
\hspace*{-3em}
&&&&\;\;
	\begin{array}{l}
\bigl((1_{\ca_i(A_i,D_i)})_{i\in I},(\id_{B_j})_{j\in J},
\\
(\id_{D_i})_{i\in I},(1_{\cb_j(B_j,E_j)})_{j\in J},
\\
(\centerdot)_I,(\centerdot)_J,g_{(A_i),(B_j),(D_i),(E_j)}\bigr)
	\end{array}
\\
\\
\dMapsTo &&\dMapsTo &&\dMapsTo
\\
\\
	\begin{array}{l}
\bigl((B_j)_{j\in J}f_{(A_i),(D_i)},
\\
(D_i)_{i\in I}\bar f_{(B_j),(E_j)},\centerdot\bigr)
	\end{array}
&&&&\;\;
	\begin{array}{l}
\bigl((1_{\ca_i(A_i,D_i)})_{i\in I},(1_{\cb_j(B_j,E_j)})_{j\in J},
\\
g_{(A_i),(B_j),(D_i),(E_j)}\bigr)
	\end{array}
\\
&\rdMapsTo &&\ldMapsTo 
\\
&&\;\;g''_{(A_i),(B_j),(D_i),(E_j)}
\end{diagram}
Therefore, \(g''_{(A_i),(B_j),(D_i),(E_j)}=g_{(A_i),(B_j),(D_i),(E_j)}\).
Hence, $g''=g$ and $\varPsi\centerdot\varPhi=\id$.
\end{proof}

Recall that the closed symmetric multicategory \(\VCat\) gives rise to a symmetric multicategory \(\und\VCat\) enriched in \(\VCat\) \cite[Proposition~4.10]{BesLyuMan-book}.
In particular, for each map \(\phi\:I\to J\) in \(\Mor\cs\) and \(\ca_i,\cb_j,\cc\in\Ob\VCat\), \(i\in I\), \(j\in J\), there exists a unique morphism
\[ \mu^{\und\VCat}_\phi\:\bigl(\und\VCat((\ca_i)_{i\in\phi^{-1}j};\cb_j)\bigr)_{j\in J}, 
\und\VCat((\cb_j)_{j\in J};\cc)\to\und\VCat((\ca_i)_{i\in I};\cc).
\]
This generalizes the horizontal composition of $\mcv$\n-transformations as discussed in \secref{sec-Compositions}.

\subsection{Completeness of the multicategory of \texorpdfstring{$\mcV$}V-categories}
\begin{proposition}\label{pro-VCat-has-products}
Let $\mcv$ be a locally small symmetric complete multicategory.
The multicategory $\VCat$ has small products.
\end{proposition}

\begin{proof}
Let \((\ca_i)_{i\in J}\) be a family of $\mcV$\n-categories, $J\in\Set$.
Then there is a $\mcV$\n-quiver $\ca$ with \(\Ob\ca=\prod_{i\in J}\Ob\ca_i\), \(\ca\bigl((A_i)_{i\in J},(D_i)_{i\in J}\bigr)=\prod_{i\in J}\ca_i(A_i,D_i)\).
Equip it with identity morphisms via
\begin{diagram}
\prod_{i\in J}\mcV\bigl(;\ca_i(A_i,A_i)\bigr) &\cong \mcV\bigl(;\prod_{i\in J}\ca_i(A_i,A_i)\bigr) = &\mcV\bigl(;\ca((A_i)_{i\in J},(A_i)_{i\in J})\bigr)\\
(\id^{\ca_i}_{A_i})_{i\in J} &\rMapsTo &\id^\ca_{(A_i)_{i\in J}}
\end{diagram}
and with composition via
\begin{multline*}
\prod_{i\in J}\mcV\bigl(\ca_i(A_i,D_i),\ca_i(D_i,E_i);\ca_i(A_i,E_i)\bigr) \rTTo^{\prod_{i\in J}\mcV(\pr_i,\pr_i;1)}
\\
\prod_{i\in J}\mcV\bigl(\prod_{j\in J}\ca_j(A_j,D_j),\prod_{k\in J}\ca_k(D_k,E_k);\ca_i(A_i,E_i)\bigr)
\\
\cong \mcV\bigl(\prod_{j\in J}\ca_j(A_j,D_j),\prod_{k\in J}\ca_k(D_k,E_k);\prod_{i\in J}\ca_i(A_i,E_i)\bigr),
\\
(\kappa_{A_i,D_i,E_i})_{i\in J} \mapsto ((\pr_i,\pr_i)\centerdot\kappa_{A_i,D_i,E_i})_{i\in J} \mapsto \kappa^\ca_{(A_i),(D_i),(E_i)}.
\end{multline*}
In detail:
\begin{multline}
\prod_{i\in J}\mcV\bigl(\ca_i(A_i,D_i),\ca_i(D_i,E_i);\ca_i(A_i,E_i)\bigr) \rTTo^{\prod_{i\in J}(\dot\pr_i\times\dot\pr_i\times1)}
\\
\hskip\multlinegap \prod_{i\in J}\bigl[ \mcV\bigl(\prod_{j\in J}\ca_j(A_j,D_j);\ca_i(A_i,D_i)\bigr)
\times \mcV\bigl(\prod_{k\in J}\ca_k(D_k,E_k);\ca_i(D_i,E_i)\bigr) \hfill
\\[-0.5em]
\hfill \times \mcV\bigl(\ca_i(A_i,D_i),\ca_i(D_i,E_i);\ca_i(A_i,E_i)\bigr)\bigr] \hskip\multlinegap
\\
\rTTo^{\prod_{i\in J}\mu_{\mathsf{II}}} \prod_{i\in J}\mcV\bigl(\prod_{j\in J}\ca_j(A_j,D_j),\prod_{k\in J}\ca_k(D_k,E_k);\ca_i(A_i,E_i)\bigr)
\\
\cong \mcV\bigl(\prod_{j\in J}\ca_j(A_j,D_j),\prod_{k\in J}\ca_k(D_k,E_k);\prod_{i\in J}\ca_i(A_i,E_i)\bigr),
\\
(\kappa_{A_i,D_i,E_i})_{i\in J} \mapsto (\pr_i,\pr_i,\kappa_{A_i,D_i,E_i})_{i\in J} \mapsto ((\pr_i,\pr_i)\centerdot\kappa_{A_i,D_i,E_i})_{i\in J} \mapsto \kappa^\ca_{(A_i),(D_i),(E_i)}.
\label{eq-kA(Ai)(Di)(Ei)}
\end{multline}

We have for all $i\in J$
\begin{diagram}[h=2.2em]
\prod_{j\in J}\ca_j(A_j,D_j),\prod_{k\in J}\ca_k(D_k,E_k) &\rTTo^{\kappa^\ca_{(A_j),(D_j),(E_j)}} &\prod_{n\in J}\ca_n(A_n,E_n)
\\
\dTTo<{\pr_i,\pr_i} &= &\dTTo>{\pr_i}
\\
\ca_i(A_i,D_i),\ca_i(D_i,E_i) &\rTTo^{\kappa_{A_i,D_i,E_i}} &\ca_i(A_i,E_i)
\end{diagram}
In detail, $tr=lb$ where:
\begin{multline*}
\mcV\bigl(\prod_{j\in J}\ca_j(A_j,D_j),\prod_{k\in J}\ca_k(D_k,E_k);\prod_{n\in J}\ca_n(A_n,E_n)\bigr) \times \mcV\bigl(\prod_{n\in J}\ca_n(A_n,E_n);\ca_i(A_i,E_i)\bigr) 
\\
\hfill \rTTo^{\mu_{\mathsf{V}}} \mcV\bigl(\prod_{j\in J}\ca_j(A_j,D_j),\prod_{k\in J}\ca_k(D_k,E_k);\ca_i(A_i,E_i)\bigr) \hskip\multlinegap
\\
(\kappa^\ca_{(A_j),(D_j),(E_j)},\pr_i) \mapsto tr,
\end{multline*}
\begin{multline}
\mcV\bigl(\prod_{j\in J}\ca_j(A_j,D_j);\ca_i(A_i,D_i)\bigr) \times \mcV\bigl(\prod_{k\in J}\ca_k(D_k,E_k);\ca_i(D_i,E_i)\bigr)
\\
\times \mcV\bigl(\ca_i(A_i,D_i),\ca_i(D_i,E_i);\ca_i(A_i,E_i)\bigr)
	\\
\hfill \rTTo^{\mu_{\mathsf{II}}} \mcV\bigl(\prod_{j\in J}\ca_j(A_j,D_j),\prod_{k\in J}\ca_k(D_k,E_k);\ca_i(A_i,E_i)\bigr) \hskip\multlinegap
	\\
(\pr_i,\pr_i,\kappa_{A_i,D_i,E_i}) \mapsto lb.
\label{eq-AjAkAi}
\end{multline}
Clearly, \eqref{eq-kA(Ai)(Di)(Ei)} is the unique solution of the above equation $tr=lb$.

Let us prove the first of unitality identities \eqref{eq-idX1k1}.
Denote by $lhs$ the composition
\begin{multline*}
\bigl[ \ca\bigl((A_i)_{i\in J},(D_i)_{i\in J}\bigr) \rTTo^{\id_{(A_i)},1} \ca\bigl((A_i)_{i\in J},(A_i)_{i\in J}\bigr),\ca\bigl((A_i)_{i\in J},(D_i)_{i\in J}\bigr)
\\
\rTTo^{\kappa_{(A_i),(A_i),(D_i)}} \ca\bigl((A_i)_{i\in J},(D_i)_{i\in J}\bigr) \bigr].
\end{multline*}
Explicitly,
\begin{multline*}
\mcV\bigl(;\ca((A_i)_{i\in J},(A_i)_{i\in J})\bigr) \times \mcV\bigl(\ca((A_i)_{i\in J},(D_i)_{i\in J});\ca((A_i)_{i\in J},(D_i)_{i\in J})\bigr)
\\
\times \mcV\bigl(\ca((A_i)_{i\in J},(A_i)_{i\in J}),\ca((A_i)_{i\in J},(D_i)_{i\in J});\ca((A_i)_{i\in J},(D_i)_{i\in J})\bigr)
\\
\rTTo^{\mu_{\mathsf{\centerdot I}}} \mcV\bigl(\ca((A_i)_{i\in J},(D_i)_{i\in J});\ca((A_i)_{i\in J},(D_i)_{i\in J})\bigr),
\\
\bigl(\id^\ca_{(A_i)_{i\in J}},1_{\ca((A_i)_{i\in J},(D_i)_{i\in J})},\kappa^\ca_{(A_i),(A_i),(D_i)}\bigr) \mapsto lhs.
\end{multline*}
Equation $lhs=1$ is equivalent to the left path of the following diagram.
In order to prove that the last obtained element is $\pr_i$ we use the associativity property from \figref{dia-assoc-mu-multi} for maps \(\mb1 \rto{\mathsf{\centerdot I}} \mb2 \rto{\mathsf{II}} \mb2\) and get the right path in
\begin{equation*}
	\hspace*{-0.3em}
	\begin{diagram}[h=3.6em,inline]
		\begin{array}{r}
\mcV\bigl(;\prod_{j\in J}\ca_j(A_j,A_j)\bigr) \times
\\
\mcV\bigl(\prod_{l\in J}\ca_l(A_l,D_l);\prod_{k\in J}\ca_k(A_k,D_k)\bigr)
\\
\times \mcV\bigl(\prod_{j\in J}\ca_j(A_j,A_j);\ca_i(A_i,A_i)\bigr)
\\
\times \mcV\bigl(\prod_{k\in J}\ca_k(A_k,D_k);\ca_i(A_i,D_i)\bigr) \times
\\
\mcV\bigl(\ca_i(A_i,A_i),\ca_i(A_i,D_i);\ca_i(A_i,D_i)\bigr)
		\end{array}
&\lTTo^{\cong\times\cong\times\dot\pr_i\times\dot\pr_i\times1} &
		\begin{array}{l}
\prod_{j\in J}\mcV\bigl(;\ca_j(A_j,A_j)\bigr) \times
	\\
\prod_{k\in J}\mcV\bigl(\prod_{l\in J}\ca_l(A_l,D_l);\ca_k(A_k,D_k)\bigr)
	\\
\times \mcV\bigl(\ca_i(A_i,A_i),\ca_i(A_i,D_i);\ca_i(A_i,D_i)\bigr)
		\end{array}
		\\
\dTTo<{1\times1\times\mu_{\mathsf{II}}} &\rdTTo^{\mu_{\varnothing\to\mb1}\times\mu_{\mb1\to\mb1}\times1} &
		\\
		\begin{array}{r}
	\mcV\bigl(;\!\prod_{j\in J}\!\ca_j(A_j,A_j)\bigr)
	\\
	\times \mcV\bigl(\prod_{l\in J}\!\ca_l(A_l,D_l);\!\prod_{k\in J}\!\ca_k(A_k,D_k)\bigr) \times
	\\
	\mcV\bigl(\prod_{j\in J}\ca_j(A_j,A_j),\prod_{k\in J}\ca_k(A_k,D_k);\ca_i(A_i,D_i)\bigr)
		\end{array}
\hspace*{-6em} &&
		\begin{array}{l}
\mcV\bigl(;\ca_i(A_i,A_i)\bigr)
\\
\times \mcV\bigl(\prod_{l\in J}\ca_l(A_l,D_l);\ca_i(A_i,D_i)\bigr) \times
\\
\mcV\bigl(\ca_i(A_i,A_i),\ca_i(A_i,D_i);\ca_i(A_i,D_i)\bigr)
		\end{array}
		\\
&\rdTTo<{\mu_{\mathsf{\centerdot I}}} &\dTTo>{\mu_{\mathsf{\centerdot I}}}
		\\
&&\hspace*{-5em} \mcV\bigl(\prod_{l\in J}\ca_l(A_l,D_l);\ca_i(A_i,D_i)\bigr)
	\end{diagram}
\end{equation*}
On elements
\begin{diagram}[h=2em,nobalance]
\bigl( \id^\ca_{(A_j)_{j\in J}},1_{\prod_{k\in J}\ca_k(A_k,D_k)},\pr_i,\pr_i,\kappa_{A_i,A_i,D_i} \bigr) &\lMapsTo &
\bigl( (\id_{A_j})_{j\in J},(\pr_k)_{k\in J},\kappa_{A_i,A_i,D_i} \bigr)
\\
\dMapsTo &\rdMapsTo &
\\
\bigl( \id^\ca_{(A_j)_{j\in J}},1_{\prod_{k\in J}\ca_k(A_k,D_k)},(\pr_i,\pr_i)\centerdot\kappa_{A_i,A_i,D_i} \bigr) &&\bigl( \id_{A_i},\pr_i,\kappa_{A_i,A_i,D_i} \bigr)
\\
&\rdMapsTo<? &\dMapsTo
	\\
&&\pr_i
\end{diagram}

Let us prove the second of unitality identities \eqref{eq-1idYk1}.
Denote by $lhs$ the composition
\begin{multline*}
\bigl[ \ca\bigl((A_i)_{i\in J},(D_i)_{i\in J}\bigr) \rTTo^{1,\id_{(D_i)}} \ca\bigl((A_i)_{i\in J},(D_i)_{i\in J}\bigr),\ca\bigl((D_i)_{i\in J},(D_i)_{i\in J}\bigr)
\\
\rTTo^{\kappa_{(A_i),(D_i),(D_i)}} \ca\bigl((A_i)_{i\in J},(D_i)_{i\in J}\bigr) \bigr].
\end{multline*}
Explicitly,
\begin{multline*}
\mcV\bigl(\ca((A_i)_{i\in J},(D_i)_{i\in J});\ca((A_i)_{i\in J},(D_i)_{i\in J})\bigr) \times \mcV\bigl(;\ca((D_i)_{i\in J},(D_i)_{i\in J})\bigr)  
	\\
\times \mcV\bigl(\ca((A_i)_{i\in J},(D_i)_{i\in J}),\ca((D_i)_{i\in J},(D_i)_{i\in J});\ca((A_i)_{i\in J},(D_i)_{i\in J})\bigr)
	\\
\rTTo^{\mu_{\mathsf{I\centerdot}}} \mcV\bigl(\ca((A_i)_{i\in J},(D_i)_{i\in J});\ca((A_i)_{i\in J},(D_i)_{i\in J})\bigr),
	\\
\bigl(1_{\ca((A_i)_{i\in J},(D_i)_{i\in J})},\id^\ca_{(A_i)_{i\in J}},\kappa^\ca_{(A_i),(D_i),(D_i)}\bigr) \mapsto lhs.
\end{multline*}
Equation $lhs=1$ is equivalent to the left path of the following diagram.
In order to prove that the last obtained element is $\pr_i$ we use the associativity property from \figref{dia-assoc-mu-multi} for maps \(\mb1 \rto{\mathsf{I\centerdot}} \mb2 \rto{\mathsf{II}} \mb2\) and get the right path in
\begin{equation*}
	\hspace*{-0.3em}
	\begin{diagram}[h=3.6em,inline]
		\begin{array}{r}
\mcV\bigl(\prod_{l\in J}\ca_l(A_l,D_l);\prod_{k\in J}\ca_k(A_k,D_k)\bigr)
\\
\times \mcV\bigl(;\prod_{j\in J}\ca_j(D_j,D_j)\bigr)
\\
\times \mcV\bigl(\prod_{k\in J}\ca_k(A_k,D_k);\ca_i(A_i,D_i)\bigr)
\\
\times \mcV\bigl(\prod_{j\in J}\ca_j(D_j,D_j);\ca_i(D_i,D_i)\bigr)
\\
\times \mcV\bigl(\ca_i(A_i,D_i),\ca_i(D_i,D_i);\ca_i(A_i,D_i)\bigr)
		\end{array}
&\lTTo^{\cong\times\cong\times\dot\pr_i\times\dot\pr_i\times1} &
		\begin{array}{l}
\prod_{k\in J}\mcV\bigl(\prod_{l\in J}\ca_l(A_l,D_l);\ca_k(A_k,D_k)\bigr)
\\
\times \prod_{j\in J}\mcV\bigl(;\ca_j(D_j,D_j)\bigr)
\\
\times \mcV\bigl(\ca_i(A_i,D_i),\ca_i(D_i,D_i);\ca_i(A_i,D_i)\bigr)
		\end{array}
		\\
\dTTo<{1\times1\times\mu_{\mathsf{II}}} &\rdTTo^{\mu_{\mb1\to\mb1}\times\mu_{\varnothing\to\mb1}\times1} &
		\\
		\begin{array}{r}
\mcV\bigl(\prod_{l\in J}\ca_l(A_l,D_l);\prod_{k\in J}\ca_k(A_k,D_k)\bigr)
\\
\times \mcV\bigl(;\prod_{j\in J}\ca_j(D_j,D_j)\bigr) \times
\\
\mcV\bigl(\prod_{k\in J}\ca_k(A_k,D_k),\prod_{j\in J}\ca_j(D_j,D_j);\ca_i(A_i,D_i)\bigr)
		\end{array}
		\hspace*{-6em} &&
		\begin{array}{l}
\mcV\bigl(\prod_{l\in J}\ca_l(A_l,D_l);\ca_i(A_i,D_i)\bigr)
\\
\times \mcV\bigl(;\ca_i(A_i,D_i)\bigr) \times
\\
\mcV\bigl(\ca_i(A_i,D_i),\ca_i(D_i,D_i);\ca_i(A_i,D_i)\bigr)
		\end{array}
		\\
&\rdTTo<{\mu_{\mathsf{I\centerdot}}} &\dTTo>{\mu_{\mathsf{I\centerdot}}}
		\\
		&&\hspace*{-5em} \mcV\bigl(\prod_{l\in J}\ca_l(A_l,D_l);\ca_i(A_i,D_i)\bigr)
	\end{diagram}
\end{equation*}
On elements
\begin{diagram}[h=2em,nobalance]
\bigl( 1_{\prod_{k\in J}\ca_k(A_k,D_k)},\id^\ca_{(D_j)_{j\in J}},\pr_i,\pr_i,\kappa_{A_i,D_i,D_i} \bigr) &\lMapsTo &\bigl( (\pr_k)_{k\in J},(\id_{D_j})_{j\in J},\kappa_{A_i,D_i,D_i} \bigr)
	\\
	\dMapsTo &\rdMapsTo &
	\\
\bigl( 1_{\prod_{k\in J}\ca_k(A_k,D_k)},\id^\ca_{(D_j)_{j\in J}},(\pr_i,\pr_i)\centerdot\kappa_{A_i,D_i,D_i} \bigr) &&\bigl( \pr_i,\id^\ca_{(D_j)_{j\in J}},\kappa_{A_i,D_i,D_i} \bigr)
	\\
	&\rdMapsTo<? &\dMapsTo
	\\
	&&\pr_i
\end{diagram}

The associativity of $\ca$ is expressed by diagram~\eqref{dia-assoc-V-cat} which is a shorthand of equation $tr=lb$, where
\begin{multline*}
\mcV\bigl(\ca((A_i)_{i\in J},(C_i)_{i\in J});\ca((A_i)_{i\in J},(C_i)_{i\in J})\bigr)
\\
\times \mcV\bigl(\ca((C_i)_{i\in J},(D_i)_{i\in J}),\ca((D_i)_{i\in J},(E_i)_{i\in J});\ca((C_i)_{i\in J},(E_i)_{i\in J})\bigr)
\\
\times \mcV\bigl(\ca((A_i)_{i\in J},(C_i)_{i\in J}),\ca((C_i)_{i\in J},(E_i)_{i\in J});\ca((A_i)_{i\in J},(E_i)_{i\in J})\bigr) \rTTo^{\mu_{\mathsf{IV}\:\mb3\to\mb2}}
	\\
\mcV\bigl(\ca((A_i)_{i\in J},(C_i)_{i\in J}),\ca((C_i)_{i\in J},(D_i)_{i\in J}),\ca((D_i)_{i\in J},(E_i)_{i\in J});\ca((A_i)_{i\in J},(E_i)_{i\in J})\bigr),
	\\
\bigl(1_{\ca((A_i)_{i\in J},(C_i)_{i\in J})},\kappa^\ca_{(C_i),(D_i),(E_i)},\kappa^\ca_{(A_i),(C_i),(E_i)}\bigr) \mapsto tr,
\end{multline*}
\begin{multline*}
\mcV\bigl(\ca((A_i)_{i\in J},(C_i)_{i\in J}),\ca((C_i)_{i\in J},(D_i)_{i\in J});\ca((A_i)_{i\in J},(D_i)_{i\in J})\bigr) 
\\
\times \mcV\bigl(\ca((D_i)_{i\in J},(E_i)_{i\in J});\ca((D_i)_{i\in J},(E_i)_{i\in J})\bigr)
	\\
\times \mcV\bigl(\ca((A_i)_{i\in J},(D_i)_{i\in J}),\ca((D_i)_{i\in J},(E_i)_{i\in J});\ca((A_i)_{i\in J},(E_i)_{i\in J})\bigr) \rTTo^{\mu_{\mathsf{VI}\:\mb3\to\mb2}}
	\\
\mcV\bigl(\ca((A_i)_{i\in J},(C_i)_{i\in J}),\ca((C_i)_{i\in J},(D_i)_{i\in J}),\ca((D_i)_{i\in J},(E_i)_{i\in J});\ca((A_i)_{i\in J},(E_i)_{i\in J})\bigr),
	\\
\bigl(\kappa^\ca_{(A_i),(C_i),(D_i)},1_{\ca((D_i)_{i\in J},(E_i)_{i\in J})},\kappa^\ca_{(A_i),(D_i),(E_i)}\bigr) \mapsto lb.
\end{multline*}
Fix $i\in J$ and consider the projection $\pr_i$ to $i$th factor.
The required equation is equivalent to equation \(tr\centerdot\pr_i\equiv tr_i=lb_i\equiv lb\centerdot\pr_i\) between elements which are obtained below.
First of all we get $tr_i$ via the left map of the following diagram.
These expressions can be transformed using the associativity property of \figref{dia-assoc-mu-multi} for maps \(\mb3 \rto{\mathsf{IV}} \mb2 \rto{\id} \mb2\) giving the right map in
\begin{gather*}
\begin{array}{r}
\prod_{m\in J}\mcV\bigl(\prod_{j\in J}\ca_j(A_j,C_j);\ca_m(A_m,C_m)\bigr) \times
\prod_{n\in J}\mcV\bigl(\ca_n(C_n,D_n),\ca_n(D_n,E_n);\ca_n(C_n,E_n)\bigr)
	\\
	\times \mcV\bigl(\ca_i(A_i,C_i),\ca_i(C_i,E_i);\ca_i(A_i,E_i)\bigr)
\end{array}
\\
\begin{tanglec}
\hh \id \step[5] \id \step[4.5] \id
\\
\ffbox1\cong \hstep \ffbox8{\prod_{n\in J}(\dot\pr_n\!\times\!\dot\pr_n\!\times\!1)} \hstep \id \hstep
\\
\hh \id \step[4] \id \step \id \step \id \step[3.5] \id
\\
\id \step[3] \ffbox4{\prod_{n\in J}\mu_{\mathsf{II}}} \step[2.5] \id
\\
\hh \id \step[3.5] \id \step[6] \id
\\
\hh \id \step[3] \ffbox1\cong \hstep \ffbox2{\dot\pr_i} \hstep \ffbox2{\dot\pr_i} \hstep \id
\\
\hh \id \step[3.5] \id \Step \id \step[2.5] \id \step[1.5] \id
\\
\hh \hstep \id \step[3.5] \id \step[1.5] \ffbox5{\mu_{\mathsf{II}}}
\\
\hh \id \step[3.5] \id \step[4] \id \Step
\\
\hh \ffbox9{\mu_{\mathsf{IV}\:\mb3\to\mb2}} \step[1.5]
\\
\hh \id \step[1.5]
\end{tanglec}
\;=\;
\begin{tanglec}
	\hh \id \step[5] \id \step[4.5] \id
	\\
	\ffbox1\cong \hstep \ffbox8{\prod_{n\in J}(\dot\pr_n\!\times\!\dot\pr_n\!\times\!1)} \hstep \id \hstep
	\\
	\hh \id \step[4] \id \step \id \step \id \step[3.5] \id
	\\
	\id \step[3] \ffbox4{\prod_{n\in J}\mu_{\mathsf{II}}} \step[2.5] \id
	\\
	\hh \id \step[3.5] \id \step[6] \id
	\\
	\hh \id \step[3] \ffbox1\cong \hstep \ffbox2{\dot\pr_i} \hstep \ffbox2{\dot\pr_i} \hstep \id
\\
\id \step[3.5] \XX \step[2.5] \id \step[1.5] \id
\\
\hh \ffbox5{\mu_{\mathsf{I}}} \step \ffbox4{\mu_{\mathsf{V}}} \hstep \id \step
	\\
\hh \step[1.5] \id \step[5.5] \id \step[2.5] \id
	\\
\hh \step[1.5] \ffbox9{\mu_{\mathsf{IV}\:\mb3\to\mb2}}
	\\
\hh \step[1.5] \id
\end{tanglec}
\\
\mcV\bigl(\prod_{j\in J}\ca_j(A_j,C_j),\prod_{k\in J}\ca_k(C_k,D_k),\prod_{l\in J}\ca_l(D_l,E_l);\ca_i(A_i,E_i)\bigr)
\end{gather*}
On elements
\begin{diagram}[h=2em,nobalance]
\bigl((\pr_m)_{m\in J},(\kappa_{C_n,D_n,E_n})_{n\in J},\kappa_{A_i,C_i,E_i}\bigr) &&\phantom{,(\pr_n,\pr_n,\kappa_{C_n,D_n,E_n})_{n\in J},\kappa_{A_i,C_i,E_i}\bigr)}
\\
&\rdMapsTo &
	\\
&&\hspace*{-4em} \bigl(1_{\prod_{j\in J}\ca_j(A_j,C_j)},(\pr_n,\pr_n,\kappa_{C_n,D_n,E_n})_{n\in J},\kappa_{A_i,C_i,E_i}\bigr)
\\
&&\dMapsTo
\\
&&\hspace*{-5em} \bigl(1_{\prod_{j\in J}\ca_j(A_j,C_j)},((\pr_n,\pr_n)\centerdot\kappa_{C_n,D_n,E_n})_{n\in J},\kappa_{A_i,C_i,E_i}\bigr)
\\
&\ldMapsTo &
	\\
\bigl(1_{\prod_{j\in J}\ca_j(A_j,C_j)},\kappa^\ca_{(C_i),(D_i),(E_i)},\pr_i,\pr_i,\kappa_{A_i,C_i,E_i}\bigr) &&
\\
\dMapsTo &\rdMapsTo &
\\
\bigl(1_{\prod_{j\in J}\ca_j(A_j,C_j)},\kappa^\ca_{(C_i),(D_i),(E_i)},(\pr_i,\pr_i)\centerdot\kappa_{A_i,C_i,E_i}\bigr) \hspace*{-1.4em} &&\bigl(\pr_i,(\pr_i,\pr_i)\centerdot\kappa_{C_i,D_i,E_i},\kappa_{A_i,C_i,E_i}\bigr)
	\\
	&\rdMapsTo<? &\dMapsTo
	\\
&&tr_i
\end{diagram}

Secondly, we obtain $lb_i$ via the left map of the following diagram.
These expressions can be transformed using the associativity property of \figref{dia-assoc-mu-multi} for maps \(\mb3 \rto{\mathsf{VI}} \mb2 \rto{\id} \mb2\) giving the right map in
\begin{gather*}
\begin{array}{r}
\prod_{n\in J}\mcV\bigl(\ca_n(A_n,C_n),\ca_n(C_n,D_n);\ca_n(A_n,D_n)\bigr)
\times \prod_{m\in J}\mcV\bigl(\prod_{l\in J}\ca_l(D_l,E_l);\ca_m(D_m,E_m)\bigr) 
	\\
	\times \mcV\bigl(\ca_i(A_i,D_i),\ca_i(D_i,E_i);\ca_i(A_i,E_i)\bigr)
\end{array}
\\
\begin{tanglec}
\hh \step[4] \id \step[5] \id \step[2.5] \id
	\\
\ffbox8{\prod_{n\in J}(\dot\pr_n\!\times\!\dot\pr_n\!\times\!1)} \hstep \ffbox1\cong \Step \id 
	\\
\hh \step \id \step \id \step \id \step[6] \id \step[2.5] \id
	\\
\ffbox4{\prod_{n\in J}\mu_{\mathsf{II}}} \step[4] \ne3 \step[2.5] \id
	\\
\hh \Step \id \step[4] \id \step[5.5] \id
	\\
\hh \step[1.5] \ffbox1\cong \step[3.5] \id \hstep \ffbox2{\dot\pr_i} \hstep \ffbox2{\dot\pr_i} \hstep \id
	\\
\hh \Step \id \step[4] \id \step[1.5] \id \step[2.5] \id \step[1.5] \id
	\\
\hh \step[2.5] \id \step[4] \id \step \ffbox5{\mu_{\mathsf{II}}}
	\\
\hh \id \step[4] \id \step[3.5] \id
	\\
\hh \ffbox9{\mu_{\mathsf{VI}\:\mb3\to\mb2}} \hstep
	\\
\hh \id \hstep
\end{tanglec}
\;=\;
\begin{tanglec}
	\hh \step[4] \id \step[5] \id \step[2.5] \id
	\\
	\ffbox8{\prod_{n\in J}(\dot\pr_n\!\times\!\dot\pr_n\!\times\!1)} \hstep \ffbox1\cong \Step \id 
	\\
	\hh \step \id \step \id \step \id \step[6] \id \step[2.5] \id
	\\
	\ffbox4{\prod_{n\in J}\mu_{\mathsf{II}}} \step[4] \ne3 \step[2.5] \id
	\\
	\hh \Step \id \step[4] \id \step[5.5] \id
	\\
	\hh \step[1.5] \ffbox1\cong \step[3.5] \id \hstep \ffbox2{\dot\pr_i} \hstep \ffbox2{\dot\pr_i} \hstep \id
	\\
\Step \id \step[4] \XX \Step \id \step[1.5] \id
	\\
\hh \step[1.5] \ffbox5{\mu_{\mathsf{V}}} \step \ffbox3{\mu_{\mathsf{I}}} \step \id
	\\
\hh \step[4] \id \step[5] \id \step[2.5] \id
	\\
\hh \step[3.5] \ffbox9{\mu_{\mathsf{VI}\:\mb3\to\mb2}}
	\\
\hh \step[3.5] \id
\end{tanglec}
\\
\mcV\bigl(\prod_{j\in J}\ca_j(A_j,C_j),\prod_{k\in J}\ca_k(C_k,D_k),\prod_{l\in J}\ca_l(D_l,E_l);\ca_i(A_i,E_i)\bigr)
\end{gather*}
On elements
\begin{diagram}[h=2em,nobalance]
\bigl((\kappa_{A_n,C_n,D_n})_{n\in J},(\pr_m)_{m\in J},\kappa_{A_i,D_i,E_i}\bigr) &&\phantom{,\kappa_{A_n,C_n,D_n})_{n\in J},1_{\prod_{l\in J}\ca_l(D_l,E_l)},\kappa_{A_i,D_i,E_i}\bigr)}
\\
&\rdMapsTo &
\\
&&\hspace*{-4em} \bigl((\pr_n,\pr_n,\kappa_{A_n,C_n,D_n})_{n\in J},1_{\prod_{l\in J}\ca_l(D_l,E_l)},\kappa_{A_i,D_i,E_i}\bigr)
	\\
&&\dMapsTo
	\\
&&\hspace*{-5em} \bigl(((\pr_n,\pr_n)\centerdot\kappa_{A_n,C_n,D_n})_{n\in J},1_{\prod_{l\in J}\ca_l(D_l,E_l)},\kappa_{A_i,D_i,E_i}\bigr)
\\
&\ldMapsTo
\\
\bigl(\kappa^\ca_{(A_i),(C_i),(D_i)},1_{\prod_{l\in J}\ca_l(D_l,E_l)},\pr_i,\pr_i,\kappa_{A_i,D_i,E_i}\bigr) &&
	\\
	\dMapsTo &\rdMapsTo &
	\\
\bigl(\kappa^\ca_{(A_i),(C_i),(D_i)},1_{\prod_{l\in J}\ca_l(D_l,E_l)},(\pr_i,\pr_i)\centerdot\kappa_{A_i,D_i,E_i}\bigr) \hspace*{-1.4em} &&\bigl((\pr_i,\pr_i)\centerdot\kappa_{A_i,C_i,D_i},\pr_i,\kappa_{A_i,D_i,E_i}\bigr)
	\\
	&\rdMapsTo<? &\dMapsTo
	\\
&&lb_i
\end{diagram}

The expressions for $tr_i$ and $lb_i$ can be obtained in a different way, using equation $tr=lb$ from \eqref{eq-AjAkAi}.
First the left path in the following diagram for $tr_i$.
Applying the associativity property from \figref{dia-assoc-mu-multi} for maps \(\mb3 \rto\id \mb3 \rto{\mathsf{IV}} \mb2\) we get the right path in
\begin{diagram}[h=4.2em]
\begin{array}{r}
\mcV\bigl(\prod_{j\in J}\ca_j(A_j,C_j);\ca_i(A_i,C_i)\bigr)
\\
\times \mcV\bigl(\ca_i(A_i,C_i);\ca_i(A_i,C_i)\bigr)
\\
\times \mcV\bigl(\prod_{k\in J}\ca_k(C_k,D_k);\ca_i(C_i,D_i)\bigr)
\\
\times \mcV\bigl(\prod_{l\in J}\ca_l(D_l,E_l);\ca_i(D_i,E_i)\bigr) \times
\\
\mcV\bigl(\ca_i(C_i,D_i),\ca_i(D_i,E_i);\ca_i(C_i,E_i)\bigr)
\\
\times \mcV\bigl(\ca_i(A_i,C_i),\ca_i(C_i,E_i);\ca_i(A_i,E_i)\bigr)
\end{array}
&\rTTo^{1\times1\times1\times\mu_{\mathsf{IV}\:\mb3\to\mb2}} &
\begin{array}{l}
\mcV\bigl(\prod_{j\in J}\ca_j(A_j,C_j);\ca_i(A_i,C_i)\bigr)
\\
\times \mcV\bigl(\prod_{k\in J}\ca_k(C_k,D_k);\ca_i(C_i,D_i)\bigr)
\\
\times \mcV\bigl(\prod_{l\in J}\ca_l(D_l,E_l);\ca_i(D_i,E_i)\bigr)
\\
\times \mcV\bigl(\ca_i(A_i,C_i),\ca_i(C_i,D_i),
\\
\hfill \ca_i(D_i,E_i);\ca_i(A_i,E_i)\bigr)
\end{array}
\\
\dTTo<{\mu_{\mathsf{I}}\times\mu_{\mathsf{II}}\times1} &&\dTTo>{\mu_{\mathsf{III}}}
\\
\begin{array}{r}
\mcV\bigl(\prod_{j\in J}\ca_j(A_j,C_j);\ca_i(A_i,C_i)\bigr) \times
\\
\mcV\bigl(\prod_{k\in J}\ca_k(C_k,D_k);\prod_{l\in J}\ca_l(D_l,E_l);
\\
\ca_i(C_i,E_i)\bigr)
\\
\times \mcV\bigl(\ca_i(A_i,C_i),\ca_i(C_i,E_i);\ca_i(A_i,E_i)\bigr)
\end{array}
&\rTTo^{\mu_{\mathsf{IV}\:\mb3\to\mb2}} &
\begin{array}{r}
\mcV\bigl(\prod_{j\in J}\ca_j(A_j,C_j),\prod_{k\in J}\ca_k(C_k,D_k),
\\
\prod_{l\in J}\ca_l(D_l,E_l);\ca_i(A_i,E_i)\bigr)
\end{array}
\end{diagram}
On elements
\begin{diagram}
\bigl(\pr_i,1_{\ca_i(A_i,C_i)},\pr_i,\pr_i,\kappa_{C_i,D_i,E_i},\kappa_{A_i,C_i,E_i}\bigr) &\rMapsTo &(\pr_i,\pr_i,\pr_i,\kappa_{A_i,C_i,D_i,E_i})
	\\
\dMapsTo &&\dMapsTo
	\\
\bigl(\pr_i,(\pr_i,\pr_i)\centerdot\kappa_{C_i,D_i,E_i},\kappa_{A_i,C_i,E_i}\bigr) &\rMapsTo &tr_i=(\pr_i,\pr_i,\pr_i)\centerdot\kappa_{A_i,C_i,D_i,E_i}
\end{diagram}
see \eqref{eq-V(CC)V(CCC)V(CCC)-V(CCCC)}.

Second the left path in the following diagram for $lb_i$.
Applying the associativity property from \figref{dia-assoc-mu-multi} for maps \(\mb3 \rto\id \mb3 \rto{\mathsf{VI}} \mb2\) we get the right path in
\begin{diagram}[h=4.2em]
	\begin{array}{r}
\mcV\bigl(\prod_{j\in J}\ca_j(A_j,C_j);\ca_i(A_i,C_i)\bigr)
\\
\times \mcV\bigl(\prod_{k\in J}\ca_k(C_k,D_k);\ca_i(C_i,D_i)\bigr) \times
\\
\mcV\bigl(\ca_i(A_i,C_i),\ca_i(C_i,D_i);\ca_i(A_i,D_i)\bigr)
\\
\times \mcV\bigl(\prod_{l\in J}\ca_l(D_l,E_l);\ca_i(D_i,E_i)\bigr) 
\\
\times \mcV\bigl(\ca_i(D_i,E_i);\ca_i(D_i,E_i)\bigr) \times
\\
\mcV\bigl(\ca_i(A_i,D_i),\ca_i(D_i,E_i);\ca_i(A_i,E_i)\bigr) 
	\end{array}
&\rTTo^{1\times1\times1\times\mu_{\mathsf{VI}\:\mb3\to\mb2}} &
	\begin{array}{l}
\mcV\bigl(\prod_{j\in J}\ca_j(A_j,C_j);\ca_i(A_i,C_i)\bigr)
\\
\times \mcV\bigl(\prod_{k\in J}\ca_k(C_k,D_k);\ca_i(C_i,D_i)\bigr)
\\
\times \mcV\bigl(\prod_{l\in J}\ca_l(D_l,E_l);\ca_i(D_i,E_i)\bigr)
\\
\times \mcV\bigl(\ca_i(A_i,C_i),\ca_i(C_i,D_i),
\\
\hfill \ca_i(D_i,E_i);\ca_i(A_i,E_i)\bigr)
	\end{array}
	\\
\dTTo<{\mu_{\mathsf{II}}\times\mu_{\mathsf{I}}\times1} &&\dTTo>{\mu_{\mathsf{III}}}
	\\
	\begin{array}{r}
\mcV\bigl(\prod_{j\in J}\ca_j(A_j,C_j),\prod_{k\in J}\ca_k(C_k,D_k);
\\
\ca_i(A_i,D_i)\bigr)
\\
\times \mcV\bigl(\prod_{l\in J}\ca_l(D_l,E_l);\ca_i(D_i,E_i)\bigr) 
\\
\times \mcV\bigl(\ca_i(A_i,D_i),\ca_i(D_i,E_i);\ca_i(A_i,E_i)\bigr) 
	\end{array}
&\rTTo^{\mu_{\mathsf{VI}\:\mb3\to\mb2}} &
	\begin{array}{r}
\mcV\bigl(\prod_{j\in J}\ca_j(A_j,C_j),\prod_{k\in J}\ca_k(C_k,D_k),
\\
\prod_{l\in J}\ca_l(D_l,E_l);\ca_i(A_i,E_i)\bigr)
	\end{array}
\end{diagram}
On elements
\begin{diagram}
(\pr_i,\pr_i,\kappa_{A_i,C_i,D_i},\pr_i,1_{\ca_i(D_i,E_i)},\kappa_{A_i,D_i,E_i}) &\rMapsTo &(\pr_i,\pr_i,\pr_i,\kappa_{A_i,C_i,D_i,E_i})
	\\
	\dMapsTo &&\dMapsTo
	\\
\bigl((\pr_i,\pr_i)\centerdot\kappa_{A_i,C_i,D_i},\pr_i,\kappa_{A_i,D_i,E_i}\bigr) &\rMapsTo &lb_i=(\pr_i,\pr_i,\pr_i)\centerdot\kappa_{A_i,C_i,D_i,E_i}
\end{diagram}
see \eqref{eq-V(CCC)V(CC)V(CCC)-V(CCCC)}.
Thus \(tr_i=lb_i\) and associativity of multiplication in $\ca$ is proven.

Equation $tr=lb$ from \eqref{eq-AjAkAi} together with the definition of \(\id^\ca\) show that \(\pr_i\:\ca\to\ca_i\) is a $\mcv$\n-functor, see \exaref{exa-I1}.

Let us show that \((\pr_j\:\ca\to\ca_j)_{j\in J}\) is a product in $\VCat$ of a family \((\ca_j)_{j\in J}\).
Here $\ca$ is constructed from this family as above.
Let $\cx_i$, \(i\in I\in\cs_\sk\), be $\mcv$\n-categories and let \(f^j\:(\cx_i)_{i\in I}\to\ca_j\) be multi-entry $\mcv$\n-functors, $j\in J$.
Consider a multi-entry $\mcv$\n-quiver morphism \(f\:(\cx_i)_{i\in I}\to\ca\) which consists of 
\begin{myitemize}
\item[---] the function \(f=\Ob f\:(X_i)_{i\in I}\mapsto\bigl((X_i)_{i\in I}f^j\bigr)_{j\in J}\);
\item[---] the collection of elements \(f_{(X_i),(Y_i)}\) from
\end{myitemize}
\begin{align*}
\prod_{j\in J}\mcv\bigl((\cx_i(X_i,Y_i))_{i\in I};\ca_j((X_i)_{i\in I}f^j,(Y_i)_{i\in I}f^j)\bigr)
&\cong\mcv\bigl((\cx_i(X_i,Y_i))_{i\in I};\prod_{j\in J}\ca_j((X_i)_{i\in I}f^j,(Y_i)_{i\in I}f^j)\bigr)
\\[-0.6em]
\bigl(f^j_{(X_i),(Y_i)}\bigr)_{j\in J} &\mapsto f_{(X_i),(Y_i)}.
\end{align*}
It is a unique multi-entry $\mcv$\n-quiver morphism $f$ with the property \(f\centerdot\pr_j=f^j\), $j\in J$.

The morphism $f$ is coherent with units, since
\[ \bigl[() \rTTo^{(\id_{X_i})_{i\in I}} (\cx_i(X_i,X_i))_{i\in I} \rTTo^{f_{(X_i),(X_i)}} \ca((X_i)_{i\in I}f,(X_i)_{i\in I}f)\bigr] =\id_{(X_i)_{i\in I}f}.
\]
Indeed, composing with \(\pr_j\:\ca\to\ca_j\) we come to the valid identity:
\[ \bigl[() \rTTo^{(\id_{X_i})_{i\in I}} (\cx_i(X_i,X_i))_{i\in I} \rTTo^{f^j_{(X_i),(X_i)}} \ca_j((X_i)_{i\in I}f^j,(X_i)_{i\in I}f^j)\bigr] =\id_{(X_i)_{i\in I}f^j}.
\]

Compare now $lb$ and $tr$ from
\begin{multline*}
\mcv\bigl(\bigl(\cx_i(X_i,Y_i)\bigr)_{i\in I};\!\prod_{j\in J}\!\ca_j\bigl((X_i)_{i\in I}f^j,(Y_i)_{i\in I}f^j\bigr)\bigr)
\times \mcv\bigl(\bigl(\cx_i(Y_i,Z_i)\bigr)_{i\in I};\!\prod_{k\in J}\!\ca_k\bigl((Y_i)_{i\in I}f^k,(Z_i)_{i\in I}f^k\bigr)\bigr) 
	\\
\times \mcv\bigl(\prod_{j\in J}\ca_j\bigl((X_i)_{i\in I}f^j,(Y_i)_{i\in I}f^j\bigr),\prod_{k\in J}\ca_k\bigl((Y_i)_{i\in I}f^k,(Z_i)_{i\in I}f^k\bigr); \prod_{l\in J}\ca_l\bigl((X_i)_{i\in I}f^l, (Z_i)_{i\in I}f^l\bigr)\bigr)
	\\
\rTTo^{\mu_{\triangledown\triangledown}} \mcv\bigl(\bigl(\cx_i(X_i,Y_i)\bigr)_{i\in I},\bigl(\cx_i(Y_i,Z_i)\bigr)_{i\in I};\prod_{l\in J}\ca_l\bigl((X_i)_{i\in I}f^l, (Z_i)_{i\in I}f^l\bigr)\bigr),
	\\
\hfill \bigl(f_{(X_i),(Y_i)},f_{(Y_i),(Z_i)},\kappa^\ca\bigr) \mapsto lb, \hskip\multlinegap
	\\
\hskip\multlinegap \prod_{i\in I}\mcv\bigl(\cx_i(X_i,Y_i),\cx_i(Y_i,Z_i);\cx_i(X_i,Z_i)\bigr) 
\times \mcv\bigl(\bigl(\cx_i(X_i,Z_i)\bigr)_{i\in I};\prod_{l\in J}\ca_l\bigl((X_i)_{i\in I}f^l,(Z_i)_{i\in I}f^l\bigr)\bigr) \hfill
	\\
\rTTo^{\mu_\chi} \mcv\bigl(\bigl(\cx_i(X_i,Y_i)\bigr)_{i\in I},\bigl(\cx_i(Y_i,Z_i)\bigr)_{i\in I};\prod_{l\in J}\ca_l\bigl((X_i)_{i\in I}f^l,(Z_i)_{i\in I}f^l\bigr)\bigr),
	\\
\bigl((\kappa_{X_i,Y_i,Z_i})_{i\in I},f_{(X_i),(Z_i)}\bigr) \mapsto tr.
\end{multline*}
Fix $l\in J$ and consider the projection to $l$th factor.
The equation $tr=lb$ is equivalent to equations \(tr\centerdot\pr_l\equiv tr_l=lb_l\equiv lb\centerdot\pr_l\) between elements which are obtained from these formulas:
\begin{multline}
\prod_{i\in I}\mcv\bigl(\cx_i(X_i,Y_i),\cx_i(Y_i,Z_i);\cx_i(X_i,Z_i)\bigr)
\times\mcv\bigl(\bigl(\cx_i(X_i,Z_i)\bigr)_{i\in I};\ca_l\bigl((X_i)_{i\in I}f^l,(Z_i)_{i\in I}f^l\bigr)\bigr)
\\
\rTTo^{\mu_\chi} \mcv\bigl(\bigl(\cx_i(X_i,Y_i)\bigr)_{i\in I},\bigl(\cx_i(Y_i,Z_i)\bigr)_{i\in I};\ca_l\bigl((X_i)_{i\in I}f^l,(Z_i)_{i\in I}f^l\bigr)\bigr),
\\
\bigl((\kappa_{X_i,Y_i,Z_i})_{i\in I},f^l_{(X_i),(Z_i)}\bigr) \mapsto tr_l.
\label{eq-trl}
\end{multline}
The expression for $lb_l$ given by the left path of the following diagram is transformed using the associativity property from \figref{dia-assoc-mu-multi} written for maps \(I\sqcup I \rTTo^{\triangledown\triangledown} \mb2 \rto\id \mb2\) to the right path:
\begin{gather*}
\begin{array}{c}
	\prod_{j\in J}\!\mcv\bigl(\bigl(\cx_i(X_i,Y_i)\bigr)_{i\in I};\ca_j\bigl((X_i)_{i\in I}f^j,(Y_i)_{i\in I}f^j\bigr)\bigr)
	\\
	\times \prod_{k\in J}\!\mcv\bigl(\bigl(\cx_i(Y_i,Z_i)\bigr)_{i\in I};\ca_k\bigl((Y_i)_{i\in I}f^k,(Z_i)_{i\in I}f^k\bigr)\bigr)
	\\
\times \mcv\bigl(\ca_l\bigl((X_i)_{i\in I}f^l,(Y_i)_{i\in I}f^l\bigr),\ca_l\bigl((Y_i)_{i\in I}f^l,(Z_i)_{i\in I}f^l\bigr);
\ca_l\bigl((X_i)_{i\in I}f^l, (Z_i)_{i\in I}f^l\bigr)\bigr)
\end{array}
\\
\begin{tanglec}
\hh \id \step[1.5] \id \step[6] \id
\\
\hh \ffbox1\cong \hstep \ffbox1\cong \hstep \ffbox2{\dot\pr_l} \hstep \ffbox2{\dot\pr_l} \hstep \id \hstep
\\
\hh \id \step[1.5] \id \Step \id \step[2.5] \id \step[1.5] \id
	\\
\hh \hstep \id \step[1.5] \id \step[1.5] \ffbox5{\mu_{\mathsf{II}}}
	\\
\hh \id \step[1.5] \id \step[4] \id \Step
	\\
\hh \ffbox7{\mu_{\triangledown\triangledown}} \step[1.5]
	\\
\hh \id \step[1.5]
\end{tanglec}
\;=\;
\begin{tanglec}
\hh \id \step[1.5] \id \step[6] \id
	\\
\hh \ffbox1\cong \hstep \ffbox1\cong \hstep \ffbox2{\dot\pr_l} \hstep \ffbox2{\dot\pr_l} \hstep \id \hstep
	\\
\id \step[1.5] \XX \step[2.5] \id \step[1.5] \id
	\\
\hh \ffbox3{\mu_{\triangledown\:\!I\to1}} \hstep \ffbox4{\mu_{\triangledown\:I\to1}} \hstep \id \hstep
	\\
\hh \step \id \step[4] \id \step[2.5] \id 
	\\
\hh \hstep \ffbox8{\mu_{\triangledown\triangledown}}
	\\
\hh \hstep \id
\end{tanglec}
\\
\mcv\bigl(\bigl(\cx_i(X_i,Y_i)\bigr)_{i\in I},\bigl(\cx_i(Y_i,Z_i)\bigr)_{i\in I};\ca_l\bigl((X_i)_{i\in I}f^l, (Z_i)_{i\in I}f^l\bigr)\bigr)
\end{gather*}
On elements
\begin{diagram}[h=1.2em,nobalance]
\bigl((f^j_{(X_i),(Y_i)})_{j\in J},(f^k_{(Y_i),(Z_i)})_{k\in J},\kappa_{(X_i)_{i\in I}f^l,(Y_i)_{i\in I}f^l,(Z_i)_{i\in I}f^l}\bigr) &&\phantom{,\kappa_{(X_i)_{i\in I}f^l,(Y_i)_{i\in I}f^l,(Z_i)_{i\in I}f^l}\bigr)}
\\
\\
\dMapsTo &&\hspace*{-6em} \bigl(f^l_{(X_i),(Y_i)},f^l_{(Y_i),(Z_i)},\kappa_{(X_i)_{i\in I}f^l,(Y_i)_{i\in I}f^l,(Z_i)_{i\in I}f^l}\bigr)
\\
&\ruMapsTo &
\\
\bigl(f_{(X_i),(Y_i)},f_{(Y_i),(Z_i)},\pr_l,\pr_l,\kappa_{(X_i)_{i\in I}f^l,(Y_i)_{i\in I}f^l,(Z_i)_{i\in I}f^l}\bigr) &&
\\
\dMapsTo &&\dMapsTo
\\
	\\
\bigl(f_{(X_i),(Y_i)},f_{(Y_i),(Z_i)},(\pr_l,\pr_l)\centerdot\kappa_{(X_i)_{i\in I}f^l,(Y_i)_{i\in I}f^l,(Z_i)_{i\in I}f^l}\bigr) &\rMapsTo &lb_l
\end{diagram}

Since $f^l$ is a multi-entry $\mcv$\n-functor, the morphism $lb_l$ coincides with $tr_l$ from formula~\eqref{eq-trl}.
Therefore, $lb=tr$ and $f$ is a multi-entry $\mcv$\n-functor.
We conclude that \((\pr_j\:\ca\to\ca_j)_{j\in J}\) is a product in $\VCat$ of a family \((\ca_j)_{j\in J}\).
\end{proof}

\begin{proposition}\label{pro-VCat-has-equalizers}
Let $\mcv$ be a locally small symmetric complete multicategory.
The multicategory $\VCat$ has equalizers.
\end{proposition}

\begin{proof}
Let \(\cA\pile{\rTTo^f\\ \rTTo_g}\cb\in\VCat\) be a pair of parallel $\mcv$\n-functors.
Define a subset $\Ob\cK=\{X\in\Ob\ca\mid Xf=Xg\}$.
Denote by \(\Ob e\:\Ob\cK\to\Ob\ca\) the inclusion map.
For $X,Y\in\Ob\cK$ define an object \(\cK(X,Y)\in\mcv\) and a morphism \(e_{X,Y}\in\mcv\) via an equalizer diagram (in multicategory $\mcv$)
\[ \cK(X,Y) \rTTo^{e_{X,Y}} \ca(X,Y) \pile{\rTTo^{f_{X,Y}}\\ \rTTo_{g_{X,Y}}} \cb(Xf=Xg,Yf=Yg).
\]
This defines a $\mcv$\n-quiver $\cK$.
Let us show that the $\mcv$\n-subquiver \(\cK\subset\ca\) is a $\mcv$\n-subcategory.

Identity morphism for $X\in\Ob\cK$ is obtained via equalizer property for the empty family.
Given \(\id^\ca_X\) factorizes in a unique way as shown on the diagram from \defref{def-equalizers}
\begin{diagram}[h=1.8em,w=4em,LaTeXeqno]
&&\cK(X,X)
	\\
&\ruTTo^{\id^\cK_X} &\dTTo>{e_{X,X}}
	\\
() &\rTTo_{\id^\ca_X} &\ca(X,X) &\pile{\rTTo^{f_{X,X}}\\ \rTTo_{g_{X,X}}} &\cb(Xf,Xf)
\label{dia-idK}
\end{diagram}

The left-bottom path in the following diagram is a fork, that is, \((e_{X,Y},e_{Y,Z})\centerdot\kappa^\ca\centerdot f_{X,Z}=(e_{X,Y},e_{Y,Z})\centerdot\kappa^\ca\centerdot g_{X,Z}\),
\begin{diagram}[w=5em,LaTeXeqno]
\cK(X,Y),\cK(Y,Z) &\rTTo^{\exists!\kappa^\cK_{X,Y,Z}} &\cK(X,Z)
\\
\dTTo<{e_{X,Y},e_{Y,Z}} &= &\dTTo>{e_{X,Z}}
\\
\ca(X,Y),\ca(Y,Z) &\rTTo^{\kappa^\ca_{X,Y,Z}} &\ca(X,Z) &\pile{\rTTo^{f_{X,Z}}\\ \rTTo_{g_{X,Z}}} &\cb(Xf,Zf)
\label{dia-kappa-K}
\end{diagram}
In fact, due to \eqref{dia-AAABBB} for $f$ and $g$ the left-bottom path composes to the same parallel arrows as
\begin{equation*}
	\cK(X,Y),\cK(Y,Z) \rTTo^{e_{X,Y},e_{Y,Z}} \ca(X,Y),\ca(Y,Z) \pile{\rTTo^{f_{X,Y},f_{Y,Z}}\\ \rTTo_{g_{X,Y},g_{Y,Z}}} \cb(Xf,Yf),\cb(Yf,Zf) 
	\rto{\kappa^\cb} \cb(Xf,Zf).
\end{equation*} 
Therefore, there is a unique top arrow \(\kappa^\cK_{X,Y,Z}\) in this diagram which makes it commutative.
We take this arrow as a composition in $\cK$.
It is associative and unital since \(e_{-,-}\) are monomorphisms, more precisely, enjoy the property of \defref{def-equalizers}.
Furthermore, diagrams \eqref{dia-idK} and \eqref{dia-kappa-K} show that $e$ is a $\mcv$\n-functor (compare with \eqref{eq-id-AAA-BAFAF} and diagram~\eqref{dia-AAABBB}).
Clearly, \(e\:\cK\to\ca\) is an equalizer of \((f,g)\) as required in \defref{def-equalizers}.
\end{proof}

\subsection{Summary}
\begin{theorem}\label{thm-machine}
Let $\mcv$ be a locally small symmetric closed complete multicategory.
Then so is $\VCat$, the multicategory of small $\mcv$\n-categories and multi-entry $\mcv$\n-functors.
\end{theorem}

\begin{proof}
This is proven in Propositions \ref{pro-locally-small-symmetric-multicategory}, \ref{pro-VCat-closed}, \ref{pro-VCat-has-products} and \ref{pro-VCat-has-equalizers}.
\end{proof}

\section{First examples}
\subsection{Compositions and whiskerings}\label{sec-Compositions-whiskerings}
\begin{lemma}\label{lem-id1ev=p}
	Let \(F,G:(\ca_i)_{i\in I}\to\cc\) be multi-entry $\mcv$\n-functors.
	Then
	\begin{multline*}
\mu^\mcv_{\inj_2\:\mb1\hookrightarrow I\sqcup\mb1}\: \bigl[ \prod_{i\in I} \mcv\bigl(;\ca_i(A_i,A_i)\bigr) \bigr]
\times \mcv\bigl(\und\VCat\bigl((\ca_i)_{i\in I};\cc\bigr)(F,G);\und\VCat\bigl((\ca_i)_{i\in I};\cc\bigr)(F,G)\bigr)
\\
\times \mcv\bigl(\bigl(\ca_i(A_i,A_i)\bigr)_{i\in I},\und\VCat\bigl((\ca_i)_{i\in I};\cc\bigr)(F,G);\cc\bigl((A_i)_{i\in I}F;(A_i)_{i\in I}G\bigr)\bigr)
\\
\hfill \to \mcv\bigl(\und\VCat\bigl((\ca_i)_{i\in I};\cc\bigr)(F,G);\cc\bigl((A_i)_{i\in I}F;(A_i)_{i\in I}G\bigr)\bigr), \hskip\multlinegap
\\
\bigl( (\id_{A_i})_{i\in I},1, (\ev^\VCat_{(\ca_i)_{i\in I};\cc})_{(A_i)_{i\in I},F,(A_i)_{i\in I},G} \bigr) \mapsto p_{(A_i)_{i\in I}}.
	\end{multline*}
\end{lemma}

\begin{proof}
Let us write the left map of the following diagram using presentation~\eqref{eq-ev-VQu-E} of \(\ev^\VCat\), where $x$ is to be determined.
Applying the associativity property at \figref{dia-assoc-mu-multi} for maps \(\mb1 \rMono^{\inj_2} I\sqcup\mb1 \rto{\mathsf{\triangledown I}} \mb2\) we rewrite this as the right map in
\begin{gather*}
	\begin{array}{c}
\bigl[ \prod_{i\in I} \mcv\bigl(;\ca_i(A_i,A_i)\bigr) \bigr] \times
\mcv\bigl(\und\VCat\bigl((\ca_i)_{i\in I};\cc\bigr)(F,G);\und\VCat\bigl((\ca_i)_{i\in I};\cc\bigr)(F,G)\bigr)
	\\
	\times \mcv\bigl((\ca_i(A_i,A_i))_{i\in I};\cc((A_i)_{i\in I}F,(A_i)_{i\in I}F)\bigr)
	\\
	\times \mcv\bigl(\und\VCat\bigl((\ca_i)_{i\in I};\cc\bigr)(F,G);\cc((A_i)_{i\in I}F,(A_i)_{i\in I}G)\bigr) 
	\\
\times \mcv\bigl(\cc((A_i)_{i\in I}F,(A_i)_{i\in I}F),\cc((A_i)_{i\in I}F,(A_i)_{i\in I}G);
\cc((A_i)_{i\in I}F, (A_i)_{i\in I}G)\bigr)
\end{array}
\\
\begin{tanglec}
\hh	\id \hstep \id \step[1.5] \id \step \id \step \id 
	\\
\hh \step \id \hstep \id \hstep \ffbox4{\mu_{\mathsf{\triangledown I}\:\!\!I\sqcup\mb1\to\mb2}}
	\\
\hh \id \hstep \id \step[2.5] \id \step
	\\
\hh \ffbox5{\mu_{\inj_2\:\mb1\hookrightarrow I\sqcup\mb1}}
	\\
\hh \id
\end{tanglec}
\;=\;
\begin{tanglec}
\id \Step \XX \step \id \step \id 
	\\
\hh \ffbox3{\mu_{\varnothing\hookrightarrow I}} \step \ffbox2{\mu_{1_{\mb1}}} \hstep \id \hstep 
	\\
\hh \step \id \step[3.5] \id \step[1.5] \id 
	\\
\hh \step \ffbox6{\mu_{\mathsf{\centerdot I}\:\mb1\hookrightarrow\mb2}}
	\\
\hh \step \id
\end{tanglec}
\\
\mcv\bigl(\und\VCat\bigl((\ca_i)_{i\in I};\cc\bigr)(F,G);\cc\bigl((A_i)_{i\in I}F;(A_i)_{i\in I}G\bigr)\bigr)
\end{gather*}
On elements
\begin{diagram}
\bigl( (\id_{A_i})_{i\in I},1, F_{(A_i),(A_i)},p_{(A_i)_{i\in I}},\centerdot \bigr) &\rMapsTo &\bigl( \id_{(A_i)_{i\in I}F},p_{(A_i)_{i\in I}},\centerdot \bigr)
	\\
	\dMapsTo &&\dMapsTo
	\\
\bigl( (\id_{A_i})_{i\in I},1, (\ev^\VCat_{(\ca_i)_{i\in I};\cc})_{(A_i)_{i\in I},F,(A_i)_{i\in I},G} \bigr) &\rMapsTo &x=p_{(A_i)_{i\in I}}
\end{diagram}
	This proves the lemma.
\end{proof}

\subsubsection{Compositions}\label{sec-Compositions}
Let $\mcC$ be a closed symmetric multicategory.
As noticed in \cite[Proposition~4.10]{BesLyuMan-book} for each map \(\phi\:I\to J\) in \(\Mor\cs\) and \(X_i,Y_j,Z\in\Ob\mcC\), \(i\in I\), \(j\in J\), there exists a unique morphism
\[ \mu^{\und\mcC}_\phi\:\bigl(\und\mcC((X_i)_{i\in\phi^{-1}j};Y_j)\bigr)_{j\in J},\und\mcC\bigl((Y_j)_{j\in J};Z\bigr) \to\und\mcC\bigl((X_i)_{i\in I};Z\bigr)
\]
that makes the bottom square in diagram
\begin{diagram}[nobalance,LaTeXeqno,bottom]
(X_i)_{i\in I} &
\\
\dTTo<{(1_{X_i})_{i\in I},(\dot F^j)_{j\in J},\dot G} 
\\
(X_i)_{i\in I},\bigl(\und\mcC((X_i)_{i\in\phi^{-1}j};Y_j)\bigr)_{j\in J},\und\mcC\bigl((Y_j)_{j\in J};Z\bigr) &\rTTo^{(1_{X_i})_{i\in I},\mu^{\und\mcC}_\phi} &(X_i)_{i\in I},\und\mcC\bigl((X_i)_{i\in I};Z\bigr)
	\\
\dTTo<{(\ev^\mcC_{(X_i)_{i\in\phi^{-1}j};Y_j})_{j\in J},1^\mcC_{\und\mcC((Y_j)_{j\in J};Z)}} &= &\dTTo>{\ev^\mcC_{(X_i)_{i\in I};Z}}
	\\
(Y_j)_{j\in J},\und\mcC\bigl((Y_j)_{j\in J};Z\bigr) & \rTTo^{\ev^\mcC_{(Y_j)_{j\in J};Z}} &Z
	\label{equ-mu-undC-def}
\end{diagram}
commute.
Here \(F^j\:(X_i)_{i\in\phi^{-1}j}\to Y_j\), $j\in J$, \(G\:(Y_j)_{j\in J}\to Z\) are morphisms in $\mcC$.
This composition law turns \(\und\mcC\) into a $\mcC$\n-multicategory.
As a corollary \(ttr=tlb\), where
\begin{multline*}
\bigl[\prod_{i\in I}\mcC(X_i;X_i)\bigr] \times \bigl[\prod_{i\in I}\mcC(X_i;X_i)\bigr] \times \bigl[\prod_{j\in J}\mcC\bigl(;\und\mcC\bigl((X_i)_{i\in\phi^{-1}j};Y_j\bigr)\bigr)\bigr] \times \mcC\bigl(;\und\mcC\bigl((Y_j)_{j\in J};Z\bigr)\bigr)
\\
\times \mcC\bigl(\bigl(\und\mcC((X_i)_{i\in\phi^{-1}j};Y_j)\bigr)_{j\in J},\und\mcC\bigl((Y_j)_{j\in J};Z\bigr);\und\mcC\bigl((X_i)_{i\in I};Z\bigr)\bigr) \times \mcC\bigl((X_i)_{i\in I},\und\mcC\bigl((X_i)_{i\in I};Z\bigr);Z\bigr)
\\
\rTTo^{(\prod_I\mu_{\id_{\mb1}})\times\mu_{\varnothing\to\mb2}\times1}
\bigl[\prod_{i\in I}\mcC(X_i;X_i)\bigr] \times \mcC\bigl(;\und\mcC((X_i)_{i\in I};Z)\bigr) \times \mcC\bigl((X_i)_{i\in I},\und\mcC\bigl((X_i)_{i\in I};Z\bigr);Z\bigr)
\\
\hfill \rTTo^{\mu_{\inj_1\:I\to I\sqcup\mb1}} \mcC\bigl((X_i)_{i\in I};Z\bigr), \hskip\multlinegap
\\
\bigl((1_{X_i})_{i\in I},(1_{X_i})_{i\in I},(\dot F^j)_{j\in J},\dot G,\mu^{\und\mcC}_\phi,\ev^\mcC_{(X_i)_{i\in I};Z}\bigr) \mapsto \bigl((1_{X_i})_{i\in I},[(\dot F^j)_{j\in J},\dot G]\mu^{\und\mcC}_\phi,\ev^\mcC_{(X_i)_{i\in I};Z}\bigr)
\\
\mapsto \{[(\dot F^j)_{j\in J},\dot G]\mu^{\und\mcC}_\phi\}\varphi_{(X_i)_{i\in I};;Z} =ttr.
\end{multline*}
At the last step \eqref{eq-mu-in1I-I1} is used.
Also
\begin{multline*}
\bigl[\prod_{i\in I}\mcC(X_i;X_i)\bigr] \times \bigl[\prod_{j\in J}\mcC\bigl(;\und\mcC\bigl((X_i)_{i\in\phi^{-1}j};Y_j\bigr)\bigr)\bigr] \times \bigl[\prod_{j\in J}\mcC\bigl((X_i)_{i\in\phi^{-1}j},\und\mcC\bigl((X_i)_{i\in\phi^{-1}j};Y_j\bigr);Y_j\bigr)\bigr]
\\
\times \mcC\bigl(;\und\mcC\bigl((Y_j)_{j\in J};Z\bigr)\bigr) \times \mcC\bigl(\und\mcC\bigl((Y_j)_{j\in J};Z\bigr);\und\mcC\bigl((Y_j)_{j\in J};Z\bigr)\bigr) \times \mcC\bigl((Y_j)_{j\in J},\und\mcC\bigl((Y_j)_{j\in J};Z\bigr);Z\bigr)
\\
\rTTo^{(\prod_{j\in J}\mu_{\inj_1\:\phi^{-1}j\hookrightarrow\phi^{-1}j\sqcup\mb1})\times\mu_{\varnothing\to\mb1}\times1}
\bigl[\prod_{j\in J}\mcC\bigl((X_i)_{i\in\phi^{-1}j};Y_j\bigr)\bigr] \times \mcC\bigl(;\und\mcC\bigl((Y_j)_{j\in J};Z\bigr)\bigr)
\\
\hfill \times \mcC\bigl((Y_j)_{j\in J},\und\mcC\bigl((Y_j)_{j\in J};Z\bigr);Z\bigr) \rTTo^{\mu_{\phi\centerdot\inj_1\:I\to J\sqcup\mb1}} \mcC\bigl((X_i)_{i\in I};Z\bigr), \hskip\multlinegap
\\
\hskip\multlinegap \bigl((1_{X_i})_{i\in I},(\dot F^j)_{j\in J},(\ev^\mcC_{(X_i)_{i\in\phi^{-1}j};Y_j})_{j\in J},\dot G,1_{\und\mcC((Y_j)_{j\in J};Z)},\ev^\mcC_{(Y_j)_{j\in J};Z}\bigr) \hfill
\\
\mapsto \bigl((F^j)_{j\in J},\dot G,\ev^\mcC_{(Y_j)_{j\in J};Z}\bigr) \mapsto (F^j)_{j\in J}\centerdot_\phi G =tlb.
\end{multline*}
\eqref{eq-mu-in1I-I1} is used again twice.
Hence, \([(\dot F^j)_{j\in J},\dot G]\mu^{\und\mcC}_\phi=[(F^j)_{j\in J}\centerdot_\phi G]\dot{\;}\).

In particular, we can apply this discussion to the multicategory \(\mcC=\VCat\).
We deduce that on objects the $\mcv$\n-functor \(\mu^{\und\VCat}_\phi\) gives map \(\Ob\mu^{\und\VCat}_\phi\:\)
\[ \bigl[\prod_{j\in J} \Ob\und\VCat\bigl((\ca_i)_{i\in\phi^{-1}j};\cb_j\bigr)\bigr] \times \Ob\und\VCat\bigl((\cb_j)_{j\in J};\cc\bigr) \to \Ob\und\VCat\bigl((\ca_i)_{i\in I};\cc\bigr)
\]
which coincides with
\begin{equation}
\mu^\VCat_\phi\: \bigl[\prod_{j\in J} \VCat\bigl((\ca_i)_{i\in\phi^{-1}j};\cb_j\bigr)\bigr] \times \VCat\bigl((\cb_j)_{j\in J};\cc\bigr) \to \VCat\bigl((\ca_i)_{i\in I};\cc\bigr).
\label{eq-Ob-mu-VCat-phi=mu-VCat-phi}
\end{equation}

Let us study multi-entry $\mcv$\n-functor
\[ \mu^{\und\VCat}_\phi\: \bigl(\und\VCat\bigl((\ca_i)_{i\in\phi^{-1}j};\cb_j\bigr)\bigr)_{j\in J}, \und\VCat\bigl((\cb_j)_{j\in J};\cc\bigr) \to \und\VCat\bigl((\ca_i)_{i\in I};\cc\bigr).
\]

\subsubsection{Left whiskering}
Let \(F^j\:(\ca_i)_{i\in\phi^{-1}j}\to\cb_j\), $j\in J$, be multi-entry $\mcv$\n-functors.
Consider the left whiskering $\mcv$\n-functor
\begin{multline*}
LW =\bigl[\und\VCat\bigl((\cb_j)_{j\in J};\cc\bigr) \rTTo^{(\dot F^j)_{j\in J},1} \bigl(\und\VCat\bigl((\ca_i)_{i\in\phi^{-1}j};\cb_j\bigr)\bigr)_{j\in J}, \und\VCat\bigl((\cb_j)_{j\in J};\cc\bigr)
\\
\rTTo^{\mu^{\und\VCat}_\phi} \und\VCat\bigl((\ca_i)_{i\in I};\cc\bigr)\bigr].
\end{multline*}
On objects it takes \(G\:(\cb_j)_{j\in J}\to\cc\) to \((F^j)_{j\in J}\centerdot_\phi G\) as we have seen.
As a consequence of the bottom square of \eqref{equ-mu-undC-def} there is a commutative square in $\VCat$
\begin{diagram}[nobalance]
	(\ca_i)_{i\in I},\und\VCat\bigl((\cb_j)_{j\in J};\cc\bigr) &\rTTo^{(1_{\ca_i})_{i\in I},LW} &(\ca_i)_{i\in I},\und\VCat\bigl((\ca_i)_{i\in I};\cc\bigr)
	\\
	\dTTo<{(F^j)_{j\in J},1} &= &\dTTo>{\ev^\VCat_{(\ca_i)_{i\in I};\cc}}
	\\
	(\cb_j)_{j\in J},\und\VCat\bigl((\cb_j)_{j\in J};\cc\bigr) &\rTTo^{\ev^\VCat_{(\cb_j)_{j\in J};\cc}} &\cc
\end{diagram}
Hence, the morphism $LW\in\VCat$ is adjunct to the multi-entry $\mcv$\n-functor
\[ LW^\dagger =\bigl[ (\ca_i)_{i\in I},\und\VCat\bigl((\cb_j)_{j\in J};\cc\bigr) \rTTo^{(F^j)_{j\in J},1} (\cb_j)_{j\in J},\und\VCat\bigl((\cb_j)_{j\in J};\cc\bigr) \rTTo^{\ev^\VCat_{(\cb_j)_{j\in J};\cc}} \cc \bigr].
\]
More precisely,
\begin{multline*}
\mu_{\phi\sqcup1\:I\sqcup\mb1\to J\sqcup\mb1}\: \bigl[\prod_{j\in J} \VCat\bigl((\ca_i)_{i\in\phi^{-1}j};\cb_j\bigr)\bigr]
\times \VCat\bigl(\und\VCat\bigl((\cb_j)_{j\in J};\cc\bigr);\und\VCat\bigl((\cb_j)_{j\in J};\cc\bigr)\bigr) 
\\
\hfill \times \VCat\bigl((\cb_j)_{j\in J},\und\VCat\bigl((\cb_j)_{j\in J};\cc\bigr);\cc\bigr)
\to \VCat\bigl((\ca_i)_{i\in I},\und\VCat\bigl((\cb_j)_{j\in J};\cc\bigr);\cc\bigr) \hskip\multlinegap
\\
\bigl( (F^j)_{j\in J},1,\ev^\VCat_{(\cb_j)_{j\in J};\cc} \bigr) \mapsto LW^\dagger.
\end{multline*}

\begin{proposition}
	On morphisms
\[ LW\: \und\VCat\bigl((\cb_j)_{j\in J};\cc\bigr)(G,H) \to \und\VCat\bigl((\ca_i)_{i\in I};\cc\bigr)\bigl((F^j)_{j\in J}\centerdot_\phi G,(F^j)_{j\in J}\centerdot_\phi H\bigr)
\]
coincides with the morphism between ends, the top morphism in commutative square
	\begin{diagram}[LaTeXeqno,nobalance]
\int_{(B_j\in\cb_j)_{j\in J}}\hspace*{-3em} \cc\bigl((B_j)_{j\in J}G,(B_j)_{j\in J}H\bigr) &\rTTo &\int_{(A_i\in\ca_i)_{i\in I}}\hspace*{-3em} \cc\bigl((A_i)_{i\in I}\bigl((F^j)_{j\in J}\centerdot_\phi G\bigr),(A_i)_{i\in I}\bigl((F^j)_{j\in J}\centerdot_\phi H\bigr)\bigr)
		\\
\dTTo<{p_{((A_i)_{i\in\phi^{-1}j}F^j)_{j\in J}}} 
		\\
\cc\bigl(((A_i)_{i\in\phi^{-1}j}F^j)_{j\in J}G,((A_i)_{i\in\phi^{-1}j}F^j)_{j\in J}H\bigr) \hspace*{-6em} &&\dTTo>{p_{(A_i)_{i\in I}}}
\\
&\rdEq
\\
&&\cc\bigl((A_i)_{i\in I}\bigl((F^j)_{j\in J}\centerdot_\phi G\bigr),(A_i)_{i\in I}\bigl((F^j)_{j\in J}\centerdot_\phi H\bigr)\bigr)
\label{dia-p=LWp}
	\end{diagram}
\end{proposition}

\begin{proof}
	Let us find $LW=\Psi(LW^\dagger)$ from \eqref{eq-(Ai)fbar} with $J=\mb1$, \(g=LW^\dagger\).
	The map on objects \(LW\:G\mapsto(G)f\) is obtained from \eqref{eq-Ob-mu-VCat-phi=mu-VCat-phi} or directly from \eqref{eq-Bf-A-AB-C-VCat}.
	We have for all \(G\in\VCat\bigl((\cb_j)_{j\in J};\cc\bigr)\) a multi-entry $\mcv$\n-functor
	\begin{align*}
(G)f &=\bigl[ (\ca_i)_{i\in I} \rTTo^{(\Id)_I,\dot G} (\ca_i)_{i\in I},\und\VCat\bigl((\cb_j)_{j\in J};\cc\bigr)
\rTTo^{(F^j)_{j\in J},1} (\cb_j)_{j\in J},\und\VCat\bigl((\cb_j)_{j\in J};\cc\bigr)
\\
&\hspace*{28em} \rTTo^{\ev^\VCat_{(\cb_j)_{j\in J};\cc}} \cc \bigr]
\\
	&=\bigl[ (\ca_i)_{i\in I} \rTTo^{(F^j)_{j\in J}} (\cb_j)_{j\in J} \rTTo^{(\Id)_I,\dot G} 
	(\cb_j)_{j\in J},\und\VCat\bigl((\cb_j)_{j\in J};\cc\bigr) \rTTo^{\ev^\VCat_{(\cb_j)_{j\in J};\cc}} \cc \bigr]
	\\
	&=\bigl[ (\ca_i)_{i\in I} \rTTo^{(F^j)_{j\in J}} (\cb_j)_{j\in J} \rto G \cc \bigr] =(F^j)_{j\in J}\centerdot_\phi G
	\end{align*}
	due to \eqref{eq-mu-in1I-I1}.
$LW=\Psi(LW^\dagger)$ on morphisms is found from the left map in the following diagram.
Applying the associativity property from \figref{dia-assoc-mu-multi} for maps \(\mb1 \rMono^{\inj_2} I\sqcup\mb1 \rTTo^{\phi\sqcup 1} J\sqcup\mb1\) we get the right map in
\begin{gather*}
\begin{array}{c}
\bigl[ \prod_{i\in I} \mcv\bigl(;\ca_i(A_i,A_i)\bigr) \bigr] \times
\mcv\bigl(\und\VCat\bigl((\cb_j)_{j\in J};\cc\bigr)(G,H);\und\VCat\bigl((\cb_j)_{j\in J};\cc\bigr)(G,H)\bigr)
	\\
\times \bigl[ \prod_{j\in J} \mcv\bigl(\bigl(\ca_i(A_i,A_i)\bigr)_{i\in\phi^{-1}j};
\cb_j\bigl((A_i)_{i\in\phi^{-1}j}F^j,(A_i)_{i\in\phi^{-1}j}F^j\bigr)\bigr) \bigr]
	\\
\times \mcv\bigl(\und\VCat\bigl((\cb_j)_{j\in J};\cc\bigr)(G,H);\und\VCat\bigl((\cb_j)_{j\in J};\cc\bigr)(G,H)\bigr)
	\\
\times \mcv\bigl(\bigl(\cb_j\bigl((A_i)_{i\in\phi^{-1}j}F^j,(A_i)_{i\in\phi^{-1}j}F^j\bigr)\bigr)_{j\in J},
\und\VCat\bigl((\cb_j)_{j\in J};\cc\bigr)(G,H); \hfill
	\\
\hfill \cc\bigl(((A_i)_{i\in\phi^{-1}j}F^j)_{j\in J}G,((A_i)_{i\in\phi^{-1}j}F^j)_{j\in J}H\bigr)\bigr)
\end{array}
\\
\begin{tanglec}
\hh	\id \hstep \id \step[1.5] \id \Step \id \Step \id 
	\\
\step \id \hstep \id \hstep \ffbox6{\mu_{\phi\sqcup1\:I\sqcup\mb1\to J\sqcup\mb1}}
	\\
\hh \id \hstep \id \step[3.5] \id \Step
	\\
\ffbox5{\mu_{\inj_2\:\mb1\hookrightarrow I\sqcup\mb1}} \Step
	\\
\hh \id \Step
\end{tanglec}
\;=\;
\begin{tanglec}
\id \step[4.5] \XX \step \id \step \id 
	\\
\ffbox6{\prod_{j\in J}\mu_{\varnothing\to\phi^{-1}j}} \step \ffbox2{\mu_{1_{\mb1}}} \hstep \id \step
	\\
\hh \Step \id \step[5] \id \step[1.5] \id 
	\\
\hh \step[1.5] \ffbox8{\mu_{\inj_2\:\mb1\hookrightarrow J\sqcup\mb1}}
	\\
\hh \step[1.5] \id
\end{tanglec}
\\
\mcv\bigl(\und\VCat\bigl((\cb_j)_{j\in J};\cc\bigr)(G,H);
\cc\bigl(((A_i)_{i\in\phi^{-1}j}F^j)_{j\in J}G,((A_i)_{i\in\phi^{-1}j}F^j)_{j\in J}H\bigr)\bigr)
\end{gather*}
On elements
\begin{diagram}
\begin{array}{l}
\bigl( (\id_{A_i})_{i\in I},1,(F^j_{(A_i)_{i\in\phi^{-1}j},(A_i)_{i\in\phi^{-1}j}})_{j\in J},1,
\\
(\ev^\VCat_{(\cb_j)_{j\in J};\cc})_{((A_i)_{i\in\phi^{-1}j}F^j)_{j\in J},G,((A_i)_{i\in\phi^{-1}j}F^j)_{j\in J},H}\bigr) 
\end{array}
\hspace*{-2em} &&\phantom{,G,((A_i)_{i\in\phi^{-1}j}F^j)_{j\in J},H\bigr)}
\\
\dMapsTo &\rdMapsTo &
\\
\bigl( (\id_{A_i})_{i\in I},1, LW^\dagger_{(A_i)_{i\in I},G,(A_i)_{i\in I},H}\bigr) &&\hspace*{-2em}
\begin{array}{l}
\bigl( (\id_{(A_i)_{i\in\phi^{-1}j}F^j})_{j\in J},1, 
\\
(\ev^\VCat_{(\cb_j)_{j\in J};\cc})_{((A_i)_{i\in\phi^{-1}j}F^j)_{j\in J},G,((A_i)_{i\in\phi^{-1}j}F^j)_{j\in J},H}\bigr)
\end{array}
\\
&\rdMapsTo &\dMapsTo
	\\
&&p_{((A_i)_{i\in\phi^{-1}j}F^j)_{j\in J}} =LW\centerdot p_{(A_i)_{i\in I}}
\end{diagram}
	Here we have used \lemref{lem-id1ev=p}.
	Therefore, the map $LW$ placed on the top of diagram~\eqref{dia-p=LWp} makes it commutative.
	Since there is no more than one such map, this proves the statement.
\end{proof}

\begin{corollary}
	The map of natural transformations, the top arrow in the commutative square
	\begin{diagram}[h=1.2em,nobalance]
\VCat\bigl((\cb_j)_{j\in J};\cc\bigr)(G,H) &\rTTo &\VCat\bigl((\ca_i)_{i\in I};\cc\bigr)\bigl((F^j)_{j\in J}\centerdot_\phi G,(F^j)_{j\in J}\centerdot_\phi H\bigr)
		\\
\dTTo<\cong
		\\
&&\dTTo>\cong
		\\
\mcv\Bigl(;\int_{(B_j\in\cb_j)_{j\in J}} \hspace*{-3.5em} \cc\bigl((B_j)_{j\in J}G,(B_j)_{j\in J}H\bigr)\Bigr)
\\
&\rdTTo^{\mcv(;LW)}
\\
&&\hspace*{-3.7em} \mcv\Bigl(;\int_{(A_i\in\ca_i)_{i\in I}} \hspace*{-3em} \cc\bigl(\bigl((A_i)_{i\in\phi^{-1}j}F^j\bigr)_{j\in J}G,\bigl((A_i)_{i\in\phi^{-1}j}F^j\bigr)_{j\in J}H\bigr)\Bigr)
	\end{diagram}
takes a natural transformation \(\lambda=(\lambda_{(B_j)_{j\in J}})\:G\to H\:(\cb_j)_{j\in J}\to\cc\) with the components \(\lambda_{(B_j)_{j\in J}}\in\mcv\bigl(;\cc\bigl((B_j)_{j\in J}G,(B_j)_{j\in J}H\bigr)\bigr)\) to \(\nu=(\nu_{(A_i\in\ca_i)_{i\in I}})\:(F^j)_{j\in J}\centerdot_\phi G\to(F^j)_{j\in J}\centerdot_\phi H\:(\ca_i)_{i\in I}\to\cc\), where
\[ \nu_{(A_i\in\ca_i)_{i\in I}}=\lambda_{((A_i)_{i\in\phi^{-1}j}F^j)_{j\in J}}\in\mcv\bigl(;\cc\bigl(((A_i)_{i\in\phi^{-1}j}F^j)_{j\in J}G,((A_i)_{i\in\phi^{-1}j}F^j)_{j\in J}H\bigr)\bigr).
\]
\end{corollary}

\begin{proof}
	Follows from the above statement and \propref{pro-natural-V-transformations}.
\end{proof}

\subsubsection{Right whiskering}
Let \(H\:(\cb_j)_{j\in J}\to\cc\) be a multi-entry $\mcv$\n-functor.
Consider the right whiskering $\mcv$\n-functor
\begin{multline*}
RW =\bigl[\bigl(\und\VCat\bigl((\ca_i)_{i\in\phi^{-1}j};\cb_j\bigr)\bigr)_{j\in J}
\\
\rTTo^{(1)_{j\in J},\dot H} \bigl(\und\VCat\bigl((\ca_i)_{i\in\phi^{-1}j};\cb_j\bigr)\bigr)_{j\in J}, \und\VCat\bigl((\cb_j)_{j\in J};\cc\bigr)
\rTTo^{\mu^{\und\VCat}_\phi} \und\VCat\bigl((\ca_i)_{i\in I};\cc\bigr)\bigr].
\end{multline*}
On objects it takes \(\bigl(F^j\:(\ca_i)_{i\in\phi^{-1}j}\to\cb_j\bigr)_{j\in J}\) to \((F^j)_{j\in J}\centerdot_\phi H\) as we have seen.
As a consequence of the bottom square of \eqref{equ-mu-undC-def} there is a commutative square in $\VCat$
\begin{diagram}[nobalance]
	(\ca_i)_{i\in I},\bigl(\und\VCat\bigl((\ca_i)_{i\in\phi^{-1}j};\cb_j\bigr)\bigr)_{j\in J} &\rTTo^{(1_{\ca_i})_{i\in I},RW} &(\ca_i)_{i\in I},\und\VCat\bigl((\ca_i)_{i\in I};\cc\bigr)
	\\
	\dTTo<{(\ev^\VCat_{(\ca_i)_{i\in\phi^{-1}j};\cb_j})_{j\in J}} &= &\dTTo>{\ev^\VCat_{(\ca_i)_{i\in I};\cc}}
	\\
	(\cb_j)_{j\in J} &\rTTo^H &\cc
\end{diagram}
Hence, the morphism $RW\in\VCat$ is adjunct to the multi-entry $\mcv$\n-functor
\[ RW^\dagger =\bigl[ (\ca_i)_{i\in I},\bigl(\und\VCat\bigl((\ca_i)_{i\in\phi^{-1}j};\cb_j\bigr)\bigr)_{j\in J} \rTTo^{(\ev^\VCat_{(\ca_i)_{i\in\phi^{-1}j};\cb_j})_{j\in J}} (\cb_j)_{j\in J} \rTTo^H \cc \bigr].
\]
More precisely,
\begin{multline*}
\mu_{(\phi,\id_J)\:I\sqcup J\to J}\: \bigl[\prod_{j\in J} \VCat\bigl((\ca_i)_{i\in\phi^{-1}j},\und\VCat\bigl((\ca_i)_{i\in\phi^{-1}j};\cb_j\bigr);\cb_j\bigr)\bigr] \times \VCat\bigl((\cb_j)_{j\in J};\cc\bigr)
\\
\hfill \to \VCat\bigl((\ca_i)_{i\in I},\bigl(\und\VCat\bigl((\ca_i)_{i\in\phi^{-1}j};\cb_j\bigr)\bigr)_{j\in J};\cc\bigr) \hskip\multlinegap
\\
\bigl( \bigl(\ev^\VCat_{(\ca_i)_{i\in\phi^{-1}j};\cb_j}\bigr)_{j\in J},H \bigr) \mapsto RW^\dagger.
\end{multline*}

\begin{proposition}
	On morphisms
\[ RW\colon\! \bigl(\und\VCat\bigl((\ca_i)_{i\in\phi^{-1}j};\cb_j\bigr)(F^j,G^j)\bigr)_{j\in J}\! \to \und\VCat\bigl((\ca_i)_{i\in I};\cc\bigr)\bigl((F^j)_{j\in J}\centerdot_\phi H,(G^j)_{j\in J}\centerdot_\phi H\bigr)
\]
	coincides with the morphism between ends, the top morphism in
\begin{diagram}[h=2.3em,LaTeXeqno]
\Bigl(\int_{(A_i\in\ca_i)_{i\in\phi^{-1}j}} \hspace*{-5em} \cb_j\bigl((A_i)_{i\in\phi^{-1}j}F^j,(A_i)_{i\in\phi^{-1}j}G^j\bigr)\Bigr)_{j\in J} \hspace*{-0.5em} \to \hspace*{-0.5em} \int_{(A_i\in\ca_i)_{i\in I}} \hspace*{-3.7em} \cc\bigl((A_i)_{i\in I}\bigl((F^j)_{j\in J}\centerdot_\phi H\bigr),(A_i)_{i\in I}\bigl((G^j)_{j\in J}\centerdot_\phi H\bigr)\bigr)
\\
	\dTTo<{(p_{(A_i)_{i\in\phi^{-1}j}})_{j\in J}} \hspace*{18.3em} \dTTo>{p_{(A_i)_{i\in I}}}
\\
	\bigl(\cb_j\bigl((A_i)_{i\in\phi^{-1}j}F^j,(A_i)_{i\in\phi^{-1}j}G^j\bigr)\bigr)_{j\in J} \rto H \cc\bigl((A_i)_{i\in I}\bigl((F^j)_{j\in J}\centerdot_\phi H\bigr),(A_i)_{i\in I}\bigl((G^j)_{j\in J}\centerdot_\phi H\bigr)\bigr)
\label{dia-question-RW}
\end{diagram}
\end{proposition}

\begin{proof}
	Let us find $RW=\Psi(RW^\dagger)$ from \eqref{eq-(Ai)fbar} with \(g=RW^\dagger\).
	The map on objects \(RW\:(F^j)_{j\in J}\mapsto(F^j)_{j\in J}f\) is obtained from \eqref{eq-Ob-mu-VCat-phi=mu-VCat-phi} or directly from \eqref{eq-Bf-A-AB-C-VCat}.
	We have a multi-entry $\mcv$\n-functor
	\begin{align*}
(F^j)_{j\in J}f &=\bigl[ (\ca_i)_{i\in I} \rTTo^{(\Id)_I,(\dot F^j)_{j\in J}} (\ca_i)_{i\in I},\bigl(\und\VCat\bigl((\ca_i)_{i\in\phi^{-1}j};\cb_j\bigr)\bigr)_{j\in J}
\\
&\hspace*{14em} \rTTo^{(\ev^\VCat_{(\ca_i)_{i\in\phi^{-1}j};\cb_j})_{j\in J}} (\cb_j)_{j\in J} \rto H \cc \bigr]
\\
&=\bigl[ (\ca_i)_{i\in I} \rTTo^{(F^j)_{j\in J}} (\cb_j)_{j\in J} \rto H \cc \bigr].
	\end{align*}
$RW=\Psi(RW^\dagger)$ on morphisms is found from the left map of the following diagram.
Applying the associativity property from \figref{dia-assoc-mu-multi} for maps \(J \rTTo^{\inj_2} I\sqcup J \rTTo^{(\phi,1_J)} J\) we get the right map in
\begin{gather*}
\begin{array}{c}
\hspace*{-0.5em} \bigl[ \prod_{i\in I} \mcv\bigl(;\ca_i(A_i,A_i)\bigr) \bigr]
\times \bigl[ \prod_{j\in J} \mcv\bigl(\und\VCat\bigl((\ca_i)_{i\in\phi^{-1}j};\cb_j\bigr)(F^j,G^j);
\und\VCat\bigl((\ca_i)_{i\in\phi^{-1}j};\cb_j\bigr)(F^j,G^j)\bigr) \bigr] 
	\\
\hspace*{-0.5em} \times \bigl[ \prod_{j\in J} \mcv\bigl(\bigl(\ca_i(A_i,A_i)\bigr)_{i\in\phi^{-1}j},\und\VCat\bigl((\ca_i)_{i\in\phi^{-1}j};\cb_j\bigr)(F^j,G^j);
\cb_j\bigl((A_i)_{i\in\phi^{-1}j}F^j,(A_i)_{i\in\phi^{-1}j}G^j\bigr)\bigr) \bigr] 
	\\
\times \mcv\bigl(\bigl(\cb_j\bigl((A_i)_{i\in\phi^{-1}j}F^j,(A_i)_{i\in\phi^{-1}j}G^j\bigr)\bigr)_{j\in J};
\cc\bigl(((A_i)_{i\in\phi^{-1}j}F^j)_{j\in J}H,((A_i)_{i\in\phi^{-1}j}G^j)_{j\in J}H\bigr)\bigr)
\end{array}
\\
\begin{tanglec}
\hh	\id \hstep \id \step[1.5] \id \step[4] \id 
	\\
\hh \step \id \hstep \id \hstep \ffbox6{\mu_{(\phi,1_J)\:I\sqcup J\to J}}
	\\
\hh \id \hstep \id \step[3.5] \id \Step
	\\
\hh \ffbox5{\mu_{\inj_2\:J\hookrightarrow I\sqcup J}} \Step
	\\
\hh \id \Step
\end{tanglec}
\;=\;
\begin{tanglec}
\hh \id \step[3] \id \step[3] \id \Step \id 
	\\
\ffbox9{\prod_{j\in J}\mu_{\inj_2\:\mb1\hookrightarrow\phi^{-1}j\sqcup\mb1}} \hstep \id \step[1.5]
	\\
\hh \step[3] \id \step[5] \id 
	\\
\hh \step[3] \ffbox6{\mu_{1_J}}
	\\
\hh \step[3] \id
\end{tanglec}
\\
\mcv\bigl(\bigl(\und\VCat\bigl((\ca_i)_{i\in\phi^{-1}j};\cb_j\bigr)(F^j,G^j)\bigr)_{j\in J};
\cc\bigl(((A_i)_{i\in\phi^{-1}j}F^j)_{j\in J}H,((A_i)_{i\in\phi^{-1}j}G^j)_{j\in J}H\bigr)\bigr)
\end{gather*}
On elements
\[
\begin{diagram}[inline]
	\begin{array}{l}
\bigl( (\id_{A_i})_{i\in I},(1)_J,
\\
\bigl((\ev^\VCat_{(\ca_i)_{i\in\phi^{-1}j};\cb_j})_{(A_i)_{i\in\phi^{-1}j},F^j,(A_i)_{i\in\phi^{-1}j},G^j}\bigr)_{j\in J}, H\bigr)	\end{array}
	&\rMapsTo &
\bigl( (p_{(A_i)_{i\in\phi^{-1}j}})_{j\in J},H \bigr)
	\\
	\dMapsTo &&\dMapsTo
	\\
\bigl( (\id_{A_i})_{i\in I},(1)_J, RW^\dagger_{(A_i)_{i\in I},(F^j),(A_i)_{i\in I},(G^j)}\bigr) &\rMapsTo &RW\centerdot p_{(A_i)_{i\in I}}=(p_{(A_i)_{i\in\phi^{-1}j}})_{j\in J}\centerdot H
\end{diagram}
\]
where we use \lemref{lem-id1ev=p}.
	Therefore, the map $RW$ placed on the top of diagram~\eqref{dia-question-RW} makes it commutative.
	Since there is no more than one such map, the proposition is proved.
\end{proof}

\begin{corollary}
	The map of natural transformations
	\begin{multline*}
\prod_{j\in J} \VCat\bigl((\ca_i)_{i\in\phi^{-1}j};\cb_j\bigr)(F^j,G^j)
\rTTo^{(\prod_J\cong)\times\dot{RW}} \prod_{j\in J}
\mcv\Bigl(;\int_{(A_i\in\ca_i)_{i\in\phi^{-1}j}} \hspace*{-5em} \cb_j\bigl((A_i)_{i\in\phi^{-1}j}F^j,(A_i)_{i\in\phi^{-1}j}G^j\bigr)\Bigr) \times
\\
\mcv\Bigl(\Bigl(\int_{(A_i\in\ca_i)_{i\in\phi^{-1}j}} \hspace*{-5em} \cb_j\bigl((A_i)_{i\in\phi^{-1}j}F^j,(A_i)_{i\in\phi^{-1}j}G^j\bigr)\Bigr)_{j\in J};\!\!\int_{(A_i\in\ca_i)_{i\in I}} \hspace*{-3.8em} \cc\bigl(((A_i)_{i\in\phi^{-1}j}F^j)_{j\in J}H,((A_i)_{i\in\phi^{-1}j}G^j)_{j\in J}H\bigr)\Bigr)
\\
\rTTo^{\mu_{\varnothing\to J}} \mcv\Bigl(;\int_{(A_i\in\ca_i)_{i\in I}} \hspace*{-1em} \cc\bigl(((A_i)_{i\in\phi^{-1}j}F^j)_{j\in J}H,((A_i)_{i\in\phi^{-1}j}G^j)_{j\in J}H\bigr)\Bigr)
\\
\rTTo^\cong \VCat\bigl((\ca_i)_{i\in I};\cc\bigr)\bigl((F^j)_{j\in J}\centerdot_\phi H,(G^j)_{j\in J}\centerdot_\phi H\bigr)
	\end{multline*}
takes a tuple of natural transformations \(\bigl(\lambda^j\:F^j\to G^j\:(\ca_i)_{i\in\phi^{-1}j}\to\cb_j\bigr)_{j\in J}\) with the components \(\lambda^j_{(A_i)_{i\in\phi^{-1}j}}\in\mcv\bigl(;\cb_j\bigl((A_i)_{i\in\phi^{-1}j}F^j,(A_i)_{i\in\phi^{-1}j}G^j\bigr)\bigr)\) to \(\nu=(\nu_{(A_i\in\ca_i)_{i\in I}})\:(F^j)_{j\in J}\centerdot_\phi H\to(G^j)_{j\in J}\centerdot_\phi H\:(\ca_i)_{i\in I}\to\cc\), 
\[ \nu_{(A_i\in\ca_i)_{i\in I}}\in\mcv\bigl(;\cc\bigl(((A_i)_{i\in\phi^{-1}j}F^j)_{j\in J}H,((A_i)_{i\in\phi^{-1}j}G^j)_{j\in J}H\bigr)\bigr),
\]
where \(\nu_{(A_i\in\ca_i)_{i\in I}}=(\lambda^j_{(A_i)_{i\in\phi^{-1}j}})_{j\in J}\centerdot_\phi H_{((A_i)_{i\in\phi^{-1}j}F^j)_{j\in J},((A_i)_{i\in\phi^{-1}j}G^j)_{j\in J}}\).
\end{corollary}

\begin{proof}
	Follows from the above statement and \propref{pro-natural-V-transformations}.
	In fact, the composition in the top--right path in
	\[ \hspace*{-3em}
	\begin{diagram}[h=1em,inline]
(\;) &&\hspace*{-1em} \int_{(A_i\in\ca_i)_{i\in I}} \hspace*{-3.4em} \cc\bigl((A_i)_{i\in I}\bigl((F^j)_{j\in J}\centerdot_\phi H\bigr),(A_i)_{i\in I}\bigl((G^j)_{j\in J}\centerdot_\phi H\bigr)\bigr)
	\\
&\ruTTo(2,4)_{RW} &
	\\
	\dTTo<{(\lambda^j)_{j\in J}}
	\\
	\\
\Bigl(\int_{(A_i\in\ca_i)_{i\in\phi^{-1}j}} \hspace*{-4.5em} \cb_j\bigl((A_i)_{i\in\phi^{-1}j}F^j,(A_i)_{i\in\phi^{-1}j}G^j\bigr)\Bigr)_{j\in J} \hspace*{-2em} &&\dTTo>{p_{(A_i)_{i\in I}}}
	\\
	\\
	\dTTo<{(p_{(A_i)_{i\in\phi^{-1}j}})_{j\in J}} &= &
	\\
	\\
\bigl(\cb_j\bigl((A_i)_{i\in\phi^{-1}j}F^j,(A_i)_{i\in\phi^{-1}j}G^j\bigr)\bigr)_{j\in J} \hspace*{-0.5em} &\rTTo^H &\cc\bigl((A_i)_{i\in I}\bigl((F^j)_{j\in J}\centerdot_\phi H\bigr),(A_i)_{i\in I}\bigl((G^j)_{j\in J}\centerdot_\phi H\bigr)\bigr)
	\end{diagram}
	\]
	equals the composition in the left--bottom path.
\end{proof}

\subsection{Representable multicategories}\label{sec-Representable-multicategories}
\begin{proposition}\label{pro-represented-symmetric-monoidal-category}
	If \(\mcv=\wh\cv\) is represented by a symmetric monoidal category \(\cv\), then \(\mcv\CAT\) is representable by a symmetric monoidal category.
\end{proposition}

\begin{proof}
	We consider $\cv$ as an unbiased (strong lax) symmetric monoidal category \((\cv,\tens_\cv^I,\lambda_\cv^f)\).
	Denote the future strong lax symmetric monoidal category representing \(\mcv\CAT\) by \((\cW,\tens_\cW^I,\lambda_\cW^f)\), where \(\cW(\ca,\cb)=\mcv\CAT(\ca;\cb)\).
	Objects of $\cW$ are $\mcv$\n-categories.
	By definition, a $\mcv$\n-category $\cc$ is
	\begin{myitemize}
		\item[---] a small set $\Ob\cc$ of objects;
		\item[---] for each pair of objects \((A,D)\in(\Ob\cc)^2\) an object \(\cc(A,D)\) of $\cv$;
		\item[---] for each triple of objects \((A,D,E)\in(\Ob\cc)^3\) a morphism \(\kappa_{A,D,E}\:\cc(A,D)\tens\cc(D,E)\to\cc(A,E)\in\cv\) -- the composition;
		\item[---] for any object \(A\in\Ob\cc\) a morphism \(\id_A\:\1\to\cc(A,A)\in\cv\) -- the identity morphism\index{$\id_A$ -- identity element of an enriched category}
	\end{myitemize}
	such that
	\begin{myitemize}
		\item[---] for each quadruple of objects \((B,A,D,E)\) of $\cc$ the associativity holds:
		\begin{diagram}[nobalance,LaTeXeqno]
			\cc(B,A)\tens\cc(A,D)\tens\cc(D,E) &\rTTo^{\hspace*{-0.7em}\lambda^{\mathsf{IV}\:\mb3\to\mb2}\hspace*{-0.3em}} &\cc(B,A)\tens\bigl(\cc(A,D)\tens\cc(D,E)\bigr) &\rTTo^{1\tens\kappa_{A,D,E}} &\cc(B,A)\tens\cc(A,E)
			\\
			\dTTo>{\lambda^{\mathsf{VI}\:\mb3\to\mb2}} &&= &&\dTTo<{\kappa_{B,A,E}}
			\\
			\bigl(\cc(B,A)\tens\cc(A,D)\bigr)\tens\cc(D,E) &\rTTo^{\kappa_{B,A,D}\tens1\hspace*{-2em}} &\cc(B,D)\tens\cc(D,E) &\rTTo^{\kappa_{B,D,E}} &\cc(B,E)
			\label{dia-assocC}
		\end{diagram}
		\item[---] for each pair of objects \((A,D)\) of $\cc$
		\begin{align}
			\bigl[ \cc(A,D) \rTTo^{\lambda^{\mathsf{\centerdot I}\:\mb1\to\mb2}} \1\tens\cc(A,D) \rTTo^{\id_A\tens1} \cc(A,A)\tens\cc(A,D) \rTTo^{\kappa_{A,A,D}} \cc(A,D) \bigr] &=1,
			\label{eq-idX-tens-1k1}
			\\
			\bigl[ \cc(A,D) \rTTo^{\lambda^{\mathsf{I\centerdot}\:\mb1\to\mb2}} \cc(A,D)\tens\1 \rTTo^{1\tens\id_D} \cc(A,D)\tens\cc(D,D) \rTTo^{\kappa_{A,D,D}} \cc(A,D) \bigr] &=1.
			\label{eq-1-tens-idYk1}
		\end{align}
	\end{myitemize}
	Thus, $\cc$ is a $\cv$\n-category.
	
	A morphism of $\cW$ is a $\mcv$\n-functor \(F\:\ca\to\cb\) (see \exaref{exa-I1}), that is,
	\begin{myitemize}
		\item[---] a function \(F=\Ob F\:\Ob\ca\to\Ob\cb\);
		\item[---] a collection of elements \(F=F_{A,E}\in\cv\bigl(\ca(A,E),\cb(AF,EF)\bigr)\);
	\end{myitemize}
	such that
	\begin{diagram}
		\ca(A,D)\tens\ca(D,E) &\rTTo^{\kappa_{A,D,E}} &\ca(A,E)
		\\
		\dTTo<{F_{A,D}\tens F_{D,E}} &= &\dTTo>{F_{A,E}}
		\\
		\cb(AF,DF)\tens\cb(DF,EF) &\rTTo^{\kappa_{AF,DF,EF}} &\cb(AF,EF)
	\end{diagram}
	and
	\begin{equation*}
		\bigl[ \1 \rTTo^{\id_A} \ca(A,A) \rTTo^{F_{A,A}} \cb(AF,AF) \bigr] =\id_{AF},
	\end{equation*}
so $F$ is just a $\cv$\n-functor and \(\cW=\cv\CAT\).
	
	We are going to prove that $\cv$\n-quiver $\ca$ with
	\begin{myitemize}
		\item[---] \(\Ob\ca=\prod_{i\in I}\Ob\ca_i\) with the convention
		\begin{equation}
			\prod_{i\in\mb1}S \text{ is identified with $S$ for an arbitrary small set $S$}
			\label{eq-convention}
		\end{equation}
		\item[---] \(\ca\bigl((A_i)_{i\in I},(D_i)_{i\in I}\bigr)=\tens_\cv^{i\in I}\ca_i(A_i,D_i)\);
		\item[---] the composition
		\begin{multline*}
			\ca\bigl((A_i)_{i\in I},(D_i)_{i\in I}\bigr)\tens\ca\bigl((D_i)_{i\in I},(E_i)_{i\in I}\bigr) =\bigl(\tens_\cv^{i\in I}\ca_i(A_i,D_i)\bigr)\tens\bigl(\tens_\cv^{i\in I}\ca_i(D_i,E_i)\bigr)
			\\
			\rTTo^{(\lambda^{\triangledown\sqcup\triangledown\:I\sqcup I\to\mb2})^{-1}} \tens_\cv^{I\sqcup I}\bigl[\bigl(\ca_i(A_i,D_i)\bigr)_{i\in I},\bigl(\ca_i(D_i,E_i)\bigr)_{i\in I}\bigr]
			\\
			\rTTo^{\lambda^{\chi\:I\sqcup I\to I}} \tens_\cv^{i\in I}\bigl[\ca_i(A_i,D_i)\tens\ca_i(D_i,E_i)\bigr] \rTTo^{\tens_\cv^{i\in I}\kappa_{A_i,D_i,E_i}\;} \tens_\cv^{i\in I}\ca_i(A_i,E_i),
		\end{multline*}
		where \(\triangledown\sqcup\triangledown\) and \(\chi\) are given respectively by \eqref{eq-triangledowntriangledown} and \eqref{eq-chi};
		\item[---] the identity morphism \(\1 \rTTo^{\lambda_\cv^{\varnothing\to I}} \tens_\cv^{i\in I}\1 \rTTo^{\tens_\cv^{i\in I}\id_{A_i}\;} \tens_\cv^{i\in I}\ca_i(A_i,A_i)=\ca\bigl((A_i)_{i\in I},(A_i)_{i\in I}\bigr)\)
	\end{myitemize}
	is a $\cv$\n-category.
	
	Let us prove associativity condition~\eqref{dia-assocC} for $\ca$.
	Due to naturality of \(\lambda^f\) the top-right path of diagram~\eqref{dia-assocC} can be presented as
	\begin{multline*}
		[\tens^{i\in I}\ca_i(B_i,A_i)]\tens[\tens^{i\in I}\ca_i(A_i,D_i)]\tens[\tens^{i\in I}\ca_i(D_i,E_i)] \rTTo^{\lambda^{\mathsf{IV}\:\mb3\to\mb2}}
		\\
		[\tens^{i\in I}\ca_i(B_i,A_i)]\tens\bigl\{[\tens^{i\in I}\ca_i(A_i,D_i)]\tens[\tens^{i\in I}\ca_i(D_i,E_i)]\bigr\} \rTTo^{1\tens(\lambda^{\triangledown\sqcup\triangledown\:I\sqcup I\to\mb2})^{-1}}
		\\
		[\tens^{i\in I}\ca_i(B_i,A_i)]\tens\tens^{I\sqcup I}\bigl[\bigl(\ca_i(A_i,D_i)\bigr)_{i\in I},\bigl(\ca_i(D_i,E_i)\bigr)_{i\in I}\bigr] \rTTo^{1\tens\lambda^{\chi\:I\sqcup I\to I}}
		\\
		[\tens^{i\in I}\ca_i(B_i,A_i)]\tens\tens^{i\in I}\bigl[\ca_i(A_i,D_i)\tens\ca_i(D_i,E_i)\bigr] \rTTo^{(\lambda^{\triangledown\sqcup\triangledown\:I\sqcup I\to\mb2})^{-1}}
		\\
		\tens^{I\sqcup I}\bigl[\bigl(\ca_i(B_i,A_i)\bigr)_{i\in I},\bigl(\ca_i(A_i,D_i)\tens\ca_i(D_i,E_i)\bigr)_{i\in I}\bigr] \rTTo^{\lambda^{\chi\:I\sqcup I\to I}}
		\\
		\tens^{i\in I}\{\ca_i(B_i,A_i)\tens[\ca_i(A_i,D_i)\tens\ca_i(D_i,E_i)]\} \rTTo^{\tens^{i\in I}[1\tens\kappa_{A_i,D_i,E_i}]\;}
		\\
		\tens^{i\in I}[\ca_i(B_i,A_i)\tens\ca_i(A_i,E_i)] \rTTo^{\tens^{i\in I}\kappa_{B_i,A_i,E_i}\;} \tens^{i\in I}\ca_i(B_i,E_i).
	\end{multline*}
	Similarly the left-bottom path of diagram~\eqref{dia-assocC} can be presented.
	Its commutativity follows from equations~\eqref{dia-assocC} for each of $\ca_i$, an identity in $\cv$
	\begin{multline*}
		\bigl\{(\tens^{i\in I}T_i)\tens(\tens^{i\in I}U_i)\tens(\tens^{i\in I}V_i) \rTTo^{\lambda^{\mathsf{IV}\:\mb3\to\mb2}}
		(\tens^{i\in I}T_i)\tens[(\tens^{i\in I}U_i)\tens(\tens^{i\in I}V_i)] \rTTo^{1\tens(\lambda^{\triangledown\sqcup\triangledown\:I\sqcup I\to\mb2})^{-1}}
		\\
		(\tens^{i\in I}T_i)\tens\tens^{I\sqcup I}[(U_i)_{i\in I},(V_i)_{i\in I}] \rTTo^{1\tens\lambda^{\chi\:I\sqcup I\to I}}
		(\tens^{i\in I}T_i)\tens\tens^{i\in I}(U_i\tens V_i) \rTTo^{(\lambda^{\triangledown\sqcup\triangledown\:I\sqcup I\to\mb2})^{-1}}
		\\
		\tens^{I\sqcup I}[(T_i)_{i\in I},(U_i\tens V_i)_{i\in I}] \rTTo^{\lambda^{\chi\:I\sqcup I\to I}}
		\tens^{i\in I}[T_i\tens(U_i\tens V_i)] \bigr\} =
		\\
		\bigl\{(\tens^{i\in I}T_i)\tens(\tens^{i\in I}U_i)\tens(\tens^{i\in I}V_i)
		\rTTo^{\lambda^{\mathsf{V\mkern-5.9mu I\mkern5mu}\:I\sqcup I\sqcup I\to I}} \tens^{i\in I}(T_i\tens U_i\tens V_i) 
		\rTTo^{\tens^{i\in I}\lambda^{\mathsf{IV}\:\mb3\to\mb2}} \tens^{i\in I}[T_i\tens(U_i\tens V_i)] \bigr\},
	\end{multline*}
	where \(\mathsf{V\mkern-8.8mu I\mkern8mu}\:I\sqcup I\sqcup I\to I\) denotes the map
	\[
	\begin{pmatrix}
		1 &\dots &I &I+1 &\dots &2I &2I+1 &\dots &3I-1 &3I
		\\
		1 &\dots &I &1   &\dots &I  &1    &\dots &I-1  &I
	\end{pmatrix}
	,
	\]
	and a similar identity with left and right switched.
	Proof of these identities follows from coherence for strong lax symmetric monoidal categories -- all expressions, which are compositions of tensor products of \(\lambda\)'s and their inverses are equal, if their sources and targets are pairwise equal.
	To prove this coherence principle notice that a strong lax symmetric monoidal category is strongly equivalent to a symmetric strictly monoidal category \cite[Theorem~3.1.6]{math.CT/0305049}, for which this coherence is obvious.
	
	Let us prove the first unitality condition \eqref{eq-idX-tens-1k1} for $\ca$:
	\begin{multline*}
		\bigl\{\tens^{i\in I}\ca_i(A_i,D_i) \rTTo^{\lambda^{\centerdot\mathsf{I}\:\mb1\to\mb2}} \1\tens[\tens^{i\in I}\ca_i(A_i,D_i)] \rTTo^{\lambda^{\varnothing\to I}\tens1} \, [\tens^{i\in I}\1]\tens[\tens^{i\in I}\ca_i(A_i,D_i)] \rTTo^{[\tens^{i\in I}\id_{A_i}]\tens1} 
		\\
		[\tens^{i\in I}\ca_i(A_i,A_i)]\tens[\tens^{i\in I}\ca_i(A_i,D_i)] \rTTo^{(\lambda^{\triangledown\sqcup\triangledown\:I\sqcup I\to\mb2})^{-1}} \tens^{I\sqcup I}\bigl[\bigl(\ca_i(A_i,A_i)\bigr)_{i\in I},\bigl(\ca_i(A_i,D_i)\bigr)_{i\in I}\bigr]
		\\
		\rTTo^{\lambda^{\chi\:I\sqcup I\to I}} \tens^{i\in I}\bigl[\ca_i(A_i,A_i)\tens\ca_i(A_i,D_i)\bigr] \rTTo^{\tens^{i\in I}\kappa_{A_i,A_i,D_i}\;} \tens^{i\in I}\ca_i(A_i,D_i) \bigr\} =1.
	\end{multline*}
	Naturality of \(\lambda\) reduces the above to the identity in $\cv$ (here \(V_i\in\Ob\cv\) for \(i\in I\)):
	\begin{multline}
		\bigl\{\tens^{i\in I}V_i \rTTo^{\lambda^{\centerdot\mathsf{I}\:\mb1\to\mb2}} \1\tens(\tens^{i\in I}V_i) \rTTo^{\lambda^{\varnothing\to I}\tens1} (\tens^{i\in I}\1)\tens(\tens^{i\in I}V_i)
		\\
		\rTTo^{(\lambda^{\triangledown\sqcup\triangledown\:I\sqcup I\to\mb2})^{-1}} \tens^{I\sqcup I}[(\1)_{i\in I},(V_i)_{i\in I}] \rTTo^{\lambda^{\chi\:I\sqcup I\to I}} \tens^{i\in I}(\1\tens V_i) \bigr\} =\tens^{i\in I}\lambda^{\centerdot\mathsf{I}\:\mb1\to\mb2}.
		\label{eq-Vi-1Vi-lambda}
	\end{multline}
	Proof of the above follows from coherence for strong lax symmetric monoidal categories.
	
	
	The proof of second unitality equation~\eqref{eq-1-tens-idYk1} also follows from coherence for strong lax symmetric monoidal categories.
	
	Let us show that for each finite ordered family of $\cv$\n-categories \((\ca_i)_{i\in I}\) the functor \(\cW\to\Set\), \(\cb\mapsto\mcv\bigl((\ca_i)_{i\in I};\cb\bigr)\) is representable.
	As the representing object we take $\ca$ described above.
	Thus condition~(1) of Definition~8.1 \cite{Hermida:MultiRep} is satisfied, see Remark~8.2 [{\it ibid}].
	We do not check condition~(2) (of Definition~8.1 [{\it ibid}]).
	Instead we deduce by \cite[Theorem~3.24]{BesLyuMan-book} that \(\mcv\CAT\) is representable by a lax monoidal category $\cW$ and we describe its lax monoidal structure explicitly.
	
	We are going to establish a natural in $\cb$ bijection \(\mcv\CAT\bigl((\cA_i)_{i\in I};\cb\bigr)\cong\cW(\ca,\cb)\), so that the object function \(\Ob F\:\prod_{i\in I}\Ob\ca_i\to\Ob\cb\) is the same for both elements of these sets, related by the bijection.
	These elements are collections of morphisms
	\[ F_{(A_i),(D_i)}\: \tens^{i\in I} \ca_i(A_i,D_i) \to \cb\bigl((A_i)_{i\in I}F,(D_i)_{i\in I}F\bigr) \in\cv,
	\]
	subject to certain equations.
An element \((F_{(A_i),(D_i)})\in\mcv\CAT\bigl((\cA_i)_{i\in I};\cb\bigr)\) has to make the exterior of diagram
	\begin{diagram}[nobalance,LaTeXeqno]
		\tens^{I\sqcup I}\bigl[\bigl(\ca_i(A_i,D_i)\bigr)_{i\in I},\bigl(\ca_i(D_i,E_i)\bigr)_{i\in I}\bigr] &\rTTo^{\lambda^{\chi\:I\sqcup I\to I}} &\tens^{i\in I}\bigl[\ca_i(A_i,D_i)\tens\ca_i(D_i,E_i)\bigr]
		\\
		\dTTo<{\lambda^{\triangledown\sqcup\triangledown\:I\sqcup I\to\mb2}} &= &\dTTo>{\tens^{i\in I}\kappa_{A_i,D_i,E_i}}
		\\
		[\tens^{i\in I}\ca_i(A_i,D_i)]\tens[\tens^{i\in I}\ca_i(D_i,E_i)] &\rTTo^{\kappa^\ca} &\tens^{i\in I}\ca_i(A_i,E_i)
		\\
		\dTTo<{F_{(A_i),(D_i)}\tens F_{(D_i),(E_i)}} &= &\dTTo>{F_{(A_i),(E_i)}}
		\\
		\cb\bigl((A_i)_{i\in I}F,(D_i)_{i\in I}F\bigr)\tens\cb\bigl((D_i)_{i\in I}F,(E_i)_{i\in I}F\bigr) &\rTTo^{\kappa^\cb} &\cb\bigl((A_i)_{i\in I}F,(E_i)_{i\in I}F\bigr)
		\label{dia-tensIIAA-B}
	\end{diagram}
	commute and to satisfy \eqref{eq-coherence-with-units} in the form:
	\begin{equation}
		\bigl[\1 \rTTo^{\lambda^{\varnothing\to I}} \tens^{i\in I}\1 \rTTo^{\tens^{i\in I}\id_{A_i}\;} \tens^{i\in I}\ca_i(A_i,A_i) \rTTo^{F_{(A_i),(A_i)}} \cb\bigl((A_i)_{i\in I}F,(A_i)_{i\in I}F\bigr)\bigr] =\id_{(A_i)_{i\in I}F}.
		\label{eq-1tens1-B}
	\end{equation}
	An element \((F_{(A_i),(D_i)})\in\cW(\cA,\cb)\) has to make the lower square of diagram~\eqref{dia-tensIIAA-B} commutative and to satisfy \eqref{eq-1tens1-B}.
	Therefore, elements of
	\[ \mcv\CAT\bigl((\cA_i)_{i\in I};\cb\bigr)\subset\prod_{(A_i),(D_i)}\cv\bigl(\tens^{i\in I}\ca_i(A_i,D_i),\cb\bigl((A_i)_{i\in I}F,(D_i)_{i\in I}F\bigr)\bigr) \supset \cW(\ca,\cb)
	\]
	satisfy equivalent systems of equations.
	Hence, an identification 
	\[ \rho^{-1}\:\mcv\CAT\bigl((\cA_i)_{i\in I};\cb\bigr) \rto\cong \cW(\ca,\cb).
	\]
	Here $\rho$ can be written as the composition
	\[ \rho = \bigl[ \cW(\ca,\cb) \cong \mb1\times\cW(\ca,\cb) \rTTo^{\tau\times1} \mcv\CAT((\ca_i)_{i\in I};\ca)\times\cW(\ca,\cb)
	\rTTo^{\mu^{\mcv\CAT}_{I\to\mb1}} \mcv\CAT\bigl((\cA_i)_{i\in I};\cb\bigr) \bigr],
	\]
	where $\tau$ is the following multi-entry $\mcv$\n-functor: 
	\(\Ob\tau=\id\:\prod_{i\in I}\Ob\ca_i\to\prod_{i\in I}\Ob\ca_i\), on morphisms
	\[ \tau=\id\in\cv\bigl(\tens_\cv^{i\in I}\ca_i(A_i,D_i),\tens_\cv^{i\in I}\ca_i(A_i,D_i)\bigr) =\bigl\{F\in\mcv\CAT\bigl((\cA_i)_{i\in I};\ca\bigr)\mid \Ob F=\id_{\Ob\ca}\bigr\}.
	\]
	In order to prove formula for $\rho$ notice that composition with the identity map is an isomorphism.
	
	According to \cite[Theorem~3.24]{BesLyuMan-book} the category $\cW$ is lax symmetric monoidal with the tensor products \(\tens_\cW^I\:\cW^I\to\cW\), \((\cA_i)_{i\in I}\mapsto\tens_\cW^{i\in I}\ca_i=\ca\) as above, which give on morphisms
	\begin{multline*}
		\tens_\cW^{i\in I} =\Bigl\{\prod_{i\in I}\cW(\ca_i,\cb_i) \cong \Bigl[\prod_{i\in I}\cW(\ca_i,\cb_i)\Bigr]\times\mb1 \rTTo^{1\times\tau} \Bigl[\prod_{i\in I}\cW(\ca_i,\cb_i)\Bigr]\times\mcv\CAT\bigl((\cb_i)_{i\in I};\tens_\cW^{i\in I}\cb_i\bigr)
		\\
		\rTTo^{\mu^{\mcv\CAT}_{\id_I}} \mcv\CAT\bigl((\ca_i)_{i\in I};\tens_\cW^{i\in I}\cb_i\bigr) \rto[\cong]{\rho^{-1}} \cW(\tens_\cW^{i\in I}\ca_i,\tens_\cW^{i\in I}\cb_i)\Bigr\}, \;
		(F^i:\ca_i\to\cb_i)_{i\in I} \mapsto F,
	\end{multline*}
	where \(\Ob F=\prod_{i\in I}\Ob F^i\), \(\Mor F=\tens_\cv^{i\in I}F^i_{A^i,D^i}\:\tens_\cv^{i\in I}\ca_i(A^i,D^i)\to\tens_\cv^{i\in I}\cb_i(A^iF^i,D^iF^i)\).
	
	The same \cite[Theorem~3.24]{BesLyuMan-book} gives that for any map \(f\:I\to J\in\cs_\sk\) the $\cv$\n-functor \(\lambda_\cW^f\) is obtained via map
	\begin{multline*}
		\Bigl\{\Bigl[\prod_{j\in J}\mcv\CAT\bigl((\ca_i)_{i\in f^{-1}j};\tens_\cW^{i\in f^{-1}j}\ca_i\bigr)\Bigr]\times \mcv\CAT\bigl((\tens_\cW^{i\in f^{-1}j}\ca_i)_{i\in I};\tens_\cW^{j\in J}\tens_\cW^{i\in f^{-1}j}\ca_i\bigr)
		\\
		\rTTo^{\mu^{\mcv\CAT}_f} \mcv\CAT\bigl((\ca_i)_{i\in I};\tens_\cW^{j\in J}\tens_\cW^{i\in f^{-1}j}\ca_i\bigr) \rto[\cong]{\rho^{-1}} \cW(\tens_\cW^{i\in I}\ca_i,\tens_\cW^{j\in J}\tens_\cW^{i\in f^{-1}j}\ca_i)\Bigr\}, \;
		\tau^{J\sqcup\mb1} \mapsto \lambda_\cW^f.
	\end{multline*}
	Easy computation gives \(\Ob\lambda_\cW^f\) as the obvious bijection \(\prod_{i\in I}\Ob\ca_i\rto\cong\prod_{j\in J}\prod_{i\in f^{-1}j}\Ob\ca_i\).
	On morphisms \(\lambda_\cW^f\) acts as \(\lambda_\cv^f\:\tens_\cv^{i\in I}\ca_i(A_i,D_i)\to\tens_\cv^{j\in J}\tens_\cv^{i\in f^{-1}j}\ca_i(A_i,D_i)\) according to \cite[Proposition~3.22]{BesLyuMan-book}.
\end{proof}

\begin{proposition}
Assume that $\cv$ is Cartesian (closed under arbitrary small products).
Equip $\cv$ with finite products as monoidal multiplication.
Then $\cv\CAT$ is also Cartesian and its monoidal structure from \propref{pro-represented-symmetric-monoidal-category} has finite products as monoidal multiplications,
\end{proposition}

\begin{proof}

	Let $\cv$ be Cartesian.
It is proven in \propref{pro-VCat-has-products} in the general case of categories enriched over a symmetric multicategory $\mcv$ that the multicategory $\VCat$ is Cartesian.
Hence \(\cv\CAT\) is Cartesian.
It is shown in the proof of the same proposition that \(\pr_i\:\ca\equiv\tens^{j\in I}\ca_j\to\ca_i\), \((A_j)_{j\in I}\mapsto A_i\), \(\pr_i\:\prod_{j\in I}\ca_j(A_j,D_j)\to\ca_i(A_i,D_i)\) are $\mcv$\n-functors.
Moreover, it is shown that \(\ca\equiv\tens^{i\in I}\ca_i\) is the product \(\prod_{i\in I}\ca_i\) of a family \((\ca_i)_{i\in I}\).
\end{proof}

As shown in \thmref{thm-machine} + \cite[Proposition~4.8]{BesLyuMan-book} for a symmetric closed complete monoidal category $\cv$, the category \(\cv\CAT\) also has all these structures.
Equivalence of closedness of $\cv$ and \(\mcv=\wh\cv\) is proven precisely in \cite[Proposition~4.8]{BesLyuMan-book}.
As we have noticed, if monoidal category $\cv$ is Cartesian, so is \(\cv\CAT\).

\subsection{Strict 2-categories}\label{sec-Strict-2-categories}
\begin{example}
Let \(\mcv=\wh\1\), final multicategory with \(\Ob\1=\mb1=\{1\}\), and \(\1((1)_{\mb n};1)=\mb1\) for all \(n\in\ZZ_{\ge0}\).
Here $\1$ is a one-object-one-morphism category.
It has the only lax (strict) monoidal structure.
Then \(\Ob:\wh\1\CAT\to\mcSet\) is an isomorphism of \(\wh\1\CAT\) to \(\mcSet\), the symmetric multicategory of small sets, represented by $\Set$, the Cartesian closed category of small sets.
	Indeed, a small \(\1\)\n-category $\cc$ is a small set $\Ob\cc$ of objects.
	The other choices are unique.
	This ensures that required equations hold true.
	A multi-entry \(\1\)\n-functor \(F\:(\ca_i)_{i\in I}\to\cb\) is a function \(F=\Ob F\:\Ob\ca_1\times\dots\times\Ob\ca_I\to\Ob\cb\), that is, a morphism in \(\mcSet\).
\end{example}

\begin{example}
Let \(\mcv=\wh{(\mcSet,\times)}\).
	This multicategory is closed with \(\und\mcSet=\mcSet\).
	Objects of \(\mcSet\CAT\) are (ordinary) small (and locally small) categories.
	Multi-entry \(\mcSet\)\n-functors \(F\:(\ca_i)_{i\in I}\to\cc\) are (ordinary) functors \(F\:\prod_{i\in I}\ca_i\to\cc\).
We obtain the Cartesian category $\Cat$.
The object of $\mcSet$\n-transformations \(F\to G\:(\ca_i)_{i\in I} \to\cc=\) the enriched in $\mcSet$ end
\[ \int_{(A_i\in\ca_i)_{i\in I}}\cc((A_i)_{i\in I}F,(A_i)_{i\in I}G),
\]
the equalizer in multicategory $\mcSet$ of pair of morphisms \eqref{eq-int-c(FE-GE)-VQu}.
	It coincides with the set of natural transformations \(\lambda\:F\to G\:\prod_{i\in I}\ca_i\to\cc\), which are, of course, elements \(\lambda_{(A_i\in\ca_i)_{i\in I}}\in\prod_{(A_i\in\ca_i)_{i\in I}}\cc\bigl((A_i)_{i\in I}F,(A_i)_{i\in I}G\bigr)\) such that
	\begin{diagram}
		(A_i)_{i\in I}F &\rTTo^{F_{(A_i),(D_i)}} &(D_i)_{i\in I}F
		\\
		\dTTo<{\lambda_{(A_i)_{i\in I}}} &= &\dTTo>{\lambda_{(D_i)_{i\in I}}}
		\\
		(A_i)_{i\in I}G &\rTTo^{G_{(A_i),(D_i)}} &(D_i)_{i\in I}G
	\end{diagram}
\end{example}

\begin{example}
	Let \(\mcv=\mcSet\CAT\).
	A $\mcv$\n-category $\ca$ is a category enriched over the Cartesian closed category $\Cat$ of small categories.
	Thus, it is the same as a strict 2\n-category.
	A $\mcv$\n-functor \(F\:\ca\to\cb\) is a map \(F=\Ob F\:\Ob\ca\to\Ob\cb\) with functors \(F=F_{A,E}\:\ca(A,E)\to\cb(AF,EF)\) such that
	\begin{diagram}
		\ca(A,D)\times\ca(D,E) &\rTTo^{\kappa_{A,D,E}} &\ca(A,E)
		\\
		\dTTo<{F_{A,D}\times F_{D,E}} &= &\dTTo>{F_{A,E}}
		\\
		\cb(AF,DF)\times\cb(DF,EF) &\rTTo^{\kappa_{AF,DF,EF}} &\cb(AF,EF)
	\end{diagram}
	and \(F_{A,A}\:\ca(A,A)\to\cb(AF,AF)\) maps the identity object to the identity object.
Thus, $F$ is a strict 2\n-functor and \(\cW=\cv\CAT\) is the Cartesian category 2-$\Cat$.
	
The subcategory \(\und{\mcSet\CAT\CAT}(\ca,\cc)(F,G)\subset\prod_{A\in\Ob\ca}\cc(AF,AG)\) (see \eqref{eq-int-c(FE-GE)-VQu}) is equipped with the functors
	\[ p_D =\bigl[ \und{\mcSet\CAT\CAT}(\ca,\cc)(F,G) \rMono \prod_{A\in\Ob\ca}\cc(AF,AG) \rTTo^{\pr_D} \cc(DF,DG) \bigr].
	\]
	By definition, it is the biggest subcategory, for which
	\begin{diagram}[nobalance,LaTeXeqno,h=2.2em]
\ca(A,D)\times\und{\mcSet\CAT\CAT}(\ca,\cc)(F,G) \rto{F_{A,D}\times p_D} \cc(AF,DF)\times\cc(DF,DG)
		\\
\dTTo<{G_{A,D}\times p_A} \hspace*{11em} = \hspace*{11em} \dTTo>\centerdot
		\\
\cc(AG,DG)\times\cc(AF,AG) \rto{c} \cc(AF,AG)\times\cc(AG,DG) \rto{\centerdot} \cc(AF,DG)
		\label{dia-und-mcSet-CAT-CAT}
	\end{diagram}
	In particular, objects of \(\und{\mcSet\CAT\CAT}(\ca,\cc)(F,G)\) are collections of 1\n-cells \(\lambda=(\lambda_A)_{A\in\Ob\ca}\), \(\lambda_A\in\Ob\cc(AF,AG)\), such that for all \(\nu\:f\to g\in\ca(A,D)\) the following square
	\begin{diagram}
		AF &\pile{\rTTo^{\hspace*{1em}fF_{A,D}\hspace*{1em}} \\ \sss\Downarrow\nu F_{A,D}\\ \rTTo_{gF_{A,D}}} &DF
		\\
		\dTTo<{\lambda_A} &= &\dTTo>{\lambda_D}
		\\
		AG &\pile{\rTTo^{\hspace*{1em}fG_{A,D}\hspace*{1em}} \\ \sss\Downarrow\nu G_{A,D}\\ \rTTo_{gG_{A,D}}} &DG
	\end{diagram}
	commutes in $\cc$, that is, 
	\[
	\begin{diagram}[inline]
	AF &\rTTo^{fF_{A,D}} &DF
	\\
	\dTTo<{\lambda_A} &= &\dTTo>{\lambda_D}
	\\
	AG &\rTTo^{fG_{A,D}} &DG
	\end{diagram}
	\ ,\qquad
	\Bigl( AF \rTTo^{\lambda_A} AG \pile{\rTTo^{fG_{A,D}} \\ \sss\Downarrow\nu G_{A,D}\\ \rTTo_{gG_{A,D}}} DG\Bigr)
	=\Bigl( AF \pile{\rTTo^{fF_{A,D}} \\ \sss\Downarrow\nu F_{A,D}\\ \rTTo_{gF_{A,D}}} DF \rTTo^{\lambda_D} DG\Bigr)
	\]
	in the sense of strong (2\n-categorical) composition in $\cc$.
	Terminology is that of Gray \cite[\S I,2.3]{Gray}.
	Here \(\lambda_A\in\Ob\cc(AF,AG)\), \(\lambda_D\in\Ob\cc(DF,DG)\), \(\nu F_{A,D}\:fF_{A,D}\to gF_{A,D}\in\cc(AF,DF)\), \(\nu G_{A,D}\:fG_{A,D}\to gG_{A,D}\in\cc(AG,DG)\), \(\lambda_A\centerdot(\nu G_{A,D})\:\lambda_A\centerdot(fG_{A,D})\to\lambda_A\centerdot(gG_{A,D})\in\cc(AF,DG)\), \((\nu F_{A,D})\centerdot\lambda_D\:(fF_{A,D})\centerdot\lambda_D\to(gF_{A,D})\centerdot\lambda_D\in\cc(AF,DG)\).
	The last equation says that \(\lambda_A\centerdot(\nu G_{A,D})=(\nu F_{A,D})\centerdot\lambda_D\).
	Therefore, the collection \(\lambda\) is a $\Cat$\n-natural transformation \cite[\S I,2.3]{Gray} =strict 2\n-natural transformation (1\n-transfor in terminology of Crans \cite{Crans:transfors}).
	
Let \(\lambda,\mu\in\Ob\und{\mcSet\CAT\CAT}(\ca,\cc)(F,G)\),
\[ m=(m_A)_{A\in\Ob\ca}\in\und{\mcSet\CAT\CAT}(\ca,\cc)(F,G)(\lambda,\mu).
\]
Then for any 1\n-cell \(f\in\Ob\ca(A,D)\) we have \(fF_{A,D}\in\Ob\cc(AF,DF)\), \(fG_{A,D}\in\Ob\cc(AG,DG)\), \(\lambda_A,\mu_A\in\Ob\cc(AF,AG)\), \(\lambda_D,\mu_D\in\Ob\cc(DF,DG)\), \(m_A\in\cc(AF,AG)(\lambda_A,\mu_A)\), and, furthermore, \(m_D\in\cc(DF,DG)(\lambda_D,\mu_D)\).
	We have also
\begin{multline*}
m_A\centerdot(fG_{A,D})\in\cc(AF,DG)(\lambda_A\centerdot(fG_{A,D}),\mu_A\centerdot(fG_{A,D}))
\\
\hfill =\cc(AF,DG)(\lambda_A\centerdot(fG_{A,D}),(fF_{A,D})\centerdot\mu_D), \hskip\multlinegap
	\\
\hskip\multlinegap (fF_{A,D})\centerdot m_D\in\cc(AF,DG)((fF_{A,D})\centerdot\lambda_D,(fF_{A,D})\centerdot\mu_D) \hfill
\\
=\cc(AF,DG)(\lambda_A\centerdot(fG_{A,D}),(fF_{A,D})\centerdot\mu_D),
\end{multline*}
	where \(\centerdot\) is the composition in 2\n-category $\cc$.
	So the condition on the collection $m$ is \(m_A\centerdot(fG_{A,D})=(fF_{A,D})\centerdot m_D\), or, in terms of pastings,
	\[
	\begin{diagram}[inline]
	AF &\rTTo^{fF_{A,D}} &DF
	\\
	\dTTo<{\lambda_A} \stackrel{m_A}{\Rightarrow} \dTTo>{\mu_A} &\ruId^= &\dTTo>{\mu_D}
	\\
	AG &\rTTo^{fG_{A,D}} &DG
	\end{diagram}
	\quad=\quad
	\begin{diagram}[inline]
	AF &\rTTo^{fF_{A,D}} &DF
	\\
	\dTTo<{\lambda_A} &\ruId^= &\dTTo<{\lambda_D} \stackrel{m_D}{\Rightarrow} \dTTo>{\mu_D}
	\\
	AG &\rTTo^{fG_{A,D}} &DG
	\end{diagram}
	\quad.
	\]
	Therefore, \(\und{\mcSet\CAT\CAT}(\ca,\cc)(F,G)(\lambda,\mu)\) consists of modifications \(m\:\lambda\to\mu\:F\to G\:\ca\to\cc\) (see e.g. \cite[\S I,2.3]{Gray}).
	On the other hand, for any 2-cell $\nu$ of $\ca$ and any modification $m$ diagram~\eqref{dia-und-mcSet-CAT-CAT} evaluated on element \((\nu,m)\) commutes (exercise).
	Thus, \(\und{\mcSet\CAT\CAT}(\ca,\cc)(F,G)\) is precisely the category of strict 2\n-natural transformations and their modifications.
\end{example}

\section{Short spaces}\label{sec-Short-spaces}
Similarly to \cite[Section~2]{Lyu-Filt-coCat} we consider a partially ordered commutative monoid $\LL$ equipped with the operation $+$ and neutral element 0.
Of course, we assume that \(a\le b\), \(c\le d\) imply \(a+c\le b+d\).
We assume that $\LL$ satisfies the following conditions:
\begin{enumerate}
	\renewcommand{\labelenumi}{(\roman{enumi})}
	\item for all $a,b\in\LL$ there is $c\in\LL$ such that $a\le c$, 
	$b\le c$ (that is, \((\LL,\le)\) is directed);
	
	\item for all $a,b\in\LL$ there is $c\in\LL$ such that $c\le a$, 
	$c\le b$ (that is, \(\LL^\op\) is directed);
	
	\item for all $a\in\LL$ there is $c\in\LL$ such that $a+c\ge0$;

\item for \(a\le b=c+d\in\LL\) there exist \(f\le c\in\LL\) and \(g\le d\in\LL\) such that \(f+g=a\).
\end{enumerate}
If $\LL$ is a directed group (satisfies (i)), then $\LL$ satisfies (ii), (iii) and (iv) as well for obvious reasons.

\subsection{First properties}
Let $\KK$ denote one of two fields, $\RR$ or $\CC$.
By a (generalised) seminorm on a $\KK$\n-vector space $V$ we mean a function \(\|\cdot\|\:V\to[0,\infty]\), such that
\begin{enumerate}
	\renewcommand{\labelenumi}{(\roman{enumi})}
	\item for \(c\in\KK\) and \(x\in V\) we have \(\|cx\|=\vert c\vert\cdot\|x\|\) (with the convention \(0\cdot\infty=0\)) (absolute homogeneity);
	
	\item \(\|x+y\|\le\|x\|+\|y\|\) for \(x,y\in V\) (triangle inequality).
\end{enumerate}

\begin{remark}\label{rem-null-space-seminormed}
	Let $(V,\|\cdot\|)$ be a seminormed $\KK$\n-vector space.
	Then the null space \(\ker\|\cdot\|=\{x\in V\mid \|x\|=0\}\) is a $\KK$\n-vector subspace.
\end{remark}

\begin{definition}
Let $\LL$ be a partially ordered commutative monoid.
A \emph{short space}\index{short space} is a $\KK$\n-vector space \((V,(\|\cdot\|_l)_{l\in\LL})\) with a family of seminorms indexed by $\LL$, such that for any $x\in V$ there is $l\in\LL$ with \(\|x\|_l<\infty\) and the inequality \(a\le b\in\LL\) implies \(\|x\|_a\le\|x\|_b\).
\end{definition}

\begin{example}\label{exa-filtered-is-short}
	Let \((V,(\cF^lV)_{l\in\LL})\) be a filtered $\KK$\n-vector space.
	With each subspace $\cF^lV$ a seminorm is associated
	\[ \|x\|_l =
	\begin{cases}
		0, &\text{ if } x\in\cF^lV,
		\\
		\infty, &\text{ if } x\in V\setminus\cF^lV.
	\end{cases}
	\]
	Thus, \(\ker\|\cdot\|_l=\cF^lV\) and \((V,(\|\cdot\|_l)_{l\in\LL})\) is a short space.
	
	Vice versa, a short space \((V,(\|\cdot\|_l)_{l\in\LL})\) with \(\|V\|_l\subset\{0,\infty\}\) for all $l\in\LL$ determines an $\LL$\n-filtered $\KK$\n-vector space \((V,(\cF^lV)_{l\in\LL})\) with \(\cF^lV=\{x\in V\mid \|x\|_l=0\}\) (see \remref{rem-null-space-seminormed}).
\end{example}

\begin{definition}\label{def-Multicategory-Short}
Symmetric multicategory $\mcShort_\LL$ has short spaces as objects.
Morphisms\index{morphism of short spaces} are short multilinear maps:
	\[ f\: \bigl(X_1,(\sS{_1}\|\cdot\|_l)_{l\in\LL}\bigr) \times \bigl(X_2,(\sS{_2}\|\cdot\|_l)_{l\in\LL}\bigr) \times\dots\times \bigl(X_n,(\sS{_n}\|\cdot\|_l)_{l\in\LL}\bigr) \to \bigl(Y,(\|\cdot\|_l)_{l\in\LL}\bigr)
	\]
	such that for all \(l_1\), \dots, \(l_n\in\LL\) and all \(x_1\in X_1\), \dots, \(x_n\in X_n\) we have
	\[ \|f(x_1,x_2,\dots,x_n)\|_{l_1+\dots+l_n} \le \sS{_1}\|x_1\|_{l_1} \cdot\ldots\cdot \sS{_n}\|x_n\|_{l_n}
	\]
	(here \(0\cdot\infty=\infty\)).
When $n=1$, $\LL=0$, $X_1$ and $Y$ are Banach spaces, short maps are defined as above and are widely used in calculus.
Composition of multilinear maps
\[ \mu_\phi\: \bigl[ \prod_{j\in J}\mcShort_\LL\bigl((X_i)_{i\in\phi^{-1}j};Y_j\bigr) \bigr] \times \mcShort_\LL\bigl((Y_j)_{j\in J};Z\bigr) \to \mcShort_\LL\bigl((X_i)_{i\in I};Z\bigr);
\]
indexed by a map \(\phi\:I\to J\in\cS_\sk\) is given by substituting the results of \((g_j\:(X_i)_{i\in\phi^{-1}j}\to Y_j)_{j\in J}\) into \(f\:(Y_j)_{j\in J}\to Z\), thus, \(\mu_\phi\:((g_j)_{j\in J},f)\mapsto((g_j)_{j\in J})f\).
The identity morphism \(1_X\in\mcShort_\LL(X;X)\) is the identity map $\id_X\:X\to X$.
\end{definition}


\begin{proposition}\label{pro-Short-closed}
The multicategory $\mcShort_\LL$ is closed: the internal hom object is a $\KK$\n-vector subspace
\[ \und{\mcShort_\LL}(X_1,\dots,X_n;Z)\subset\ML_\KK(X_1\times\dots\times X_n,Z)=\und{\wh{\KK\Vect}}(X_1,\dots,X_n;Z)
\]
of $\KK$\n-multilinear maps.
	The latter is equipped with seminorms
	\begin{multline}
		\|f\|_l =\inf\{c\in\RR_{>0} \mid \forall (x_1,\dots,x_n)\in X_1\times\dots\times X_n\quad \forall (\lambda_1,\dots,\lambda_n)\in\LL^n
		\\
		\|f(x_1,x_2,\dots,x_n)\|_{\lambda_1+\dots+\lambda_n+l} \le c\cdot\sS{_1}\|x_1\|_{\lambda_1} \cdot\ldots\cdot \sS{_n}\|x_n\|_{\lambda_n} \}.
		\label{eq-fl-map}
	\end{multline}
	The subspace \(\und{\mcShort_\LL}(X_1,\dots,X_n;Z)\) is defined as
	\[ \{f\in\ML_\KK(X_1\times\dots\times X_n,Z) \mid \exists l\in\LL \; \|f\|_l<\infty \}.
	\]
\end{proposition}

\begin{proof}
The evaluation multi-entry functor $\ev$ is defined as
\begin{diagram}[nobalance]
\bigl[ X_1,\dots,X_n,\und{\mcShort_\LL}(X_1,\dots,X_n;Z) &\rTTo^{(1,\dots,1,\hookrightarrow)} &X_1,\dots,X_n,\und{\wh{\KK\Vect}}(X_1,\dots,X_n;Z) \rto\ev Z \bigr],
		\\
		(x_1,x_2,\dots,x_n,f) &\rMapsTo &(x_1,x_2,\dots,x_n)f.
\end{diagram}
It is a short map since \(\|(x_1,x_2,\dots,x_n)f\|_{\lambda_1+\dots+\lambda_n+l}\le\sS{_1}\|x_1\|_{\lambda_1}\cdot\ldots\cdot\sS{_n}\|x_n\|_{\lambda_n}\cdot\|f\|_l\).
	As \(\wh{\KK\Vect}\) is closed, for every \(\xi\:X_1,\dots,X_n,Y_1,\dots,Y_m\to Z\in\wh{\KK\Vect}\) there exists a unique \(\psi\:Y_1,\dots,Y_m\to\und{\wh{\KK\Vect}}(X_1,\dots,X_n;Z)\in\wh{\KK\Vect}\) such that 
	\[ \xi =\bigl[X_1,\dots,X_n,Y_1,\dots,Y_m \rTTo^{(1,\dots,1,\psi)} X_1,\dots,X_n,\und{\wh{\KK\Vect}}(X_1,\dots,X_n;Z) \rTTo^\ev Z \bigr].
	\]
	The proposition claims that $\xi$ is short iff \(\im\psi\subset\und{\mcShort_\LL}(X_1,\dots,X_n;Z)\) and
	\[ \psi\:Y_1,\dots,Y_m\to\und{\mcShort_\LL}(X_1,\dots,X_n;Z)
	\]
	is short.
	Let us prove the claim.
	We have
\[ (x_1,\dots,x_n,y_1,\dots,y_m)\xi=(x_1,\dots,x_n)(y_1,\dots,y_m)\psi.
\]
	The statement can be rephrased as equivalence of two inequalities:
\begin{multline}
		\|(x_1,\dots,x_n)(y_1,\dots,y_m)\psi\|_{\lambda_1+\dots+\lambda_n+\mu_1+\dots+\mu_m}
\\
\hfill \le \sS{_{X_1}}\|x_1\|_{\lambda_1} \cdot\ldots\cdot \sS{_{X_n}}\|x_n\|_{\lambda_n}\cdot \sS{_{Y_1}}\|y_1\|_{\mu_1} \cdot\ldots\cdot \sS{_{Y_m}}\|y_m\|_{\mu_m}, \hskip\multlinegap
		\label{eq-(x)(y)psi}
\end{multline}
\begin{equation}
\|(y_1,\dots,y_m)\psi\|_{\mu_1+\dots+\mu_m} \le \sS{_{Y_1}}\|y_1\|_{\mu_1} \cdot\ldots\cdot \sS{_{Y_m}}\|y_m\|_{\mu_m}.
		\label{eq-(y)psi}
\end{equation}		
	\eqref{eq-(x)(y)psi} implies \eqref{eq-(y)psi} because the requirement of \eqref{eq-fl-map} is satisfied by \(c=\sS{_{Y_1}}\|y_1\|_{\mu_1}\cdot\ldots\cdot\sS{_{Y_m}}\|y_m\|_{\mu_m}\).
	Vice versa, \eqref{eq-(y)psi} implies that for any \(\eps>0\)
\begin{multline*}
\|(x_1,\dots,x_n)(y_1,\dots,y_m)\psi\|_{\lambda_1+\dots+\lambda_n+\mu_1+\dots+\mu_m}
\\
\le \sS{_{X_1}}\|x_1\|_{\lambda_1} \cdot\ldots\cdot \sS{_{X_n}}\|x_n\|_{\lambda_n}\cdot (\sS{_{Y_1}}\|y_1\|_{\mu_1} \cdot\ldots\cdot \sS{_{Y_m}}\|y_m\|_{\mu_m}+\eps).
\end{multline*}
	Therefore, \eqref{eq-(x)(y)psi} holds.
\end{proof}

\begin{remark}
	The category \(\Short_\LL\) is defined as the case $n=1$ of \defref{def-Multicategory-Short}.
The category $\snS$ is defined as \(\Short_0\) for $\LL=0$.
It has seminormed spaces $(V,\|\cdot\|)$ as objects and short maps as morphisms.
Define the multicategory of seminormed spaces \(\mcsnS=\mcShort_0\), where $\LL=0$.

	\exaref{exa-filtered-is-short} gives a symmetric multifunctor 
	\(\iota\:\wh{\KK\Vect_\LL}\to\mcShort_\LL\).
	The image of $\Ob\iota$ consists of short spaces \((V,(\|\cdot\|_l)_{l\in\LL})\) with 
	\(\|V\|_l\subset\{0,\infty\}\) for all $l\in\LL$.
	Besides \(\Ob\iota\) the multifunctor consists of bijections
	\[ \iota\:\wh{\KK\Vect_\LL}(M_1,\dots,M_n;N) \to \mcShort_\LL(\iota M_1,\dots,\iota M_n;\iota N).
	\]
\end{remark}

\begin{lemma}\label{lem-b-l1ln-a-lambda1lambdan}
Let \(a\le b=l_1+\dots+l_n\in\LL\) for $n\ge1$.
Then there are elements \(\lambda_1,\dots,\lambda_n\in\LL\) such that \(\lambda_k\le l_k\) for all \(1\le k\le n\) and \(\lambda_1+\dots+\lambda_n=a\).
\end{lemma}

\begin{proof}
For $n=1$ the statement is obvious.
For $n=2$ it is the condition (iv).
The case $n>2$ reduces to the case $n-1$.
Indeed, we have \(a\le b=(l_1+\dots+l_{n-1})+l_n\).
By condition (iv) there are elements \(f\le l_1+\dots+l_{n-1}\), \(g\le l_n\), such that \(f+g=a\).
By the $n-1$ case there are \(\lambda_1,\dots,\lambda_{n-1}\in\LL\) such that \(\lambda_k\le l_k\) for \(k\in\ZZ\), \(1\le k\le n-1\) and \(\lambda_1+\dots+\lambda_{n-1}=f\).
Set \(\lambda_n=g\).
Then \(\lambda_1+\dots+\lambda_n=a\) and the lemma is proved by induction.
\end{proof}

\begin{proposition}
The category \(\Short_\LL\) has a colax symmetric monoidal structure related to the symmetric multicategory \(\mcShort_\LL\).
\end{proposition}

\begin{proof}
Given short spaces \(X_k=\bigl(X_k,(\sS{_k}\|\cdot\|_l)_{l\in\LL}\bigr)\), \(1\le k\le n\), we construct another short space \(X=\bigl(X,(\sS{_X}\|\cdot\|_l)_{l\in\LL}\bigr)\) and a map \(\Short_\LL(X,Y)\to\mcShort_\LL(X_1,\dots,X_n;Y)\), functorial in \(Y=\bigl(Y,(\sS{_Y}\|\cdot\|_l)_{l\in\LL}\bigr)\in\Short_\LL\).
As a $\KK$\n-vector space $X=X_1\tdt X_n$.
Using the free-forgetful adjunction between sets and $\KK$\n-vector spaces
\[ V\: \Set \leftrightarrows \KK\Vect\: U
\]
we may write $X$ as the quotient
\[ 0 \to W \to V\bigl(\prod_{k=1}^nX_k\bigr) \to X_1\tdt X_n \to 0,
\]
where the subspace $W$ is spanned by vectors
\begin{gather*}
c(x_1,\dots,x_k,\dots,x_n) - (x_1,\dots,cx_k,\dots,x_n), \quad 1\le k\le n, \quad c\in\KK,
\\
(x_1,\dots,x_k,\dots,x_n) + (x_1,\dots,y_k,\dots,x_n) - (x_1,\dots,x_k+y_k,\dots,x_n), \quad 1\le k\le n, \quad y_k\in X_k.
\end{gather*}
The following functions are seminorms on \(X_1\tdt X_n\) (this is left to the reader as an exercise)
\[ \|x\|_l =\inf\bigl\{ \sum_{j\in J}\sS{_1}\|x^j_1\|_{l^j_1}\cdot\ldots\cdot\sS{_n}\|x^j_n\|_{l^j_n} \mid x=\sum_{j\in J} x^j_1\tdt x^j_n, |J|<\infty, \forall l^j_k\in\LL \: \forall j\in J \; l^j_1 +\dots+ l^j_n =l \bigr\}.
\]
Here \(0\cdot\infty=\infty\).
It follows from \lemref{lem-b-l1ln-a-lambda1lambdan} that inequality \(a\le b\in\LL\) implies \(\|x\|_a\le\|x\|_b\).
Any \(x\in X\) has a presentation as a finite sum \(x=\sum_{j\in I}x^j_1\tdt x^j_n\).
For each $x^j_k$ there is \(l^j_k\in\LL\) such that \(\sS{_k}\|x^j_k\|_{l^j_k}<\infty\).
By (ii) there is \(l\in\LL\) such that \(l\le l^j_1+\dots+l^j_n\) for all \(j\in J\).
By \lemref{lem-b-l1ln-a-lambda1lambdan} there are elements \(\lambda^j_k\in\LL\) such that \(\lambda^j_k\le l^j_k\) and \(\lambda^j_1+\dots+\lambda^j_n=l\).
Hence,
\[ \|x\|_l \le \sum_{j\in I}\sS{_1}\|x^j_1\|_{l^j_1}\cdot\ldots\cdot\sS{_n}\|x^j_n\|_{l^j_n} <\infty.
\]
We deduce that \(X=\bigl(X_1\tdt X_n,(\|\cdot\|_l)_{l\in\LL})\bigr)\) is a short space.

For $n=0$ we define \(\1=\tens^0\) as $K$ with seminorms
\[ \|c\|_l =
\begin{cases}
0, &\quad \text{if } l\not\ge0,
	\\
|c|, &\quad \text{if } l\ge0.
\end{cases}
\]
Let \(f\:X_1\tdt X_n\to Y\in\Short_\LL\).
Then
\[ \|f(x_1\tdt x_n)\|_{l_1+\dots+l_n} \le \|x_1\tdt x_n\|_{l_1+\dots+l_n} \le \sS{_1}\|x_1\|_{l_1} \cdot\ldots\cdot \sS{_n}\|x_n\|_{l_n}.
\]
Hence, \(\Short_\LL(X_1\tdt X_n,Y)\subset\mcShort_\LL(X_1,\dots,X_n;Y)\).

For a map $f\:I\to K\in\cs_\sk$ we have the standard bijections \(\lambda^f\:\tens^{k\in K}\tens^{i\in f^{-1}k}X_i\to\tens^{i\in I}X_i\), exhibiting \(\KK\Vect\) as a colax strong symmetric monoidal category.
Let us prove that \(\lambda^f\in\Short_\LL\).
Let us decompose simultaneously a vector \(x\in\tens^{i\in I}X_i\) and its preimage \(x'=(\lambda^f)^{-1}(x)\in\tens^{k\in K}\tens^{i\in f^{-1}k}X_i\).
Namely, consider a decomposition \(x'=\sum_{m\in M}\tens^{k\in K}y^m_k\), where \(y^m_k\in\tens^{i\in f^{-1}k}X_i\) in its turn is decomposed as \(y^m_k=\sum_{n\in N^m_k}\tens^{i\in f^{-1}k}z^n_i\).
Hence,
\[ x' =\sum_{m\in M}\tens^{k\in K}\sum_{n\in N^m_k}\tens^{i\in f^{-1}k}z^n_i =\sum_{m\in M}\sum_{n\in\prod_{k\in K}N^m_k}\tens^{k\in K}\tens^{i\in f^{-1}k}z^n_i =\sum_{j\in J}\tens^{k\in K}\tens^{i\in f^{-1}k}x^j_i,
\]
where \(J=\coprod_{m\in M}\prod_{k\in K}N^m_k\) and \(x^j_i=z^n_i\) for \(j=(m,n)\in J\) with \(n=(n^m_k)_{k\in K}\).
If for some \((l^m_k)_{k\in K}^{m\in M}\) with \(\sum_{k\in K}l^m_k=l\) \(\forall m\in M\) and for all \(m\in M\) and \(k\in K\) \(\|y^m_k\|_{l^m_k}<\infty\), then \(\|x'\|_l<\infty\).
If for some \(m\in M\) and \(k\in K\) \(\|y^m_k\|_{l^m_k}=\infty\), then \(\sum_{m\in M}\prod_{k\in K}\|y^m_k\|_{l^m_k}=\infty\).
In both cases there exists a decomposition of $x'$ such that for an arbitrary positive $\eps$
\[ \sum_{m\in M}\prod_{k\in K}\sum_{n\in N^m_k}\prod_{i\in f^{-1}k}\|z^n_i \|_{\lambda^n_i} \le \|x'\|_l +\eps,
\]
where for all \(m\in M\), \(k\in K\) and \(n\in N^m_k\) \(\sum_{i\in f^{-1}k}\lambda^n_i=l^m_k\).

The relevant decomposition of $x$ is \(x=\sum_{j\in J}\tens^{i\in I}x^j_i\) accompanied with indices \(\ell^j_i=\lambda^n_i\), where \(j=(m,n)\in J\) with \(n=(n^m_k)_{k\in K}\).
Note that \(\sum_{i\in I}\ell^j_i=\sum_{k\in K}\sum_{i\in f^{-1}k}\lambda^n_i=\sum_{k\in K}l^m_k=l\).
So we have
\[ \|x\|_l \le \sum_{j\in J}\prod_{i\in I}\|x^j_i\|_{\ell^j_i} =\sum_{m\in M}\prod_{k\in K}\sum_{n\in N^m_k}\prod_{i\in f^{-1}k}
\|z^n_i \|_{\lambda^n_i} \le \|x'\|_l +\eps.
\]
Therefore, \(\|x\|_l\le\|x'\|_l\) and the map \(\lambda^f\) is short.

The collections \(\bigl((\lambda^f)^{-1}\bigr)_{f\:I\to K}\) satisfy conditions (i), (ii) (see \eqref{eq-lambda-fg-non-enriched}) since they equip \(\KK\Vect\) with the structure of unbiased symmetric monoidal category.
Thus, \(\Short_\LL\) is a colax symmetric monoidal category.
\end{proof}

\subsection{Completeness of the multicategory of short spaces}
\begin{proposition}
The product \(\prod_{i\in I}^{\Short_\LL}M_i\) of a family of short spaces \(\bigl((M_i,(\sS{_i}\|\cdot\|_l)_{l\in\LL})\bigr)_{i\in I}\) exists and consists of elements \(m=(m_i)_{i\in I}\in\prod_{i\in I}^{\KK\Vect}M_i\) such that for at least one $l\in\LL$ the value
	\[ \sS{_{\prod}}\|m\|_l =\sup_{i\in I}\sS{_i}\|m_i\|_l
	\]
	is finite.
	This formula defines seminorms $\sS{_{\prod}}\|\cdot\|_l$ for \(\prod_{i\in I}^{\Short_\LL}M_i\).
\end{proposition}

\begin{proof}
	There are embeddings of $\KK$\n-vector spaces
	\begin{multline*}
		\Short_\LL\bigl(N,\prod_{i\in I}^{\Short_\LL}M_i\bigr) \subset \KK\Vect\bigl(N,\prod_{i\in I}^{\Short_\LL}M_i\bigr) \subset \KK\Vect\bigl(N,\prod_{i\in I}^{\KK\Vect}M_i\bigr)
		\\
		\cong \prod_{i\in I}^{\KK\Vect}\KK\Vect(N,M_i) \supset \prod_{i\in I}^{\KK\Vect}\Short_\LL(N,M_i).
	\end{multline*}
	Let us consider an arbitrary \(f\:N\to\prod_{i\in I}^{\KK\Vect}M_i\in\KK\Vect\) and the corresponding family \((f_i\:N\to M_i\in\KK\Vect)_{i\in I}\).
	We have to prove that $f$ is short iff $f_i$ is short for all $i\in I$.
	
	Assume that \(f\:N\to\prod_{i\in I}^{\Short_\LL}M_i\in\Short_\LL\).
	It means that for all $n\in N$ and for all $l\in\LL$
	\[ \sup_{i\in I}\sS{_i}\|f_i(n)\|_l =\sS{_{\prod}}\|f(n)\|_l \le \sS{_N}\|n\|_l.
	\]
	Therefore, \(\sS{_i}\|f_i(n)\|_l\le\sS{_N}\|n\|_l\) for all $i\in I$, for all $n\in N$ and for all $l\in\LL$.
	Hence, \(f_i\in\Short_\LL\).
	
	Assume now that \(f_i\in\Short_\LL\) for all $i\in I$.
	Thus, \(\sS{_i}\|f_i(n)\|_l\le\sS{_N}\|n\|_l\) for all $i\in I$, for all $n\in N$ and for all $l\in\LL$.
	Therefore,
	\begin{equation}
		\sS{_{\prod}}\|f(n)\|_l =\sup_{i\in I}\sS{_i}\|f_i(n)\|_l \le \sS{_N}\|n\|_l.
		\label{eq-fnl-nl}
	\end{equation}
	For any $n\in N$ there is $l\in\LL$ such that \(\sS{_{\prod}}\|f(n)\|_l \) is finite.
	That is, \(f(N)\subset\prod_{i\in I}^{\Short_\LL}M_i\).
	Inequality \eqref{eq-fnl-nl} shows that \(f\:N\to\prod_{i\in I}^{\Short_\LL}M_i\) is short.
\end{proof}

\begin{proposition}\label{pro-multicategory-Short-has-small-products}
The multicategory $\mcShort_\LL$ has small products (see \defref{def-multicategory-small-products}).
\end{proposition}

\begin{proof}
Given a family \(\bigl(f_i\:(X_j)_{j\in\mb n}\to V_i\in\mcv\bigr)_{i\in I}\) there is a unique morphism \(f\:(X_j)_{j\in\mb n}\to\prod^{\KK\Vect}_{i\in I}V_i\) such that for all \(i\in I\)
\[ f_i =\bigl[ (X_j)_{j\in\mb n} \rto f \prod^{\KK\Vect}_{i\in I}V_i \rTTo^{\pr_i} V_i \bigr],
\]
since the multicategory \(\wh{\KK\Vect}\) is representable.
For any $n$\n-tuple of elements \((x_j\in X_j)_{j\in\mb n}\) there is an $n$\n-tuple of elements \((l_j\in\LL)_{j\in\mb n}\) such that \(\sS{_{X_j}}\|x_j\|_{l_j}<\infty\).
Then
\begin{multline*}
\sS{_{\prod}}\|f(x_1,x_2,\dots,x_n)\|_{l_1+\dots+l_n} =\sS{_{\prod}}\|(f(x_1,x_2,\dots,x_n))_{i\in I}\|_{l_1+\dots+l_n} 
\\
=\sup_{i\in I} \sS{_{V_i}}\|f(x_1,x_2,\dots,x_n)\|_{l_1+\dots+l_n} \le \sS{_{X_1}}\|x_1\|_{l_1} \cdot\ldots\cdot \sS{_{X_n}}\|x_n\|_{l_n} <\infty
\end{multline*}
Therefore, $f$ takes values in \(\prod_{i\in I}^{\Short_\LL}V_i\).
Moreover, \(f\:(X_j)_{j\in\mb n}\to\prod^{\Short_\LL}_{i\in I}V_i\) is short.
\end{proof}

\begin{proposition}\label{pro-kernel}
A morphism \(h\:B\to A\in\snS\) has a kernel (equalizer of $h$ and 0) in $\snS$, which coincides with the kernel \(K=\Ker h\) in $\KK\Vect$.
The subspace \(K\subset B\) inherits the seminorm from $B$.
\end{proposition}

\begin{proof}
	In $\KK\Vect$ the kernel \((K=\Ker h,i=\ker h)\) exists and satisfies the property which is based on the diagram
	\begin{diagram}
		K &\rMono^i &B &\pile{\rTTo^h \\ \rTTo_0} &A
		\\
		&\luTTo<n &\uTTo>j
		\\
		&&D
	\end{diagram}
	Namely,
	\[ \forall\, j\quad j\centerdot h=0 \implies (\exists!\,n\;\:\; n\centerdot i=j).
	\]
	We have to prove the same property in $\snS$.
	First of all, $i$ is short.
	Hence, if $n$ is short, then \(j=n\centerdot i\) is short as well.
	If $j$ is short, then for all $d\in D$
	\[ \sS{_K}\|nd\| =\sS{_B}\|ind\| =\sS{_B}\|jd\| \le \sS{_D}\|d\|.
	\]
	Hence, $n$ is short.
\end{proof}

\begin{corollary}
	By \cite[Corollary~V.2.2]{MacLane} the category $\snS$ (and more generally $\Short_\LL$) is complete.
	The limit of a diagram \(I\to\Short_\LL\), \(i\mapsto\bigl(M_i,(\sS{_i}\|\cdot\|_l)_{l\in\LL}\bigr)\) is
	\begin{equation}
		\lim_{i\in I}(M_i\in\Short_\LL) =\bigl(\prod_{i\in\Ob I}^{\Short_\LL}M_i\bigr)\bigcap\lim_{i\in I}(M_i\in\KK\Vect),
		\label{eq-lim-prodClim}
	\end{equation}
	where the both $\KK$\n-vector spaces are viewed as subspaces of \(\prod_{i\in\Ob I}^{\KK\Vect}M_i\).
	The seminorms on the subspace \(\lim_{i\in I}(M_i\in\Short_\LL)\subset\prod_{i\in\Ob I}^{\Short_\LL}M_i\) are induced from the latter short space.
\end{corollary}

\begin{proof}
	According to \cite[Theorem~V.2.2]{MacLane} the rows of diagram
	\begin{diagram}[w=5em,h=2.9em]
		\lim_{i\in I}(M_i\in\Short_\LL) &\rMono &\prod_{i\in\Ob I}^{\Short_\LL}M_i &\pile{\rTTo^f \\ \rTTo_g} &\prod_{u\in\Mor I}^{\Short_\LL}M_{\tgt u}
		\\
		\dMono &&\dMono &&\dMono
		\\
		\lim_{i\in I}(M_i\in\KK\Vect) &\rMono &\prod_{i\in\Ob I}^{\KK\Vect}M_i &\pile{\rTTo^f \\ \rTTo_g} &\prod_{u\in\Mor I}^{\KK\Vect}M_{\tgt u}
	\end{diagram}
	(where \(\pr_u\circ f=\pr_{\tgt u}\), \(\pr_u\circ g=M_u\circ\pr_{\src u}\)) are equalizers.
	The both squares on the right (one with upper arrows and another with lower arrows) commute.
	One easily deduces \eqref{eq-lim-prodClim}.
\end{proof}

\begin{corollary}\label{cor-Short-complete}
The multicategory \(\mcShort_\LL\) is complete.
\end{corollary}

\begin{proof}
Given a functor \(I\to\Short_\LL\) and a family of morphisms \(h_i\:(X_j)_{j\in\mb n}\to M_i\in\mcShort_\LL\), \(i\in\Ob I\), such that
\[ h_k =\bigl[ (X_j)_{j\in\mb n} \rTTo^{h_i} M_i \to M_k \bigr]
\]
for each \(i\to k\in I\), we see that the map \(h=(h_i)\:(X_j)_{j\in\mb n}\to\prod^{\KK\Vect}_{i\in I}M_i\) takes values in each of the subspaces \(\prod^{\mcShort_\LL}_{i\in I}M_i\) (by \propref{pro-multicategory-Short-has-small-products}) and \(\lim_{i\in I}(M_i\in\KK\Vect)\).
Hence, in their intersection \(\lim_{i\in I}(M_i\in\Short_\LL)\).
Since \(h=(h_i)\:(X_j)_{j\in\mb n}\to\prod^{\Short_\LL}_{i\in I}M_i\in\mcShort_\LL\) (again by \propref{pro-multicategory-Short-has-small-products}) we have \(h\:(X_j)_{j\in\mb n}\to\lim_{i\in I}(M_i\in\Short_\LL)\in\mcShort_\LL\).
\end{proof}

\begin{exercise}
The category $\Short_\LL$ has small coproducts.
The categorical coproduct of a family \((X_i)_{i\in I}\) of short spaces is given by the subspace \(\coprod_{i\in I}X_i\subset\prod_{i\in I}X_i\), consisting of elements \((x_i)_{i\in I}\) with $x_i=0$ except a finite number of them, equipped with seminorms
\[ \|(x_i)_{i\in I}\|_l =\sum_{i\in I}\|x_i\|_l.
\]
\end{exercise}

\appendix
\section{Symmetric groups and symmetric multicategories}
\subsection{Action of symmetric groups on a symmetric multicategory}
Let \(\sigma\:J\to K\in\cs_\sk\) be a bijection.
Let \((Y_j)_{j\in J}\), \((Z_k)_{k\in K}\), $W$ be (families of) objects of a symmetric multicategory $\mcv$ such that \(Z_k=Y_{\sigma^{-1}k}\).
Similarly to \cite[Lemma~A.2.2]{math.CT/0305049} define a map
\begin{multline*}
r_\sigma =\bigl\{ \mcv\bigl((Y_{\sigma^{-1}k})_{k\in K};W\bigr) \rTTo^{(\dot1_{Y_{\sigma^{-1}k}})_{k\in K}\times1} \bigl[\prod_{k\in K}\mcV(Y_{\sigma^{-1}k};Z_k)\bigr] \times \mcv\bigl((Y_{\sigma^{-1}k})_{k\in K};W\bigr)
\\
\rto{\mu_\sigma} \mcv\bigl((Y_j)_{j\in J};W\bigr) \bigr\}.
\end{multline*}

The following statement is implied by the proof of \cite[Theorem~A.2.4]{math.CT/0305049}.

\begin{proposition}\label{pro-mu-psi-r-sigma-mu-phi}
Let, furthermore, \(\psi=(I\rto\phi J\rto\sigma K)\in\cs_\sk\) and \((X_i)_{i\in I}\) be a family of objects of $\mcv$.
Then
\begin{multline}
\mu_\psi =\bigl\{ \bigl[ \prod_{k\in K}\mcV\bigl((X_i)_{i\in\psi^{-1}k};Y_{\sigma^{-1}k}\bigr) \bigr] \times\mcv\bigl((Y_{\sigma^{-1}k})_{k\in K};W\bigr)
\\
\rTTo^{\prod_{\sigma^{-1}}\times r_\sigma} \bigl[ \prod_{j\in J}\mcV\bigl((X_i)_{i\in\phi^{-1}(j)};Y_j\bigr) \bigr] \times\mcv\bigl((Y_j)_{j\in J};W\bigr) \rTTo^{\mu_\phi} \mcv\bigl((X_i)_{i\in I};W\bigr) \bigr\}.
\label{eq-mu-psi-P-sigma-1-r-sigma-mu-phi}
\end{multline}
\end{proposition}

\begin{proof}
Applying the associativity property from \figref{dia-assoc-mu-multi} for maps \(I \rto\phi J \rto\sigma K\) we get the sought equation \vpageref{dia-Action-symmetric-multicategory}.
\begin{figure}
	\begin{center}
		\boldmath
		\resizebox{!}{.934\texthigh}{\rotatebox{90}{%
				\begin{diagram}[nobalance,inline,height=2.4em,thick]
\HmeetV &\lLine &\phantom a && \hspace*{-8.5em} \bigl[ \prod_{k\in K}\mcV\bigl((X_i)_{i\in\psi^{-1}k};Y_{\sigma^{-1}k}\bigr) \bigr] \times\mcv\bigl((Y_{\sigma^{-1}k})_{k\in K};W\bigr) &\rLine &\HmeetV
\\
&&&&\dBTo<\cong>{\prod_{\sigma^{-1}}\times1}
\\
&&&& \hspace*{-7.5em} \bigl[ \prod_{j\in J}\mcV\bigl((X_i)_{i\in\phi^{-1}(j)};Y_j\bigr) \bigr] \times\mcv\bigl((Y_{\sigma^{-1}k})_{k\in K};W\bigr)
\\
&&&&\dBTo>{1\times(\dot1_{Z_k})_{k\in K}\times1}
\\
\dLine>1 &&&& \hspace*{-15em} \bigl[ \prod_{j\in J} \mcv\bigl((X_i)_{i\in\phi^{-1}j};Y_j\bigr)\bigr] \times \bigl[ \prod_{k\in K} 	\mcv\bigl(Y_{\sigma^{-1}k};Z_k\bigr) \bigr] \times \mcv\bigl((Y_{\sigma^{-1}k})_{k\in K};W\bigr)
					\\
&&& \ldBTo_\cong &&&\dLine<{\prod_{\sigma^{-1}}\times r_\sigma}
					\\
&&\bigl[ \prod_{k\in K} \bigl( \mcv\bigl((X_i)_{i\in\phi^{-1}\sigma^{-1}k};Y_{\sigma^{-1}k}\bigr) \times \mcv\bigl(Y_{\sigma^{-1}k};Z_k\bigr) \bigr) \bigr] \times \mcv\bigl((Y_{\sigma^{-1}k})_{k\in K};W\bigr) \hspace*{-4em}
					&& \dBTo>{1\times\mu_\sigma}
					\\
&&&=
					\\
&&\dBTo<{(\prod_{k\in K}\mu_{\triangledown\:\phi^{-1}\sigma^{-1}k\to\{\sigma^{-1}k\}})\times1} && \hspace*{-5em} 
\bigl[ \prod_{j\in J} \mcv\bigl((X_i)_{i\in\phi^{-1}j};Y_j\bigr) \bigr] \times \mcv\bigl((Y_j)_{j\in J};W\bigr) &\lBTo &\HmeetV
					\\
					\\
\HmeetV & \rBTo &\bigl[ \prod_{k\in K} \mcv\bigl((X_i)_{i\in\psi^{-1}k};Z_k\bigr) \bigr] \times \mcv\bigl((Z_k)_{k\in K};W\bigr) 
&& \dBTo>{\mu_\phi} 
					\\
&&& \rdBTo^{\mu_\psi} & 
					\\
&&&& \mcv\bigl((X_i)_{i\in I};W\bigr)
				\end{diagram}
		}}
	\end{center}
	\caption{Action of symmetric groups on a symmetric multicategory
		\label{dia-Action-symmetric-multicategory}}
\end{figure}
\end{proof}

\begin{corollary}\label{cor-Action-symmetric-groups-symmetric-multicategory}
Assume that both $\phi$ and \(\sigma\) are bijections from $\cs_\sk$, \(\psi=(I\rto\phi J\rto\sigma K)\).
Then\index{action of a symmetric group on a symmetric multicategory}
\[ r_\psi =\bigl[ \mcv\bigl((Y_{\sigma^{-1}k})_{k\in K};W\bigr) \rto{r_\sigma} \mcv\bigl((Y_j)_{j\in J};W\bigr) \rto{r_\phi} \mcv\bigl((Y_{\phi i})_{i\in I};W\bigr) \bigr].
\]
\end{corollary}

\begin{proof}
Consider \(X_i=Y_{\phi i}\), hence, \(Y_j=X_{\phi^{-1}j}\).
Rewrite \eqref{eq-mu-psi-P-sigma-1-r-sigma-mu-phi} as
\begin{multline}
\mu_\psi =\bigl\{ \bigl[ \prod_{k\in K}\mcV\bigl(X_{\psi^{-1}k};Y_{\sigma^{-1}k}\bigr) \bigr] \times \mcv\bigl((Y_{\sigma^{-1}k})_{k\in K};W\bigr)
	\\
\rTTo^{\prod_{\sigma^{-1}}\times r_\sigma} \bigl[ \prod_{j\in J}\mcV\bigl(X_{\phi^{-1}j};Y_j\bigr) \bigr] \times  \mcv\bigl((Y_j)_{j\in J};W\bigr) \rTTo^{\mu_\phi} \mcv\bigl((X_i)_{i\in I};W\bigr) \bigr\}.
\label{eq-mu-psi-sigma-r-sigma-mu-phi}
\end{multline}
Substitute \((1_{Y_{\sigma^{-1}k}})_{k\in K}\) into the first factor.
We get from the left hand side of \eqref{eq-mu-psi-sigma-r-sigma-mu-phi}
\begin{multline*}
\bigl\{ \mcv\bigl((X_{\psi^{-1}k})_{k\in K};W\bigr) \rTTo^{(\dot1_{X_{\psi^{-1}k}})_{k\in K}\times1} \bigl[\prod_{k\in K}\mcV(X_{\psi^{-1}k};X_{\psi^{-1}k})\bigr] \times \mcv\bigl((X_{\psi^{-1}k})_{k\in K};W\bigr)
\\
\rto{\mu_\psi} \mcv\bigl((X_i)_{i\in I};W\bigr) \bigr\} =r_\psi.
\end{multline*}
From the right hand side of \eqref{eq-mu-psi-sigma-r-sigma-mu-phi} we get
\begin{multline*}
\bigl\{ \mcv\bigl((Y_{\sigma^{-1}k})_{k\in K};W\bigr) \rTTo^{(\dot1_{Y_{\sigma^{-1}k}})_{k\in K}\times1} \bigl[\prod_{k\in K}\mcV(X_{\psi^{-1}k};Y_{\sigma^{-1}k})\bigr] \times \mcv\bigl((Y_{\sigma^{-1}k})_{k\in K};W\bigr)
	\\
\hfill \rTTo^{\prod_{\sigma^{-1}}\times r_\sigma} \bigl[\prod_{j\in J}\mcV(X_{\phi^{-1}j};Y_j)\bigr] \times \mcv\bigl((Y_j)_{j\in J};W\bigr) \rto{\mu_\phi} \mcv\bigl((X_i)_{i\in I};W\bigr) \bigr\} \hskip\multlinegap
\\
\hskip\multlinegap =\bigl\{ \mcv\bigl((Y_{\sigma^{-1}k})_{k\in K};W\bigr) \rto{r_\sigma} \mcv\bigl((Y_j)_{j\in J};W\bigr) \hfill
\\
\hfill \rTTo^{(\dot1_{Y_j})_{j\in J}\times1} \bigl[\prod_{j\in J}\mcV(X_{\phi^{-1}j};Y_j)\bigr] \times \mcv\bigl((Y_j)_{j\in J};W\bigr) \rto{\mu_\phi} \mcv\bigl((X_i)_{i\in I};W\bigr) \bigr\} \hskip\multlinegap
\\
=\bigl\{ \mcv\bigl((Y_{\sigma^{-1}k})_{k\in K};W\bigr) \rto{r_\sigma} \mcv\bigl((Y_j)_{j\in J};W\bigr) \rto{r_\phi} \mcv\bigl((X_i)_{i\in I};W\bigr) \bigr\}.
\end{multline*}
Therefore, \(r_{\phi\centerdot\sigma}=r_\sigma\centerdot r_\phi\).
\end{proof}

The second identity axiom implies that \(r_{\id}=\id\).
Thus, we have an action of a symmetric group on the set of homomorphism sets of a symmetric multicategory $\mcv$.
Often this action is included in the definition of a symmetric multicategory, which we do not do.

\begin{example}
Assume that $\cv$ is a complete closed symmetric monoidal category with \(\tens^{\mb1}=\Id\).
For \(\mcv=\wh\cv\) (see \cite[Proposition~3.22]{BesLyuMan-book}) we get \(r_\sigma=\cv(\lambda^\sigma,W)\:\cv(\tens^{k\in K}Y_{\sigma^{-1}k},W)\to\cv(\tens^{j\in J}Y_j,W)\), 
where \(\lambda^\sigma\:\tens^{j\in J}Y_j\to\tens^{k\in K}Y_{\sigma^{-1}k}\) is the action of symmetric group on tensor products via symmetries.
\end{example}

The following equivariance property seems to be explicitly stated in the literature for the first time, although it should be implied by the proof of \cite[Theorem~A.2.4]{math.CT/0305049}.

\begin{proposition}\label{pro-equivariance-property}
Let the square in $S_\sk$, where vertical arrows are bijections,
\begin{diagram}
I &\rTTo^\phi &J
\\
\dTTo<\pi>\cong &&\dTTo<\cong>\sigma
\\
L &\rTTo^\psi &K
\end{diagram}
commute.
Then there is the equivariance property
\begin{multline}
\bigl\{ \bigl[ \prod_{k\in K}\mcV\bigl((X_{\pi^{-1}l})_{l\in\psi^{-1}k};Y_{\sigma^{-1}k}\bigr) \bigr] \times \mcv\bigl((Y_{\sigma^{-1}k})_{k\in K};W\bigr)
\\
\rTTo^{\prod_{\sigma^{-1}}\times1} \bigl[ \prod_{j\in J}\mcV\bigl((X_{\pi^{-1}l})_{l\in\pi\phi^{-1}j};Y_j\bigr) \bigr] \times \mcv\bigl((Y_{\sigma^{-1}k})_{k\in K};W\bigr)
\\
\hfill \rTTo^{\prod_{j\in J}r_{\varpi_j}\times r_\sigma} \bigl[ \prod_{j\in J}\mcV\bigl((X_i)_{i\in\phi^{-1}j};Y_j\bigr) \bigr] \times\mcv\bigl((Y_j)_{j\in J};W\bigr) \rto{\mu_\phi} \mcv\bigl((X_i)_{i\in I};W\bigr) \bigr\} \hskip\multlinegap
\\
\hskip\multlinegap =\bigl\{ \bigl[ \prod_{k\in K}\mcV\bigl((X_{\pi^{-1}l})_{l\in\psi^{-1}k};Y_{\sigma^{-1}k}\bigr) \bigr] \times \mcv\bigl((Y_{\sigma^{-1}k})_{k\in K};W\bigr) \hfill
\\
\hfill \rTTo^{\prod_{k\in K}r_{\pi_k}\times1} \bigl[ \prod_{k\in K}\mcV\bigl((X_i)_{i\in\pi^{-1}\psi^{-1}k};Y_{\sigma^{-1}k}\bigr) \bigr] \times\mcv\bigl((Y_{\sigma^{-1}k})_{k\in K};W\bigr) \rTTo^{\mu_{\pi\centerdot\psi}} \mcv\bigl((X_i)_{i\in I};W\bigr) \bigr\} \hskip\multlinegap
\\
=\bigl\{ \bigl[ \prod_{k\in K}\mcV\bigl((X_{\pi^{-1}l})_{l\in\psi^{-1}k};Y_{\sigma^{-1}k}\bigr) \bigr] \times \mcv\bigl((Y_{\sigma^{-1}k})_{k\in K};W\bigr)
\rto{\mu_\psi} \mcv\bigl((X_{\pi^{-1}l})_{l\in L};W\bigr)
\\
\rto{r_\pi} \mcv\bigl((X_i)_{i\in I};W\bigr) \bigr\}.
\label{eq-mur-miao}
\end{multline}
Here \(\varpi_j=\pi|\:\phi^{-1}j\to\pi\phi^{-1}j=\psi^{-1}\sigma j\) and \(\pi_k=\varpi_{\sigma^{-1}k}=\pi|\:\pi^{-1}\psi^{-1}k\to\psi^{-1}k\) are bijections.
\end{proposition}

\begin{proof}
Denote \(Z_k=Y_{\sigma^{-1}k}\).
Applying the associativity property from \figref{dia-assoc-mu-multi} for maps \(I \rto\phi J \rto\sigma K\) 
we get the proof of the first equation from \eqref{eq-mur-miao} \vpageref{dia-equivariance-symmetric-multicategory}.
	\begin{figure}
		\begin{center}
			\boldmath
			\resizebox{!}{.934\texthigh}{\rotatebox{90}{%
					\begin{diagram}[nobalance,inline,height=2.4em,thick]
\HmeetV &\lLine &\phantom a && \hspace*{-8.5em} \bigl[ \prod_{k\in K}\mcV\bigl((X_{\pi^{-1}l})_{l\in\psi^{-1}k};Y_{\sigma^{-1}k}\bigr) \bigr] \times\mcv\bigl((Y_{\sigma^{-1}k})_{k\in K};W\bigr) &\rLine &\HmeetV
						\\
						&&&&\dBTo<\cong>{\prod_{\sigma^{-1}}\times1}
						\\
&&&& \hspace*{-7.5em} \bigl[ \prod_{j\in J}\mcV\bigl(X_{\pi^{-1}l})_{l\in\pi\phi^{-1}j};Y_j\bigr) \bigr] \times\mcv\bigl((Y_{\sigma^{-1}k})_{k\in K};W\bigr)
						\\
\dLine>{\prod_{j\in J}r_{\pi_k}\times1} &&&&\dBTo>{\prod_{j\in J}r_{\varpi_j}\times(\dot1_{Z_k})_{k\in K}\times1}
						\\
&&&& \hspace*{-15em} \bigl[ \prod_{j\in J} \mcv\bigl((X_i)_{i\in\phi^{-1}j};Y_j\bigr)\bigr] \times \bigl[ \prod_{k\in K} 	\mcv\bigl(Y_{\sigma^{-1}k};Z_k\bigr) \bigr] \times \mcv\bigl((Y_{\sigma^{-1}k})_{k\in K};W\bigr)
						\\
						&&& \ldBTo_\cong &&&\dLine<{(\prod_{\sigma^{-1}}\centerdot\prod_{j\in J}r_{\varpi_j})\times r_\sigma}
						\\
&&\bigl[ \prod_{k\in K} \bigl( \mcv\bigl((X_i)_{i\in\phi^{-1}\sigma^{-1}k};Y_{\sigma^{-1}k}\bigr) \times \mcv\bigl(Y_{\sigma^{-1}k};Z_k\bigr) \bigr) \bigr] \times \mcv\bigl((Y_{\sigma^{-1}k})_{k\in K};W\bigr) \hspace*{-4em}
						&& \dBTo>{1\times\mu_\sigma}
						\\
						&&&=
						\\
	&&\dBTo<{(\prod_{k\in K}\mu_{\triangledown\:\phi^{-1}\sigma^{-1}k\to\{\sigma^{-1}k\}})\times1} && \hspace*{-5em} 
	\bigl[ \prod_{j\in J} \mcv\bigl((X_i)_{i\in\phi^{-1}j};Y_j\bigr) \bigr] \times \mcv\bigl((Y_j)_{j\in J};W\bigr) &\lBTo &\HmeetV
						\\
						\\
\HmeetV & \rBTo &\bigl[ \prod_{k\in K} \mcv\bigl((X_i)_{i\in\pi^{-1}\psi^{-1}k};Z_k\bigr) \bigr] \times \mcv\bigl((Z_k)_{k\in K};W\bigr) 
						&& \dBTo>{\mu_\phi} 
						\\
						&&& \rdBTo_{\mu_{\pi\centerdot\psi}} & 
						\\
						&&&& \mcv\bigl((X_i)_{i\in I};W\bigr)
					\end{diagram}
			}}
		\end{center}
		\caption{Equivariance of action of symmetric groups on a symmetric multicategory
			\label{dia-equivariance-symmetric-multicategory}}
	\end{figure}

In order to prove the second equation from \eqref{eq-mur-miao} we substitute into the former expression the definition of $r$:
\begin{multline*}
\bigl[ \prod_{k\in K}\mcV\bigl((X_{\pi^{-1}l})_{l\in\psi^{-1}k};Z_k\bigr) \bigr] \times \mcv\bigl((Z_k)_{k\in K};W\bigr)
\rTTo^{\prod_{k\in K}[(\dot1_{X_{\pi^{-1}l}})_{l\in\psi^{-1}k}\times1]\times1}
\\
\prod_{k\in K} \bigl[ \prod_{l\in\psi^{-1}k}\mcV(X_{\pi^{-1}l};X_{\pi^{-1}l}) \times \mcV\bigl((X_{\pi^{-1}l})_{l\in\psi^{-1}k};Z_k\bigr) \bigr] \times \mcv\bigl((Z_k)_{k\in K};W\bigr)
\\
\rTTo^{\prod_{k\in K}\mu_{\pi_k}\times1} \bigl[ \prod_{k\in K}\mcV\bigl((X_i)_{i\in\pi^{-1}\psi^{-1}k};Z_k\bigr) \bigr] \times\mcv\bigl((Z_k)_{k\in K};W\bigr) \rTTo^{\mu_{\pi\centerdot\psi}} \mcv\bigl((X_i)_{i\in I};W\bigr).
\end{multline*}
Transforming this with the help of the associativity property from \figref{dia-assoc-mu-multi} for maps \(I \rto\pi L \rto\psi K\) we get
\begin{multline*}
\bigl\{ \bigl[ \prod_{k\in K}\mcV\bigl((X_{\pi^{-1}l})_{l\in\psi^{-1}k};Z_k\bigr) \bigr] \times \mcv\bigl((Z_k)_{k\in K};W\bigr)
\rTTo^{(\dot1_{X_{\pi^{-1}l}})_{l\in L}\times1\times1}
\\
\bigl[ \prod_{l\in L} \mcV(X_{\pi^{-1}l};X_{\pi^{-1}l}) \bigr] \times \bigl[ \prod_{k\in K} \mcV\bigl((X_{\pi^{-1}l})_{l\in\psi^{-1}k};Z_k\bigr) \bigr] \times \mcv\bigl((Z_k)_{k\in K};W\bigr)
	\\
\hfill \rTTo^{1\times\mu_\psi} \bigl[ \prod_{l\in L} \mcV(X_{\pi^{-1}l};X_{\pi^{-1}l}) \bigr] \times\mcv\bigl((X_{\pi^{-1}l})_{l\in L};W\bigr) \rTTo^{\mu_\pi} \mcv\bigl((X_i)_{i\in I};W\bigr) \bigr\} \hskip\multlinegap
\\
\hskip\multlinegap =\bigl\{ \bigl[ \prod_{k\in K}\mcV\bigl((X_{\pi^{-1}l})_{l\in\psi^{-1}k};Z_k\bigr) \bigr] \times \mcv\bigl((Z_k)_{k\in K};W\bigr) \rto{\mu_\psi} \mcv\bigl((X_{\pi^{-1}l})_{l\in L};W\bigr) \hfill
\\
\rTTo^{(\dot1_{X_{\pi^{-1}l}})_{l\in L}\times1} \bigl[ \prod_{l\in L} \mcV(X_{\pi^{-1}l};X_{\pi^{-1}l}) \bigr] \times \mcv\bigl((X_{\pi^{-1}l})_{l\in L};W\bigr) \rto{\mu_\pi} \mcv\bigl((X_i)_{i\in I};W\bigr) \bigr\}.
\end{multline*}
This is the last expression from \eqref{eq-mur-miao} with expanded $r_\pi$.
\end{proof}


\ifcase\value{whichClass}
{Institute of Mathematics, NAS Ukraine},
{3 Tereshchenkivska st.}, {Kyiv}, {01024}, {Ukraine}
\fi

\end{document}